%% file: companion.tex
\documentclass{amsart}[11pt]

\usepackage{etex}
\usepackage[lmargin=1in,rmargin=1in,tmargin=1in,bmargin=1in]{geometry}
\usepackage{amsmath,amsthm,amsfonts,amssymb,verbatim}
\usepackage{graphicx}
\usepackage{tabularx}
\usepackage{overpic}
\usepackage{tikz}\usetikzlibrary{decorations.markings}
\usetikzlibrary{arrows}
\usepackage{rotating}
\usepackage{hyperref}
\usepackage{framed}
\usepackage{xcolor}
\usepackage{placeins}
\usepackage{booktabs}
\usepackage{pinlabel}
 \usepackage{picins}
\usepackage{color}
\usepackage{multirow}
\usepackage{overpic}
\usepackage{colortbl}
\usepackage{geometry}
\usepackage{slashed}
\usepackage[font=footnotesize,labelfont=bf]{caption}
\usepackage{arydshln}
\usepackage{blkarray}

\usepackage{wrapfig}

\DeclareMathAlphabet{\mathpzc}{OT1}{pzc}{m}{it}

\def\co{\colon\thinspace}

\newcommand{\Z}{\mathbb{Z}}
\newcommand{\Q}{\mathbb{Q}}
\newcommand{\R}{\mathbb{R}}
\newcommand{\F}{\mathbb{F}}

\newcommand{\sA}{\mathcal{A}}
\newcommand{\sB}{\mathcal{B}}
\newcommand{\sC}{\mathcal{C}}
\newcommand{\sI}{\mathcal{I}}
\newcommand{\sL}{\mathcal{L}}

\newcommand{\sT}{\mathcal{T}}

\newcommand{\sM}{\mathcal{M}}

\DeclareMathOperator{\Hom}{Hom}

\newcommand{\bu}{\bullet}
\newcommand{\ci}{\circ}

\newcommand{\Tp}{T} % a punctured torus. 
\newcommand{\tildeT}{\widetilde{T}}
\newcommand{\barT}{\overline{T}}

\newcommand{\spinc}{\operatorname{Spin}^c}

\newcommand{\Ztwo}{\Z/2\Z}

\newcommand{\gr}{\operatorname{gr}}
\newcommand{\hol}{\operatorname{hol}}

\newcommand{\cu}{\boldsymbol{c}_{\alpha,\beta}}
\newcommand{\bg}[1]{\boldsymbol{\gamma_{#1}}}

\newcommand{\bargamma}{\overline{\gamma}}

\newcommand{\tracks}{{\boldsymbol\vartheta}}

\newcommand{\im}{\operatorname{im}}

\newcommand{\HFhat}{\widehat{\mathit{HF}}}
\newcommand{\HFplus}{\mathit{HF}^{+}}
\newcommand{\HFminus}{\mathit{HF}^{-}}
\newcommand{\CFminus}{\mathit{CF}^-}
\newcommand{\CFplus}{\mathit{CF}^+}
\newcommand{\CFhat}{\widehat{\mathit{CF}}}
\newcommand{\CFKminus}{\mathit{CFK}^-}
\newcommand{\HFKminus}{\mathit{HFK}^-}
\newcommand{\gCFKminus}{\mathit{g}\mathit{CFK}^-}
\newcommand{\CFKhat}{\widehat{\mathit{CFK}}}
\newcommand{\gCFKhat}{\mathit{g}\widehat{\mathit{CFK}}}
\newcommand{\HFKhat}{\widehat{\mathit{HFK}}}

\newcommand{\KH}{\mathit{KH}}

\newcommand{\CFD}{\widehat{\mathit{CFD}}}
\newcommand{\CFA}{\widehat{\mathit{CFA}}}

\newcommand{\CFAA}{\widehat{\mathit{CFAA}}}

\newcommand{\CFDA}{\widehat{\mathit{CFDA}}}
\newcommand{\halfid}{\widehat{\mathit{CFDA}}(\frac{\mathbb{I}}{2})}

\newcommand{\HFK}{\widehat{\mathit{HFK}}}

\newcommand{\SFH}{\mathit{SFH}}

\newcommand{\lp}{\boldsymbol\ell}

\newcommand{\HFa}{\mathit{HF}}

\newcommand{\Alg}{\mathcal{A}}
\newcommand{\Aop}{\mathcal{A}^{\operatorname{op}}}

\newcommand{\spin}{\mathfrak{s}}

\newcommand{\spinbar}{\overline{\spin}}
\newcommand{\spinhat}{\hat{\spin}}

\newcommand{\curves}[1]{\HFhat({#1})}
\newcommand{\liftcurves}[1]{\overline{\mathit{HF}}({#1})}
\newcommand{\tildecurves}[1]{\widetilde{ \boldsymbol \gamma}({#1})}
\newcommand{\barcurves}[1]{\HFhat({#1})}

\newcommand{\id}{\operatorname{id}}

\newcommand{\wtp}{\check{w}}
\newcommand{\wtm}{\hat{w}}

\newcommand{\grDmodtwo}{\gr^D_{\Z_2}}
\newcommand{\grAmodtwo}{\gr^A_{\Z_2}}
\newcommand{\grCFD}{\gr}
\newcommand{\grbox}{\gr^\boxtimes}
\newcommand{\x}{{\bf x}}

\newcommand{\bbA}{\mathbf{A}}
\newcommand{\bbE}{\mathbf{E}}
\newcommand{\FA}{F_\mathbf{A}}
\newcommand{\FE}{F_\mathbf{E}}
\newcommand{\bbI}{\mathbf{I}}
\newcommand{\FI}{F_\mathbf{I}}

\newcommand{\hor}
	{\raisebox{-6pt}{\includegraphics[scale=0.6]{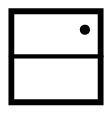}}}
\newcommand{\ver}
	{\raisebox{-6pt}{\includegraphics[scale=0.6]{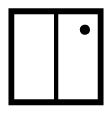}}}

\newtheorem{theorem}{Theorem}%[section]

\newtheorem{corollary}[theorem]{Corollary}
\newtheorem{proposition}[theorem]{Proposition}
\newtheorem{lemma}[theorem]{Lemma}

\newtheorem*{namedtheorem}{\theoremname}
\newcommand{\theoremname}{testing}

\theoremstyle{definition}
\newtheorem{definition}[theorem]{Definition}
\newtheorem{question}[theorem]{Question}

\newtheorem{remark}[theorem]{Remark}
\newtheorem{example}[theorem]{Example} 

\title[Heegaard Floer homology for manifolds with torus
boundary]{Heegaard Floer homology for manifolds with torus
boundary: properties and examples.}
\date{\today}
\author[Jonathan Hanselman]{Jonathan Hanselman}
\address {Department of Mathematics, Princeton University.\newline \it{E-mail address:} \tt{jh66@math.princeton.edu}}

\author[Jacob Rasmussen]{Jacob Rasmussen}
\address {Department of Pure Mathematics and Mathematical Statistics, University of Cambridge.\newline \it{E-mail address:} \tt{J.Rasmussen@dpmms.cam.ac.uk}}

\author[Liam Watson]{Liam Watson}
\thanks{JH was partially supported by NSF RTG grant DMS-1148490; JR was partially supported by EPSRC grant EP/M000648/1; LW was partially supported by a Marie Curie career integration grant, by a Canada Research Chair, and by an NSERC discovery/accelerator grant; JR and LW were Isaac Newton Institute program participants while part of this work was completed and acknowledge partial support from EPSRC grant EP/K032208/1; additionally, LW was partially supported by a grant from the Simons Foundation while at the Isaac Newton Institute}
\address {Department of Mathematics, University of British Columbia.\newline \it{E-mail address:} \tt{liam@math.ubc.ca}}

\begin{document}
\maketitle

\begin{abstract} 
This is a companion paper to earlier work of the authors \cite{HRW}, which interprets the Heegaard Floer homology for a manifold with torus boundary in terms of immersed curves in a punctured torus. We prove a variety of properties of this invariant, paying particular attention to its relation to knot Floer homology, the Thurston norm, and the Turaev torsion. We also give a geometric description of the gradings package from bordered Heegaard Floer homology and establish a symmetry under \(\spinc\) conjugation; this symmetry gives rise to genus one mutation invariance in Heegaard Floer homology for closed three-manifolds. Finally, we include more speculative discussions on relationships with Seiberg-Witten theory, Khovanov homology, and $\mathit{HF}^\pm$. Many examples are included. 
\end{abstract}

\input{sections/introduction}

\section{Immersed curves as invariants of manifolds with torus boundary}\label{sec:loop}\input{sections/setup}

\section{Gradings}\label{sec:grad}\input{sections/grading}

\section{Symmetries of the invariant}\label{sec:bim}\input{sections/elliptic}

\section{Knot Floer homology}\label{sec:knot_floer}\input{sections/knot_floer}

\section{Turaev torsion and Thurston norm}\label{sec:prop}\input{sections/properties}

\section{Seiberg Witten theory}\label{sec:seiberg-witten}\input{sections/seiberg-witten.tex}

\section{Khovanov homology and the two-fold branched cover}\label{sec:Kh} \input{sections/Kh.tex}

\section{Speculation on minus type invariants}\label{sec:minus}\input{sections/minus}

\bibliographystyle{plain}
\bibliography{references/bibliography}

\end{document}

%% file: sections/introduction.tex
% !TEX root = ../companion.tex
%introduction.tex

Bordered Heegaard Floer homology provides a toolkit for studying the Heegaard Floer homology of a three-manifold $Y$ decomposed along an essential surface. This theory was introduced and developed by Lipshitz, Ozsv\'ath, and Thurston \cite{LOT}, and has been studied in some detail in the case or essential tori as these are relevant to questions related to the JSJ decomposition of $Y$. In the authors' previous work \cite{HRW}, a geometric interpretation of the bordered Heegaard Floer homology of a three-manifold with torus boundary $M$ is established. In particular, we proposed:

\begin{definition} Let $M$ be a compact oriented three-manifold with torus boundary; fix a base point $z\in \partial M$. 
The invariant $\HFhat(M)$ is a collection of immersed curves in $\partial M\setminus z$ decorated with local systems, up to regular homotopy of the curves and isomorphism of the local systems.  
\end{definition}  
From now on, the phrase `manifold with torus boundary' will be used to refer to a manifold as in the definition; such manifolds will generally be denoted by \(M\), while closed three-manifolds will be denoted by \(Y\). 

We emphasize that $\HFhat(M)$ both determines and is determined by the bordered Floer homology of $M$; its existence is a consequence of a structure theorem for type D structures \cite[Theorem 5]{HRW}. This structure theorem is constructive, and a computer implementation of the algorithm has been given by Thouin \cite{Thouin}. The utility of this interpretation is illustrated by the following:

\begin{theorem}[{\cite[Theorem 2]{HRW}}]\label{thm:pairing}
Supose that $Y=M_0\cup_h M_1$ where the $M_i$ are manifolds with torus boundary and $h\co \partial M_1\to \partial M_0$ is an orientation reversing homeomorphism for which $h(z_1)=z_0$. Then \[\HFhat(Y) \cong \mathit{HF}(\boldsymbol\gamma_0,\boldsymbol\gamma_1)\]
where  $\mathit{HF}(\cdot,\cdot)$ is the (immersed) Lagrangian intersection Floer homology of $\boldsymbol\gamma_0=\HFhat(M_0)$ and $\boldsymbol\gamma_1=h\big(\HFhat(M_1)\big)$ computed in $\partial M_0\setminus z_0$. 
\end{theorem} 

Consistent with bordered theory, throughout this paper we will work with coefficients in the two-element field $\F$. 

\begin{figure}[t]
\labellist
  \pinlabel {$\mu_0$} at 220 321
  \pinlabel {$\lambda_0$} at 280 248
   \pinlabel {$\mu_0$} at 133 198
  \pinlabel {$\lambda_0$} at 223 101
   \pinlabel {$ \color{gray}{\mu_1}$} at 248 101
  \pinlabel {$\color{gray}{\mu_1+\lambda_1}$} at 341 198
         \endlabellist
\includegraphics[scale=.8]{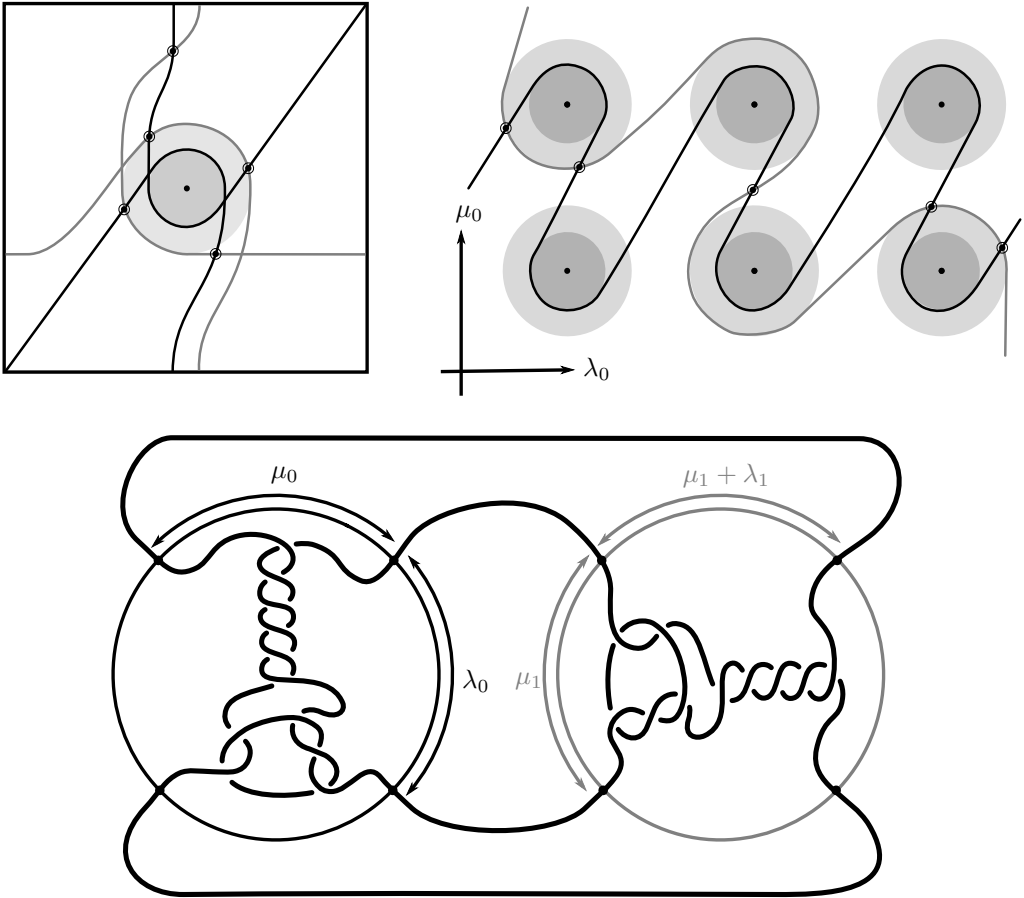}
\caption{Intersecting the curves $\HFhat(M_0)$ and $\color{gray}{h(\HFhat(M_1))}$ shown in $\partial M_0\setminus z_0$ (above, left) and in the cover $\R^2\setminus\Z^2$ determined by cutting along  $\mu_0$ and $\lambda_0$. The resulting $Y=M_0\cup_h M_1$ is the two fold branched cover of a link; the torus decomposition determines a tangle decomposition. By abuse of notation, the images of the slopes identified by the homeomorphism $h$ have been labeled by the corresponding slopes in the cover. For an explicit construction of this tangle see, for example, \cite{Montesinos1976,Watson2012}. Note that the Seifert structure (over the disk with cone points of orders 2 and 3) on the $M_i$ is manifested in each tangle.}
\label{fig:intro-figure}
\end{figure}

\subsection*{Executive summary by example: splicing trefoils} In practice, Theorem \ref{thm:pairing} reduces the computation of $\dim\HFhat(Y)$ to minimal intersection counts; various applications of this principle follow \cite{HRW}. To illustrate, we briefly review the setup with an example.

Let $M_i$ denote the complement of the right hand trefoil for $i=0,1$, with $(\mu_i,\lambda_i)$ the standard meridian-longitude pair. The closed three-manifold $Y=M_0\cup_h M_1$ obtained via the homeomorphism $h$ determined by $\lambda_0=h(\mu_1)$ and $\mu_0=h(\mu_1+\lambda_1)$ is an integer homology sphere. For readers familiar with bordered Floer homology, this setup is compatible with \[\HFhat(Y)\cong H_*\big(\CFA(M_0,\mu_0,\lambda_0)\boxtimes\CFD(M_1,\mu_1+\lambda_1,\mu_1)\big)\] where the triples $(M_0,\mu_0,\lambda_0)$ and $(M_1,\mu_1+\lambda_1,\mu_1)$ are bordered three-manifolds (or, trefoil exteriors with fixed bordered structures) \cite{LOT}. Following Theorem \ref{thm:pairing}, the dimension of the vector space $\HFhat(Y)$ can be found by the minimal intersection between $\HFhat(M_0)$ and $h(\HFhat(M_1))$, see Figure \ref{fig:intro-figure}, hence $\dim\HFhat(Y)=5$. 

This is actually as small as possible: in \cite[Theorem 8]{HRW} we show that if a three-manifolds $Y$ contains an essential separating torus then $\dim\HFhat(Y)\ge5$. In fact, it follows from our proof that up to orientation reversal, there is a unique prime toroidal integer homology sphere $Y$ with $\dim\HFhat(Y)=5$.
As a consequence, due to the spectral sequence from Khovanov homology to $\HFhat$ of the branched double cover, any link $L$ for which $\dim\widetilde{\mathit{Kh}}(L)<5$ cannot contain an essential Conway sphere \cite[Corollary 11]{HRW}. It would be interesting to know the smallest possible value of  $\dim\widetilde{\mathit{Kh}}(L)$ for links \(L\) containing an essential Conway sphere. The example  above can be realized as the two-fold branched cover of the knot $K=T_0\cup T_1$ shown in Figure \ref{fig:intro-figure}, for which we compute that  $\dim\widetilde{\mathit{Kh}}(K)=63$, but this is far from optimal; the Conway knot, for example, has   $\dim\widetilde{\mathit{Kh}}(K)=33.$

This companion paper has three basic goals. The first is to give an overview of the invariant together with some interesting examples. The second is to describe a range of its basic properties, some of which were briefly mentioned in \cite{HRW}; we give a more careful discussion here. The third is to discuss some more speculative connections between \(\HFhat(M)\) and other invariants, including Seiberg-Witten theory, Khovanov homology, and  \(\mathit{HF}^\pm\). Below, we give a more detailed outline of the contents of individual sections.
%First section: Review of the invariants, the trivial local system case and loop calculus, some examples including Floer solid tori, SW 

\subsection*{Section \ref{sec:loop}: A survey} We begin with a broad overview of the invariant $\HFhat(M)$ and review the setup for Theorem \ref{thm:pairing}. With the aim of providing an accessible survey of the material in \cite{HRW}, we largely focus on the special case where the local systems present are one dimensional, which (following \cite{HW}) we refer to as {\it loop type}.  In this case, studying $\curves M$ amounts to simply studying immersed curves in the punctured torus. In particular, in Section \ref{sub:loop-type} we give a greatly simplified construction of the curves $\curves M$ from $\CFD(M,\alpha,\beta)$, provided the latter is given in terms of a sufficiently nice basis. 

While the loop type condition may seem like a strong restriction, it is enjoyed by a wide range of examples and  is quite useful in practice. For instance, any $M$ admitting more than one L-space Dehn filling is loop type. In fact, the authors are currently unable to construct a single manifold \(M\) for which $\curves M$ is verifiably {\em not} loop type. While this is most likely due to a lack of sufficiently complicated examples, it seems that one does not loose much conceptually by restricting to this special case.

In this vein, the remainder of Section \ref{sec:loop} discusses some interesting examples of loop type manifolds. In Section \ref{sub:loop-calculus} we review some machinery for constructing manifolds with this property, including large classes of graph manifolds, which was first introduced by the first and last author in \cite{HW}. In \ref{sub:bundles} we explicitly compute the invariant $\curves{M_g}$, where $M_g$ is the product of $S^1$ and an orientable surface of genus $g$ with one boundary component. Combined with Theorem \ref{thm:pairing}, we recover a formula for $\dim\HFhat(S^1\times \Sigma_g)$ first proved by Ozsv\'ath-Szab\'o \cite[Theorem 9.3]{OSz2004-knot} and Jabuka-Mark \cite[Theorem 4.2]{JM2008}.

\begin{theorem}\label{crl:surface-bundle}
For $g\ge 0$, the total dimension of $\HFhat(S^1 \times \Sigma_g)$ is $2^g + \binom{2g}{g} + 2\sum_{i=1}^g (2i-1)\binom{2g}{g+i}$.
\end{theorem}

Finally, in Section \ref{sub:HFtori} we discuss the class of {\it Heegaard Floer solid tori}, whose definition was introduced by the third author (see \cite{solid-tori}, for example). In particular, we will show 
\begin{theorem}
If $M$ is a manifold with torus boundary which admits an L-space filling, then the following conditions are equivalent:  $\HFhat(M)$ is invariant under Dehn twists along the rational longitude; and the Dehn filling $M(\alpha)$ is an L-space for all slopes $\alpha$ other than the rational longitude.  
\end{theorem}
The proof of the theorem passes through a third characterization in terms of the immersed curves $\HFhat(M)$; see Theorem \ref{thm:HFST}. Manifolds satisfying the conditions are called {Heegaard Floer solid tori}.
 The solid torus is an obvious example; a more interesting example to keep in mind is the twisted $I$-bundle over the Klein bottle \cite{BGW2013} (see also \cite{HW,LW2014,solid-tori}).

%Third section: the complete grading package

%In sections \ref{sec:grad}--\ref{sec:prop} we describe some properties of \(\HFhat(M)\). 

\subsection*{Section \ref{sec:grad}: The grading package} Bordered Floer homology has a somewhat idiosyncratic grading by a quotient of a non-commutative group, which includes relative versions of the $\spinc(M)$ grading, the Maslov grading, and the simpler $\Ztwo$ grading. We show that this grading information can be encoded with some mild additional decorations on the curve invariant $\curves M$. This was set up previously for the spin$^c$ grading and $\Ztwo$ grading \cite{HRW} to the extent that it was required for the applications in our earlier work; our aim here is to review the complete grading package, and interpret this grading geometrically for $\HFhat(M)$. In particular, we give a geometric interpretation of the Maslov grading which seems interesting in its own right.

No decorations are required to encode grading information if $\curves{M}$ has a single component for each $\spinc$ structure $\spin$; in general, the decoration takes the form of arrows connecting different components of $\curves{M}$ associated with the same $\spinc$ structure. Given a set of parametrizing curves $(\alpha,\beta)$ for $\partial M$, the gradings on $\CFD(M,\alpha,\beta;\spin)$ can be extracted from geometric information on the corresponding decorated curves. The $\spinc$ grading of an intersection point of the curves with $\alpha$ or $\beta$, which corresponds to a generator of $\CFD(M,\alpha,\beta;\spin)$, is given by the position of the point in a chosen lift of to a cover of $\partial M\setminus z$ by $\R^2\setminus\Z^2$. The $\Ztwo$ grading is given by a choice of orientation on the curves, while the Maslov grading measures areas bounded by paths in a certain representative of $\curves{M}$.

Given two sets of decorated curves, we can endow their Floer homology with relative $\spinc$, Maslov, and $\Ztwo$ gradings; these gradings will be defined in Section \ref{sec:gradings-overview}. We will show that these gradings recover the corresponding gradings on the box tensor product of the corresponding type A and type D structures, thus proving the following grading refined version of the pairing theorem:

\begin{theorem}\label{thm:gradings}
The isomorphism in Theorem~\ref{thm:pairing} is an isomorphism of relatively graded vector spaces. More precisely, $\HFa(\boldsymbol\gamma_0, \boldsymbol\gamma_1)$ decomposes over spin$^c$ structures and carries a relative Maslov grading on each spin$^c$ structure, and these agree with the spin$^c$ decomposition and relative Maslov grading for $\HFhat(Y)$.
\end{theorem}

\begin{remark}
We used an alternate way of keeping track of \(\spinc\) structures in \cite{HRW}. This relies on the fact that there is a natural covering space \(\barT_{M}\) of \(\partial M \setminus z\)  with the property that for each \(\spin \in \spinc(M)\), the part of \(\HFhat(M)\)  associated to \(\spin\) lifts to \(\barT\). We denote this lift by  \(\HFhat(M,\spin)\). 
The decorations mentioned above uniquely determine it.
 \end{remark}

%Fourth section: the elliptic involution

\subsection*{Section \ref{sec:bim}: Symmetries} In this section, we discuss two symmetries of the invariant. The first describes the behavior of the invariant under orientation reversal.
\parpic[r]{
 \begin{minipage}{50mm}
 \centering
 \includegraphics[scale=0.7]{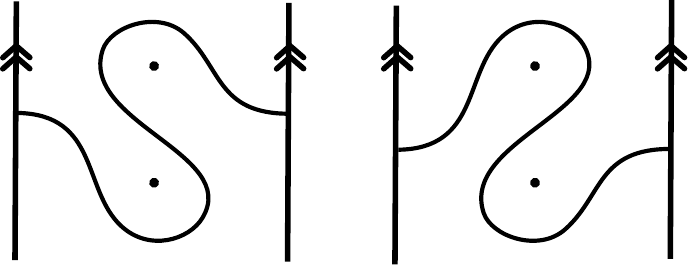}
 \captionof{figure}{\(\HFhat\) for the complements of the left and right-handed trefoils}
 \label{fig:left-right-trefoil}
  \end{minipage}%
  }
 This is a direct geometric translation of known properties of bordered Floer invariants. In short: $\curves{-M} = \curves{M}$ as curves, but we must remember that the orientation of $\partial M$ is different on the two sides of the equation. Thus when we identify $\partial M$ with a square or draw the curves in $\R^2 \setminus \Z^2$, as we usually do, the orientation reversal corresponds to a reflection across the homological longitude. For example, the curves shown in Figure~\ref{fig:left-right-trefoil} represent the invariants of the left and right-handed trefoil.

The second theorem in this section describes the behavior of the invariant under conjugation symmetry. 

\begin{theorem}\label{thm:spin-c-sym}
The invariant $\HFhat(M)$ is symmetric under the elliptic involution of \(\partial M \setminus z\).  Here, the involution is chosen so that \(z\) is a fixed point.  \end{theorem}

This corresponds to the fact that the curves in  Figure~\ref{fig:left-right-trefoil} are symmetric under reflection through the origin. (This is the midpoint of the segment which joins the two lifts of \(z\) shown in the figure.) 

\begin{remark} The original statement of Theorem \ref{thm:pairing} in  \cite{HRW}
said that \(\HFhat(Y) = \mathit{HF}(\boldsymbol\gamma_0,\boldsymbol\gamma_1')\), where $\boldsymbol\gamma_1'=\bar h(\HFhat(M_1))$. Here, $\bar h$ denotes the composition of  $h$ with the elliptic involution. By combining this with Theorem~\ref{thm:spin-c-sym} we are able to derive the version of Theorem \ref{thm:pairing} stated at the beginning of the paper. 
\end{remark}

This symmetry of the bordered Floer invariants had long been suspected and was already known for certain classes of manifolds. For example, for graph manifold rational homology tori, the symmetry holds because the elliptic involution actually extends to a diffeomorphism of the whole manifold. For complements of knots in the three-sphere, this symmetry was established by Xiu using properties of knot Floer homology  \cite{Xiu}. 
Its existence in general answers another natural question, which has been in the air for some time:

\begin{corollary}\label{crl:mutation}
Heegaard Floer homology is invariant under genus one mutation. In other words, $\HFhat(M_1\cup_hM_2) \simeq \HFhat(M_1\cup_{\overline{h}}M_2)$, where \(\overline{h}\) is the composition of \(h\) with the elliptic involution. 
\end{corollary}

The proof of Theorem \ref{thm:spin-c-sym} is surprisingly subtle, and relies on our structure theorem in an essential way. Work of Lipshitz, Ozsv\'ath, and Thurston identifies the algebraic symmetry associated with $\spinc$ conjugation, which amounts to considering the action of the torus algebra via box tensor product on type D structures \cite[Theorem 3]{LOT2011}. We compare the algebra (as a type DA bimodule) with the bimodule associated with the elliptic involution, and ultimately establish that while these two bimodules are different, the behaviour (of the functors induced on the Fukaya category) is the same on any set of immersed curves that arise as the invariants of three-manifolds with torus boundary. Along the way, we prove the following result, which may be of independent interest.

\begin{proposition}
No component of \(\HFhat(M)\) is a small circle linking the basepoint.
\end{proposition}

%Fourth section: properties and examples. Knot Floer, torsion, 

\subsection*{Section \ref{sec:knot_floer}:  Knot Floer homology.} 
If  \(K\) is a knot in a closed oriented three-manifold \(Y\), its 
complement is a manifold with torus boundary. Conversely, if \(M\) is a manifold with torus boundary and \(\mu\) is a filling slope on \(\partial M\), there is a knot 
\(K_\mu \subset M(\mu)\), where \(M(\mu)\) is the Dehn filling of slope \(\mu\) and \(K_\mu\) is the core of the Dehn filling. There is a close relationship between 
\(\HFhat(M)\) and the knot Floer homology of \(K_\mu\). In one direction we have the following result:

\begin{theorem} Suppose \(K\) is a knot in \(S^3\) and \(M\) is the complement of \(K\). Then \(\HFhat(M)\) is determined by the knot Floer chain complex \(\mathit{CFK}^-(K)\). 
\end{theorem}

This is a consequence of a theorem of Lipshitz, Ozsv{\'a}th and Thurston, which says that \(\CFD(M)\) is determined by \(\mathit{CFK}^-(K)\). Using the arrow calculus of \cite{HRW}, we give an effective algorithm for determining \(\HFhat(M)\) from \(\mathit{CFK}^-(K)\). 

%\labellist \tiny
%\pinlabel {$L_{0}$} at 35 70
%\pinlabel {$L_{1}$} at 35 100
%\pinlabel {$L_{-1}$} at 35 20
%\pinlabel {$L_{1/2}$} at 130 85
%\pinlabel {$L_{-1/2}$} at 130 35
%\endlabellist
\parpic[r]{
 \begin{minipage}{60mm}
 \centering
 \includegraphics[scale=0.8]{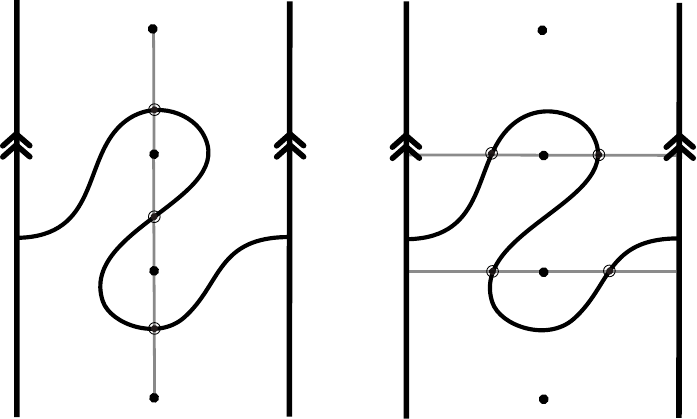}
 \captionof{figure}{Knot Floer homology of the trefoil.}
 \label{fig:HFK(T)}
  \end{minipage}%
  }
Conversely, it follows directly from the definition of \(\HFhat(M)\) that  \( \HFK(K_\mu) = \HFa(\HFhat(M),L_\mu)\), where \(L_\mu\) is the noncompact Lagrangian of slope \(\mu\) passing through the puncture point. 
As usual, there is a refined version of this statement which takes \(\spinc\) structures into account. 
The relevant set of \(\spinc\) structures -- \(\spinc(M, \gamma_\mu)\) -- was defined by Juh{\'a}sz \cite{Juhasz1}. It   is an \(H_1(M)\)-torsor. 
Suppose that \(\spin \in \spinc(M)\) and let \(\pi_\mu\co\spinc(M, \gamma_\mu) \to \spinc(M)\) be the restriction map.
%Let \(\pi_\mu\co\spinc(M, \gamma_\mu) \to \spinc(M)\) be the restriction map, and suppose that \(\spin \in \spinc(M)\). 
There is a natural bijection between  \(\pi_\mu^{-1}(\spin)\) and the set of lifts of \(L_{\mu}\) to the covering space \(\barT_{M,\spin}\). Denote the lift corresponding to  \(\spinbar \in \pi_\mu^{-1}(\spin)\) by \(L_{\mu,\spinbar}\). Then we have:
\begin{proposition}
	\label{prop:HFK}
\(	\HFK(K_\mu,\spinbar) = \HFa(\HFhat(M,\spin),L_{\mu,\spinbar}).\)
\end{proposition}

As an example, suppose 
	 \(M\) is the  complement of the right-hand trefoil, and let \(m\) and \(\ell\) be its standard meridian and longitude. Then \(\spinc(M ,\gamma_m) = \{\spin_{i} \,|\, i \in \Z\} \). The lifts \(L_{m, \spin_i}\)  are shown on the left-hand side of  Figure~\ref{fig:HFK(T)}; the groups \(\HFa(\HFhat(M), L_{m,\spin_i})\) give the knot Floer homology of the trefoil. For comparison, \(\spinc(M, \gamma_\ell) = \{\mathfrak{t}_{i} \,\ell\, i \in \Z+\frac{1}{2}\} \). The lifts  \(L_{\ell, \mathfrak{t}_i}\)  are shown in the right-hand side of the  figure. It is easy to see that \(\HFa(\HFhat(M), L_{l,\mathfrak{t}_i}) = \F^2 \) if \(i = \pm \frac{1}{2}\), and is \(0\) otherwise, as was first calculated by Eftekhary \cite{Eftekhary2005}.

\subsection*{Section \ref{sec:prop}:  Turaev torsion and Thurston norm.} 
It is well known that  knot Floer homology determines  these invariants, so it must be possible to express them in terms of \(\HFhat(M)\). In fact, the relation is very simple and geometric. In this introduction, we restrict our attention to the case where \(H_1(M)= \Z\), but the general case is treated Section \ref{sec:prop}. 

The Turaev torsion is a function \(\tau\co\spinc(M,\partial M) \to \Z\). When \(H_1(M)=\Z\), \(\spinc(M)\) contains a unique element \(\spin\), and \(\spinc(M,\partial M)\) can be identified with \(\{\spinbar_i \, | \, i \in \Z+\frac{1}{2}\}\) in such a way that 
\[\sum \tau(\spinbar_i)t^i = \frac{\Delta(M)}{t^{-1/2}-t^{1/2}}\]
where \(\Delta(M)\) is the Alexander polynomial of \(M\) and the right-hand side is to be expanded in positive powers of \(t\). We have 

\begin{theorem}\label{thm:torsion} For \(i \in \Z+\frac{1}{2}\), 
\(\tau(\spinbar_i) = \gamma_i \cdot \barcurves{M,\spin}\), where \(\gamma_i\) is a path running from  lift of \(z\) at height \(i\) in \(T_M\) towards \(-\infty\). 
\end{theorem}

For example, if \(M \) is the complement of the right-hand trefoil, we see from Figure~\ref{fig:HFK(T)} that the \(\gamma_i \cdot \barcurves{M,\spin} = 0\) for \(i=\frac{1}{2}\) and \(i\leq  - \frac{3}{2}\), while  \(\gamma_i \cdot \barcurves{M,\spin} = 1\) for  \(i=- \frac{1}{2}\) and \(i\geq  \frac{3}{2}\). This agrees with the fact that 
\[ \frac{\Delta(M)}{t^{-1/2}-t^{1/2}} = \frac{t^{-1}-1+t}{t^{-1/2}-t^{1/2}} = t^{-1/2} + t^{3/2}+t^{5/2}+\cdots\] 

Similarly,  we can relate \(\HFhat(M)\) to the Thurston norm:

\begin{proposition} Suppose that \(H_1(M) = \Z\), and  
let \(k_+\) be the largest value of \(k\) such that \(z_k\) cannot be connected to \(+\infty\) by a path in \(T_M\) disjoint from \(\barcurves{M,\spin}\). Similarly, let \(k_-\) be the smallest value of \(k\) such that \(z_k\) cannot be connected to \(-\infty\) by a path disjoint from \(\barcurves{M,\spin}\).  
If \(\Sigma\) is a minimal genus surface generating \(H_2(M,\partial M)\), then    $ 2g(\Sigma) - 1= k_+ - k_- $.
\end{proposition}

By combining Theorem~\ref{thm:torsion} with the characterization of L-space Dehn fillings given in \cite{HRW}, we give a simple new proof of the first main theorem of \cite{RR}, which characterizes the set of L-space filling slopes of a Floer simple manifold in terms of the Turaev torsion.

\subsection*{Section \ref{sec:seiberg-witten}: Relation to Seiberg-Witten theory}
In the final three sections, we explore some more speculative connections between \(\HFhat(M)\) and other subjects. The first of these is Seiberg-Witten theory. The Seiberg-Witten equations on four-manifolds with \(T^3\) boundary (or more accurately, an end modeled on \(T^3 \times [0, \infty)\) were studied by Morgan, Mrowka, and Szab{\'o} \cite {MMS1997}; very similar statements hold for three-manifolds with torus boundary. We discuss the relation between the set of solutions to the Seiberg-Witten equations on \(M\) and \(\HFhat(M)\), focussing on the case of Seifert-fibred spaces.  Although proving any general relation seems difficult (and the payoff uncertain), these considerations motivated a lot of our initial thinking about \(\HFhat(M)\), and are a useful guide in many contexts. 

\subsection*{Section \ref{sec:Kh}: Relation to Khovanov homology}
A well-know theorem of Ozsv{\'a}th and Szab{\'o} \cite{OSz2005} shows that if \(K\) is a knot in \(S^3\) there is a spectral sequence from the Khovanov homology of \(-K\) to \(\HFhat(\Sigma_K)\), where \(\boldsymbol\Sigma_K\) is the branched double cover of \(K\). Here we explore the analog of this statement for a four-ended tangle \(T\), whose branched double cover \(\boldsymbol \Sigma_T\) is a manifold with torus boundary. We discuss the relation between the underlying categories in which the two invariants live, and describe the form the analog of the Ozsv{\'a}th-Szab{\'o} spectral sequence should have. Finally, we consider some specific examples, including rational tangles, which are relatively easy, and the \((2,-2)\) and \((2,-3)\) pretzel tangles, which are more interesting. 

\subsection*{Section \ref{sec:minus}: Relation to \(\mathit{HF}^-\)} 
The theory of bordered Floer homology for \(\HFminus\) is currently being developed by Lipshitz, Ozsv{\'a}th and Szab{\'o}.  One might hope that this theory can be used to enhance  \(\HFhat(M)\)  to an invariant \(\HFminus(M)\) which carries full information about \(\HFminus\) of Dehn fillings on \(M\). It is  natural to ask if there are conditions under which everything about \(\HFminus(M)\) is actually determined by \(\HFminus(M)\). 
Although it is relatively easy to construct examples where \(\HFhat(M)\) cannot tell us everything, it is equally clear that there are many cases in which it effectively does. In this final section, we consider some examples of both types and speculate briefly about what conditions might be enough to ensure that \(\HFhat(M)\) carries full infomation about \(\HFminus\) of Dehn fillings on \(M\).

\subsection*{Acknowledgements}
The authors would like to thank Cameron Gordon, Peter Kronheimer, Yank\i\ Lekili, Tye Lidman, Robert Lipshitz, Peter Ozsv{\'a}th, Sarah Rasmussen, Ivan Smith, Zoltan Szab{\'o}, and Claudius Zibrowius for helpful discussions (some of them dating back a very long time). Part of this work was carried out while the third author was visiting Montr\'eal as CIRGET research fellow, part was carried out while the second and third authors were participants in the program {\em Homology Theories in Low Dimensions} at the Isaac Newton Institute, and part while the third author was visiting the CRM as a Simons Visiting Professor. The authors would like to thank CIRGET, the CRM, and the Newton Institute for their support.

%% file: sections/setup.tex
% !TEX root = ../companion.tex
%setup.tex

We begin by describing the invariant $\HFhat(M)$ associated with a three-manifold $M$ with torus boundary, its relationship to bordered Floer homology, and our interpretation of Lipshitz, Ozsv\'ath, and Thurston's pairing theorem in terms of Langrangian intersection Floer homology. 

\subsection{Modules over the torus algebra} We give a quick overview of the modules that arise in bordered Floer theory, restricting attention to the case of torus boundary. A less terse overview is given in \cite{HRW}.

\parpic[r]{
 \begin{minipage}{40mm}
 \centering
 \begin{tikzpicture}[scale=0.75,>=stealth', thick] 
\def \radius {1.5cm} \def \outer {2cm}
 \node at ({360/(4) * (1- 1)}:\radius) {$\boldsymbol{\cdot}$};
  \node at ({360/(4) * (2- 1)}:\radius) {$\boldsymbol{\cdot}$};
  \node at ({360/(4) * (3 -1)}:\radius) {$\boldsymbol{\cdot}$};
   \node at ({360/(4) * (4 -1)}:\radius) {$\boldsymbol{\cdot}$};
  \draw[<-,shorten <= 0.125cm, shorten >= 0.155cm, dashed]({360/4 * (1 - 1)}:\radius) arc ({360/4 * (1- 1)}:{360/4 * (2-1)}:\radius);
 \draw[<-,shorten <= 0.125cm, shorten >= 0.125cm]({360/4 * (2 - 1)}:\radius) arc ({360/4 * (2- 1)}:{360/4 * (3-1)}:\radius);
 \draw[<-,shorten <= 0.125cm, shorten >= 0.125cm]({360/4 * (3 - 1)}:\radius) arc ({360/4 * (3- 1)}:{360/4 * (4-1)}:\radius);
 \draw[<-,shorten <= 0.125cm, shorten >= 0.125cm]({360/4 * (4 - 1)}:\radius) arc ({360/4 * (4- 1)}:{360/4 * (5-1)}:\radius);
  \node at ({360/(4) * (1- (1/2))}:\outer) {${\rho_{0}}$};
  \node at ({360/(4) * (2-(1/2))}:\outer) {$\rho_3$};
  \node at ({360/(4) * (3-(1/2))}:\outer) {$\rho_2$};
  \node at ({360/(4) * (4-(1/2))}:\outer) {$\rho_1$};
   \node at ({360/(4) * (1)}:\outer) {$\iota_0$};
\node at ({360/(4) * (2)}:\outer) {$\iota_3$};
\node at ({360/(4) * (3)}:\outer) {$\iota_2$};
\node at ({360/(4) * (4)}:\outer) {$\iota_1$};
\end{tikzpicture}

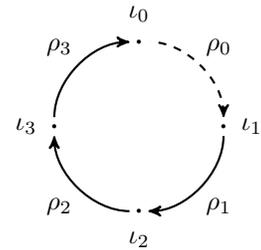
\captionof{figure}{A simple quiver.}\label{fig:path-for-B}
  \end{minipage}%
  }
The torus algebra $\Alg$ is obtained as the quotient of a particularly simple path algebra. Ignoring (for the moment) the dashed edge labelled $\rho_0$, let $\sB$ be the path algebra (over $\F$) of the quiver shown in Figure \ref{fig:path-for-B}. Then $\Alg$ is obtained in two steps: we first quotient by the ideal $\langle \iota_0+\iota_2, \iota_1+\iota_3\rangle$ and then quotient the result by the ideal $\langle \rho_3\rho_2,\rho_2\rho_1\rangle$. It will sometimes be convenient to write $\mu(a,b)=ab$ for the multiplication in $\Alg$, and we will use the shorthand $\rho_I=\rho_{I_1}\rho_{I_2}$ where $I$ is an increasing sequence in $\{1,2,3\}$ and $I=I_1I_2$. Denote by $\sI\subset\Alg$ the subring of idempotents, generated by $\iota_\bu=\iota_1=\iota_3$ and $\iota_\ci=\iota_0=\iota_3$. Note that, as a vector space, $\Alg$ is generated by $\{\iota_\bu,\iota_\ci,\rho_1,\rho_2,\rho_3,\rho_{12},\rho_{23},\rho_{123}\}$.

A slightly larger algebra, which yields $\Alg$ as a quotient, is obtained from the quiver in Figure \ref{fig:path-for-B} (this time including the $\rho_0$ edge) modulo the ideal $\langle \rho_I | \text{the\ string\ } I \text{\ contains\ more\ than\ one\ } 0 \rangle$. Denoting this algebra by $\widetilde{\sB}$, the algebra $\widetilde{\Alg}$ is obtained (as before) in two steps: we first quotient by the ideal $\langle \iota_0+\iota_2, \iota_1+\iota_3\rangle$ and then quotient the result by the ideal $\langle \rho_3\rho_2,\rho_2\rho_1,\rho_1\rho_0, \rho_0\rho_3\rangle$. Note that $\Alg\cong\widetilde{\Alg}/\langle\rho_0\rangle$, and that $\sI$ is the subring of idempotents in $\widetilde{\Alg}$ as well. The element $U=\rho_{1230}+\rho_{2301}+\rho_{3012}+\rho_{0123}$ is central in $\widetilde{\Alg}$.
%, so that $\Alg$ is obtained from $\widetilde{\Alg}$ by setting $U=0$.

Bordered Floer homology introduces a particular class of left-modules over $\Alg$ called type D structures. A type D structure over $\Alg$ is a left $\sI$-module $V$ where  $\sI\subset\Alg$ is the idempotent subring (so that the underlying $\F$-vector space satisfies $V\cong V_\bullet\oplus V_\circ$), equipped with an $\sI$-linear map $\delta\co V\to \Alg\otimes V$ such that \[\big((\mu\otimes\id)\circ(\id\otimes\delta) \circ\delta\big)(x)=0\] for all $x\in V$. Notice that this compatibility condition on $\delta$ ensures that $\partial(a\otimes x)=a\cdot\delta(x)$ squares to zero, where $a\cdot(b\otimes x) = \mu(a,b)\otimes x$, so that $\Alg\otimes V$ is a left differential module over $\Alg$. All tensor products are taken over $\sI$.

Given a type D structure $(V,\delta)$, an extension is a pair $(V,\tilde\delta)$ where $\tilde\delta\co V\to \widetilde{\Alg}\otimes V$ satisfying \[\big((\tilde\mu\otimes\id)\circ(\id\otimes\tilde\delta) \circ\tilde\delta\big)(x)=U\otimes x\] for all $x\in V$ and such that $\tilde\delta|_{U=0}=\delta$. Whenever an extension exists, the type D structure $(V,\delta)$ is called extendable. It turns out that extensions, when they exist, are unique up to isomorphism as $\widetilde{\Alg}$-modules \cite{HRW}. This class of objects has geometric significance: If $(M,\alpha,\beta)$ is a bordered three-manifold with torus boundary, so that $\alpha$ and $\beta$ specify a handle decomposition of the punctured torus $\partial M\setminus z$, the bordered invariant $\CFD(M,\alpha,\beta)$ is an extendable type D structure. This is essentially due to Lipshitz, Ozsv\'ath, and Thurston; see \cite[Appendix A]{HRW}. Our structure theorem states that every extendable type D structure over $\Alg$ is equivalent to a collection of immersed curves decorated with local systems \cite[Theorem 5]{HRW}. We illustrate this with a simple example; see Figure \ref{fig:review-example}. 

\parpic[r]{
 \begin{minipage}{50mm}
 \centering
 \includegraphics[scale=.8]{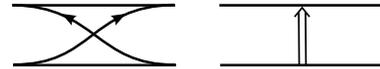}%\vspace*{-15pt}
 \captionof{figure}{Pairs of edges are replaced by crossover arrows.}
 \label{fig:crossover-convention}
  \end{minipage}%
  }
The starting point for our geometric interpretation of type D structures (and their extensions) is the observation that the description of a type D structure in terms of a decorated graph, where the vertex set generates $V$ and the labeled edge set describes the map $\delta$, may equivalently be given in terms of an immersed train track in the torus minus a marked point. Furthermore, extended type D structures admit a convenient shorthand, wherein particular pairs of arrows are replaced by crossover arrows; see Figure \ref{fig:crossover-convention}. The work to be done towards a structure theorem for bordered invariants \cite[Theorem 1]{HRW} is to exhibit an algorithm by which all crossover arrows are either removed or run between parallel strands. And, towards the paring theorem \cite[Theorem 2]{HRW}, one checks that the box tensor product chain complex is left invariant (up to chain homotopy equivalence) when this algorithm is implemented. 

 \begin{figure}[t]
\labellist 
\pinlabel {$\cong$} at 255 331 %\pinlabel {$\Alg$} at 55 331
\pinlabel {(i)} at 52 130 \pinlabel {(ii)} at 187 130 \pinlabel {(iii)} at 320 130
\pinlabel {(iv)} at 52 -10 \pinlabel {(v)} at 187 -10 \pinlabel {(vi)} at 320 -10
\small
\pinlabel {$z$} at 79 355
\pinlabel {$\rho_0$} at 100 370 \pinlabel {$\rho_1$} at 100 287 \pinlabel {$\rho_2$} at 15 287 \pinlabel {$\rho_3$} at 15 370
  \pinlabel {$3$} at 224 298  \pinlabel {$3$} at 320 298
   \pinlabel {$1$} at 224 363.5  \pinlabel {$1$} at 224 341 \pinlabel {$1$} at 320 363.5
    \pinlabel {$12$} at 167 352  \pinlabel {$12$} at 270 320
         \endlabellist
\includegraphics[scale=.8]{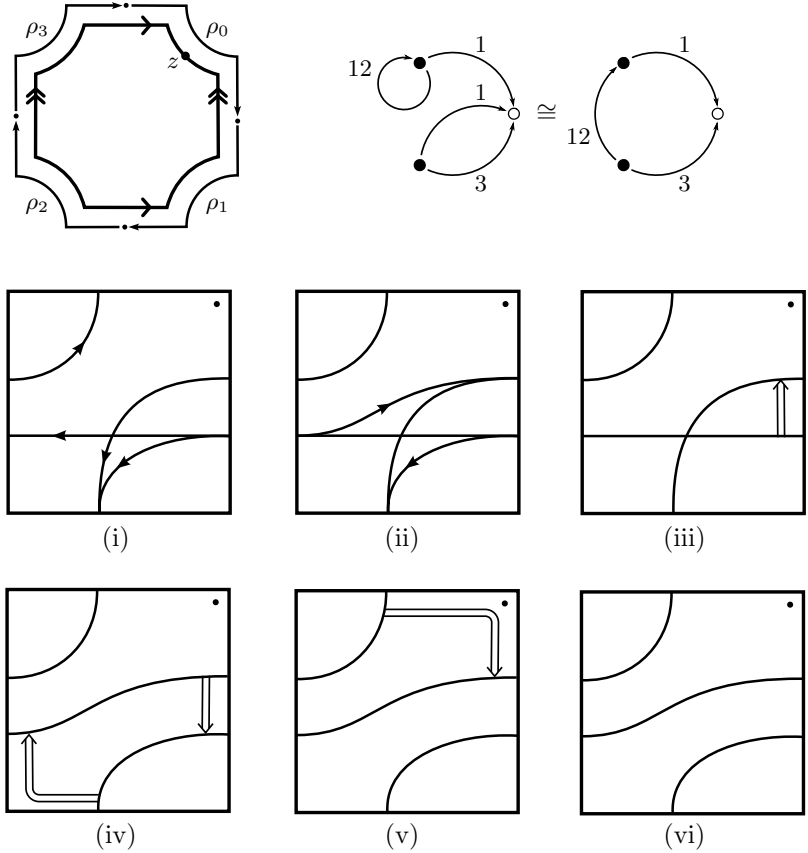}
\vspace*{10pt}
\caption{Upper left: identifying the elements of the algebra $\Alg$ with the torus, punctured at $z$ and cut open along a fixed 1-hadle decomposition. Upper right: two type D structures represented as decorated graphs. These are isomorphic as $\Alg$-modules, as illustrated, in 6 steps. (i) The decorated graph admits an equivalent representation an immersed train track in the marked torus. Note that all tangencies are either vertical or horizontal, and intersection with the vertical gives generators in the $\iota_\bu$ summand while intersection with the horizontal gives  generators in the $\iota_\ci$ summand. (ii) A choice of extension is made, where in our convention unoriented edges should be read as two-way edges. (iii) We introduce the crossover arrow notation as a shorthand for collecting pairs of oriented edges. (iv) Applying \cite[Proposition 24]{HRW} we can pass to another diagram representing the same type D structure but in which all crossover arrows run clockwise. (v) Clockwize-running crossover arrows covering a corner may be removed by a change of basis of the form $x\mapsto x + \rho_I\otimes y$ when the crossover arrow gives $\delta(x)=\rho_I\otimes y$, and crossover arrows can be pushed over a handle by a change of basis of the form $x\mapsto x+y$. (v) In this way, an algorithm can be given that removes all arrows (see \cite[Section 3.7]{HRW}) unless they connect two strands that remain parallel, in which case a local system provids the appropriate book keeping tool. Notice that, the resulting unoriented immersed curve specifies both a type D structure and and extension without ambiguity.}
\label{fig:review-example}
\end{figure}

\parpic[r]{
\labellist 
\small
\pinlabel {$\big(\F^2, \left(\begin{smallmatrix} 1 & 1 \\ 0 & 1\end{smallmatrix}\right)\!\big)$} at 195 60
\endlabellist
 \begin{minipage}{70mm}
 \centering
 \includegraphics[scale=.8]{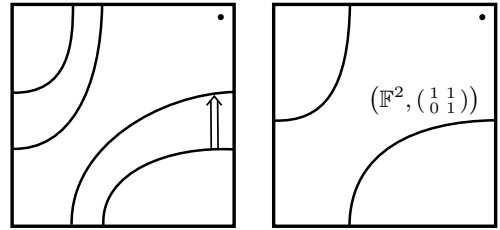}
 \captionof{figure}{Collections of crossover arrows between parallel curves can be represented using local systems.}
 \label{fig:small-local-system}
  \end{minipage}%
  }
The end result of the aforementioned algorithm leads naturally to the appearance of local systems, that is, finite dimensional vector spaces over $\F$ that are equipped with an automorphism. Indeed, since the only remaining crossover arrows run between parallel curves, the dimension of this vector space is given by the number of parallel curve-components while the crossover arrows give a graphical shorthand for an automorphism; see Figure \ref{fig:small-local-system}. As such, the case where there are no crossover arrows remaining corresponds, strictly speaking, to the case where all local system are one-dimensional. We will refer to this one-dimensional local system as the trivial local system, and simply record the immersed curve in this trivial case. Note that the case of trivial local systems corresponds to the loop type case that appears in the literature \cite{HRRW,HW,Zibrowius}. This also provides us with a graphical representation of a local system $(V,\Phi)$ over an immersed curve, namely, one replaces the curve with $\dim(V)$ parallel copies of the curve in question and encodes the endomorphism $\Phi$ using crossover arrows. This can always be done by confining the crossover arrows to a prescribed part of the curve; we will refer to this process as {\em expanding} the local system. 

\subsection{The case of trivial local systems}\label{sub:loop-type} In practice, many examples of type D structures arising as the bordered invariants of three-manifolds with torus boundary carry trivial local systems. Following \cite{HRW}, a manifold $M$ is loop type if $\HFhat(M)$ caries a trivial local system. At present, the authors are not aware of a three-manifold $M$ for which the invariant $\HFhat(M)$ carries a non-trivial local system, though we emphasize that this is most likely tied to a general lack of examples rather than being indicative of a simplification that holds for all bordered invariants. However, there are certain classes of manifolds which are known to be loop type. For example:

\begin{proposition}\label{prp:Floer-simple}
Manifolds with torus boundary that are Floer simple are loop type. 
\end {proposition}
\begin{proof}[Outline of proof] Recall that $M$ is Floer simple if it admits more than one L-space filling \cite{RR}, where a closed manifold $Y$ is an L-space whenever it is a rational homology sphere for which $\dim\HFhat(Y)=|H_1(Y;\Z)|$. It is observed in \cite{RR} that the class of Floer simple manifolds coincides with the class of  {\em simple} loop type manifolds introduced in \cite{HW}; see \cite{HRRW} for a concise statement and proof. This latter class yields a description in terms of immersed curves with trivial local system directly. \end{proof}
\begin{remark}In fact, more can be said in the Floer simple case: The number of curve components in $\HFhat(M)$ agrees with the number of spin$^c$ structures on $M$. Thus for each $\spin\in\spinc(M)$, $\HFhat(M,\spin)$ is a single immersed curve with trivial local system. This should be compared, for instance, with the special case where $M$ is the complement of an L-space knot in $S^3$.   \end{remark}

The final statement in this proof outline alludes to an alternate construction of the curves in the loop type case. The manifolds for which $\HFhat(M)$ carries trivial local systems are precisely the loop type manifolds introduced by the first and third authors \cite{HW}; they are characterized by the property that, for some choice of basis, the type D structure $\CFD(M,\alpha,\beta)$ associated with $M$ and some parametrization $(\alpha,\beta)$ is represented by a valence 2 graph. The goal of this subsection is to give a greatly simplified construction of the curve invariant $\HFhat(M)$ given loop type manifold $M$ (and, in particular, given such a preferred basis for $\CFD(M,\alpha,\beta)$). This is an instructive special case to consider as it bypasses the train tracks and arrow sliding algorithm needed for the general case, while still capturing the typical behavior of the curve invariants. Indeed, as mentioned above, no non-loop type examples are currently known, and this case may be sufficient for many applications. We remark, however, that even if a manifold is loop type, computing $\CFD(M,\alpha,\beta)$ may produce a basis which does not satisfy the loop type condition. Finding a basis that does can be a non-trivial task. In this case the arrow sliding procedure from \cite{HRW} can be thought of as a graphical algorithm for finding a loop type basis for $\CFD(M,\alpha,\beta)$.  See, for example, Figures~\ref{fig:connect-sum} and \ref{fig:connect-sum2} in Section \ref{sec:curves-from-CFK}.

To describe this simplified construction, we first consider a collection of enhanced decorated graphs, where the vertex set takes values in $\{\bu^+,\bu^-,\ci^+,\ci^-\}$ and the edge set is directed and each edge is labeled by an element in $\{1,2,3,12,23,123\}$. Theses graphs are required to satisfy two additional conditions: First, ignoring signs for the moment, the edge orientations are compatible with the black and white vertex labeling so that  
\raisebox{-4pt}{$\begin{tikzpicture}[thick, >=stealth',shorten <=0.1cm,shorten >=0.1cm] 
\draw[->] (0,0) -- (1,0);\node [above] at (0.45,-0.07) {$\scriptstyle{1}$};
\node at (0,0) {$\bu$};\node at (1,0) {$\circ$};\end{tikzpicture}$},
\raisebox{-4pt}{$\begin{tikzpicture}[thick, >=stealth',shorten <=0.1cm,shorten >=0.1cm] 
\draw[->] (0,0) -- (1,0);\node [above] at (0.45,-0.07) {$\scriptstyle{2}$};
\node at (0,0) {$\circ$};\node at (1,0) {$\bu$};\end{tikzpicture}$}, 
\raisebox{-4pt}{$\begin{tikzpicture}[thick, >=stealth',shorten <=0.1cm,shorten >=0.1cm] 
\draw[->] (0,0) -- (1,0);\node [above] at (0.45,-0.07) {$\scriptstyle{3}$};
\node at (0,0) {$\bu$};\node at (1,0) {$\circ$};\end{tikzpicture}$}, 
\raisebox{-4pt}{$\begin{tikzpicture}[thick, >=stealth',shorten <=0.1cm,shorten >=0.1cm] 
\draw[->] (0,0) -- (1,0);\node [above] at (0.45,-0.07) {$\scriptstyle{12}$};
\node at (0,0) {$\bu$};\node at (1,0) {$\bu$};\end{tikzpicture}$},
\raisebox{-4pt}{$\begin{tikzpicture}[thick, >=stealth',shorten <=0.1cm,shorten >=0.1cm] 
\draw[->] (0,0) -- (1,0);\node [above] at (0.45,-0.07) {$\scriptstyle{23}$};
\node at (0,0) {$\circ$};\node at (1,0) {$\circ$};\end{tikzpicture}$}, and
\raisebox{-4pt}{$\begin{tikzpicture}[thick, >=stealth',shorten <=0.1cm,shorten >=0.1cm] 
\draw[->] (0,0) -- (1,0);\node [above] at (0.45,-0.07) {$\scriptstyle{123}$};
\node at (0,0) {$\bu$};\node at (1,0) {$\circ$};\end{tikzpicture}$}. Second, the signs on the verticies change only when edges labelled by $2$ or 123 are traversed, so that
\raisebox{-4pt}{$\begin{tikzpicture}[thick, >=stealth',shorten <=0.1cm,shorten >=0.1cm] 
\draw[->] (0,0) -- (1,0);\node [above] at (0.45,-0.07) {$\scriptstyle{2}$};
\node at (-0.125,0.0625) {$\circ^{\pm}$};\node at (1.125,0.0625) {$\bu^{\mp}$};\end{tikzpicture}$} and
\raisebox{-4pt}{$\begin{tikzpicture}[thick, >=stealth',shorten <=0.1cm,shorten >=0.1cm] 
\draw[->] (0,0) -- (1,0);\node [above] at (0.45,-0.07) {$\scriptstyle{123}$};
\node at (-0.125,0.0625) {$\bu^{\pm}$};\node at (1.125,0.0625) {$\circ^{\mp}$};\end{tikzpicture}$}. Any such enhanced decorated graph encodes a reduced $\Z/2\Z$-graded type D structure over $\Alg$. We are only interested in relatively $\Z/2\Z$ graded type D structures, so the sign labels will be well defined only up to changing the sign on every vertex of the graph. This is treated in detail in Section \ref{sec:grad}; see also \cite[Section 6]{HRW}. 

\parpic[r]{
 \begin{minipage}{60mm}{
\begin{tikzpicture}[>=stealth',scale=0.8] 
		\node at (0,0) {$\bullet$}; 
		\node at (0,1) {$\bullet$}; 
		\node at (0,2) {$\bullet$};
\draw[thick, ->,shorten <=0.1cm] (0,0) -- (1,0); \node at (0.5,0.25) {$\scriptstyle{123}$};
\draw[thick, ->, shorten <=0.1cm] (0,1) -- (1,1); \node at (0.5,1.25) {$\scriptstyle{12}$};
\draw[thick, ->, shorten <=0.1cm] (0,2) -- (1,2); \node at (0.5,2.25) {$\scriptstyle{1}$}; \node at (0.5,-0.75){\bf{I}$_\bullet$};
		\node at (2,0) {$\bullet$}; 
		\node at (2,1) {$\bullet$}; 
		\node at (2,2) {$\bullet$};
\draw[thick, <-,shorten <=0.1cm] (2,0) -- (3,0); \node at (2.5,0.25) {$\scriptstyle{12}$};
\draw[thick, <-, shorten <=0.1cm] (2,1) -- (3,1); \node at (2.5,1.25) {$\scriptstyle{2}$};
\draw[thick, ->, shorten <=0.1cm] (2,2) -- (3,2); \node at (2.5,2.25) {$\scriptstyle{3}$}; \node at (2.5,-0.75){\bf{II}$_\bullet$};
		\node at (4,0) {$\circ$}; 
		\node at (4,1) {$\circ$}; 
		\node at (4,2) {$\circ$};
\draw[thick, <-,shorten <=0.1cm] (4,0) -- (5,0); \node at (4.5,0.25) {$\scriptstyle{1}$};
\draw[thick, ->, shorten <=0.1cm] (4,1) -- (5,1); \node at (4.5,1.25) {$\scriptstyle{23}$};
\draw[thick, ->, shorten <=0.1cm] (4,2) -- (5,2); \node at (4.5,2.25) {$\scriptstyle{2}$}; \node at (4.5,-0.75){\bf{I}$_\circ$};
		\node at (6,0) {$\circ$}; 
		\node at (6,1) {$\circ$}; 
		\node at (6,2) {$\circ$};
\draw[thick, <-,shorten <=0.1cm] (6,0) -- (7,0); \node at (6.5,0.25) {$\scriptstyle{123}$};
\draw[thick, <-, shorten <=0.1cm] (6,1) -- (7,1); \node at (6.5,1.25) {$\scriptstyle{23}$};
\draw[thick, <-, shorten <=0.1cm] (6,2) -- (7,2); \node at (6.5,2.25) {$\scriptstyle{3}$}; \node at (6.5,-0.75){\bf{II}$_\circ$};
	\end{tikzpicture}
	\vspace*{-5pt}
	} 
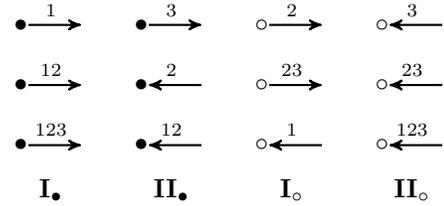
\captionof{figure}{Vertex types for valence 2 decorated graphs.}\label{fig:vertex-types}
  \end{minipage}
  }
The extendability condition on type D structures coming from manifolds with torus boundary places further constraints on the corresponding enhanced decorated graphs. In particular, if we sort incident edges at a vertex into types according to Figure \ref{fig:vertex-types}, at every vertex there is at least one edge of type {\bf I}$_{\bu/\circ}$ and at least one of type {\bf II}$_{\bu/\circ}$. Restricting to the case of valence 2 graphs, we conclude that there is exactly one edge of type {\bf I}$_{\bu/\circ}$ and exactly one of type {\bf II}$_{\bu/\circ}$ at each vertex. It is straightforward to see that any valence two decorated graph satisfying this condition at each vertex represents an extendable type D structure. A calculus for working with this class of valence two graphs was introduced in \cite{HW}.

From a valence 2 enhanced directed graph as described above, it is fairly easy to produce an (oriented) set of immersed curves following the general procedure in \cite{HRW}; the valence 2 condition implies that the initial train track representing the directed graph is in fact an immersed multicurve, so no removal of crossover arrows or other simplification is required. We will now describe an even quicker shortcut for producing a curve from a valence 2 enhanced directed graph of the type described above. The key observation is that the graph is determined by its vertex labels; the arrow labels and directions are redundant. 
More precisely, a component $\lp$ of the graph determines a cyclic list of symbols in $\{\bu^+,\bu^-,\circ^+,\bu^-\}$ coming from the vertex labels, where the order in which the vertices are read is determined by the convention that type {\bf I}$_\bu$ arrows precede $\bu^+$ vertices and follow $\bu^-$ vertices, while {\bf I}$_\circ$ arrows follow $\circ^+$ vertices and precede $\circ^-$ vertices. It is straightforward to check that this convention is consistent---that is, that any vertex of $\lp$ determines the same cyclic ordering on the vertices in $\lp$. Moreover, it is clear that given this convention the graph $\lp$ can be reconstructed from the cyclic list of vertex labels. For example, a $\bu^+$ followed by a $\circ^+$ must be connected by an arrow labelled by either $1$ or $3$, since these are the only arrows that connect $\bu$ to $\circ$ without changing sign; since $\bu^+$ vertices are followed by type {\bf II}$_\bu$ arrows, the arrow must be labelled by $3$. We find it convenient to replace the labels $\bu^\pm$ with $\beta^\pm$ and $\circ^\pm$ with $\alpha^\pm$. We have shown that $\lp$ is equivalent to a cyclic word in the letters $\{\alpha^\pm, \beta^\pm\}$, which we denote $\cu(\lp)$. Note that changing the sign on every vertex (equivalently, shifting the $\Z/2\Z$ grading on the corresponding type D structure by one) has the effect of inverting every letter $\cu(\lp)$ and reversing the cyclic order. Finally, we observe that such a cyclic word gives rise to an oriented immersed curve in the parametrized punctured torus. $\cu(\lp)$ may be viewed as an element of the free group generated by $\alpha$ and $\beta$ mod conjugation; the free group is precisely the fundamental group of the punctured torus, generated by the parametrizing curves $\alpha$ and $\beta$, and taking conjugacy classes amounts to taking non-basepointed loops. Recall that when $\lp$ comes from $\CFD(M,\alpha,\beta)$ for some bordered manifold $M$, $\alpha$ and $\beta$ are parametrizing curves for $\partial M$ and thus $\cu(\lp)$ defines an oriented immersed curve in $\partial M \setminus z$. To summarize, the case where $M$ is the complement of the right-hand trefoil is shown in Figure \ref{fig:trefoil-construction} (note that in this example, $\mu = \alpha$ and $\lambda = \beta$).

\begin{figure}[t]
\labellist 
\pinlabel {$\lambda$} at 272 21 \pinlabel {$\mu$} at 240.5 54
\pinlabel {$-$} at 150 79 %\pinlabel {$\lambda^{-1}$} at 154 79
\pinlabel {$-$} at 150 55 %\pinlabel {$\lambda^{-1}$} at 154 55
\pinlabel {$+$} at 150 30 %\pinlabel {$\lambda$} at 147 30
\pinlabel {$+$} at 108 118
\pinlabel {$-$} at 92 118
\pinlabel {$-$} at 78 118
\pinlabel {$+$} at 62 118
\small
\pinlabel {$+$} at 291 103 \pinlabel {$-$} at 274.5 103 \pinlabel {$-$} at 274.5 208
\pinlabel {$-$} at 324 55 \pinlabel {$+$} at 324 151 \pinlabel {$-$} at 324 274
\pinlabel {$+$} at 385 208
\pinlabel {$123$} at 136 270 \pinlabel {$3$} at 141 217
\pinlabel {$1$} at 102 303 \pinlabel {$23$} at 102 188
\pinlabel {$3$} at 50 292 \pinlabel {$1$} at 50 198
\pinlabel {$2$} at 25 250
\pinlabel $-$ at 166 242
\pinlabel $+$ at 139 304
\pinlabel $+$ at 69 321
\pinlabel $+$ at 21 284
\pinlabel $-$ at 20 211
\pinlabel $-$ at 67 173
\pinlabel $-$ at 135 185
\tiny
\pinlabel {$z$} at 130 97
\endlabellist
\includegraphics[scale=.8]{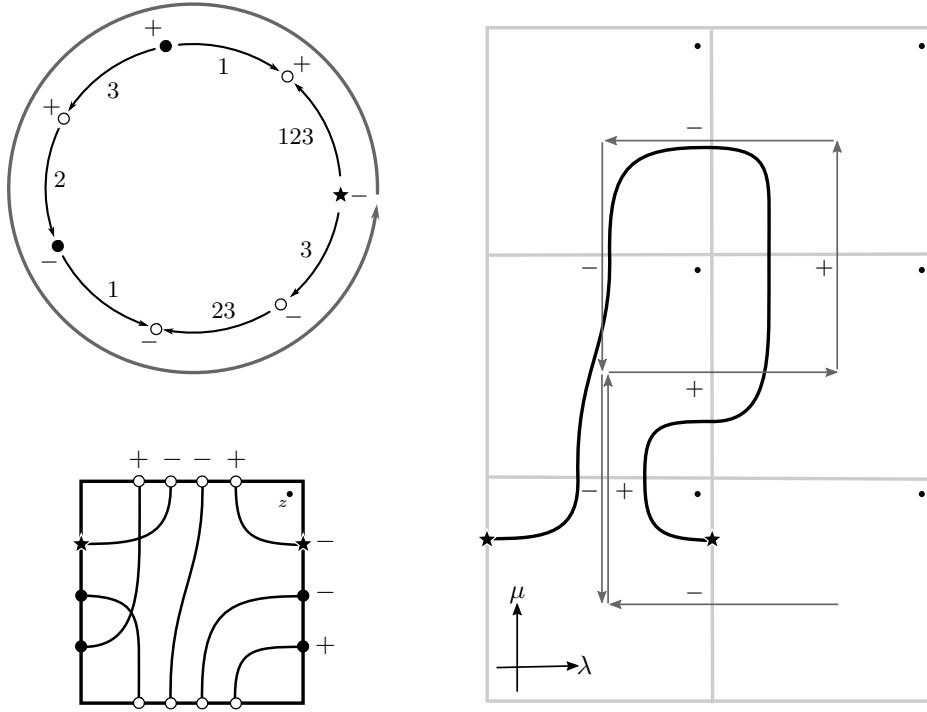}
\caption{Building the curve associated with the right-hand trefoil complement in two different ways: On the one hand, a loop-type type D structure always admits an extension for which the train track formalism immediately yields an immersed curve. On the other, by recording the signs of the intersection of this curve with the bordered arcs, a word in the free-group on two elements is obtained. In this example, the word $\lambda^{-1}\mu\lambda\mu\lambda^{-1}\mu^{-2}$ is shown, for comparison with the associated curve, in the cover $\R^2\setminus\Z^2$ of the marked boundary of the trefoil complement associated with the preferred meridian-longitude basis $\{\mu,\lambda\}$. }
\label{fig:trefoil-construction}
\end{figure}

\begin{proposition}
The two constructions are equivalent: if $M$ is a loop type manifold with type D structure described by an enhanced decorated graph $\lp$ that is valence 2, then $\cu(\lp)$ agrees with $\HFhat(M)$. (The same is true for any mod 2 graded extendable type D structure that can be described with an enhanced decorated valence 2 graph.) 
\end{proposition}
\begin{proof}As suggested by the example in Figure \ref{fig:trefoil-construction}, it is enough to identify the signs on the vertices with the intersection between the $\alpha$ and $\beta$ curves and the (oriented) immersed curve $\HFhat(M)$. \end{proof}

Let \(\barT_M\) be the cover of \(T_M = \partial M \setminus z\) associated with the kernel of the composition $\pi_1(\partial M\setminus z)\to\pi_1(\partial M)\to H_1(\partial M) \to H_1(M).$ 
When $H_1(M) \cong \Z$ this covering space can be identified with an infinite cylinder, with the lift of $z$ covered by an integer's worth of points. There is a natural lift of $\HFhat(M)$ to \(\barT_M\), which we denote by  $\HFhat(M,\spin_0)$. (Here \(\spin_0\) is the unique \(\spinc\) structure on \(M\).)  Some simple examples are shown in Figure \ref{fig:first-4-knots}. These examples follow quickly from the knot Floer homology together with the the conversion from this invariant to bordered invariants, given in \cite[Chapter 11]{LOT}. More generally, applying the work of Petkova \cite{Petkova2013}, many more examples are provided by thin knots.

\begin{proposition}If $M_K$ is the complement of a thin knot $K$ in the three-sphere, then $\HFhat(M_K)$ is loop type and is  determined by the Alexander polynomial and signature of $K$. \qed\end{proposition} 

For complements of thin knots, the immersed curves $\HFhat(M_K)$ are rather simple. One component winds around the lifts of \(z\) in a manner analogous to the curve for the trefoil complement, but with total height \(\sigma(K)\). All the other components are figure eights linking two adjacent lifts of $z$.

\begin{figure}[t]
\includegraphics[scale=.8]{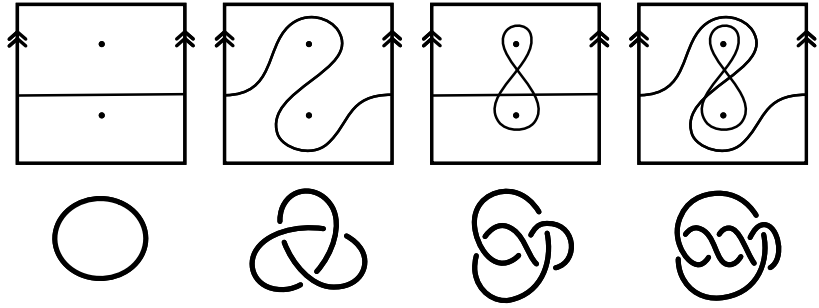}
\caption{Some small examples: $\HFhat(M)$ when $M$ is the complement of the trivial knot, the right-hand trefoil, the figure eight, and the three-twist knot}
\label{fig:first-4-knots}
\end{figure}

\subsection{Loop calculus and graph manifolds}\label{sub:loop-calculus} The remainder of this section is devoted to describing further families of loop type manifolds. Toward this end, we review some notation for loops from \cite{HW}. We start with a valence 2 decorated graph satisfying the vertex condition above. Breaking this graph along $\bullet$-vertices creates segments of certain types, and we record a loop as a cyclic word in letters representing these segments. These come in two families, according to those which are stable and unstable relative to the Dehn twist taking the bordered manifold $(M,\alpha,\beta)$ to $(M,\alpha,\alpha+\beta)$, and are described in Figure \ref{fig:puzzle-stable} and Figure \ref{fig:puzzle-unstable}, respectively. Reading a loop with a fixed orientation, these segments may appear forward or backwards; we use a bar to denote backward oriented segments. For example, the cyclic words $(d_1 d_2 d_3)$ and $(\bar d_3 \bar d_2 \bar d_1)$ represent the same loop, read with different orientations. Either cyclic word suffices to define the loop, but recall that fixing an orientation of the loop is equivalent to fixing the $\Ztwo$ grading. To keep track of this grading information, we will denote $\CFD$ by a collection of these cyclic words representing \emph{oriented} loops. Recall that since the $\Ztwo$ grading is only a relative grading in each spin$^c$ structure, only the relative orientations among loops in the same spin$^c$ structure are well defined. The extendability condition places constraints on how these segments can fit together, which is indicated by the puzzle piece ends in the figures.

\begin{figure}[ht]
\labellist
\pinlabel {\begin{tikzpicture}[>= stealth', scale=0.7]
\node at (0,0) {$\bullet$}; \node at (1.5,0) {$\circ$};  \node at (3,0) {$\circ$};  \node at (4.5,0) {$\bullet$};
\draw[->, thick, shorten >=0.1cm,shorten <=0.1cm]  (0,0) to (1.5,0);
\draw[->, thick, shorten >=0.1cm,shorten <=0.1cm, dashed]  (1.5,0) to (3,0);
\draw[->, thick, shorten >=0.1cm,shorten <=0.1cm] (3,0) to (4.5,0);
\node at (0.75,0.2) {$\scriptstyle{3}$};\node at (2.25,0.2) {$\scriptstyle{23}$};\node at (3.75,0.2) {$\scriptstyle{2}$};
\end{tikzpicture}} at 76.5 32 
\pinlabel {$\underbrace{\phantom{aaaaaa}}_k$} at 76.5 18
\pinlabel {$a_{k}$} at 72 -12
\endlabellist
\includegraphics[scale=0.7]{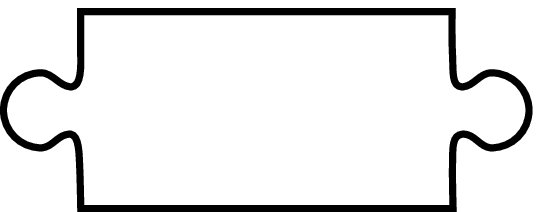}
\qquad\qquad\qquad
\labellist
\pinlabel {\begin{tikzpicture}[>= stealth', scale=0.7]
 \node at (1.5,0) {$\circ$};  \node at (3,0) {$\circ$}; 
 \draw[->, thick, shorten >=0.1cm,shorten <=0.3cm]  (0,0) to (1.5,0);
\draw[->, thick, shorten >=0.1cm,shorten <=0.1cm, dashed]  (1.5,0) to (3,0);
\draw[<-, thick, shorten >=0.3cm,shorten <=0.1cm] (3,0) to (4.5,0);
\node at (0.75,0.2) {$\scriptstyle{123}$};\node at (2.25,0.2) {$\scriptstyle{23}$};\node at (3.75,0.2) {$\scriptstyle{1}$};
\end{tikzpicture}} at 76.5 32  
\pinlabel {$\underbrace{\phantom{aaaaaa}}_k$} at 76.5 18
\pinlabel {$b_{k}$} at 72 -12
\endlabellist
\includegraphics[scale=0.7]{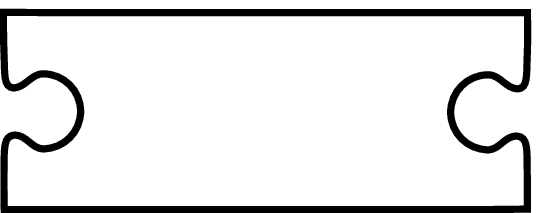}
\vspace*{5pt}
\caption{Stable pieces in standard notation.}\label{fig:puzzle-stable}
\end{figure}

\begin{figure}[ht]
\labellist
\pinlabel {\begin{tikzpicture}[>= stealth', scale=0.7]
 \node at (1.5,0) {$\circ$};  \node at (3,0) {$\circ$}; \node at (4.5,0) {$\bullet$};
\draw[->, thick, shorten >=0.1cm,shorten <=0.3cm]  (0,0) to (1.5,0);
\draw[<-, thick, shorten >=0.1cm,shorten <=0.1cm, dashed]  (1.5,0) to (3,0);
\draw[<-, thick, shorten >=0.1cm,shorten <=0.1cm] (3,0) to (4.5,0);
\node at (0.75,0.25) {$\scriptstyle1$};\node at (2.25,0.25) {$\scriptstyle{23}$};\node at (3.75,0.25) {$\scriptstyle3$};
\end{tikzpicture}} at 79.5 32  
\pinlabel {$\underbrace{\phantom{aaaaaa}}_k$} at 76 18
\pinlabel {$d_{-k}$} at 72 -12
\endlabellist
\includegraphics[scale=0.7]{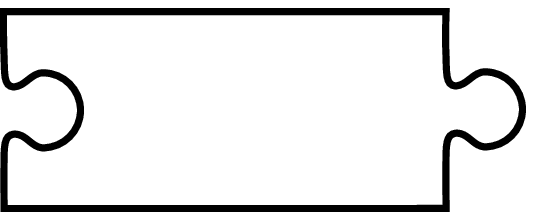}
\qquad
\labellist
\pinlabel {\begin{tikzpicture}[>= stealth', scale=0.7]
 \node at (4.5,0) {$\bullet$};
\draw[->, thick, shorten >=0.1cm,shorten <=0.3cm]  (0,0) to (4.5,0);
\node at (2.25,0.25) {$\scriptstyle{12}$};
\end{tikzpicture}} at 79.5 32  
\pinlabel {$d_{0}$} at 72 -12
\endlabellist
\includegraphics[scale=0.7]{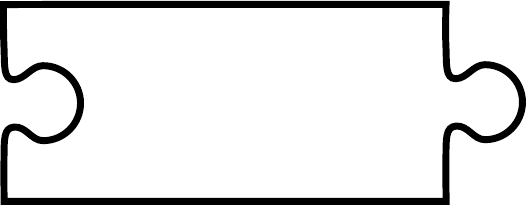}
\qquad
\labellist
\pinlabel {\begin{tikzpicture}[>= stealth', scale=0.7]
\node at (1.5,0) {$\circ$};  \node at (3,0) {$\circ$};  \node at (4.5,0) {$\bullet$};
\draw[->, thick, shorten >=0.1cm,shorten <=0.3cm]  (0,0) to (1.5,0);
\draw[->, thick, shorten >=0.1cm,shorten <=0.1cm, dashed]  (1.5,0) to (3,0);
\draw[->, thick, shorten >=0.1cm,shorten <=0.1cm] (3,0) to (4.5,0);
\node at (0.85,0.25) {$\scriptstyle{123}$};\node at (2.2,0.25) {$\scriptstyle{23}$};\node at (3.7,0.25) {$\scriptstyle2$};
\end{tikzpicture}} at 79.5 32  
\pinlabel {$\underbrace{\phantom{aaaaaa}}_k$} at 76 18
\pinlabel {$d_{k}$} at 72 -12
\endlabellist
\includegraphics[scale=0.7]{figures/in-out} 
\vspace*{5pt}
\caption{Unstable pieces in standard notation: the effect of the Dehn twist taking $\CFD(M,\alpha,\beta)$ to $\CFD(M,\alpha,\alpha+\beta)$ takes $d_k$ to $d_{k+1}$ Note that, in a variant of this notation used in \cite{HW}, we can avoid nonpositive subscripts by introducing two new letters, with $e = d_0$ and $c_k = \bar d_{-k}$.}\label{fig:puzzle-unstable}
\end{figure}

This machinery is particularly well suited for the study of graph manifolds, making it relatively easy to calculate the curve-set $\HFhat(M)$. Following \cite[Section 6]{HW}, a (bordered) graph manifold can be constructed from solid tori using three operations, all of which admit nice descriptions in terms of their action on the puzzle piece components of a loop. The operations $\mathcal{E}$ and $\mathcal{T}$ amount to reparametrizing the boundary; the only topologically significant operation is the \emph{merge} operation $\mathcal{M}$, which takes two manifolds with parametrized torus boundary $M_1$ and $M_2$ and produces a new manifold $\mathcal{M}(M_1, M_2)$ by gluing $M_1$ and $M_2$  to two boundary components of $\mathcal{P} \times S^1$, where $\mathcal{P}$ is $S^2$ with three disks removed (the particular gluing is determined by boundary parametrizations on $\partial M_1$ and $\partial M_2$). The following is a slight reformulation of \cite[Lemma 6.5]{HW}:

\begin{lemma}\label{lem:merge}
Suppose $M_1$ and $M_2$ are loop type manifolds equipped with boundary parametrizations. The manifold $\mathcal{M}(M_1, M_2)$ is loop type if for at least one $i \in \{1,2\}$, $M_i$ is simple loop type and the slope in $\partial M_i$ which glues to the fiber slope of $\mathcal{P} \times S^1$ is in $\mathcal{L}^\circ_{M_i}$. If this holds for both $i$, then $\mathcal{M}(M_1, M_2)$ is in fact simple loop type.
\end{lemma}

The following is an immediate consequence:

\begin{proposition}\label{prop:loop-type-graph-mfds}
A graph manifold $M$ with torus boundary is loop type if every component of the JSJ decomposition contains at most two boundary components.
\end{proposition}
\begin{proof}
We induct on the number of JSJ components. If there is only one, then $M$ is Seifert fibered with torus boundary and thus simple loop type. Otherwise, let $N$ be the JSJ component containing $\partial M$ and let $M' = M \setminus N$. $N$ is Seifert fibered with two boundary components; it can be obtained by gluing a Seifert fibered manifold $N'$ with one boundary component to $\mathcal{P} \times S^1$, gluing fiber slope to fiber slope. Thus $M$ can be viewed as $\mathcal{M}(M', N')$. $N'$ is simple loop type, the fiber slope is in $\mathcal{L}^\circ_{N'}$, and $M'$ is loop type by induction, so by Lemma \ref{lem:merge} $M$ is also loop type.
\end{proof}

%\subsection{Loop-type integer homology solid tori}\label{ex:loop-type-ZHS}
Proposition \ref{prop:loop-type-graph-mfds} provides a large family of loop type manifolds, many of which do not have multiple L-space fillings (that is, are not simple loop type). In fact, as the following example demonstrates, many do not have even a single L-space filling. Let $M$ be the graph manifold with boundary determined by the plumbing tree in Figure \ref{fig:loop-type-ZHS}. $M$ has two JSJ pieces, one with two boundary components (counting $\partial M$) and the other with one boundary component. By Proposition \ref{prop:loop-type-graph-mfds}, $M$ is loop type. Using the algorithm described in \cite{HW}, it is not difficult to compute $\CFD(M, \alpha, \beta)$ where $\alpha$ and $\beta$ are a fiber and a curve in the base surface, respectively, of the $S^1$ corresponding to the boundary vertex; the result is a single loop. Using the loop notation of \cite{HW}, this invariant can be represented as follows:
%$$(e\bar b_1 c_1 c_1 \bar a_1 e \bar b_2 a_1 b_1 \bar a_2 \bar b_1 c_2 \bar a_1 \bar c_1 \bar b_1 c_2 \bar a_1 \bar c_1 \bar b_1 c_2 \bar a_1 \bar b_2 a_1 b_1 \bar a_2)$$
$$(\bar a_1 d_2 e \bar b_1 c_1 \bar a_1 e d_1 \bar b_1 \bar e \bar a_1 d_1 d_1 \bar b_1 \bar e \bar a_1 d_1 d_1 \bar b_1 \bar e \bar a_1 d_1 e \bar b_1 c_1 \bar a_1 e d_2 \bar b_1 \bar d_1) $$
The corresponding curve $\curves M$ is shown (lifted to the plane) in Figure \ref{fig:loop-type-ZHS}. We see that there are no L-space fillings, since for any slope $\frac{p}{q}$ there is a straight line of slope $\frac{p}{q}$ which is in minimal position with $\curves M$ and intersects $\curves M$ multiple times. (Similar examples of such manifolds are also described in \cite{SR15}). The fact that there is only one loop in $\CFD(M,\alpha,\beta)$ reflects the fact that in this example $M$ is an integer homology solid torus. 
%The existence of loop type integer homology solid tori with no L-space filling means that Theorem \ref{thm:ZHS3} does extend the main theorem of \cite{HL} to new cases. 
It is not difficult to produce more examples of loop type integer homology tori with no L-space fillings. For example, an integer homology sphere graph manifold with at most two boundary tori on each JSJ component is loop type by Proposition \ref{prop:loop-type-graph-mfds} and if the plumbing tree has the tree in Figure \ref{fig:loop-type-ZHS} as a subtree it follows from \cite[Theorem 14]{HRW} that there are no L-space fillings.
\begin{figure}[t]
\begin{tikzpicture}
%\node (a) at (0,0) {$\bullet$};
\node (b) at (1,0) {$\bullet$};
\node (c) at (1,1) {$\bullet$};
\node (d) at (1,-1) {$\bullet$};
\node (e) at (2,0) {$\bullet$};
\node (f) at (3,0) {$\bullet$};
\node (g) at (3.3,1) {$\bullet$};
\node (h) at (3.3,-1) {$\bullet$};

%\node[above right = 1pt] at (a) {$0$};
\node[above right = 1pt] at (b) {$0$};
\node[above left] at (c) {$-2$};
\node[below left] at (d) {$-3$};
\node[above right = 1pt] at (e) {$7$};
\node[right = 2pt] at (f) {$1$};
\node[above right] at (g) {$2$};
\node[below right] at (h) {$3$};

\draw (1,1) -- (1,-1);
\draw (1,0) -- (3,0);
\draw (3.3,1) -- (3,0) -- (3.3,-1);
\draw[dashed] (1,0) -- (0,0);
\end{tikzpicture}\qquad\includegraphics[scale=0.55]{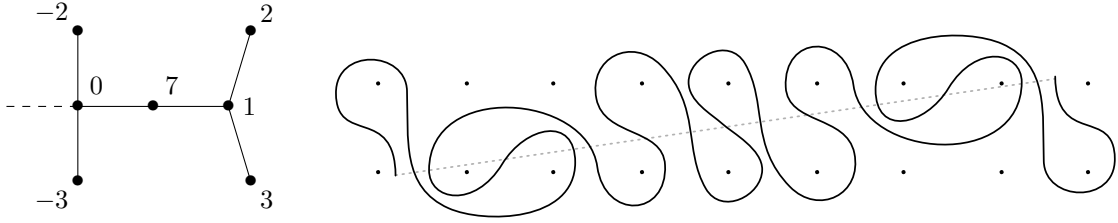}
\caption{The curve $\liftcurves{M}$, where $M$ is the graph manifold with plumbing tree shown. Note that, relative to the chosen basis, the longitude of $M$ (dotted line) has slope $1/7$. }
\label{fig:loop-type-ZHS}
\end{figure}

\subsection{Surface bundles}\label{sub:bundles} The examples discussed above are all rational homology solid tori; for an interesting class of examples with $b_1 > 1$ we consider products of once-punctured surfaces with $S^1$. Let $S_{g,1}$ denote the surface of genus $g$ with one boundary component, and let $M_g$ denote $S_{g,1}\times S^1$. To parametrize the boundary $\partial M_g = T^2$, let $\mu$ be a fiber $\{p\} \times S^1$ for $p\in \partial S_{g,1}$ and let $\lambda$ be $\partial S_{g,1}\times \{q\}$ for $q\in S^1$. We will compute $\CFD(M_g, \mu, \lambda)$.

 In the following example, all loops are in the same spin$^c$ structure, so the relative orientations are meaningful.

\begin{theorem}\label{thm:surface-bundle}
For $g\ge 0$, $M_g$ is a loop-type manifold, with bordered invariants consisting of loop-components of the form $d_0$ and $d_{2i}d_{-2i}$ for $0\le i \le g$. Specifically, when $g=0$, $\CFD(M_g,\mu,\lambda)$ consists of a single loop $(d_0)$ and, when $g>0$, this invariant has $2^g$ components of type $(d_0)$, $\frac{1}{2}\binom{2g}{g}-2^{g-1}$ components of type $(d_0d_0)$, and $\binom{2g}{g+i}$ components of type $(d_{2i}d_{-2i})$ for $1\le i \le g$, with the orientation of each $(d_{2i}d_{-2i})$ component reversed if $i$ is odd. \end{theorem}

\begin{proof}
The case when $g=0$ is immediate. To establish the result for positive genus we will induct on $g$, making use of the type DA bimodule $\mathcal{G}$ described in \cite[Section 5]{Hanselman2016}. This has the property $\CFD(M_{g+1},\mu,\lambda)\cong \mathcal{G}\boxtimes \CFD(M_g,\mu,\lambda)$, and can be explicitly calculated to show that $\mathcal{G}\cong \mathcal{G}_1\oplus\mathcal{G}_2\oplus\mathcal{G}_3$ where both $\mathcal{G}_2$ and $\mathcal{G}_3$ are the identity bimodule. A 
%complete 
list of operations describing $\mathcal{G}_1$ is given in Table \ref{tab:genus-adder}.

Write $\mathcal{G}(\lp)$ to denote the result of box-tensoring the corresponding type D structure for $\lp$ with $\mathcal{G}$. We fix the $\Ztwo$ grading so that the identity components preserve orientation; that is, so that $\mathcal{G}_2(\lp) = \lp$ and $\mathcal{G}_3(\lp) = \lp$. Using this choice, the generators $x_1$ and $x_3$ in $\mathcal{G}_1$ have grading 1. Applying the bimodule $\mathcal{G}_1$ to certain loops, we have that $\mathcal{G}_1(d_0)=( \bar d_{2}\bar d_{-2})$, $\mathcal{G}_1(d_0d_0)=(\bar d_{2}\bar d_{-2})(\bar d_{2}\bar d_{-2})$, and $\mathcal{G}_1(d_{2i}d_{-2i})=(\bar d_{2(i-1)}\bar d_{-2(i-1)})(\bar d_{2(i+1)}\bar d_{-2(i+1)})$ for $i>0$.

We compute $\mathcal{G}_1(d_{2i}d_{-2i})$ for $i>1$ explicitly -- leaving the remaining cases to the reader -- as follows:
\[
\begin{tikzpicture}[scale=1,>=stealth', thick] 
\draw [fill=lightgray,lightgray] (-0.25,0.75) rectangle (7.25,1.5);
\draw [fill=lightgray,lightgray] (-0.25,4.75) rectangle (7.25,5.5);
%\draw [fill=lightgray,lightgray] (0,2) -- (2,2) -- (0,3);
%\draw [fill=lightgray,lightgray] (5,4) -- (7,4) -- (7,3);
%vertices
\node at (0,6) {$x_1$};\node at (1,6) {$x_8$};\node at (2,6) {$x_8$};\node at (3,6) {$x_8$};\node at (4,6) {$x_8$};\node at (5,6) {$x_8$};\node at (6,6) {$x_8$};\node at (7,6) {$x_6$};
 \node at (0,5) {$\bu$}; \node at (1,5) {$\ci$}; \node at (2,5) {$\ci$}; \node at (3,5) {$\ci$}; \node at (4,5) {$\ci$}; \node at (5,5) {$\ci$}; \node at (6,5) {$\ci$};\node at (7,5) {$\bu$};
\node at (0,4) {$x_3$};\node at (1,4) {$x_7$};\node at (2,4) {$x_7$};\node at (3,4) {$x_7$};\node at (4,4) {$x_7$};\node at (5,4) {$x_7$};\node at (6,4) {$x_7$};\node at (7,4) {$x_5$};
 \node at (-1,3) {$x_2$}; \node at (0,3) {$x_4$}; \node at (7,3) {$x_2$}; \node at (8,3) {$x_4$};
 \node at (0,2) {$x_6$};\node at (1,2) {$x_8$};\node at (2,2) {$x_8$};\node at (3,2) {$x_8$};\node at (4,2) {$x_8$};\node at (5,2) {$x_8$};\node at (6,2) {$x_8$};\node at (7,2) {$x_1$};
 \node at (0,1) {$\bu$}; \node at (1,1) {$\ci$}; \node at (2,1) {$\ci$}; \node at (3,1) {$\ci$}; \node at (4,1) {$\ci$}; \node at (5,1) {$\ci$}; \node at (6,1) {$\ci$};\node at (7,1) {$\bu$};
 \node at (0,0) {$x_5$};\node at (1,0) {$x_7$};\node at (2,0) {$x_7$};\node at (3,0) {$x_7$};\node at (4,0) {$x_7$};\node at (5,0) {$x_7$};\node at (6,0) {$x_7$};\node at (7,0) {$x_3$};
 %dots 
 \node at (3.5,0) {$\cdots$}; \node at (3.5,1) {$\cdots$}; \node at (3.5,2) {$\cdots$}; \node at (3.5,4) {$\cdots$}; \node at (3.5,5) {$\cdots$}; \node at (3.5,6) {$\cdots$};
 %arrows for starting complexs
  \draw[->,shorten <= 0.125cm, shorten >= 0.125cm] (0,5)--(1,5); \draw[->,shorten <= 0.125cm, shorten >= 0.125cm] (1,5)--(2,5);\draw[->,shorten <= 0.125cm, shorten >= 0.125cm] (2,5)--(3,5);
  \draw[->,shorten <= 0.125cm, shorten >= 0.125cm] (4,5)--(5,5); \draw[->,shorten <= 0.125cm, shorten >= 0.125cm] (5,5)--(6,5);\draw[->,shorten <= 0.125cm, shorten >= 0.125cm] (6,5)--(7,5);
 \draw[->,shorten <= 0.125cm, shorten >= 0.125cm] (0,1)--(1,1); \draw[->,shorten <= 0.125cm, shorten >= 0.125cm] (1,1)--(2,1);\draw[->,shorten <= 0.125cm, shorten >= 0.125cm] (2,1)--(3,1);
  \draw[->,shorten <= 0.125cm, shorten >= 0.125cm] (4,1)--(5,1); \draw[->,shorten <= 0.125cm, shorten >= 0.125cm] (5,1)--(6,1);\draw[<-,shorten <= 0.125cm, shorten >= 0.125cm] (6,1)--(7,1);
%labels for starting complexes
  \node at (0.5,5.25) {$\scriptstyle{123}$}; \node at (1.5,5.25) {$\scriptstyle{23}$};\node at (2.5,5.25) {$\scriptstyle{23}$};\node at (4.5,5.25) {$\scriptstyle{23}$};\node at (5.5,5.25) {$\scriptstyle{23}$};\node at (6.5,5.25) {$\scriptstyle{2}$};
 \node at (0.5,1.25) {$\scriptstyle{3}$}; \node at (1.5,1.25) {$\scriptstyle{23}$};\node at (2.5,1.25) {$\scriptstyle{23}$};\node at (4.5,1.25) {$\scriptstyle{23}$};\node at (5.5,1.25) {$\scriptstyle{23}$};\node at (6.5,1.25) {$\scriptstyle{1}$};
 %other arrows
  \draw[->,shorten <= 0.25cm, shorten >= 0.25cm] (0,6)--(1,6);  \draw[->,shorten <= 0.25cm, shorten >= 0.25cm] (1,6)--(2,6);  \draw[->,shorten <= 0.25cm, shorten >= 0.25cm] (2,6)--(3,6); \draw[->,shorten <= 0.25cm, shorten >= 0.25cm] (4,6)--(5,6);  \draw[->,shorten <= 0.25cm, shorten >= 0.25cm] (5,6)--(6,6);  \draw[->,shorten <= 0.25cm, shorten >= 0.25cm] (6,6)--(7,6);
\draw[->,shorten <= 0.25cm, shorten >= 0.25cm] (0,4)--(1,4);  \draw[->,shorten <= 0.25cm, shorten >= 0.25cm] (1,4)--(2,4);  \draw[->,shorten <= 0.25cm, shorten >= 0.25cm] (2,4)--(3,4); \draw[->,shorten <= 0.25cm, shorten >= 0.25cm] (4,4)--(5,4);  \draw[->,shorten <= 0.25cm, shorten >= 0.25cm] (5,4)--(6,4);  \draw[->,shorten <= 0.25cm, shorten >= 0.25cm] (6,4)--(7,4);
\draw[->,shorten <= 0.25cm, shorten >= 0.25cm] (0,2)--(1,2);  \draw[->,shorten <= 0.25cm, shorten >= 0.25cm] (1,2)--(2,2);  \draw[->,shorten <= 0.25cm, shorten >= 0.25cm] (2,2)--(3,2); \draw[->,shorten <= 0.25cm, shorten >= 0.25cm] (4,2)--(5,2);  \draw[->,shorten <= 0.25cm, shorten >= 0.25cm] (5,2)--(6,2);  \draw[<-,shorten <= 0.25cm, shorten >= 0.25cm] (6,2)--(7,2);
\draw[->,shorten <= 0.25cm, shorten >= 0.25cm] (0,0)--(1,0);  \draw[->,shorten <= 0.25cm, shorten >= 0.25cm] (1,0)--(2,0);  \draw[->,shorten <= 0.25cm, shorten >= 0.25cm] (2,0)--(3,0); \draw[->,shorten <= 0.25cm, shorten >= 0.25cm] (4,0)--(5,0);  \draw[->,shorten <= 0.25cm, shorten >= 0.25cm] (5,0)--(6,0);  \draw[<-,shorten <= 0.25cm, shorten >= 0.25cm] (6,0)--(7,0);
 %other labels
 \node at (0.5,6.25) {$\scriptstyle{123}$}; \node at (1.5,6.25) {$\scriptstyle{23}$};\node at (2.5,6.25) {$\scriptstyle{23}$};\node at (4.5,6.25) {$\scriptstyle{23}$};\node at (5.5,6.25) {$\scriptstyle{23}$};\node at (6.5,6.25) {$\scriptstyle{23}$};
  \node at (0.5,4.25) {$\scriptstyle{123}$}; \node at (1.5,4.25) {$\scriptstyle{23}$};\node at (2.5,4.25) {$\scriptstyle{23}$};\node at (4.5,4.25) {$\scriptstyle{23}$};\node at (5.5,4.25) {$\scriptstyle{23}$};\node at (6.5,4.25) {$\scriptstyle{\varnothing}$};
    \node at (0.5,2.25) {$\scriptstyle{\varnothing}$}; \node at (1.5,2.25) {$\scriptstyle{23}$};\node at (2.5,2.25) {$\scriptstyle{23}$};\node at (4.5,2.25) {$\scriptstyle{23}$};\node at (5.5,2.25) {$\scriptstyle{23}$};\node at (6.5,2.25) {$\scriptstyle{1}$};
  \node at (0.5,0.25) {$\scriptstyle{23}$}; \node at (1.5,0.25) {$\scriptstyle{23}$};\node at (2.5,0.25) {$\scriptstyle{23}$};\node at (4.5,0.25) {$\scriptstyle{23}$};\node at (5.5,0.25) {$\scriptstyle{23}$};\node at (6.5,0.25) {$\scriptstyle{1}$};
  %final arrows
  \draw[->,shorten <= 0.25cm, shorten >= 0.25cm] (0,6).. controls (-1,5) and (-1,4) .. (-1,3);   \draw[->,shorten <= 0.25cm, shorten >= 0.25cm] (-1,3).. controls (-1,2) and (-1,1) .. (0,0); 
 \draw[->,shorten <= 0.25cm, shorten >= 0.25cm] (7,6).. controls (8,5) and (8,4) .. (8,3);   \draw[->,shorten <= 0.25cm, shorten >= 0.25cm] (8,3).. controls (8,2) and (8,1) ..(7,0); 
 \draw[->,shorten <= 0.25cm, shorten >= 0.25cm] (0,2)--(0,3);  \draw[->,shorten <= 0.25cm, shorten >= 0.25cm] (0,3) .. controls (1,3) and (1.5,2.75).. (2,2); 
   \draw[->,shorten <= 0.25cm, shorten >= 0.25cm] (7,3)--(7,4);  \draw[->,shorten <= 0.25cm, shorten >= 0.25cm] (5,4)--(7,3); 
  \draw[->,shorten <= 0.25cm, shorten >= 0.25cm] (4,4).. controls (4,3) and (7,3.5) .. (7,2);  \draw[->,shorten <= 0.25cm, shorten >= 0.25cm] (0,4).. controls (0,3) and (3,3.5) .. (3,2); 
  %final labels
  \node at (-0.75,5.4) {$\scriptstyle{3}$};  \node at (-0.75,0.6) {$\scriptstyle{23}$};
    \node at (7.75,5.4) {$\scriptstyle{23}$};  \node at (7.75,0.6) {$\scriptstyle{2}$};
       \node at (-0.25,2.5) {$\scriptstyle{23}$};  \node at (1.6,2.75) {$\scriptstyle{\varnothing}$};\node at (2.65,3.1) {$\scriptstyle{3}$};
          \node at (7.25,3.5) {$\scriptstyle{23}$};  \node at (6.25,3.6) {$\scriptstyle{\varnothing}$};\node at (5,3) {$\scriptstyle{2}$};
\end{tikzpicture}
\]
The shaded boxes contain $d_{2i}$ (above) and $d_{-2i}$ (below) for book-keeping purposes. Note that the $\bullet$-generators on the left of each shaded box are identified in the loop $(d_{2i}d_{-2i})$, as are the two rightmost $\bullet$-generators. The generators $x_i$ correspond to those in Table \ref{tab:genus-adder}, interpreted as tensored with the generator immediately above or below in the shaded box. Each $\circ$-generator in $(d_{2i}d_{-2i})$ pairs with both $x_7$ and $x_8$ in the tensor product, while each $\bullet$-generator pairs with the six generators $x_1, \ldots, x_6$. The inner loop gives $d_{2(i-1)}d_{2(i-1)}$ (after reducing the loop by cancelling unlabelled edges) and the outer loop gives $d_{2(i+1)}d_{2(i+1)}$. Since the generators $x_1$ and $x_3$ of $\mathcal{G}_1$ have grading 1, the $\bullet$-generators of the resulting two loops have the opposite sign from the $\bullet$-generators of $(d_{2i}d_{-2i})$; equivalently, the new loops have the opposite orientation.

As a result, incorporating the two identity bimodules in $\mathcal{G}$, we conclude that
\begin{align*}
& \mathcal{G}(d_0)=(\bar d_{2}\bar d_{-2})(d_0)(d_0) \\
& \mathcal{G}(d_0d_0)=(\bar d_{2}\bar d_{-2})(\bar d_{2}\bar d_{-2})(d_0d_0)(d_0d_0)\\
&\mathcal{G}(d_{2i}d_{-2i}) =(\bar d_{2(i-1)}\bar d_{-2(i-1)})(\bar d_{2(i+1)}\bar d_{-2(i+1)})(d_{2i}d_{-2i})(d_{2i}d_{-2i})
\end{align*}
where $i>0$. From this it is immediate that the number of $(d_0)$ components in $\CFD(M_g,\mu,\lambda)$ must be $2^g$. Let $m(g)$ denote the number of $(d_0d_0)$ components, $n_0(g)=2^g+2m(g)$, and $n_i(g)$ be the number of $(d_{2i}d_{-2i})$ components when $1\le i\le g$ (where the orientation is reversed if $i$ is odd). By inspection,
\[n_i(g)=n_{i-1}(g)+2n_i(g)+n_{i+1}(g)\] for $i>0$, which is precisely the recursion satisfied for $\binom{2g}{g+i}$. It remains to check that $n_0(g)=\binom{2g}{g}$, but as $m(g)=2m(g-1)+n_1(g-1)$ we have \[n_0(g) = 2^g+2m(g) = 2\cdot 2^{g-1} + 2(2m(g-1)+n_1(g-1)) = 2n_0(g-1) + 2n_1(g-1)\]
as required. 
\end{proof}

\begin{table}[t]
\begin{tabular}{|l|l|l|}
\hline
$\delta_1^2(x_1,\rho_{1})=\rho_{1}\otimes x_8$&$\delta_1^1(x_1)=\rho_3\otimes x_2 \ {}^\spadesuit$ & $\delta_1^4(x_3,\rho_{123},\rho_{23},\rho_{2})= \rho_{12}\otimes x_1$ \\
$\delta_1^2(x_1,\rho_{12})=\rho_{123}\otimes x_6 \ {}^\spadesuit$&$\delta_1^1(x_2)=\rho_{23}\otimes x_5 \ {}^\spadesuit$ & $\delta_1^4(x_3,\rho_{3},\rho_{2},\rho_{12})= \rho_{123}\otimes x_5 \ {}^\clubsuit$ \\
$\delta_1^2(x_1,\rho_{123})=\rho_{123}\otimes x_8$&$\delta_1^1(x_4)=\rho_{2}\otimes x_3 \ {}^\spadesuit$ & $\delta_1^4(x_3,\rho_{3},\rho_{23},\rho_{2})= \rho_{3}\otimes x_6 \ {}^\clubsuit$ \\
$\delta_1^2(x_3,\rho_{1})=\rho_{1}\otimes x_7$&$\delta_1^1(x_6)=\rho_{23}\otimes x_4 \ {}^\spadesuit$ & $\delta_1^4(x_3,\rho_{3},\rho_{23},\rho_{23})= \rho_{3}\otimes x_8$ \\
\cline{2-2}
$\delta_1^2(x_3,\rho_{12})=\rho_{1}\otimes x_5 \ {}^\spadesuit$&$\delta_1^3(x_2,\rho_3,\rho_2)=\rho_{2}\otimes x_1 \ {}^\clubsuit$ & $\delta_1^4(x_5,\rho_{3},\rho_{23},\rho_{2})= \rho_{2}\otimes x_1 \ {}^\clubsuit$ \\ 
$\delta_1^2(x_3,\rho_{123})=\rho_{123}\otimes x_7$&$\delta_1^3(x_3,\rho_{123},\rho_2)=\rho_{1}\otimes x_2$ & $\delta_1^4(x_7,\rho_{23},\rho_{23},\rho_{2})= \rho_{2}\otimes x_1$ \\
\cline{3-3}
$\delta_1^2(x_5,\rho_{3})=\rho_{23}\otimes x_7$&$\delta_1^3(x_3,\rho_3,\rho_2)=\rho_{3}\otimes x_4 \ {}^\clubsuit$ & 
$\delta_1^5(x_3,\rho_3,\rho_2,\rho_1,\rho_2)= \rho_1\otimes x_2\ {}^\clubsuit$
\\
$\delta_1^2(x_6,\rho_{3})= x_8$&$\delta_1^3(x_2,\rho_3,\rho_2)= x_6\ {}^\clubsuit$ &  
$\delta_1^5(x_3,\rho_3,\rho_2,\rho_{123},\rho_2)= \rho_{123}\otimes x_2\ {}^\clubsuit$
\\
\cline{3-3}
$\delta_1^2(x_7,\rho_{2})= x_5$&$\delta_1^3(x_4,\rho_3,\rho_{23})= x_8$&  
$\delta_1^6(x_3,\rho_3,\rho_2,\rho_1,\rho_{23},\rho_2)=\rho_{12}\otimes x_1 \ {}^\clubsuit$
\\
\cline{3-3}
$\delta_1^2(x_7,\rho_{23})=\rho_{23}\otimes x_7$&$\delta_1^3(x_5,\rho_3,\rho_2)= x_2 \ {}^\clubsuit$ &
$\delta_1^7(x_3,\rho_3,\rho_2,\rho_3,\rho_2,\rho_1,\rho_2)=\rho_{123}\otimes x_2 \ {}^\clubsuit$ 
 \\
$\delta_1^2(x_8,\rho_{2})=\rho_{23}\otimes x_6$&$\delta_1^3(x_7,\rho_{23},\rho_2)= x_2$ &\\
$\delta_1^2(x_8,\rho_{23})=\rho_{23}\otimes x_8$&&\\
\hline
\end{tabular}
\vskip0.1in
\caption{The operations in the type DA bimodule $\mathcal{G}_1$. The bimodule has generators \(x_1\ldots x_8\), of which \(x_1\) and \(x_3\) are in the \((\ci,\bu)\) idempotent, \(x_2,x_4,x_5\) and \(x_6\) are in the \((\ci,\ci)\) idempotent, and \(x_7\) and \(x_8\) are in the \((\bu,\ci)\) idempotent. The flags ${}^\spadesuit$ single out those operations relevant to tensoring with loops $(d_0)$ or $(d_0d_0)$ while the flags  ${}^\clubsuit$ highlight those operations that are {\em not} relevant to calculations involving loops of the form $(d_{2i}d_{-2i})$. %Operations $\delta_1^k$ for $k>4$ have been omitted as they are not needed for our calculation; note that $\delta_1^k=0$ for $k>7$.
}\label{tab:genus-adder}
\end{table}

The curves associated with the loops $(d_{2i}d_{-2i})$ are relatively simple. For example, Figure \ref{fig:corollary-example} shows a component of $\curves{M_g}$ (for $g \ge 2$) corresponding to the loop $(d_4 d_{-4})$. It is easier to picture the lifted curves $\curves{M_g, \spin_0}$ in the covering space $\barT_M$. (Here \(\spin_0\) is the unique torsion \(\spinc\) structure on \(M\); the invariant for all other \(\spinc\) structures is empty.) $\curves{M_g, \spin}$ consists of curves from a particular family of curves, $\{\tilde\gamma_i\}_{i\in\mathbb{N}}$, which are depicted in Figure \ref{fig:surface-bundle-curves}. Recall that choosing a lift $\curves{M_g,\spin}$ of the curve set $\curves{M_g}$ encodes additional grading information. Computing the spin$^c$ gradings under the tensor product with the bimodule $\mathcal{G}$ reveals that all components of $\curves{M_g,\spin}$ are centered vertically on the same horizontal line, as in the figure. The curve $\tilde\gamma_i$ corresponds to the loop $(d_{2i}d_{-2i})$. Note that the curve $\tilde\gamma_i$ is obtained from $\tilde\gamma_{i-1}$ by sliding the peak up one unit and sliding the valley down one unit. For convenience we will allow the subscript $i$ of $\tilde\gamma_i$ to be negative, with the convention that $\tilde\gamma_i$ and $\tilde\gamma_{-i}$ represent the same curve. On the level of curves, the bimodule $\mathcal{G}_1$ applied to $\tilde\gamma_i$ produces two curves $\tilde\gamma_{i-1}$ and $\tilde\gamma_{i+1}$.

As an immediate consequence of Theorem \ref{thm:surface-bundle}, we can establish Theorem \ref{crl:surface-bundle} concerning the Floer homology of $Y_g$, the closed three-manifold resulting from the product of a closed, orientable surface of genus $g$ with $S^1$. Namely, the total dimension of $\HFhat(Y_g)$ is $2^g + \binom{2g}{g} + 2\sum_{i=1}^g (2i-1)\binom{2g}{g+i}$.

%\begin{corollary}\label{crl:surface-bundle}
%For $g\ge 0$, the total dimension of $\HFhat(Y_g)$ is $2^g + \binom{2g}{g} + 2\sum_{i=1}^g (2i-1)\binom{2g}{g+i}$.
%\end{corollary}
\begin{proof}[Proof of Theorem \ref{crl:surface-bundle}]
This follows from the minimal intersection of $\curves{M_g}$ with a horizontal line, which calculates $\HFhat(Y_g)$ via Theorem \ref{thm:pairing}. The reader can verify that each of component $(d_0)$ or $(d_0d_0)$ contributes two intersection points -- in both cases admissibility forces the intersection. According to Theorem \ref{thm:surface-bundle}, the total resulting contribution for these curves must be $2^g+\binom{2g}{g}$. The contribution of curves $(d_{2i}d_{-2i})$ when $i>0$ is summarized in Figure \ref{fig:corollary-example}: Notice that the number of vertical lines is $2(2i-1)$, hence the total contribution from $(d_{2i}d_{-2i})$ components is $2\sum_{i=1}^g (2i-1)\binom{2g}{g+i}$.
\end{proof}

\begin{figure}[t]
\begin{tikzpicture}[scale=1,>=stealth', thick] 
\def \radius {2cm} \def \outer {2.25cm}
 \node at ({360/(10) * (1- 1)}:\radius) {$\bu$};
  \node at ({360/(10) * (2- 1)}:\radius) {$\ci$};
  \node at ({360/(10) * (3 -1)}:\radius) {$\ci$};
   \node at ({360/(10) * (4 -1)}:\radius) {$\ci$};
    \node at ({360/(10) * (5 -1)}:\radius) {$\ci$};
     \node at ({360/(10) * (6 -1)}:\radius) {$\bu$};
      \node at ({360/(10) * (7 -1)}:\radius) {$\ci$};
        \node at ({360/(10) * (8 -1)}:\radius) {$\ci$};
          \node at ({360/(10) * (9 -1)}:\radius) {$\ci$};
            \node at ({360/(10) * (10 -1)}:\radius) {$\ci$};
  \draw[<-,shorten <= 0.125cm, shorten >= 0.125cm]({360/10 * (1 - 1)}:\radius) arc ({360/10 * (1- 1)}:{360/10 * (2-1)}:\radius);
 \draw[<-,shorten <= 0.125cm, shorten >= 0.125cm]({360/10 * (2 - 1)}:\radius) arc ({360/10 * (2- 1)}:{360/10 * (3-1)}:\radius);
 \draw[<-,shorten <= 0.125cm, shorten >= 0.125cm]({360/10 * (3 - 1)}:\radius) arc ({360/10 * (3- 1)}:{360/10 * (4-1)}:\radius);
 \draw[<-,shorten <= 0.125cm, shorten >= 0.125cm]({360/10 * (4 - 1)}:\radius) arc ({360/10 * (4- 1)}:{360/10 * (5-1)}:\radius);
 \draw[<-,shorten <= 0.125cm, shorten >= 0.125cm]({360/10 * (5 - 1)}:\radius) arc ({360/10 * (5- 1)}:{360/10 * (6-1)}:\radius);
 \draw[->,shorten <= 0.125cm, shorten >= 0.125cm]({360/10 * (6 - 1)}:\radius) arc ({360/10 * (6- 1)}:{360/10* (7-1)}:\radius);
 \draw[->,shorten <= 0.125cm, shorten >= 0.125cm]({360/10 * (7 - 1)}:\radius) arc ({360/10 * (7- 1)}:{360/10 * (8-1)}:\radius);
  \draw[->,shorten <= 0.125cm, shorten >= 0.125cm]({360/10 * (8 - 1)}:\radius) arc ({360/10 * (8- 1)}:{360/10 * (9-1)}:\radius);
   \draw[->,shorten <= 0.125cm, shorten >= 0.125cm]({360/10 * (9 - 1)}:\radius) arc ({360/10 * (9- 1)}:{360/10 * (10-1)}:\radius);
    \draw[<-,shorten <= 0.125cm, shorten >= 0.125cm]({360/10 * (10 - 1)}:\radius) arc ({360/10 * (10- 1)}:{360/10 * (11-1)}:\radius);
  \node at ({360/(10) * (1- (1/2))}:\outer) {$\scriptstyle{2}$};
  \node at ({360/(10) * (2-(1/2))}:\outer) {$\scriptstyle{23}$};
  \node at ({360/(10) * (3-(1/2))}:\outer) {$\scriptstyle{23}$};
  \node at ({360/(10) * (4-(1/2))}:\outer) {$\scriptstyle{23}$};
  \node at ({360/(10) * (5-(1/2))}:2.4) {$\scriptstyle{123}$};
  \node at ({360/(10) * (6-(1/2))}:\outer) {$\scriptstyle{3}$};
  \node at ({360/(10) * (7-(1/2))}:\outer) {$\scriptstyle{23}$};
   \node at ({360/(10) * (8-(1/2))}:\outer) {$\scriptstyle{23}$};
    \node at ({360/(10) * (9-(1/2))}:\outer) {$\scriptstyle{23}$};
     \node at ({360/(10) * (10-(1/2))}:\outer) {$\scriptstyle{1}$};
\end{tikzpicture}\qquad\qquad
\labellist
\endlabellist
\includegraphics[scale=0.8]{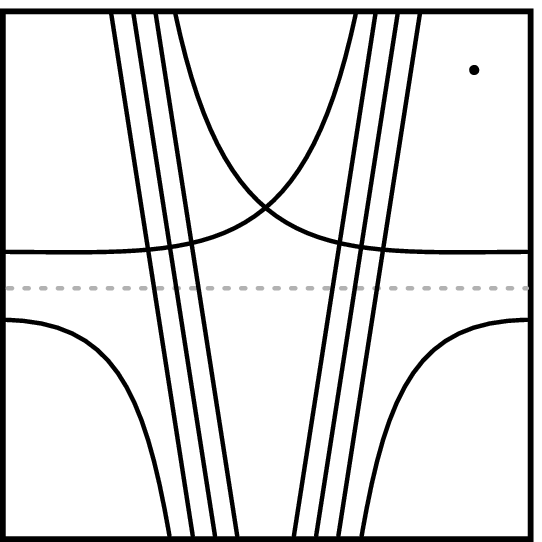}
\caption{The loop $(d_4d_{-4})$ (left) and the corresponding curve in $\Tp$ (right): Notice that the minimal intersection with a horizontal curve is $6=2(4-1)$. More generally, $(d_{2i}d_{-2i})$ gives rise to $2(2i-1)$ points of intersection, as in the proof of Theorem \ref{crl:surface-bundle}. }\label{fig:corollary-example}
\end{figure}

\begin{figure}[t]
\begin{overpic}[scale = 0.8]{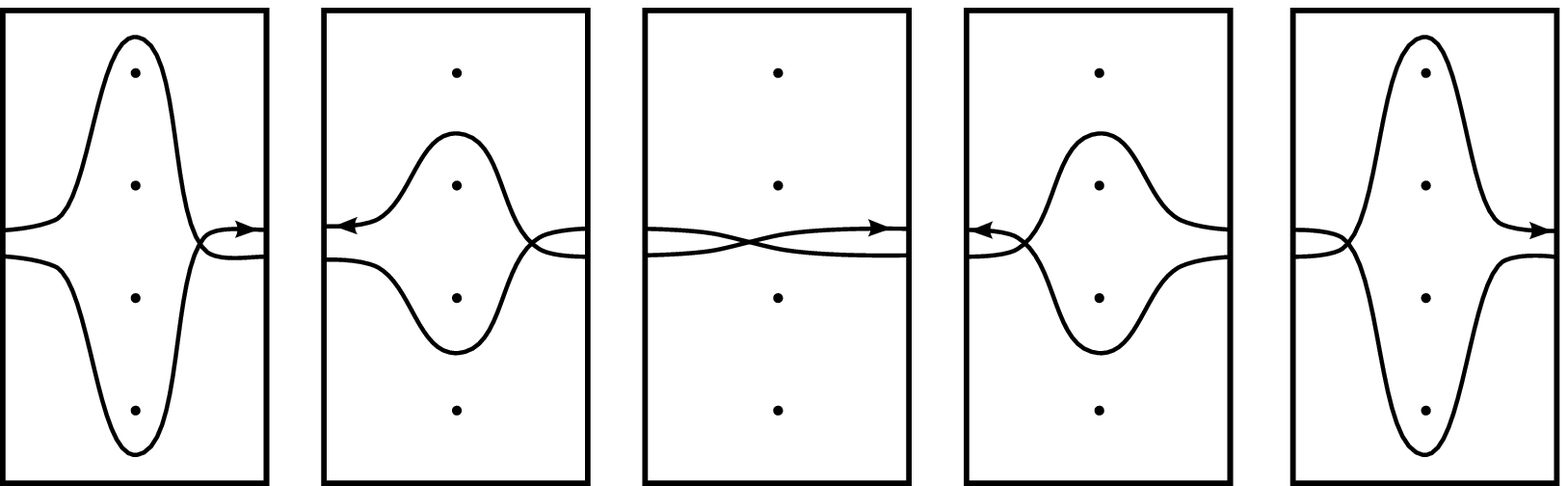}
%\put (-5,15) {$\cdots$}
\put (103,25) {$\cdots$}
\end{overpic}

%$\tilde\gamma_{-2}$ \hspace{2.1 cm} $\tilde\gamma_{-1}$ \hspace{2.1 cm} 
$\tilde\gamma_{0}$ \hspace{2.1 cm} $\tilde\gamma_{1}$ \hspace{2.1 cm} $\tilde\gamma_{2}$
\caption{Components of $\liftcurves{M_g}$ in $\barT_M$: The curve $\tilde\gamma_i$ corresponds to the loop $(d_{2i}d_{-2i})$.}
\label{fig:surface-bundle-curves}
\end{figure}

\subsection{Heegaard Floer homology solid tori}\label{sub:HFtori} L-spaces represent the class of closed three-manifolds whose Heegaard Floer homology is simplest-possible.
These rational homology spheres admit a characterization in the presence of an essential torus:

\begin{theorem}[{\cite[Theorem 14]{HRW}}]\label{thm:L-space-gluing} Let $Y$ be a closed toroidal three-manifold so that $Y=M_0\cup_h M_1$ where $M_i\ne D^2\times S^1$. Then $Y$ is an L-space if and only if $\sL_0^\circ\cup h(\sL_1^\circ)= \Q P^1$, where \[\sL_i=\{\alpha | \text{the\ Dehn\ filling\ } M_i(\alpha)\ \text{is\ an\ L-space}\}\] and $\sL_i^\circ$ denotes the interior of $\sL_i$ as a subset of $\Q P^1$. \end{theorem}  
It is not immediately clear what the analogue of {\em simplest possible} might be when $M$ is a manifold with torus boundary. Recall that L-space is short for {\em Heegaard Floer homology lens space} -- these first appear in the work of Ozsv\'ath and Szab\'o on obstructions to lens space surgeries \cite{OSz2005}, since lens spaces as a starting point are precisely those spaces constructed from a pair of solid tori, it is instructive to consider two characterizations of the solid torus. 

First, the solid torus is characterized, among orientable, compact, connected, irreducible three-manifold with torus boundary by the existence of a single essential simple closed curve in the boundary that bounds a properly embedded disk. Said another way, this is the observation that the solid torus is an $S^1$ bundle over the disk, which is a topological characterization of the solid torus that is captured by Heegaard Floer homology. %In particular, we have the following adaptation/generalisation of Ozsv\'ath and Szab\' o's result that knot Floer homology detects the trivial knot \cite[Theorem 1.2]{OSz2004} to the present setting. 

\begin{theorem}[Ni \cite{Ni2009}, reformulated]\label{thm:genus-detection}Let $M$ be an orientable, compact, connected, irreducible three-manifold with torus boundary. If $\CFD(M,\lambda,\mu) \cong \CFD(D^2\times S^1,m,l)$ for some $\mu$ dual to $\lambda$, then $M\cong D^2\times S^1$.
\end{theorem}

Note that the equivalence of differential modules implies that $H_1(M;\Z)\cong\Z$, hence $M$ is the complement of the knot $K$ in an integer homology sphere $Y$ corresponding to $\mu$-Dehn filling on $M$. The knot Floer homology of \(K\) is determined by \(\CFD(M,\lambda, \mu)\), hence  \(\HFK(K) = \HFK(U)\), where \(U\) is the unknot in \(S^3\). It follows from work of Ni \cite[Theorem 3.1]{Ni2009} that \(g(K)=g(U)=0\). Since \(M\) is irreducible, \(M\cong D^2 \times S^1\). 

A second characterization is given by the following:

\begin{theorem}[Johanssen, see Siebenmann \cite{Siebenmann1980}]\label{thm:twist}
Let $M$ be an orientable, compact, connected, irreducible three-manifold with torus boundary. If $M$ admits a homeomorphism that restricts to the boundary torus as a Dehn twist then $M\cong D^2\times S^1$.\qed
\end{theorem}

The proof of this theorem is summarized nicely in \cite[Remark on 5.1, page 206]{Siebenmann1980}. The key step is the observation that a Dehn twist along a properly embedded annulus in $M$ (with the additional assumption that $M$ is boundary irreducible) would induce a pair of cancelling Dehn twists in the boundary. The fact that the homeomorphism in question reduces to considering such an annulus follows from, and is the central application of, Johanssen's finiteness theorem. A much more general treatment (considering higher genus boundary) may be found in the work of McCullough \cite{McCullough2006}. In particular, the interested reader should compare Theorem \ref{thm:twist} with \cite[Corollary 3]{McCullough2006}.

 In contrast with Theorem \ref{thm:genus-detection}, the topological characterization of the solid torus provided by Theorem \ref{thm:twist} is not faithfully represented in Heegaard Floer theory. Consider the following:
 
\begin{definition}\label{def:hf-torus}
A rational homology solid torus $M$ is a Heegaard Floer homology solid torus if there is a homotopy equivalence of differential modules \[\CFD(M,\lambda,\mu)\cong\CFD(M,\lambda,\mu+\lambda)\] where $\lambda$ is the (rational) longitude of $M$ and $\mu$ is any slope dual to $\lambda$.   
\end{definition}

Notice that this definition may be rephrased by saying that the invariant $\CFD(M,\lambda,\mu)$ is independent (up to homotopy) of the choice of slope dual to the rational longitude when $M$ is a Heegaard Floer homology solid torus; this may be viewed as a type of Heegaard Floer homology Alexander trick, in the sense that  one may begin with an arbitrary manifold with torus boundary, and form a closed three manifold by attaching a Heegaard Floer homology solid torus. While the resulting manifold depends on a pair of gluing parameters, the Heegaard Floer homology is specified once the image of the longitude is known. This is precisely the simplification afforded to Dehn surgery by the Alexander trick. In particular, given a Heegaard Floer homology solid torus one has a means of producing  infinite families of distinct closed three-manifolds with identical Heegaard Floer homology is a straightforward manner (see Corollary \ref{crl:infinite}, below). 

\parpic[r]{
 \begin{minipage}{35mm}
 \centering
 \includegraphics[scale=.8]{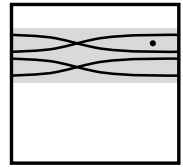}%\vspace*{-15pt}
 \captionof{figure}{Curves for the twisted $I$-bundle of the Klein bottle.}
 \label{fig:twisted-I-with-NBHD}
  \end{minipage}%
  }
A particular example of this phenomenon is provided by the twisted $I$-bundle over the Klein bottle. The bordered invariants of this manifold are computed in \cite[Proposition 6]{BGW2013} and the fact that this manifold is a Heegaard Floer homology solid torus is the content of \cite[Proposition 7]{BGW2013}. Viewed as immersed curves, these results are summarized in Figure \ref{fig:twisted-I-with-NBHD}. Further examples have appeared in \cite{HW,solid-tori}, for instance, there is an infinite family of Seifert fibered examples with base orbifold a disk with two cone points of order $n$ (the base orbifold $D^2(2,2)$ gives rise to the twisted $I$-bundle over the Klein bottle). We return to this in Section \ref{sec:seiberg-witten}. 

\begin{theorem}\label{thm:HFST} If $M$ is an orientable, compact, connected, irreducible three-manifold with torus boundary that is the complement of a knot in an L-space, the following are equivalent.
\begin{itemize}
\item[(i)] $M$ is a Heegaard Floer homology solid torus;
\item[(ii)] $\sL_M$ contains all slopes not equal to the rational longitude;
\item[(iii)] $\HFhat(M)$ carries trivial local systems and can be moved, via regular homotopy, to a small neighbourhood of the rational longitude.
\end{itemize}
\end{theorem}

\begin{proof}
The equivalence between (i) and (iii) is an immediate consequence of the equivalence between $\CFD(M,\lambda,\mu)$ and $\HFhat(M)$, and in particular, the identification of the action of the appropriate Dehn twist bimodule with a Dehn twists along $\lambda$; see \cite[Section 5]{HRW}. Similarly, the equivalence between (ii) and (iii) follows from the observation that $\sL_M$ is the collection of all slope intersecting each curve, minimally, exactly one time; see \cite[Section 7]{HRW}. Note that $\sL_M$ is empty whenever $\HFhat(M)$ carries a non-trivial local system.
\end{proof}

%\begin{remark}Actually, this should be true without the restriction to $M$ having an L-space filling....\end{remark}

This result, in combination with Theorem \ref{thm:L-space-gluing}, shows that a closed manifold obtained by gluing two Heegaard Floer homology solid tori together is an L-space (provided the rational longitudes are not identified in the process).  

Corollary \ref{crl:mutation} establishes the existence of pairs of closed orientable three-manifolds with identical \(\HFhat\), namely, a toroidal $Y$ and its genus 1  mutant $Y^\mu$. The existence of Heegaard Floer homology solid tori gives rise to infinite families of closed orientable three-manifolds enjoying the same property:

\begin{corollary}\label{crl:infinite}
For any closed orientable three-manifold $Y$ containing a Heegaard Floer homology solid torus in its JSJ decomposition, there is an infinite family of manifolds $\{Y_i\}$ for which $\HFhat(Y_i)$ does not depend on $i\in\Z$.
\end{corollary}
\begin{proof}Write $Y=Y_0 = M_0\cup _{h_0} M_1$ where one of the $M_i$ is a Heegaard Floer homology solid torus (not $D^2\times S^1$). Then $Y_i$ is defined by composing $h_0$ with $i$ Dehn twists along the rational longitude of the  Heegaard Floer homology solid torus, and $\HFhat(Y_i) \cong \HFhat(Y)$ by Definition \ref{def:hf-torus} or, explicitly, via Theorem \ref{thm:pairing} and Theorem \ref{thm:HFST}.\end{proof}

%% file: sections/grading.tex
% !TEX root = ../companion.tex

%grading.tex

%\section{Gradings}
%\label{sec:gradings}

In this section we will show that grading information on bordered Floer invariants of a manifold $M$ with torus boundary can be captured with additional decorations on the corresponding collection of immersed curves $\HFhat(M)$ and we prove Theorem~\ref{thm:gradings}, which asserts that gradings on the intersection Floer homology of curves recover the relative grading data on $\HFhat$ of a closed 3-manifold obtained by gluing along a torus.

Reviewing the notation used in \cite{HRW}, recall that $\HFhat(M)$ is a collection of immersed curves in the punctured torus $T_M$, while for any $\spin$ in $\spinc(M)$, $\HFhat(M,\spin)$ is a collection of immersed curves in the covering space $\barT_M = (H_1(M;\R)\setminus H_1(M;\Z))/\langle \lambda \rangle$. Thus with our conventions
$$ \HFhat(M) = \bigoplus_{\spin\in\spinc(M)} p\big(\HFhat(M,\spin)\big) $$
where $p\co \barT_M \to T_M$ is the projection. On each summand, the lift to the covering space carries additional information about the relative spin$^c$ grading. The goal of this section is to further decorate $\HFhat(M,\spin)$ to capture the Maslov grading as well. Once the curve set is given this extra decoration, it turns out that it can be projected to $T_M$ without losing information, and thus the spin$^c$ grading can be recorded without working in the cover $\barT_M$. Though we will not write this, since it is often convenient to work in the cover anyway, we could define $\HFhat_{\!\gr}(M,\spin)$ as the projection of $\HFhat(M,\spin)$, with this extra decoration, to $T_M$; this lives in the torus and represents the fully graded bordered Floer invariants of $M$.

\subsection{Graded Floer homology of curves with local systems}\label{sec:gradings-overview}

Theorem~\ref{thm:gradings} lets us compute the gradings on $\HFhat(M_0\cup_h M_1)$ directly from the curve invariants $\HFhat(M_0)$ and $\HFhat(M_1)$. Before proving the theorem, we define the grading structure on (Floer homology of) immersed curves and demonstrate it with some examples.

\begin{definition}
Given a collection $\boldsymbol{\gamma}$ of $n$ immersed curves in the punctured torus, possibly decorated with local systems, a \emph{set of grading arrows} is a collection of $n-1$ crossover arrows connecting the curves such that the union of the curves and crossover arrows is path-connected, i.e. such that there is a smooth immersed path between points on any two curves.
\end{definition}
For an example of a grading arrow on a collection of two curves, see Figure \ref{fig:fig8-maslov}. The grading arrows will sometimes be referred to as \emph{phantom arrows}, since they are a decoration that encodes grading information but otherwise have no effect (for instance, they are ignored when counting bigons while taking Floer homology of two curve sets). We will see that all the grading information on the bordered invariants for a manifold $M$ can be encoded with a set of grading arrows on $\curves M$. Note that when $\curves M$ contains a single curve, there are no grading arrows; that is, $\curves M$ does not require any further decoration to capture grading information. 

For $i = 0,1$, let $\bg i$ be a collection of immersed curves in the punctured torus decorated with local systems. We will further assume that every component of $\bg i$ is homologous to a multiple of $\lambda_i$, where $\lambda_i$ is a fixed homology class in $H_1(T^2)$ (when $\bg i$ is the invariant $\curves M$ for a manifold $M$ with torus boundary, $\lambda_i$ is the homological longitude of $M$). We defined the intersection Floer homology $HF(\bg 0, \bg 1)$ in \cite[Section 4]{HRW}; recall that, unless a component of $\bg 0$ is parallel to a component of $\bg1$, this is simply the vector space over $\F$ whose dimension is the geometric intersection number of $\bg0$ and $\bg1$. Provided $\bg0$ and $\bg1$ are decorated with a set of grading arrows, this vector space can be endowed with a relative spin$^c$ grading, which gives rise to a direct sum decomposition, and a relative Maslov grading on each spin$^c$ component.

For the spin$^c$ grading, consider two generators $x$ and $y$ of $HF(\bg0, \bg1)$, which are intersection points of $\bg 0$ and $\bg1$. Choose a path (not necessarily smooth) from $x$ to $y$ in $\bg0$ together with its grading arrows, and a path from $y$ to $x$ in $\bg1$ with its grading arrows. These paths are well defined up to adding full curve components of $\bg0$ and $\bg1$, so the union of the two paths gives a well defined element of $H_1(T^2) / \langle [\lambda_0], [\lambda_1] \rangle$; this element is the grading difference for the relative spin$^c$ grading. 

There is another description of the spin$^c$ grading which is often useful involving lifts of the curves to the covering space $\R^2\setminus \Z^2$. The curve set $\bg i$ together with its grading arrows has a well defined lift to the covering space  up to an overall translation, and the lift is invariant under the deck transformation corresponding to a lift of $\lambda_i$. Note that each curve in $\bg i$ has such a lift, and the grading arrows determine the relative position of the lifts of each component. Two intersection points $x$ and $y$ have the same spin$^c$ grading if and only if there are lifts $\boldsymbol{\tilde\gamma_0}$ and $\boldsymbol{\tilde\gamma_1}$ of $\bg 0$ and $\bg 1$ which pass through a lift of $x$ and a lift of $y$.

Given two intersection points $x$ and $y$ in the same spin$^c$ grading, consider lifts $\boldsymbol{\tilde\gamma_0}$ and $\boldsymbol{\tilde\gamma_1}$ of the curves passing through lifts $\tilde x$ of $x$ and $\tilde y$ of $y$, and let $p_i$ be a path from $\tilde x$ to $\tilde y$ in $\bg i$. The concatenation of $p_0$ with $-p_1$ gives a closed piecewise smooth path $p$ in $\R^2\setminus\Z^2$. The Maslov grading difference $m(y) - m(x)$ is defined to be twice the number of lattice points enclosed by $p$ (where each point is counted with multiplicity given by the winding number of $p$) plus $\frac 1 \pi$ times the total rightward rotation along the smooth segments of $p$. We assume that $p_0$ and $p_1$ are orthogonal at $x$ and $y$. For example, if $p$ is the boundary of an immersed bigon from $\tilde x$ to $\tilde y$ with two smooth sides meeting at right angles at $\tilde x$ and $\tilde y$, then the total rightward (clockwise) rotation when traveling along the smooth portions of $p_0- p_1$ is $-\pi$ (a full counterclockwise circle would be $-2 \pi$, but the two right corners of angle $\frac \pi 2$ do not contribute to the total rotation). Thus the grading difference $m(y) - m(x)$ is $-1$ plus two times the number of lattice points enclosed.

\begin{example}\label{example-intro}
Consider the splice of two trefoil complements discussed in the introduction (see Figure \ref{fig:intro-figure}). The two relevant immersed curves intersect five times; by looking at lifts of the two curves to the plane, it is clear that all five intersection points have the same spin$^c$ grading. They are connected by a string of bigons, each covering the puncture once, as in Figure \ref{fig:intro-fig-maslov}. If we label the intersections $x_1$ through $x_5$ from left to right in the figure, there are bigons from $x_2$ to $x_1$ and to $x_3$ and from $x_4$ to $x_3$ and to $x_5$. It follows that the generators $x_1, x_3$, and $x_5$ of $\HFhat(M_0\cup_h M_1)$ all have the same Maslov grading and that the grading of $x_2$ and $x_4$ is lower by one.
\end{example}

\begin{example}\label{example-trefoil}
To compute $\HFhat$ of $+3$-surgery on the right handed trefoil, we intersect the curve $\curves M$, where $M$ is the trefoil complement, with a line of slope 3. Figure \ref{fig:trefoil-3surgery} shows this arrangement both in the torus $\partial M \setminus z$ and in a lift to $\barT_M$. There are three intersection points, indicating that $\HFhat$ has dimension 3. Since three separate lifts of the straight line are needed to realize all three intersection points in the covering space, the three intersection points have different spin$^c$ gradings.
\end{example}

\begin{example}\label{example-fig8}
Let $M$ be the complement of the figure eight knot in $S^3$, and let $Y = M(1)$ be $+1$ surgery on this knot. $\HFhat(Y)$ is computed by intersecting $\curves M$ with a line of slope 1, as shown in Figure \ref{fig:fig8-maslov}. Note that since $\curves M$ contains two curves in the same spin$^c$-structure, grading information is not captured by the curves alone; the left side of Figure \ref{fig:fig8-maslov} shows $\curves M$ decorated with a grading arrow. There is a bigon connecting the intersection points $a$ and $b$ which covers the puncture once (middle part of the Figure); it follows that the Maslov grading of $b$ is one higher than that of $a$. The intersection points $c$ and $a$ can be connected by a more complicated region (right side of the figure), the boundary of which is a closed piecewise smooth path with corners at $a$ and $c$. The net clockwise rotation along the smooth pieces of the path is $\pi$, and a puncture is enclosed with winding number $-1$. It follows from this that the grading of $c$ is one higher than the grading of $a$.
\end{example}

\begin{figure}
\centering
\begin{minipage}{70mm}
 \centering
 \vspace{5 mm}
 \includegraphics[scale=0.55]{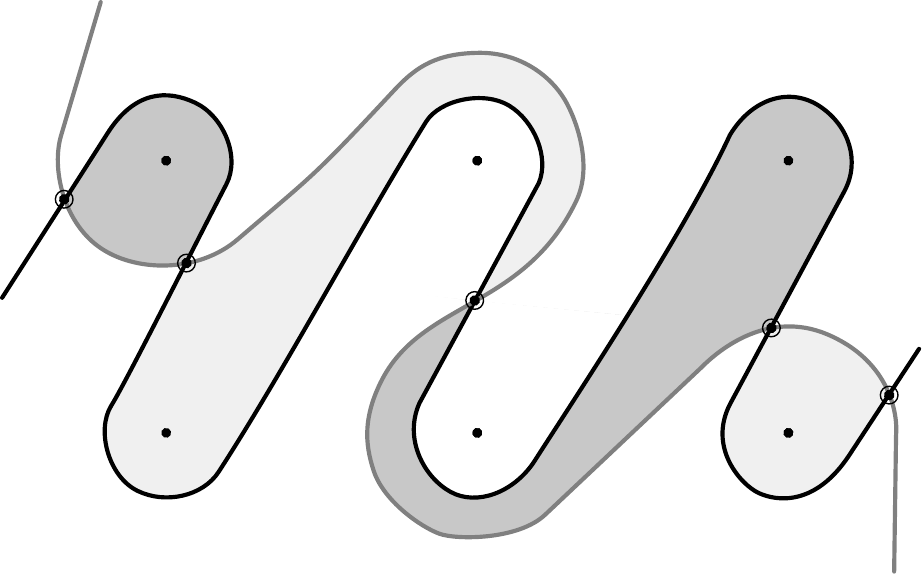}
 \vspace {3 mm}
\captionof{figure}{Lifts of curves $\curves{M_0}$ and $\color{gray} h(\curves{M_1})$ from Example \ref{example-intro} to the plane and bigons determining the Maslov grading.}\label{fig:intro-fig-maslov}
\end{minipage}%
\hspace{10 mm}
\begin{minipage}{70mm}
 \centering
 \includegraphics[scale=0.9]{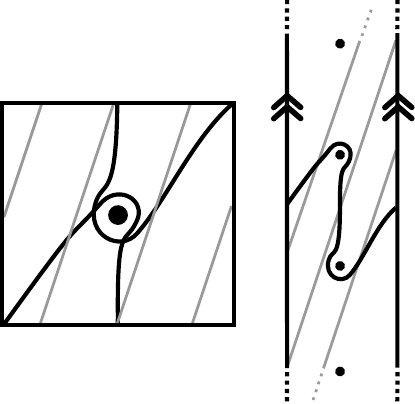}
\captionof{figure}{+3 surgery on the right handed trefoil, as described in Example \ref{example-trefoil}}\label{fig:trefoil-3surgery}
\end{minipage}
\end{figure}

\labellist
\tiny
\pinlabel {$a$} at 220 99
\pinlabel {$b$} at 173 67
\pinlabel {$c$} at 152 47

\pinlabel {$a$} at 348 99
\pinlabel {$b$} at 301 67
\pinlabel {$c$} at 280 47
\endlabellist
\begin{figure}
\includegraphics[scale=.7]{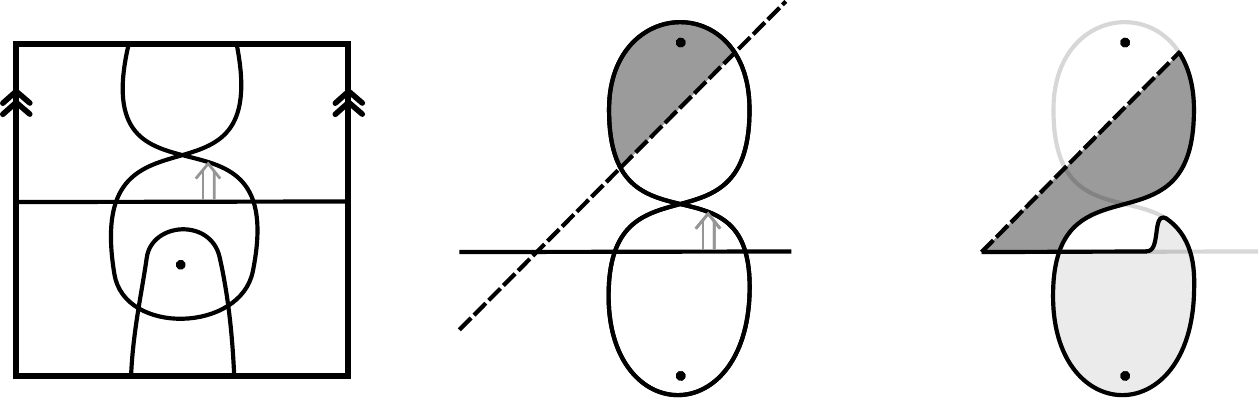}
\caption{Left: The invariant $\curves{M}$ for $M$ the complement of the figure eight knot, decorated by a phantom arrow to encode grading information. Middle: Pairing with $\curves{D^2\times S^1}$ (dotted line), corresponding to 1-surgery on $M$, viewed in the covering space $\tildeT_M$. There are three generators, all in the same spin$^c$ structure. The shaded bigon indicates that $m(b) - m(a) = 1$. Right: The shaded region shows that $m(c) - m(a) = 1$.}
\label{fig:fig8-maslov}
\end{figure}

\subsection{Gradings in bordered Floer homology}

Before proving Theorem \ref{thm:gradings}, we briefly recall the grading structure on bordered Floer homology; since we are only interested in the torus boundary case, we follow the specialization given in \cite[Chapter 11]{LOT}. Recall that  each generator \(x\) of \(\CFD(M,\alpha,\beta)\) has an associated spin$^c$ structure \(\spin(x) \in \spinc(M)\). The elements of \(\spinc(M)\) are homology classes of nonvanishing vector fields on \(M\), and \(\spinc(M)\) has the structure of an affine set modeled on \(H^2(M) \cong H_1(M, \partial M)\). The same decomposition holds for $\CFA(M,\alpha,\beta)$, so that 
\[\CFD(M,\alpha,\beta ) = \bigoplus_{\spin \in \spinc(M)} \CFD(M,\alpha,\beta;\spin)\quad\text{and}\quad\CFA(M,\alpha,\beta ) = \bigoplus_{\spin \in \spinc(M)} \CFA(M,\alpha,\beta;\spin).\]
The gradings in bordered Floer homology take the form of relative gradings on each spin$^c$-structure summand in the above decomposition.

The refined grading $\grCFD$ takes values in a non-commutative group $G$ whose elements are triples $(m; i, j)$, with $m,i,j\in\frac{1}{2}\Z$ and $i+j\in\Z$, and for which $m\in\Z$ if and only if $i, j\in\Z$ and $i+j\in 2\Z$. The half integer $m$ is the Maslov component, and the pair $v = (i,j)$ (regarded as a vector in $\frac{1}{2}\Z \times \frac{1}{2}\Z$) is the spin$^c$-component. The group law is given by
\[(m;v)\cdot(m';v')=(m+m'+\det\left(\begin{smallmatrix} -\, v\, -\\ -\, v'-\end{smallmatrix} \right);v+v').\]
The torus algebra is also graded by elements of $G$; the grading on Reeb elements is given by
$$ \textstyle \grCFD(\rho_1) = (-\frac{1}{2}; \frac{1}{2}, -\frac{1}{2}), \quad
\grCFD(\rho_2) = (-\frac{1}{2}; \frac{1}{2}, \frac{1}{2}), \quad
\grCFD(\rho_3) = (-\frac{1}{2}; -\frac{1}{2}, \frac{1}{2}),$$
along with the fact that $\grCFD(\rho_I\rho_J) = \grCFD(\rho_I)\grCFD(\rho_J)$. In $\CFD$, the grading satisfies ${\grCFD(\partial x) = \lambda^{-1}\grCFD(x)}$, where $\lambda$ is the central element $(1;0,0)$, and $\grCFD(\rho_I \otimes y) = \grCFD(\rho_I)\grCFD(y)$. In $\CFA$, the grading satisfies $\grCFD( m_{k+1}(x, a_1, \ldots, a_k)) = \grCFD(x)\grCFD(a_1)\cdots\grCFD(a_k)\lambda^{k-1}$. It follows that the change in grading associated with coefficient maps in a type D structure or a corresponding type A structure are given by Table \ref{tab:mas-D}.

The gradings are defined only up to an overall shift; that is, up to multiplying the grading of each generator on the right (for $\CFD)$ or on the left (for $\CFA$) by a fixed element $G$. Thus it is convenient to fix a reference generator $\x_0$ and declare it to have grading $(0;0,0)$. Moreover, for each manifold and choice of reference generator $\x_0$  there is a subgroup $P(\x_0)$ of $G$ such that the gradings on $\CFD$ (resp. $\CFA$) are defined only modulo the right (resp. left) action of $P(\x_0)$. $P(\x_0)$ records the gradings of periodic domains connecting $\x_0$ to itself. More precisely, $P(\x_0)$ is the image in $G$ of the set of periodic domains $\pi_2(\x_0,\x_0) \cong H_2(M,\partial M)$ (see \cite[Section 10.2]{LOT}). Note that for torus boundary $\pi_2(\x_0,\x_0) \cong \pi^\partial_2(\x_0,\x_0) \oplus \Z$, where $\pi_2^\partial(\x_0,\x_0) \cong H_2(M)$ is the set of provincial periodic domains. It follows that if $M$ is a rational homology solid torus then $P(\x_0)$ is cyclic, and otherwise it is generated by $p_0$ and $\lambda^n = (n;0,0)$ for some $p_0 \in G$ with non-zero spin$^c$ component and some integer $n$. We remark that $P(\x_0)$ is determined by topological information about $M$; for example, the spin$^c$ component of the generator of $P(\x_0)$ is determined by the homological longitude of $M$.

\begin{table}[t]
	\begin{tabular}{|c|c|c|c|}
		\hline
		\rowcolor{lightgray} Labelled edge   & Grading change   & Labelled edge   & Grading change   \\
		\hline
		\begin{tikzpicture}[thick, >=stealth'] \draw[->] (0,0) -- (1,0);\node [above] at (0.5,0) {$\scriptstyle{1}$};
		\node [left] at (0,0) {$x$}; \node [right] at (1,0) {$y$};\end{tikzpicture}
		& $\gr(x)=(\frac{1}{2};\frac{1}{2},-\frac{1}{2})\cdot\gr(y)$  
		&\begin{tikzpicture}[thick, >=stealth'] \draw[->] (0,0) -- (1,0);\node [above] at (0.5,0) {$\scriptstyle{12}$};
		\node [left] at (0,0) {$x$}; \node [right] at (1,0) {$y$};\end{tikzpicture}
		& $\gr(x)=(\frac{1}{2};1,0)\cdot\gr(y)$ \\
		\begin{tikzpicture}[thick, >=stealth'] \draw[->] (0,0) -- (1,0);\node [above] at (0.5,0) {$\scriptstyle{2}$};
		\node [left] at (0,0) {$x$}; \node [right] at (1,0) {$y$};\end{tikzpicture}& $\gr(x)=(\frac{1}{2};\frac{1}{2},\frac{1}{2})\cdot\gr(y)$
		& \begin{tikzpicture}[thick, >=stealth'] \draw[->] (0,0) -- (1,0);\node [above] at (0.5,0) {$\scriptstyle{23}$};
		\node [left] at (0,0) {$x$}; \node [right] at (1,0) {$y$};\end{tikzpicture}
		&$\gr(x)=(\frac{1}{2};0,1)\cdot\gr(y)$\\
		\begin{tikzpicture}[thick, >=stealth'] \draw[->] (0,0) -- (1,0);\node [above] at (0.5,0) {$\scriptstyle{3}$};
		\node [left] at (0,0) {$x$}; \node [right] at (1,0) {$y$};\end{tikzpicture}
		&  $\gr(x)=(\frac{1}{2};-\frac{1}{2},\frac{1}{2})\cdot\gr(y)$
		& \begin{tikzpicture}[thick, >=stealth'] \draw[->] (0,0) -- (1,0);\node [above] at (0.5,0) {$\scriptstyle{123}$};
		\node [left] at (0,0) {$x$}; \node [right] at (1,0) {$y$};\end{tikzpicture}
		&  $\gr(x)=(\frac{1}{2};\frac{1}{2},\frac{1}{2})\cdot\gr(y)$ \\
		\hline
	\end{tabular}
	\vskip0.1in

	\begin{tabular}{|c|c|c|c|}
		\hline
		\rowcolor{lightgray} Labelled edge   & Grading change   & Labelled edge   & Grading change   \\
		\hline
		\begin{tikzpicture}[thick, >=stealth'] \draw[->] (0,0) -- (1,0);\node [above] at (0.5,0) {$\scriptstyle{(3)}$};
		\node [left] at (0,0) {$x$}; \node [right] at (1,0) {$y$};\end{tikzpicture}
		& $ \gr(x)=\gr(y) \cdot (\frac{1}{2};\frac{1}{2},-\frac{1}{2})$  
		&\begin{tikzpicture}[thick, >=stealth'] \draw[->] (0,0) -- (1,0);\node [above] at (0.5,0) {$\scriptstyle{(3,2)}$};
		\node [left] at (0,0) {$x$}; \node [right] at (1,0) {$y$};\end{tikzpicture}
		& $\gr(x)=\gr(y)\cdot (\frac{1}{2};0,-1)$ \\
		\begin{tikzpicture}[thick, >=stealth'] \draw[->] (0,0) -- (1,0);\node [above] at (0.5,0) {$\scriptstyle{(2)}$};
		\node [left] at (0,0) {$x$}; \node [right] at (1,0) {$y$};\end{tikzpicture}& $\gr(x)=\gr(y)\cdot (\frac{1}{2};-\frac{1}{2},-\frac{1}{2})$
		& \begin{tikzpicture}[thick, >=stealth'] \draw[->] (0,0) -- (1,0);\node [above] at (0.5,0) {$\scriptstyle{(2,1)}$};
		\node [left] at (0,0) {$x$}; \node [right] at (1,0) {$y$};\end{tikzpicture}
		&$\gr(x)=\gr(y)\cdot (\frac{1}{2};-1,0)$\\
		\begin{tikzpicture}[thick, >=stealth'] \draw[->] (0,0) -- (1,0);\node [above] at (0.5,0) {$\scriptstyle{(1)}$};
		\node [left] at (0,0) {$x$}; \node [right] at (1,0) {$y$};\end{tikzpicture}
		&  $\gr(x)=\gr(y)\cdot(\frac{1}{2};-\frac{1}{2},\frac{1}{2})$
		& \begin{tikzpicture}[thick, >=stealth'] \draw[->] (0,0) -- (1,0);\node [above] at (0.5,0) {$\scriptstyle{(3,2,1)}$};
		\node [left] at (0,0) {$x$}; \node [right] at (1,0) {$y$};\end{tikzpicture}
		&  $\gr(x)=\gr(y)\cdot (\frac{1}{2};-\frac{1}{2},-\frac{1}{2})$ \\
		\hline
	\end{tabular}
	\vskip0.1in
	\caption{Shifts in the refined gradings on $\CFD(M,\alpha,\beta)$, top, and $\CFA(M,\alpha,\beta)$, bottom. }\label{tab:mas-D}
\end{table}

For $i = 0, 1$, consider a bordered manifold with torus boundary $(M_i, \alpha_i, \beta_i)$ with spin$^c$ structure $\spin_i$. The gradings $\grCFD$ on $N^A_0 = \CFA(M_0,\alpha_0,\beta_0;\spin_0)$ and $N_1 = \CFD(M_1,\alpha_1,\beta_1;\spin_1)$ give rise to a grading $\grbox$ on $N_0^A \boxtimes N_1$, where $\grbox(y_0\otimes y_1) = \grCFD(y_0) \cdot \grCFD(y_1)$. Fix reference generators $\x_0$ and $\x_1$ of $N_0^A$ and $N_1$, respectively, with $\iota(\x_0) = \iota(\x_1)$ so that $\x_0\otimes \x_1$ is a generator in the box tensor product. The grading $\grbox$ takes values in $P(\x_0) \backslash G / P(\x_1)$ with integer coefficients in the spin$^c$ component.

Suppose that $Y = M_0 \cup_h M_1$. Restriction gives a surjective map 
$$\pi: \spinc(Y) \to \spinc(M_0) \times \spinc(M_1).$$ 
It is not hard to see that $\pi^{-1}(\spin_0 \times \spin_1)$ is a torsor over  $H_Y= H_1(\partial M_0;\Z)/ \langle \lambda_0, h_*(\lambda_1) \rangle)$, where $\lambda_i$ is the homological longitude of $M_i$. The spin$^c$ component of $\grbox(y_0\otimes y_1)$ can be interpreted as an element of  $H_Y$; this, along with $\spin_0$ and $\spin_1$, determine the spin$^c$ grading of $y_0\otimes y_1$. If $y_0 \otimes y_1$ and $z_0\otimes z_1$ have the same spin$^c$ grading, then they have a well defined Maslov grading difference as well, obtained by acting on $\grCFD(y_0)\cdot\grCFD(y_1)$ and $\grCFD(z_0)\cdot\grCFD(z_1)$ by $P(\x_0)$ and $P(\x_1)$ to make the spin$^c$ components equal and then comparing the Maslov components.

The spin$^c$ component of the grading admits another description which we find valuable: Restricting attention to the generators in a particular idempotent \(\iota\), we define a refined spin$^c$ grading \(\spin_\iota(x) \in \spinc(M, \iota)\), which lives in an affine set modeled on \(H^2(M, \partial M) \cong H_1(M)\).  Elements of \(\spinc(M,\iota)\) are homology classes of nonvanishing vector fields with prescribed behavior on \(\partial M\), and \(\spin(x)\) is the image of \(\spin_\iota(x)\) in \(\spinc(M)\).

To compare the refined gradings of two generators, we adopt the following. Given \(\mathfrak{t} \in \spinc(M)\), let 
\[ \spinc(M, \iota, \mathfrak{t}) = \{ \spin \in \spinc(M, \iota) \, | \, \spin = \mathfrak{t} \ \text{in} \ \spinc(M)\}.\]
If \(j_*\co H_1(\partial M) \to H_1(M)\) is the map induced by inclusion, 
\(\spinc(M, \iota, \mathfrak{t})\) is an affine set modeled on \(H_M = \im j_* \cong H_1(\partial M)/\ker j_* \). 
When \(\partial M\) is a torus we let \(\spinc(M,\mathfrak{t}) = \spinc(M, \iota_0, \mathfrak{t}) \cup 
\spinc(M, \iota_1, \mathfrak{t})\) 
and define
 $$\textstyle\frac{1}{2}H_M = \{ x\in H_1(\partial M, \R) \, | \, 2x \in H_1(\partial M, \Z)\}/ \ker j_* $$
Given two generators $x$ and $y$ in $\CFD(M,\alpha, \beta; \spin)$ with idempotents $\iota_x$ and $\iota_y$, respectively, we think of the grading difference $\spin_{\iota_x}(x) - \spin_{\iota_y}(y)$ as an element of $\frac{1}{2}H_M$, which is in $H_M$ if and only if $\iota_x = \iota_y$. Equivalently, we can think of $\spin_\iota$ as a relative grading where $\spin_{\iota_x}(x)$ is an element of $\frac{1}{2}H_M$, defined only up to an overall shift.

\begin{table}[t]
\begin{tabular}{|c|c|c|c|}
		\hline
		\rowcolor{lightgray} Labeled edge  & \(\spin(y)-\spin(x)\)   & Labeled edge & \(\spin(y)-\spin(x)\)   \\
		\hline
		\begin{tikzpicture}[thick, >=stealth'] \draw[->] (0,0) -- (1,0);\node [above] at (0.5,0) {$\scriptstyle{1}$};
		\node [left] at (0,0) {$x$}; \node [right] at (1,0) {$y$};\end{tikzpicture}
		& $-(\alpha+\beta)/2$  
		&\begin{tikzpicture}[thick, >=stealth'] \draw[->] (0,0) -- (1,0);\node [above] at (0.5,0) {$\scriptstyle{12}$};
		\node [left] at (0,0) {$x$}; \node [right] at (1,0) {$y$};\end{tikzpicture}
		& $-\beta$ \\
		\begin{tikzpicture}[thick, >=stealth'] \draw[->] (0,0) -- (1,0);\node [above] at (0.5,0) {$\scriptstyle{2}$};
		\node [left] at (0,0) {$x$}; \node [right] at (1,0) {$y$};\end{tikzpicture}& $(\alpha -\beta)/2$ 
		& \begin{tikzpicture}[thick, >=stealth'] \draw[->] (0,0) -- (1,0);\node [above] at (0.5,0) {$\scriptstyle{23}$};
		\node [left] at (0,0) {$x$}; \node [right] at (1,0) {$y$};\end{tikzpicture}
		& $\alpha$\\
		\begin{tikzpicture}[thick, >=stealth'] \draw[->] (0,0) -- (1,0);\node [above] at (0.5,0) {$\scriptstyle{3}$};
		\node [left] at (0,0) {$x$}; \node [right] at (1,0) {$y$};\end{tikzpicture}
		&  $(\alpha+\beta)/2$ 
		& \begin{tikzpicture}[thick, >=stealth'] \draw[->] (0,0) -- (1,0);\node [above] at (0.5,0) {$\scriptstyle{123}$};
		\node [left] at (0,0) {$x$}; \node [right] at (1,0) {$y$};\end{tikzpicture}
		&  $(\alpha-\beta)/2$ \\
		\hline
	\end{tabular}
%	\vskip0.1in
%	\begin{tabular}{|c|c|c|c|}
%		\hline
%		\rowcolor{lightgray} Labeled edge  & \(\spin(y)-\spin(x)\)   & Labeled edge & \(\spin(y)-\spin(x)\)   \\
%		\hline
%		\begin{tikzpicture}[thick, >=stealth'] \draw[->] (0,0) -- (1,0);\node [above] at (0.5,0) {$\scriptstyle{(3)}$};
%		\node [left] at (0,0) {$x$}; \node [right] at (1,0) {$y$};\end{tikzpicture}
%		& $-(\alpha+\beta)/2$  
%		&\begin{tikzpicture}[thick, >=stealth'] \draw[->] (0,0) -- (1,0);\node [above] at (0.5,0) {$\scriptstyle{(3,2)}$};
%		\node [left] at (0,0) {$x$}; \node [right] at (1,0) {$y$};\end{tikzpicture}
%		& $-\beta$ \\
%		\begin{tikzpicture}[thick, >=stealth'] \draw[->] (0,0) -- (1,0);\node [above] at (0.5,0) {$\scriptstyle{(2)}$};
%		\node [left] at (0,0) {$x$}; \node [right] at (1,0) {$y$};\end{tikzpicture}& $(\alpha -\beta)/2$ 
%		& \begin{tikzpicture}[thick, >=stealth'] \draw[->] (0,0) -- (1,0);\node [above] at (0.5,0) {$\scriptstyle{(2,1)}$};
%		\node [left] at (0,0) {$x$}; \node [right] at (1,0) {$y$};\end{tikzpicture}
%		& $\alpha$\\
%		\begin{tikzpicture}[thick, >=stealth'] \draw[->] (0,0) -- (1,0);\node [above] at (0.5,0) {$\scriptstyle{(1)}$};
%		\node [left] at (0,0) {$x$}; \node [right] at (1,0) {$y$};\end{tikzpicture}
%		&  $(\alpha+\beta)/2$ 
%		& \begin{tikzpicture}[thick, >=stealth'] \draw[->] (0,0) -- (1,0);\node [above] at (0.5,0) {$\scriptstyle{(3,2,1)}$};
%		\node [left] at (0,0) {$x$}; \node [right] at (1,0) {$y$};\end{tikzpicture}
%		&  $(\alpha-\beta)/2$ \\
%		\hline
%	\end{tabular}
	\vskip0.1in
%	\caption{Grading shifts, as elements of $\frac{1}{2} H_M$, associated with labelled edges in $\CFD(M,\alpha,\beta)$, top, or $\CFA(M,\alpha,\beta)$, bottom.}\label{tab:spinc}
	\caption{Grading shifts, as elements of $\frac{1}{2} H_M$, associated with labelled edges in $\CFD(M,\alpha,\beta)$. The grading shifts are the same for the corresponding edges in $\CFA(M,\alpha,\beta)$, though the identification of between $\frac{1}{2}H_M$ and the notation of \cite{LOT} is different.}\label{tab:spinc}
\end{table}

\begin{lemma}
	\label{Lem:spindiff} We can identify \(\spinc(M,\mathfrak{t} )\) with a subset of \(\frac{1}{2} H_M\) in such a way that arrows  in \(\CFD(M, \alpha, \beta)\) shift the \(\spinc\) grading as shown in Table~\ref{tab:spinc}.
\end{lemma}

\begin{proof}
	This is just a rephrasing of \cite[Lemma 3.8]{RR}; compare \cite[Lemma 11.42]{LOT}. Specifically, choose some identification \(f\co\spinc(M, \iota_0,\mathfrak{t}) \to H_M\). For \(\spin \in \spinc(M, \iota_0,\mathfrak{t})\), we identify \( \spin\) with \( f(\spin)\), and for \(\spin \in \spinc(M, \iota_1,\mathfrak{t})\) we identify  \(\spin\) with  \( f(i^{-1}(\spin))-(\alpha + \beta)/2\), where \(i\co\spinc(M, \iota_0,\mathfrak{t}) \to  \spinc(M, \iota_1,\mathfrak{t})\) is the map defined in \cite[Lemma 3.8]{RR}. 
\end{proof}

Given bordered manifolds $M_0$ and $M_1$, consider generators $x_0$ and $y_0$ in $\CFA(M_1,\alpha_0,\beta_0;\spin_0)$ and $x_1$ and $y_1$ in $\CFD(M_1,\alpha_1,\beta_1;\spin_1)$. The generators $x = x_0\otimes x_1$ and $y = y_0\otimes y_1$ in the box tensor product both have spin$^c$ grading in $\pi^{-1}( \spin_0 \times \spin_1 ) \subset \spinc(Y)$, and their difference, as an element of $H_Y$, is given by $\left[\spin(x_1)-\spin(x_0)\right] - r \left[ \spin(y_1) - \spin(y_0)\right]s$, where $r$ denotes the reflection taking $\alpha$ to $-\beta$ and $\beta$ to $-\alpha$.

\begin{remark} \label{rem:compare-grading-notation} We pause to explicitly state the identification between the two conventions above, which is a potential source of confusion. Comparing Tables \ref{tab:mas-D} and \ref{tab:spinc}, note that a change in $\spin_\iota$ of $i\beta + j\alpha \in \frac{1}{2} H_M$ corresponds to a change of $(i,-j)$ in the spin$^c$ component of the grading in $\CFD$, or to a change of $(j,-i)$ to the spin$^c$ component of the grading in $\CFA$. Note also that when representing a module by a train track in the next section, our convention is to draw the train track in the $\beta$-$\alpha$ plane, where $\beta$ is taken to be the positive horizontal direction and $\alpha$ is the positive vertical direction. Thus a generator with grading $\spin_\iota(x) = j\alpha + i\beta$ will occur at coordinates $(i,j)$ in the plane.
\end{remark}

Finally, we note that the full grading can be specialized to give a relative $\Ztwo$ grading, which can be a convenient restriction when the full Maslov grading is not needed (see \cite{Petkova-decategorification}). In $\CFD$, for a generator $x$ with grading $\grCFD(x) \in G$, we define $\grDmodtwo(x)$ to be $f_+(\grCFD(x))$ if $x$ has idempotent $\iota_0$ and $f_-(\grCFD(x))$ if $x$ has idempotent $\iota_1$, where $f_\pm$ is the mod 2 reduction of the map $\tilde{f}_\pm: G\to \Z$ defined by 
$$\tilde f_\pm(m,i,j) = \begin{cases}
m+ \frac{i+j}{2} & i,j \in \Z \,\, \text{with same parity} \\
m\pm \frac{j-i}{2} & i,j \in \Z, \,\, \text{with different parity} \\
m+ \left(i \pm \frac{1}{2}\right) \left(j \pm \frac{1}{2}\right) \pm \frac{1}{2} & i,j \not\in \Z
\end{cases}$$
On connected components of $\CFD$, the following proposition gives rise to a simpler description of $\grDmodtwo$ as a relative grading; we remark that this agrees with the relative $\Ztwo$ grading defined in \cite{Petkova-decategorification}.
\begin{proposition}
Two generators $x$ and $y$ in $\CFD$ have $\grDmodtwo(x) = \grDmodtwo(y)$ if they are connected by an arrow labeled with $\rho_2, \rho_{12}, \rho_{23}$, or $\rho_{123}$ and $\grDmodtwo(x) \neq \grDmodtwo(y)$ if they are connected by an arrow labeled with $\rho_1, \rho_3$, or $\rho_\emptyset = 1$.
\end{proposition}
\begin{proof}
This follows from the following identities of the functions $f_\pm$:
\begin{align*}
f_\pm( \lambda \cdot (m;i,j) ) &=  f_\pm( (m;i,j) ) + 1\\
f_+( \grCFD(\rho_1) \cdot (m;i,j) ) &=  f_-( (m;i,j) )  \\
f_-( \grCFD(\rho_2) \cdot (m;i,j) ) &=  f_+( (m;i,j) ) + 1 \\
 f_+( \grCFD(\rho_3) \cdot (m;i,j) ) &=  f_-( (m;i,j) ) 
\end{align*}

The first is clear, since multiplying by $\lambda = (1;0,0)$ simply increases the Maslov component by one. We will prove the second identity; the remaining two are similar and left to the reader. Let $g = (m;i,j)$ and
$$g' = gr(\rho_1)\cdot (m;i,j) = \left(-\frac 12; - \frac 12, -\frac 12\right)(m;i,j) = \left( m - \frac 12 + \frac{i+j}{2}; i + \frac 12, j - \frac 12\right).$$
The first case to consider is that $i$ and $j$ are integers of the same parity. In this case
$$f_-(g) = m + \frac{i+j}{2}, \qquad f_+(g') = \left[m - \frac 12 + \frac{i+j}{2}\right] + \left( \left[i+\frac 12\right] + \frac 12\right)\left( \left[j-\frac 12\right] + \frac 12\right) +\frac 12,$$
and the difference $f_+(g')-f_-(g)$ is $(i+1)j$, which is congruent to 0 mod 2 since $i$ and $j$ have the same parity. The second case is that $i$ and $j$ are integers of opposite parity. In this case
$$f_-(g) = m + \frac{i-j}{2}, \qquad f_+(g') = \left[m - \frac 12 + \frac{i+j}{2}\right] + \left( \left[i+\frac 12\right] + \frac 12\right)\left( \left[j-\frac 12\right] + \frac 12\right) +\frac 12,$$
and the difference is $j + (i+1)j = (i+2)j \equiv 0 \pmod 2$. The third case is that $i = i' + \frac 12$, $j = j' + \frac 12$, and $i'$ and $j'$ are integers of the same parity. Note that $g' = (m + \frac{i'+j'}{2}; i'+1, j')$. We have
$$f_-(g) = m + i' j' -\frac 12, \qquad f_+(g') = m + \frac{i'+j'}{2} + \frac{j' - (i'+1)}{2},$$
and the difference is $j' - i'j' = (1-i')j' \equiv 0 \pmod 2$. Finally, if $i = i'+\frac 12$ and $j = j' +\frac 12$ but $i'$ and $j'$ have different parity, we have
$$f_-(g) = m + i' j' -\frac 12, \qquad f_+(g') = m + \frac{i'+j'}{2} + \frac{(i'+1)+j'}{2},$$
and the difference is $1 + i' + j' - i'j' \equiv 0 \pmod 2$.

To see that the proposition follows from the identities above, suppose for example that there is a $\rho_1$ arrow from $x$ to $y$. This implies the idempotents of $x$ and $y$ are $\iota_0$ and $\iota_1$, respectively, and that $\grCFD(x) = \lambda \cdot \grCFD(\rho_1) \cdot \grCFD(y)$. Combining the first two identities gives
$$\grDmodtwo(x) = f_+(\lambda \cdot \grCFD(\rho_1) \cdot \grCFD(y)) = f_+(\grCFD(\rho_1) \cdot \grCFD(y)) + 1 = f_-(\grCFD(y)) + 1= \grDmodtwo(y) + 1.$$
If instead $x$ and $y$ are connected by a $\rho_{123}$ arrow, we would use that
$$\grCFD(x) = \lambda \cdot \grCFD(\rho_{123}) \cdot \grCFD(y) = \lambda \cdot \grCFD(\rho_1) \cdot \grCFD(\rho_2) \cdot \grCFD(\rho_3) \cdot \grCFD(y)$$
and use all four identities. Checking the claim for other arrows is similar.
\end{proof}

The grading on $\CFA$ can be reduced to a mod 2 grading as well in a similar way, using the same functions $f_\pm$ except that each $\pm$ should be replaced with $\mp$ in the case that $i,j \not\in \Z$. Since a generator with grading $(m;i,j)$ in $\CFD$ corresponds to a generator with grading $(m;-j,-i)$ in $\CFA$, the mod 2 gradings on $\CFA$ and $\CFD$ agree when $i,j \in \Z$ and disagree when $i,j \not\in \Z$. In other words, the relative mod two grading $\grAmodtwo$ on $\CFA$ comes from $\grDmodtwo$ by flipping the grading for one of the two idempotents. In particular, generators of $\CFA$ have opposite gradings if they are connected by an arrow labeled with the sequences $(\rho_2)$, $(\rho_3, \rho_2, \rho_1)$, or the empty sequence and the same grading if they are connected by any other arrow. A generator $x_0\otimes x_1$ in a box tensor product $\CFA(M_0,\alpha_0,\beta_0)\boxtimes\CFD(M_1,\alpha_1,\beta_1)$ inherits the grading $\grAmodtwo(x_0) + \grDmodtwo(x_1)$, which recovers the relative $\Ztwo$ grading on $\CFhat(M_0\cup M_1)$.

%%%%%%%%

\subsection{Gradings and train tracks}

\begin{figure}[ht]
\labellist
\tiny
\pinlabel {$z$} at 63 63 
\pinlabel {$z$} at 176 63 
\pinlabel {$z$} at 291 63 
\pinlabel {$z$} at 403 63 
\pinlabel {$z$} at 516 63 
\pinlabel {$z$} at 628 63 
\small
\pinlabel {$1$} at 33 35
\pinlabel {$2$} at 161 35
\pinlabel {$3$} at 266 35
\pinlabel {$12$} at 369 29
\pinlabel {$23$} at 474 35
\pinlabel {$123$} at 593 35
\endlabellist
\includegraphics[scale=0.5]{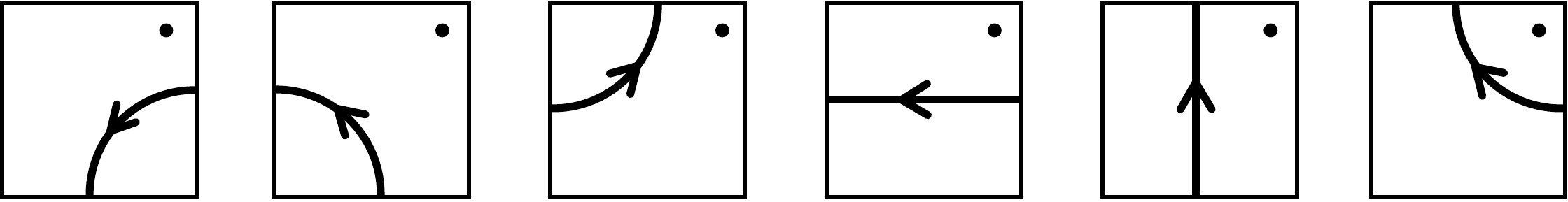}
\caption{Directed edges in $\tracks$ (labeled by $I$) corresponding to coefficient maps $D_I$ in $N$.}\label{fig:algebra-edges}
\end{figure}
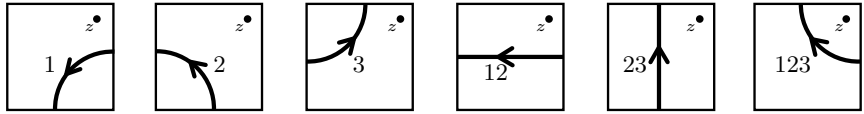

For a bordered 3-manifold $(M,\alpha,\beta)$ and a spin$^c$-structure $\spin$, consider the homotopy equivalence class of type D structures $\CFD(M,\alpha,\beta; \spin)$, and let $N$ be a reduced representative. As described in \cite[Section 2.4]{HRW}, $N$ gives rise to an immersed train track $\tracks$ in the parametrized torus $T_M = \partial M\setminus z$, which has a lift $\bar\tracks$ in $\barT_M$. Using a series of steps which correspond to homotopy equivalences of the underlying type D structure, this train track can be reduced to a curve-like train track, that is, a train track which consists of immersed curves along with crossover arrows connecting parallel segments; such a curve-like train track is interpreted as the collection curves with local systems $\curves{M;\spin}$. To prove Theorem \ref{thm:gradings}, we will show more generally that any immersed train track $\tracks$ representing $\CFD(M,\alpha,\beta; \spin)$ encodes the grading information of $\CFD(M,\alpha,\beta; \spin)$, provided $\tracks$ is path connected or decorated with extra phantom edges which make it path connected, and that the pairing of two such train tracks carries gradings which agree with those carried by the box tensor product of bordered invariants. Since this holds in particular when $\tracks$ is a curve-like train track representing (the projection to $T_M$ of) $\curves{M;\spin}$, Theorem \ref{thm:gradings} follows. In this section we will describe how to read the bordered gradings off of a train track $\tracks$ representing $\curves{M;\spin}$, and in the next section we prove the gluing result.

First, we briefly recall the construction of $\tracks$. The $\iota_0$ (resp. $\iota_1$) generators of $N$ correspond to vertices of $\tracks$ which lie on $\alpha$ (resp. $\beta)$. For a coefficient map $D_I$ and generators $x$ and $y$, a $y$ term in $D_I(x)$ corresponds to edge in the complement of $\alpha$ and $\beta$ from the vertex representing $x$ to the vertex representing $y$, according to Figure \ref{fig:algebra-edges}. Note that by construction the edges in $\tracks$ are oriented. However, for train tracks representing a type D structure over the torus algebra these orientations can be dropped, since they are determined by assuming that each edge has the basepoint $z$ on its right in $\partial M \setminus(\alpha\cup\beta)$, where $z$ is taken to be arbitrarily close to the intersection point of $\alpha$ and $\beta$, in the quadrant between the end of $\alpha$ and the end of $\beta$. By convention, we identify $\partial M \setminus(\alpha\cup\beta)$ with the square $[0,1]\times[0,1]$, where the vertical sides are $\alpha$, the horizontal sides are $\beta$, and $z$ lies at the top right corner of the square. It will be convenient to assume that the vertices of $\tracks$ all lie in a small $\epsilon$ neighborhood of the midpoints of $\alpha$ and $\beta$, the points $(0,\frac{1}{2})$ and $(\frac{1}{2}, 0)$. The train track $\tracks$ encodes both the type D structure $N$ and the corresponding type A structure $N^A = \CFAA(\mathbb{I}) \boxtimes N$.

\noindent \emph{Spin$^c$ Grading.}
Since we have constructed the train track $\tracks$ to have vertices at the midpoints of $\alpha$ and $\beta$, a (not necessarily smooth) path in $\tracks$ from one vertex $x$ to another vertex $y$ determines an element of $\frac{1}{2}H_M$, which is in $H_M$ if and only if the vertices are both on $\alpha$ or both on $\beta$. In fact, this element of $\frac{1}{2}H_M$ is precisely the difference in the spin$^c$ grading $\spin_\iota(y)-\spin_\iota(x)$ between the corresponding generators in $N$ or in $N^A$. It is sufficient to check this when the path has length one, which is straightforward upon comparing Table \ref{tab:spinc} and Figure \ref{fig:algebra-edges}.

Note that the spin$^c$ grading difference between two generators is determined by the train track alone only if the corresponding vertices are connected by a path in $\tracks$; if $\tracks$ is not path connected, some extra decoration is required to record the relative grading between components. We will achieve this by adding phantom train track edges connecting separate components, which are not counted as contributing to the differential of the underlying type D structure but which can be traversed in paths used to determine gradings. If $\tracks$ has the form of immersed curves with crossover arrows, as in the case of a curve-like train track representing $\curves{M;\spin}$, we will match this form by adding phantom crossover arrows. The choice of phantom edges is not unique; there are many possible phantom edges that will encode the same grading information. One way to find appropriate phantom edges on a train track representing $\CFD(M,\alpha,\beta;\spin)$ is to start with a representative for which the directed graph, and thus the train track, representing the type D structure is path connected (there is always such a representative---for instance, the representative computed from a nice Heegaard diagram). For this train track, no phantom edges are required. The train track can then be simplified to remove crossover arrows, but any time removing an arrow would disconnect two components the arrow should be remembered as a phantom arrow. Consider for example, the invariant for the complement of the figure eight knot in Figure \ref{fig:fig8-maslov}; the figure shows a phantom arrow connecting the two immersed curves. Adding or removing an arrow of this form is an allowable move on train tracks, corresponding to a change of basis of type D structures. Thus, if the phantom arrow were treated as a real arrow, the resulting (connected) train track still represents $\CFD$, albeit not in simplest form. Simplifying the train track by removing this arrow disconnects the train track, so relative grading information is lost unless we keep track of the phantom arrow.

If we are only interested in the spin$^c$ grading, this information can be recorded in a different way which is perhaps more natural: instead of decorating $\tracks$ with phantom edges, we enhance it by choosing a lift to a certain covering space. The lift is defined only up to an overall translation and a connected component has a unique lift up to translation, so the new information being recorded is the relative position of the lifts of different components; note that the presence of phantom arrows determines such a choice of lift. 
%Using the notation of \cite{HRW}, recall that that \(\barT_M\) is the cover of \(T_M\) whose fundamental group is the kernel of the composite map
%$ \pi_1(T_M) \to H_1(T_M) \to H_1(\partial M) \to H_1(M)$. Equivalently, $\barT_M$ is homeomorphic to the quotient $((H_1(M;\R)\setminus H_1(M;\Z)) / \langle\lambda\rangle$, where  $\lambda \in H_1(M;\Z)$ generates the kernel of the inclusion $j_*: H_1(\partial M;\Z)\to H_1(M;\Z)$. 
Note that each vertex of $\tracks$, which occurs at the midpoint of $\alpha$ or the midpoint of $\beta$, must lift to an element of $\frac{1}{2}H_M \subset \barT_M$. The relative spin$^c$ grading on $N$ determines a lift $\bar\tracks$ of $\tracks$, up to an overall translation, by requiring that the difference in spin$^c$ grading between any two generators agrees with the difference in position of the lifts of the corresponding vertices. Conversely, the relative grading can be determined by the relative position of the corresponding vertices in $\bar\tracks$. This clearly holds by construction for any vertices connected by a path in $\tracks$, but the choice of lift contains new information when $\tracks$ is not connected. Note that it is sometimes convenient to work in a higher covering space, $\tildeT_M = H_1(M;\R)\setminus H_1(M;\Z) \cong \R^2\setminus \Z^2$. Here $\tracks$ lifts to a train track $\tilde\tracks$ which is invariant under the action of $\lambda$. The position of a vertex of $\tilde\tracks$, up to the action of $\lambda$, determines an element of $\frac{1}{2}H_M$, and this is taken as the  spin$^c$ grading of the corresponding generator of $N$.
 
See for example Figure \ref{fig:maslov-example}, which shows the lift of a portion of a train track $\tracks$. The relative 
position between vertices gives the difference in spin$^c$ grading in the corresponding type D structure. If we set the generator $x_0$ to have grading $\spin(x_0) =0\in\frac{1}{2} H_M$, then the gradings of any other generators can be read from the figure. For example, $a$ has coordinates $(\frac{3}{2},\frac{1}{2})$ relative to $x_0$, so $\spin(a) = (3\beta+\alpha)/2$; similarly, $\spin(b) = \beta + \alpha$. Note that the grading difference $\spin(b)-\spin(a)$ is consistent with there being a $\rho_{2}$ labelled arrow from $a$ to $b$ (see Table \ref{tab:spinc}). In the notation of \cite{LOT}, setting $x_0$ to have grading $(0;0,0)$, the spin$^c$ component of $\gr(a)$ in $N$ is $(\frac{3}{2}, -\frac{1}{2})$ and the spin$^c$ component of $\gr(b)$ in $N$ is $(1,-1)$. In general, the generator corresponding to a vertex at coordinates $(i,j)$ relative to the origin at $x_0$ has spin$^c$ grading $(i,-j)$ (see Remark \ref{rem:compare-grading-notation}). For the spin$^c$ grading of the corresponding type A structure $N^A$, a vertex at position $(i,j)$ has spin$^c$ grading $(j,-i)$.

\noindent \emph{Maslov Grading.} 
Suppose $x$ and $y$ are generators of the type D structure $N$ which are connected by a coefficient map $D_I(x) = y$. Recall that $\gr(y) = (-1;0,0) \cdot \gr(\rho_I)^{-1} \cdot \gr(x)$. It follows that 
$$\Delta m = m(y) - m(x) = -\frac{1}{2} + \det\left(\begin{smallmatrix} -\, ( -v_{\rho_I})\, -\\ -\, v_x-\end{smallmatrix} \right) = -\frac{1}{2} + \det\left(\begin{smallmatrix} -\, (v_y - v_x)\, -\\ -\, v_x-\end{smallmatrix} \right),$$
where $m(x)$ denotes the Maslov component of $\gr(x)$ and $v_x$ denotes the spin$^c$ component as a vector in $\frac{1}{2}\Z^2$. Consistent with the theme of this section, we aim to give geometric meaning to this Maslov grading difference. 

Consider a lift $\tilde\tracks$ of $\tracks$ to the covering space $\tildeT_M$.  The generator $x$ determines a vector $w_x$ in the plane starting at the vertex of $\tilde\tracks$ corresponding to $x_0$ and ending at the vertex corresponding to $x$. Comparing conventions (c.f. Remark \ref{rem:compare-grading-notation}), note that $w_x$ is the reflection of $v_x$ in the vertical direction. Similarly, $y$ determines a vector $w_y$ which is the vertical reflection of $v_y$. Clearly $\det\left(\begin{smallmatrix} -\, (v_y - v_x)\, -\\ -\, v_x-\end{smallmatrix} \right) = -\det\left(\begin{smallmatrix} -\, (w_y - w_x )\, -\\ -\, v_x-\end{smallmatrix} \right)$, which can be interpreted as an area: it is twice the area of the triangle spanned by $w_y$, $w_x$, counted positively if it lies to the left of $w_y-w_x$ and negatively if it lies to the right. Thus this term of $m(y)-m(x)$ will be called the \emph{area contribution} to $\Delta m$. The remaining term of $-\frac{1}{2}$, which we will call the \emph{path contribution} to $\Delta m$, records the fact that the coefficient map connecting the two generators is an arrow from $x$ to $y$. Traveling from $x$ to $y$ along the relevant edge in $\tracks$ follows the edge orientation (here by edge orientation we mean the orientation coming from identifying the edge with a differential in $N$, or equivalently coming from assuming the edge keeps the basepoint $z$ on its right). If the coefficient map connecting $x$ and $y$ instead went from $y$ to $x$, traveling from $x$ to $y$ in $\tracks$ would oppose this orientation, and the path contribution to $m(y) - m(x)$ would be $+\frac{1}{2}$.

Now suppose that $x$ and $y$ are not connected by a coefficient map, but that there is a path $P$ in $\tilde\tracks$ from $x$ to $y$. The difference in Maslov gradings between successive generators passed along $P$ is defined above. Clearly $m(y)-m(x)$ is the sum of these successive $\Delta m$'s. By summing the areas of triangles, we see that the total area contribution measures twice the area enclosed by a piecewise linear deformation of $P$ from $x$ to $y$, deformed  so that successive vertices are connected by straight line segments, and a straight line from $y$ back to $x$ (see Figure \ref{fig:maslov-example}). If $P$ intersects itself, note that the area of each region is counted with multiplicity given by the winding number of the path; see, for example, Figure \ref{fig:maslov-example-crossing}. The total path contribution is $-\frac{1}{2}$ times the number of edges traversed along $P$ following the edge orientation plus $\frac{1}{2}$ times the number of edges traversed opposing the edge orientation.

We have shown that the Maslov grading difference of two generators $x$ and $y$ is determined by $\tracks$ if there is a path connecting $x$ to $y$. If $\tracks$ is decorated with a set of phantom edges which make it path connected, then this fully defines the relative Maslov grading on the corresponding type D structure $N$. The discussion above deals with the Maslov grading on the type D structure $N$, but the Maslov grading on the corresponding type A structure $N^A$ is exactly the same. Fixing a chosen generator $x_0$ with grading $(0;0,0)$ $x$ and $y$ with gradings $(m(x); \vec v_x)$ and $(m(y); \vec v_y)$ determine vectors $\vec w_x$ and $\vec w_y$, which are rotations of $\vec v_x$ and $\vec v_y$ by $\frac{\pi}{2}$. It follows that the area contribution to $\Delta m$ is the area to the left of a piecewise linear deformation of a path from $x$ to $y$, as before, and the path contribution is still $-\frac{1}{2}$ for each edge traversed following the edge orientation and $\frac{1}{2}$ for each edge traversed opposing the edge orientation.

\labellist
%leftfigure
\small
\pinlabel {$x_0$} at 270 55 
\pinlabel {$a$} at 412 125
\pinlabel {$b$} at 360 177
\pinlabel {$c$} at 315 320
\pinlabel {$d$} at 244 347	
\pinlabel {$e$} at 100 105
\pinlabel {$f$} at 53 180
\tiny
\pinlabel {$1$} at 155 219
\pinlabel {$2$} at 70 219
\pinlabel {$3$} at 70 303
%right figure
\small
\pinlabel {$x_0$} at 565 55
\pinlabel {$d$} at 570 347
\endlabellist
\begin{figure}[ht] \includegraphics[scale=0.55]{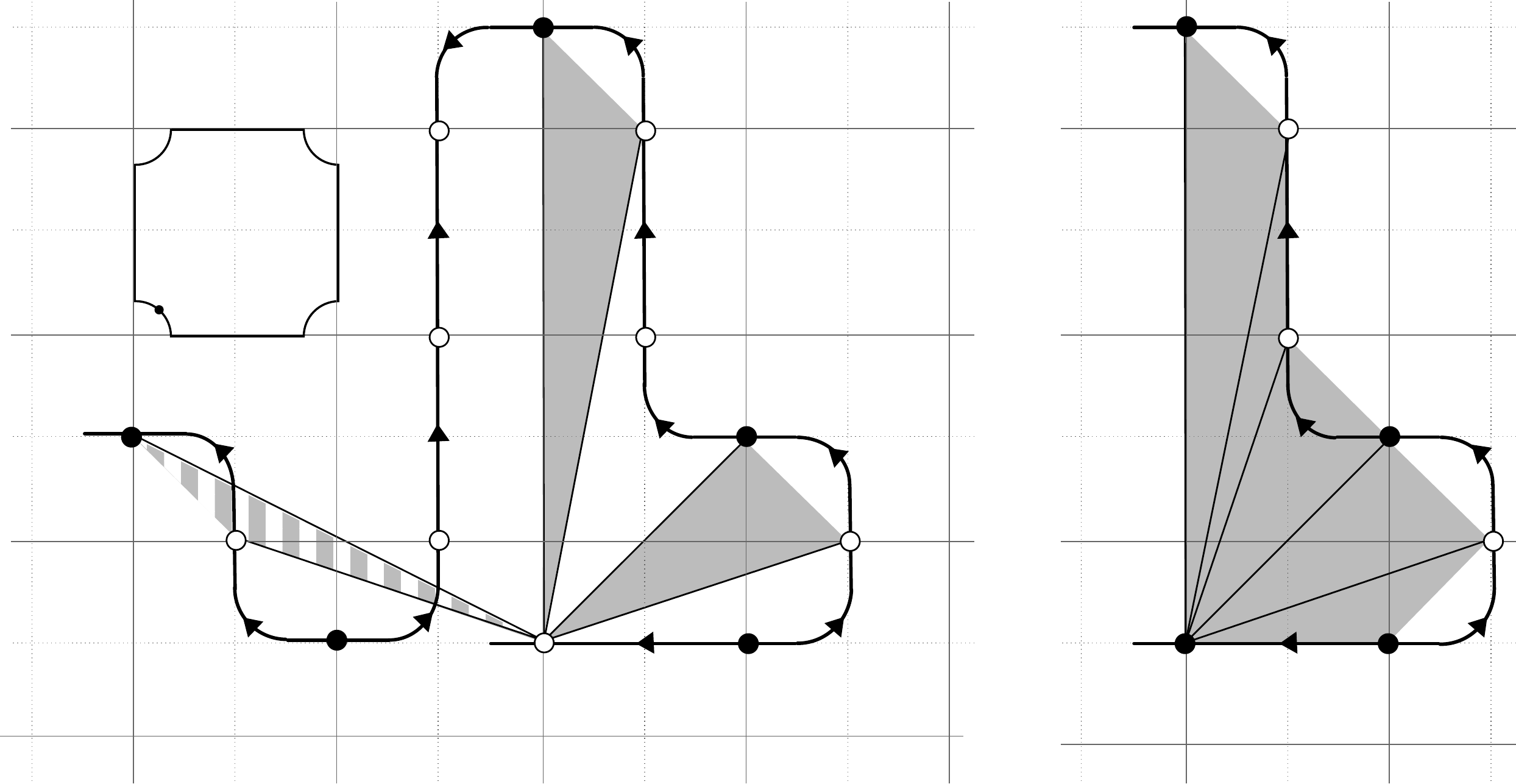}
\caption{Shifting the Maslov grading on a $\tracks$. On the left: each triangle is associated with a $\rho_2$ edge. The Maslov grading shift in each case is $-\frac{1}{2}$ plus twice the area of the triangle, with the striped area counting negatively. On the right: The Maslov grading of $d$, relative to the vertex $x_0$ fixed as the origin, is $\frac{5}{2}$. The area contribution is $\frac{9}{2}$, twice the area of the shaded shaded region, and the path contribution is -2.} \label{fig:maslov-example}
  \end{figure}%

\labellist
\small
\pinlabel {$x_0$} at 44 99
\pinlabel {$y$} at 45 295
\endlabellist
\begin{figure}
\includegraphics[scale=.4]{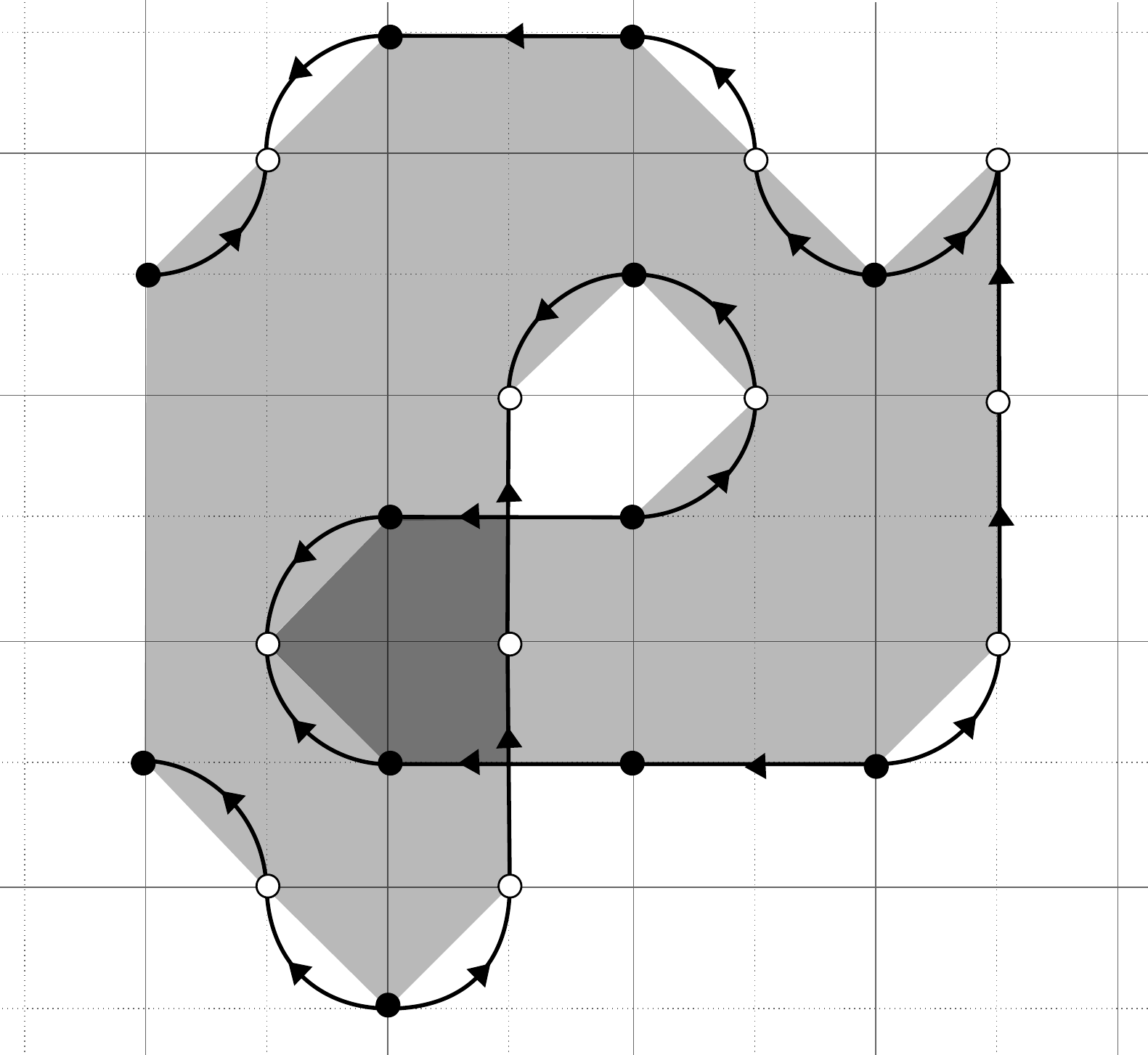}
\caption{To compute the Maslov grading of $y$, the path from $x_0$ to $y$ need not be either smooth or embedded in the plane. In this example, the area contribution to $m(y)$ is 20, twice the area of the shaded region, with the darker region counted twice. The path contribution is $-1$, so $m(y) = 19$.}
\label{fig:maslov-example-crossing}
\end{figure}

\noindent \emph{Mod 2 Grading.} The $\Z/2\Z$ reduction of the full grading on a type D structure $N$ admits a particularly simple geometric interpretation which is worth highlighting. It can be interpreted as an orientation on the train track $\tracks$ representing $N$. By this we mean a choice of orientation on each edge such that any immersed path carried by $\tracks$ either always follows or always opposes the edge orientation. This orientation should not be confused with the orientation of edges coming from viewing them as arrows in the directed graph representing $N$; to avoid this possible confusion, one can also view the $\Z/2\Z$ grading as a choice of sign on each vertex of $\tracks$, which should be viewed as reflecting the sign of the intersection point between $\alpha$ or $\beta$ and $\tracks$. This is equivalent to a choice of orientation on the small segments of $\tracks$ perpendicular to $\alpha$ or $\beta$ at each vertex. A vertex on $\alpha$ is given a positive sign (equivalently, $\tracks$ is oriented leftward near this vertex) if the corresponding $\iota_0$ generator of $N$ has $\Z/2\Z$ grading 0 and the corresponding generator of $N^A$ has $\Z/2\Z$ grading 1. A vertex on $\beta$ is given a positive sign (equivalently, $\tracks$ is oriented upward near this vertex) if the corresponding $\iota_1$ generator of $N$ or of $N^A$ has $\Z/2\Z$ grading 0. It is straightforward to check that these conventions produce a consistent orientation on $\tracks$. For example, if $N$ contains a $\rho_2$ arrow from $x$ to $y$, the corresponding edge $\tracks$ connects the top side of an intersection with $\beta$ to the right side of an intersection with $\alpha$. If $\tracks$ is oriented upward near the former, it must be oriented leftward near the latter, which is consistent with the fact that $x$ and $y$ to have equal gradings in $N$.

Such an orientation on $\tracks$ corresponds to an absolute $\Z/2\Z$ grading on a type D structure $N$. Since we are interested in $N = \CFD(M,\alpha,\beta;\spin)$, which only carries a relative grading, these orientations are well defined only up to flipping all of them. Note that up to an overall flip the orientation is determined on any path connected component of $\tracks$. If $\tracks$ is not connected but is decorated with phantom edges, the relative orientations of different components is determined by requiring that the orientation is consistent when phantom edges. Note that if we are only interested in the mod 2 grading, we could avoid using phantom edges and instead decorate $\tracks$ with a choice of relative orientation on each component. In particular, the $\Z/2\Z$ grading on $\curves{M;\spin}$ is realized as a choice of orientation on the underlying curves, up to reversing the orientation of all curves.

\subsection{Proof of the refined pairing theorem}

Having given a geometric interpretation of the grading structure on type D and type A modules, we return in this section to pairing. We will show that this grading structure defines a spin$^c$ grading on the intersection Floer homology of $\curves{M_0}$ and $h(\curves{M_1})$, and a relative Maslov grading in each spin$^c$-grading. By identifying our gradings with the gradings of \cite{LOT}, we complete the proof of Theorem \ref{thm:gradings}.

Recall that in proving the pairing theorem in \cite[Section 2]{HRW}, we worked with special representatives of the train tracks $\tracks_i$ representing the bordered invariants of $(M_i,\alpha_i,\beta_i)$, for $i = 0,1$. Specifically, $A(\tracks_i)$ is obtained by including $\tracks_i$ into the first (top right) quadrant of the square $[0,1]\times[0,1]$ and extending horizontally and vertically through the second and fourth quadrants. We continue to assume, as in the previous section, that $\tracks_i$ only intersects the parametrizing curves $\alpha_i$ and $\beta_i$ in a small neighborhood of their midpoints. It follows that in $A(\tracks_i)$ all horizontal segments and vertical segments lie arbitrarily close to the lines $y = \frac{3}{4}$ and $x = \frac{3}{4}$. The generators of $\CFA(M_i,\alpha_i,\beta_i)$, previously identified with intersections of $\tracks_i$ with $\alpha_i$ and $\beta_i$, should now be identified with (midpoints of) horizontal and vertical segments of $A(\tracks_i)$. These midpoints occur at approximately the point $(\frac{1}{4},\frac 3 4)$ for $\iota_0$ generators and $(\frac 3 4, \frac 1 4)$ for $\iota_1$ generators.

$D(\tracks_i)$ is obtained by reflecting $A(\tracks_i)$ across the anti-diagonal $y=-x$. For our purposes, $A(\tracks_0)$ and $D(\tracks_1)$ both live in $T_{M_0}$, and the reflection corresponds to the gluing map $h$; we choose the parametrization on $M_1$ such that $h(\alpha_1) = -\beta_0$ and $h(\beta_1) = -\alpha_0$. Note that $\iota_0$ (resp. $\iota_1)$ generators of $\CFD(M_i,\alpha_i,\beta_i)$ correspond to midpoints of vertical segments in the second quadrant (resp. horizontal segments in the fourth quadrant). The Floer homology of $A(\tracks_0)$ and $D(\tracks_1)$ can be identified with the homology of the box tensor product $\CFA(M_0,\alpha_0,\beta_0)\boxtimes\CFD(M_1,\alpha_1,\beta_1)$, and thus with $\HFhat(Y)$ for $Y = M_0\cup_h M_1$; for details, see \cite[Section 2]{HRW}. To complete the pairing theorem, we observed that the train tracks $A(\tracks_0)$ and $D(\tracks_1)$ can be simplified to give $\curves{M_0}$ and $h(\curves{M_1}$ without changing the Floer homology (see \cite[Section 4]{HRW}).

For the grading enhanced pairing theorem, it will be sufficient to consider the graded Floer homology of $A(\tracks_0)$ and $D(\tracks_1)$, since the grading structure is preserved by the simplifications done to obtain $\curves{M_0}$ and $h(\curves{M_1})$ from these. Moreover, since $A(\tracks_0)$ and $D(\tracks_1)$ are equipped with grading structure, we will assume that they are connected in each spin$^c$ structure, possibly after including phantom arrows.

\noindent \emph{Pairing and the spin$^c$ grading.} 
Let $x$ and $y$ be any two intersection points between $A(\tracks_0;\spin_0)$ and $D(\tracks_1;\spin_1)$, corresponding to generators $x_0 \otimes x_1$ and $y_0 \otimes y_1$ in the box tensor product of the corresponding modules. Let $p_0$ be a path from $x$ to $y$ in $A(\tracks_0;\spin_0)$; $p_0$ determines an element $i_0 \beta + j_0 \alpha$ of $\frac{1}{2}H_M$, which is the grading difference from $x_0$ to $y_0$. Recall that in bordered Floer notation, this means that $\grCFD(y_0) - \grCFD(x_0)$ has spin$^c$ component $(j_0, -i_0)$ (see Remark  \ref{rem:compare-grading-notation}). Similarly, let $p_1$ be a path from $x$ to $y$ in $D(\tracks_1;\spin_1)$ which determines an element $i_1 \beta + j_1 \alpha$ of $\frac{1}{2} H_M$. Since $D(\tracks_1;\spin_1)$ is a reflection of $A(\tracks_1;\spin_1)$ across the antidiagonal, the difference in grading between $x_1$ and $y_1$ is $-j_1 \beta + -i_1 \alpha$; in bordered Floer notation, $\grCFD(y_1) - \grCFD(x_1)$ has spin$^c$ component $(-j_1, i_1)$. Note that $\grbox(y_0\otimes y_1) - \grbox(x_0\otimes x_1)$ has spin$^c$ component $(j_0-j_1, i_1-i_0)$; $x$ and $y$ are in the same spin$^c$ structure if and only if this vector is zero, up to the action of the homological longitudes of $M_0$ and $M_1$. In fact, since $A(\tracks_0;\spin_0)$ and $D(\tracks_1;\spin_1)$ are connected and each must contain a closed path that is homotopic to $\lambda_0$ or $h(\lambda_1)$, respectively, we can change the paths $p_0$ and $p_1$ as needed to change $(j_0-j_1, i_1-i_0)$ by multiples of the images of these homological longitudes in $\frac{1}{2} H_M$. Thus, $x$ and $y$ are in the same spin$^c$ structure if and only if there are paths $p_0$ and $p_1$ from $x$ to $y$ such that $p_0$ and $p_1$ determine the same element of $\frac{1}{2} H_1(M_0;\Z)$, or equivalently, the path $p_0 - p_1$ lifts to a closed loop in $\tilde T_{M_0} \cong \R^2\setminus\Z^2$. In other words, $x$ and $y$ are in the same spin$^c$ structure if and only if there are lifts $A(\tilde\tracks_0;\spin_0)$ and $D(\tilde\tracks_1;\spin_1)$ to $\tilde T_{M_0}$ such that both $x$ and $y$ lift to intersection points.

\noindent \emph{Pairing and the Maslov  grading.} We now consider the the Maslov grading under pairing. Let $x$ and $y$ be  two intersections between $A(\tracks_0;\spin_0)$ and $D(\tracks_1;\spin_1)$, corresponding to generators $x_0\otimes x_1$ and $y_0 \otimes y_1$, and suppose that $x$ and $y$ have the same spin$^c$ grading. We can choose paths $p_0$ in $A(\tracks_0;\spin_0)$ and $p_1$ in $D(\tracks_1;\spin_1)$ from $x$ to $y$ and fix lifts $A(\tilde\tracks_0;\spin_0)$ and $D(\tilde\tracks_1;\spin_1)$ to $\tilde T_{M_0} \cong \R^2\setminus\Z^2$ such that $p_0$ and $p_1$ both lift to paths $\tilde p_0$ and $\tilde p_1$ from a lift $\tilde x$ of $x$ to a lift $\tilde y$ of $y$. Let $(i,j)$ be the vector from $\tilde x$ to $\tilde y$.

We may set both $x_0$ and $x_1$ to be reference generators for their respective modules, with grading $\grCFD(x_i) = (0;0,0)$, so that $\tilde x$ is the origin in $\tilde T_{M_0}$. Consider the grading $ \grCFD(y_0)$. The spin$^c$ component is  $(j,-i)$, since $\tilde y$ lies at coordinates $(i,j)$. The Maslov component $m(y_0)$, roughly speaking, measures twice the area to the left of a piecewise linear deformation of $\tilde p_0$ and to the right of the straight path from $\tilde x$ to $\tilde y$, along with a path contribution counting corners traversed with and against their orientation along $\tilde p_0$. Similarly, consider $\grCFD(y_1)$. The spin$^c$ component is $(-j,i)$, while the Maslov component counts twice the area to the right of the piecewise linear deformation of $\tilde p_1$ and to the left of the straight line from $\tilde x$ to $\tilde y$. Note that area to the right of $\tilde p_1$ rather than to the left is counted positively, since $D(\tilde\tracks_1;\spin_1)$ is a reflection of $(\tilde\tracks_1;\spin_1)$. Since the spin$^c$ components of $\grCFD(y_0)$ and $\grCFD(y_1)$ cancel out, we have $\grbox(y_0\otimes y_1) = (m(y_0) + m(y_1); 0, 0)$, while $\grbox(x_0\otimes x_1) = (0;0,0)$. It follows that the Maslov grading difference between $x$ and $y$ is $m(y_0)+m(y_1)$.

The sum of $m(y_0)$ and $m(y_1)$ has both an area contribution and a path contribution. The area contribution, the sum of the area contributions to $m(y_0)$ and $m(y_1)$, is twice the area to the left of a piecewise linear deformation of $\tilde p_0$ and to the right of a piecewise linear deformation of $\tilde p_1$  (see Figure \ref{fig:maslov-pairing-example}). Recall that area to the right of $\tilde p_0$ and to the left of $\tilde p_1$ is counted with negative sign; more generally, the area of any region is counted with multiplicity given by twice the winding number of the path $\tilde p_0 - \tilde p_1$ around the region. The path contribution to $m(y_0)+m(y_1)$ is the sum of the path contribution along the paths $\tilde p_0$ and $\tilde p_1$.

The preceding paragraph gives a geometric definition of the Maslov grading of $y$ relative to that of $x$. However, this description is somewhat impractical. In particular, the areas involved are not preserved by homotopy of the train tracks and a particular piecewise linear version of $\tilde p_0$ and $\tilde p_1$ is required. To fix this, we will perturb our train tracks in a specific way and compute an adjusted area contribution and an adjusted path contribution. Recall that $A(\tracks_0)$ and $D(\tracks_1)$ lie near the lines $x = \frac 1 4, x = \frac 3 4, y = \frac 1 4$, and $\frac 3 4$, and that generators (that is, midpoints of horizontal and vertical segments corresponding to generators) lie near the points $(\frac 1 4, \frac 3 4)$ and $(\frac 3 4, \frac 1 4)$. We will homotope these train tracks to lie in a neighborhood of the lines $x = \frac 1 2$ and $y = \frac 1 2$; thus horizontal segments in $A(\tracks_0)$ shift down by $\frac 1 4$ while vertical segments shift left by $\frac 1 4$. Generators now lie near the point $(\frac 1 2, \frac 1 2)$, and we take the paths $p_0$ and $p_1$ to be piecewise linear connecting successive generators. The adjusted area contribution to $m(y)$ is twice the area bounded by the adjusted paths $\tilde p_0$ and $\tilde p_1$. This area is composed of one by one blocks, each centered on a puncture (see for example the middle of Figure \ref{fig:maslov-pairing-example}); it follows that the adjusted area contribution is simply twice the number of punctures enclosed by the loop $\tilde p_0 - \tilde p_1$ (again, counted with appropriate multiplicity).

The adjusted path contribution is defined so that the adjusted path and area contributions combine to the same value as the original path and area contributions. In the original path contribution, any corner contributed $-\frac 1 2$ if traversed following the corner orientation and $\frac 1 2$ if traversed with opposite orientation. For each corner, the adjusted path contribution is $-\frac 1 2$ minus the area gained near the corner when the curves are shifted. See Figure \ref{fig:figures-maslov-area-shift} for the adjusted corner contributions in $A(\tracks_0)$. Note that the contribution is $-\frac 1 2$ for a left turn, $\frac 1 2$ for a right turn, and $0$ for no turn. By reflecting these figures, it is easy to see the adjusted path contribution is the opposite for the path $\tilde p_1$ in $D(\tracks_1)$; however, if we follow the path $\tilde p_1$ backwards, forming a closed path $\tilde p_0 - \tilde p_1$, then the adjusted path contribution is the same everywhere. Thus the total path contribution to $m(y)$ is $\frac R \pi$, where $R$ is the net clockwise rotation of the path $\tilde p_0 - \tilde p_1$ in radians, ignoring any cusps and the corners at $\tilde x$ and $\tilde y$.

This gives the following intrinsic characterization of the relative Maslov grading on the Floer homology of (connected) immersed train tracks:

\begin{definition}
Let $\tracks_0, \tracks_1$ be immersed train tracks in $T$ which are connected, and let $x$ and $y$ be two intersection points which are in the same spin$^c$ structure in $HF(\tracks_0, \tracks_1)$. For $i = 0,1$, let $p_i$ be a path from $x$ to $y$ in $\tracks_i$ such that $p_0 - p_1$ lifts to a closed piecewise smooth path $\gamma$ in $\R^2 \setminus \Z^2$. The Maslov grading difference $m(y)-m(x)$ is given by twice the number of lattice points enclosed by $\gamma$ (where each point is counted with multiplicity the winding number of $\gamma$) plus $\frac 1 \pi$ times the net total righward rotation along the smooth segments of $\gamma$.
\end{definition}

This is clearly invariant under regular homotopies of $\tracks_0$ and $\tracks_1$, provided we assume that $\tracks_0$ and $\tracks_1$ intersect orthogonally. We take this as the definition of the relative Maslov grading on each spin$^c$ component of $HF(\tracks_0, \tracks_1)$; since we showed that it agrees with the relative Maslov grading on $\HFhat(M_0 \cup_h M_1)$, this completes the proof of the grading enhanced pairing theorem.

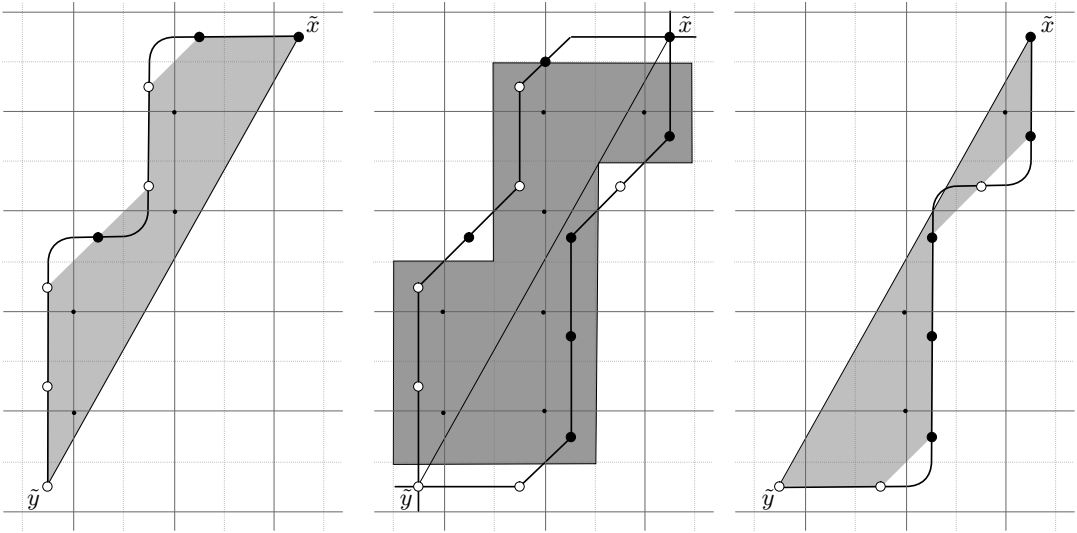
\begin{figure}[ht] 
\labellist
\small
\pinlabel {$\tilde x$} at 300 490
\pinlabel {$\tilde x$} at 660 490
\pinlabel {$\tilde x$} at 1010 490

\pinlabel {$\tilde y$} at 30 30
\pinlabel {$\tilde y$} at 390 30
\pinlabel {$\tilde y$} at 740 30

%\pinlabel {$\frac{1}{16}$} at 305 -20
%\pinlabel {$-\frac{1}{4}$} at 444 -20
%\pinlabel {$-\frac{1}{4}$} at 562 -20
%\pinlabel {$-\frac{1}{16}$} at 661 -20
\endlabellist
\includegraphics[scale=0.4]{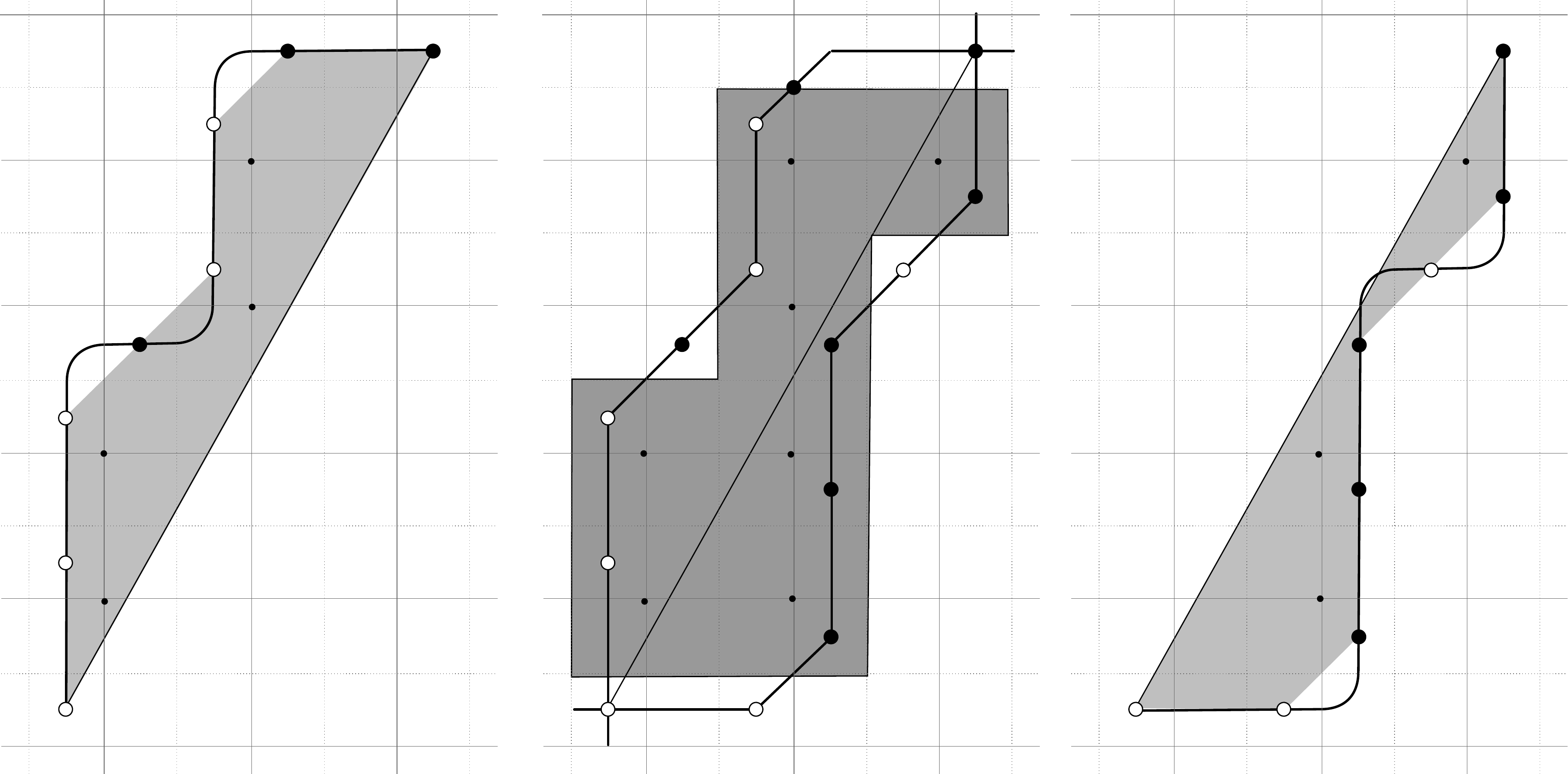}
\vspace*{10pt}
\caption{Computing the Maslov grading difference between two intersection points $\tilde x$ and $\tilde y$. Left: the grading difference in $\CFA$ is twice the area of the shaded region, plus a contribution from each corner of the curve. Right: the grading difference in $\CFD$ is twice the area of the shaded region, plus a contribution for each corner. Middle: The Maslov grading difference is twice the area enclosed by the two piecewise linear curves, plus a contribution from each corner. This is equivalent to twice the adjusted area (shaded) plus an adjusted corner contribution (see Figure \ref{fig:figures-maslov-area-shift}).}
\label{fig:maslov-pairing-example}
\end{figure}

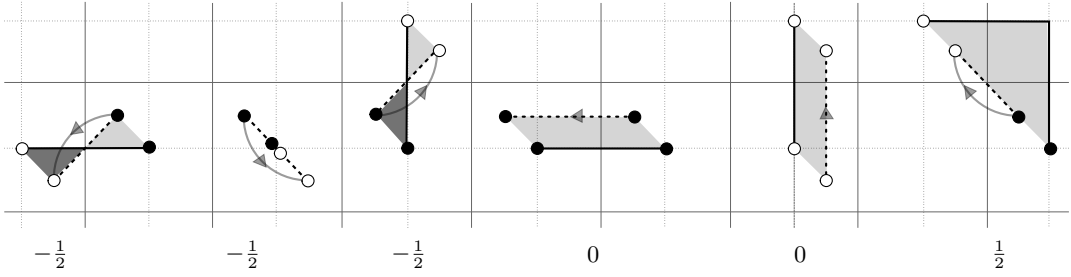
\begin{figure}[ht] 
\labellist
\small
\pinlabel {$-\frac{1}{2}$} at 35 -20
\pinlabel {$-\frac{1}{2}$} at 180 -20
\pinlabel {$-\frac{1}{2}$} at 305 -20
\pinlabel {$0$} at 444 -20
\pinlabel {$0$} at 600 -20
\pinlabel {$\frac{1}{2}$} at 750 -20
\endlabellist
\includegraphics[scale=0.5]{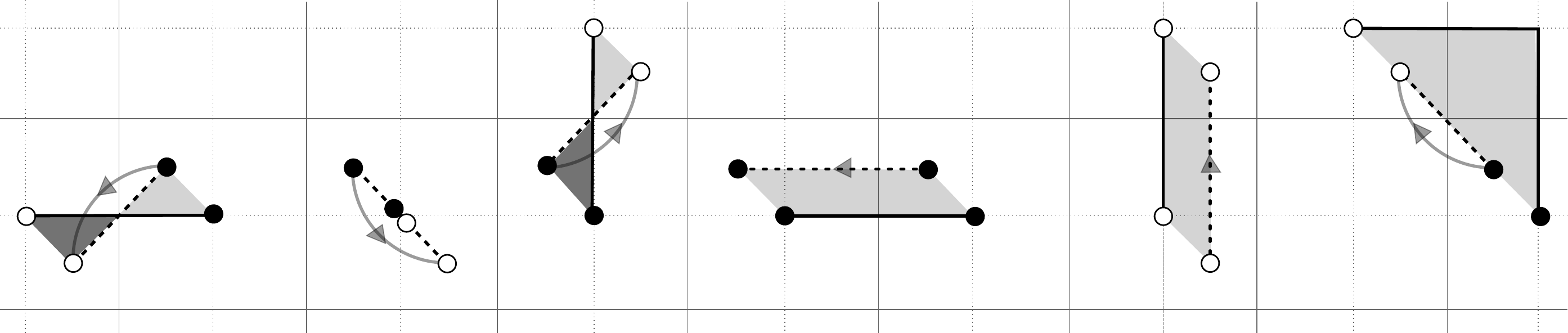}
\vspace*{10pt}
\caption{Adjusting the corners in $A(\tracks_0)$. The gray curve is part of the path $\tilde p_0$ in its usual form, the dotted line is the piecewise linear path connecting generators, before adjustment, and the solid line is the path after adjustment. Area to the left of $\tilde p_0$ lost by the shift is shaded lightly, while area gained by the shift is shaded darkly. The adjusted corner contribution, when following the corner orientation, $-\frac 1 2$ plus twice the net area lost, is given.}\label{fig:figures-maslov-area-shift}
\end{figure}

\noindent \emph{Pairing and the $\Z/2\Z$ grading.} 
Finally we remark that the $\Z/2\Z$ reduction of the Maslov grading on the Floer homology of $A(\tracks_0;\spin_0)$ and $D(\tracks_1;\spin_1)$ admits a nice description that does not require computing the full Maslov grading: it is given by the sign of intersection points, where $(\tracks_0;\spin_0)$ and $(\tracks_1;\spin_1)$ carry orientations encoding their $\Z/2\Z$ gradings following the conventions in the previous section. It is easy to check that at each intersection point, this agrees with the sum of the $\Z/2\Z$ gradings associated to the intersecting horizontal and vertical segments, and thus to the corresponding grading in $\HFhat(Y)$.  For example, an intersection point at $(\frac 1 4, \frac 3 4)$ is the intersection of a horizontal segment in $A(\tracks_0;\spin_0)$, corresponding to an $\iota_0$ generator $x_0$ in $\CFA(M_0,\alpha_0,\beta_0;\spin_0)$, and a vertical segment in $D(\tracks_1;\spin_1)$ corresponding to an $\iota_0$ generator $x_1$ in $\CFD(M_1,\alpha_1,\beta_1;\spin_1)$. If the intersection point has positive sign (corresponding to grading 0), then either the segments are oriented upward/leftward or downward/rightward. In the first case, both $x_0$ and $x_1$ have grading 0, while in the second case $x_0$ and $x_1$ both have grading 1. Similarly, if the intersection has negative sign then $x_0$ and $x_1$ have opposite gradings. It is simple to check the corresponding relationship for intersection points near $(\frac 3 4, \frac 1 4)$.  Note that flipping the orientation on either $\tracks_0$ or $\tracks_1$ changes the sign of every intersection point, so this gives a well defined relative grading on each preimage $\pi^{-1}(\spin_1 \times \spin_2)$.

%% file: sections/elliptic.tex
% !TEX root = ../companion.tex
%elliptic.tex

\subsection{Orientation reversal}
We begin by describing the effect on $\curves M$ of reversing the orientation on $M$. In fact, we will show that as decorated curves in $\partial M$, the invariant does not change under this orientation reversal; however, since representing $\partial M$ on the page depends on the choice of orientation, our figures for $\partial M$ will change by a reflection.

Given a bordered manifold $(M,\alpha,\beta)$, the effect of orientation reversal on the type D structure $\CFD(M,\alpha,\beta)$ is easy to describe: If $\mathcal{H}$ is a bordered Heegaard diagram representing $(M,\alpha,\beta)$, then reversing the orientation on $\mathcal{H}$ gives a bordered Heegaard diagram $-\mathcal{H}$ for $(-M, \beta, \alpha)$. The holomorphic curves counted in the definition of $\CFD$ are unchanged, but their direction is reversed. Since the labeling of Reeb chords along the boundary of the Heegaard diagram depends on its orientation, the labels $\rho_1$ and $\rho_3$ are reversed. Moreover, the labeling and orientation of the two arcs $\alpha_1^a$ and $\alpha_2^a$ on the Heegaard diagram is reversed; this changes the idempotent associated to every generator.
\begin{remark}
We remind the reader of a potential notational confusion: the parametrizing curves $\alpha$ and $\beta$ in our notation for a bordered manifold should not be confused with $\alpha$ and $\beta$ curves in a bordered Heegaard diagram $\mathcal{H}$. Rather, $\alpha$ and $\beta$ correspond to the arcs $\alpha_2^a$ and $-\alpha_1^a$ in $\mathcal{H}$, respectively, so that their roles are reversed under orientation reversal of $\mathcal{H}$.
\end{remark}
In summary, a directed graph representing $\CFD(-M,\beta,\alpha)$ can be obtained from a graph representing $\CFD(M,\alpha,\beta)$ by switching the labeling (between $\bullet$ and $\circ$) on each vertex, reversing the direction of every arrow, and interchanging $\rho_1$'s with $\rho_3$'s and $\rho_{12}$'s with $\rho_{23}$'s. Comparing with the construction of train tracks in the parametrized torus from decorated graphs, it is clear that this corresponds to reflection across the diagonal line in the square, i.e. the curve $\alpha + \beta$.

Recall that drawing curves for $M$ in a standard parametrized torus depends not only on the curves $\curves M\setminus z$ in $\partial M$ but also on the map from $\partial M$ to the standard torus determined by a parametrization of $\partial M$. In this case the reflection in the square exactly corresponds to the (orientation reversing) change of parametrization taking $(\alpha,\beta)$ to $(\beta, \alpha)$. In other words, $\curves M$ and $\curves{-M}$ are the same as decorated immersed curves in $\partial M\setminus z$, though since the orientation of $\partial M$ is reversed we compose any map to the standard torus with an orientation reversing diffeomorphism of the torus. Our convention in this paper is to parametrize the torus such that $\beta$ is the rational longitude of $M$; in order to maintain this convention, we should use the reflection $\alpha \to -\alpha$ for orientation reversal. Thus in our figures, orientation reversal corresponds to reflection in the vertical direction. See, for instance, Figure \ref{fig:orientation-reversal}, where the invariant for the left handed trefoil is obtained from the invariant for the right handed trefoil by reflecting in the vertical direction. Since the figure eight knot is amphichiral, the curve associated with its complement is symmetric with respect to this reflection.

\begin{figure}
\includegraphics[scale=.8]{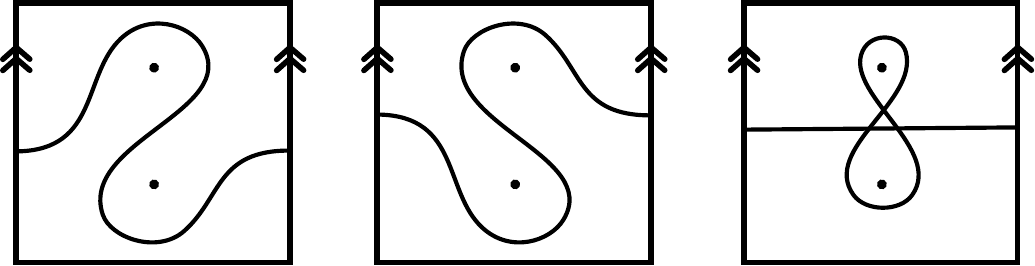}
\caption{The invariant $\HFhat(M)$ for the complements of the right-hand trefoil, left-handed trefoil, and the figure eight. Note that the two trefoil curves are related by vertical reflection and the invariant of the figure eight is symmetric under this reflection.}
\label{fig:orientation-reversal}
\end{figure}

\subsection{Spin$^{\mathbf c}$ conjugation and the elliptic involution} The goal of this section is to prove Theorem \ref{thm:spin-c-sym}, which identifies the action of the elliptic involution on the punctured torus with Spin$^c$ conjugation. As a starting point,  \cite[Theorem 3]{LOT2011} asserts that
\[\CFD(M,\spin) \cong \CFA(M,c(\spin))\]
where the left hand side is regarded as a right type A structure. Recall that a type D structure determines a left-differential module over $\Alg$ via box tensor product with $\Alg$; and that a left $\Aop$ module is the same as a right $\Alg$ module. In particular, Spin$^c$ conjugation results from tensoring with the algebra $\Alg$, and \cite[Theorem 3]{LOT2011}  may be rephrased in terms of type D structures

\[\CFD(M,c(\spin)) \cong \bbA\boxtimes\CFD(M,\spin)\]

where $\bbA$ is the type DA bimodule structure on $\Alg$ (viewed as a bimodule). We aim to prove:

\begin{theorem}\label{A-E-bim} For any manifold with torus boundary $\CFD(M,c(\spin) )\cong \bbE\boxtimes\CFD(M,\spin)$ where $\bbE$ is the type DA structure associated with the elliptic involution on the punctured torus. \end{theorem}

A formal consequence of this result is a symmetry under the elliptic involution:

\begin{corollary}The elliptic involution acts trivially on the (unlabelled) curve-set $\HFhat(M)$.\qed \end{corollary}

%\begin{proof}This is an entirely formal application of Theorem \ref{A-E-bim}: Write $\CFD(M)=\oplus_{\spin\in\spinc(M)}\CFD(M,\spin)}$.\end{proof}

The proof of Theorem \ref{A-E-bim} involves a third bimodule: the half-identity DA bimodule $\bbI=\halfid$. Direct calculation of $\bbE$ (see \cite{LOT2013}) shows that this bimodule is precisely $\bbA\boxtimes\bbI$. Relating these three bimodules we have the following amalgam of results:

\begin{proposition}\label{prop:all-the-bimodules}
Let $\mathcal{M}$ be the category of type D structures associated with manifolds having torus boundary. The following diagram commutes:
\[
  \begin{tikzpicture}[scale=0.75,>=stealth', thick] 
\node  at (0,0) {$\mathcal{M}$};
\node  at (-2,2) {$\mathcal{M}$};
\node  at (2,2) {$\mathcal{M}$};
\draw[->,shorten <= 0.3cm,shorten >= 0.3cm] (0,0) to[bend right=15] (2,2);
\draw[->,shorten <= 0.3cm,shorten >= 0.3cm] (0,0) to[bend left=15] (-2,2);
\draw[->,shorten <= 0.3cm,shorten >= 0.3cm] (-2,2) to[bend left=20] (2,2);
\draw[<-,shorten <= 0.3cm,shorten >= 0.3cm] (-2,2) to[bend right=20] (2,2);
\node  at (0,2.75) {$\bbI\boxtimes-$};\node  at (0,1.25) {$\bbI\boxtimes-$};
\node  at (-2,0.75) {$\bbE\boxtimes-$};\node  at (2,0.75) {$\bbA\boxtimes-$};
   \end{tikzpicture}
\]
\end{proposition}
\begin{proof}This commutative diagram combines various works of Lipshitz, Ozsv\'ath and Thurston. The operation on the right-hand side of the diagram is studied in \cite{LOT2011} (as described above) while the operation on the left-hand side of the diagram can be calculated explicitly  following the methods developed in \cite{LOT2013} (see Figure \ref{fig:AZ}). The key point is that the Heegaard diagram for \(A\) glued to
the Heegaard diagram for \(I\) gives precisely the Heegaard diagram for \(E\). From this we see directly that $\bbE$ and $\bbI\boxtimes\bbA$ agree. Finally, recall that  the identity bimodule $\CFDA(\mathbb{I})$ fixes type D stuctures \cite{LOT-bimodules}; again consulting  \cite{LOT2013} (or computing directly), $\CFDA(\mathbb{I})$ agrees with $\bbI\boxtimes\bbI$. Hence $\bbI\boxtimes-$ is an involution, completing the commutative diagram. \end{proof}

As a result, the proof of Theorem \ref{A-E-bim} reduces to establishing the behaviour of the half-identity bimodule on curves. Before carrying this out, for the reader's convenience, we review the constructions involved in Proposition \ref{prop:all-the-bimodules}.

\begin{figure}[t]
\labellist 
\tiny
\pinlabel {$\iota_0$} at 22 145
\pinlabel {$\rho_1$} at 22 100
\pinlabel {$\rho_{12}$} at 24 59
\pinlabel {$\rho_{123}$} at 25.5 17
\pinlabel {$\iota_1$} at 49 107
\pinlabel {$\rho_2$} at 49 60
\pinlabel {$\rho_{23}$} at 51 17
\pinlabel {$\rho_{3}$} at 76.5 17.5
%left
\pinlabel {$3$} at -3 45 \pinlabel {$2$} at -3 85 \pinlabel {$1$} at -3 125 %type A (side)
\pinlabel {$1$} at 30 -3 \pinlabel {$2$} at 58 -3 \pinlabel {$3$} at 88 -2 %typs D (bottom)
%right
\pinlabel {$1$} at 244 45 \pinlabel {$2$} at 244 85 \pinlabel {$3$} at 244 125 %type D (side)
\pinlabel {$1$} at 352.5 56 \pinlabel {$2$} at 339 107 \pinlabel {$3$} at 308 148 %type A 
\endlabellist
\includegraphics[scale=.8]{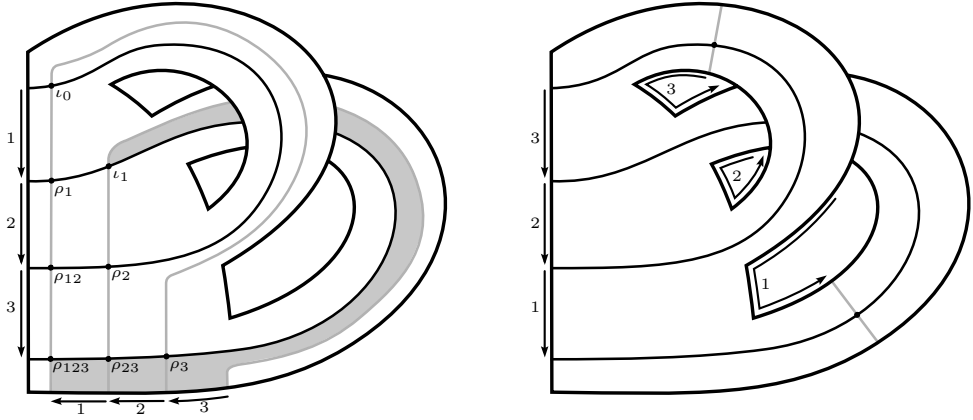}
\caption{Heegaard diagrams for the bimodules of interest, where the $\beta$-arcs are grey. On the left (compare \cite{LOT2011}): the (nice) diagram $\mathsf{AZ}$ computing the type DA structure $\bbA$. Note that the type A side corresponds to the right action of $\Alg$ on $\Alg$ by multiplication, which we have identified with the vertical edge of the diagram; the intersections are labeled by the corresponding elements in $\Alg$. The shaded domain gives rise to the dashed arrow in Figure \ref{fig:master-bimodule}. On the right (compare \cite{LOT2013}): the diagram computing $\bbI=\halfid$, where all domains are polygons. The type D structure corresponds to the vertical edge, so that $\bbE=\bbI\boxtimes\bbA$ is obtained by identifying the $\beta$-arcs. The resulting diagram no longer has a domain  corresponding to the dashed edge in Figure \ref{fig:master-bimodule}.}
\label{fig:AZ}
\end{figure}
\begin{figure}[h!]
\labellist 
\tiny
\pinlabel {$\boldsymbol\iota_0$} at 10 96
\pinlabel {$\boldsymbol\rho_{12}$} at 82 104
\pinlabel {$\boldsymbol\rho_{23}$} at 242 104
\pinlabel {$\boldsymbol\iota_1$} at 317 95
\pinlabel {$\boldsymbol\rho_1$} at 114 156
\pinlabel {$\boldsymbol\rho_{123}$} at 215 156
\pinlabel {$\boldsymbol\rho_2$} at 116 36
\pinlabel {$\boldsymbol\rho_{3}$} at 211 36
%type D
\pinlabel {$\rho_1\otimes$} at 112 71 \pinlabel {$\rho_1\otimes$} at 210 120  \pinlabel {$\rho_1\otimes$} at 279 180
\pinlabel {$\rho_2\otimes$} at 210 71 \pinlabel {$\rho_2\otimes$} at 67 45
\pinlabel {$\rho_3\otimes$} at 260 45
\pinlabel {$\rho_{123}\otimes$} at 260 147
%type A
\pinlabel {$\cdot\rho_1$} at 67 145
 \pinlabel {$\cdot\rho_2$} at 110 120  \pinlabel {$\cdot\rho_2$} at 280 15
 \pinlabel {$\cdot\rho_3$} at 48 15  \pinlabel {$\cdot\rho_3$} at 160 73  \pinlabel {$\cdot\rho_3$} at 160 118
  \pinlabel {$\cdot\rho_{12}$} at 55 90
  \pinlabel {$\cdot\rho_{23}$} at 160 155  \pinlabel {$\cdot\rho_{23}$} at 270 90
  \pinlabel {$\cdot\rho_{123}$} at 50 185
\endlabellist
\includegraphics[scale=.8]{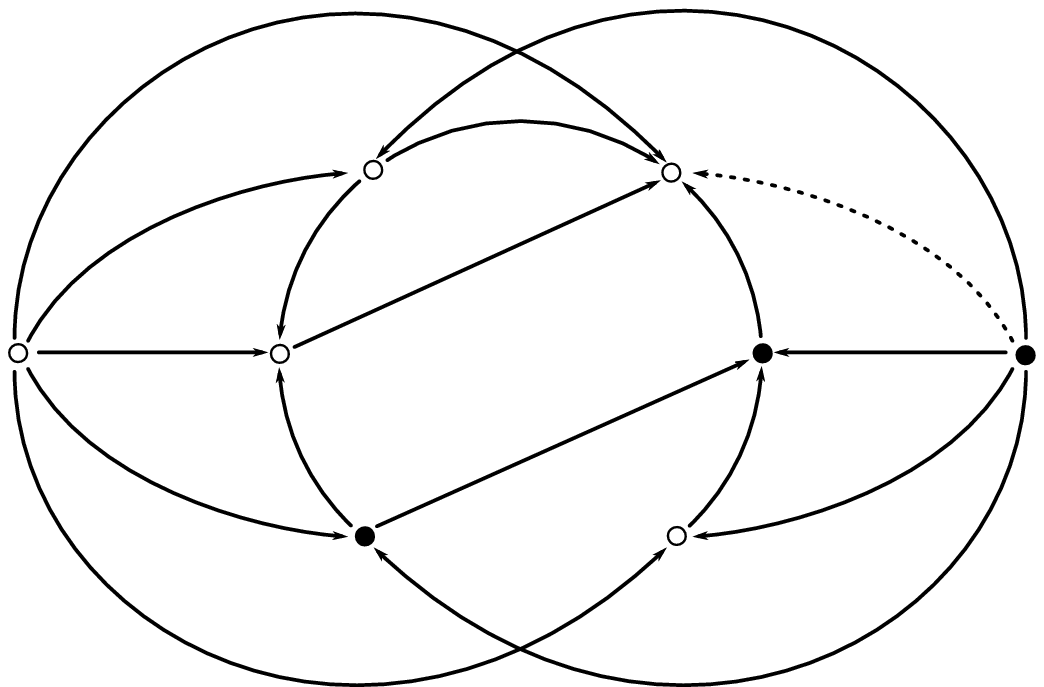}
\caption{The bimodule $\bbA$, where forgetting the dashed edge (equivalently, box-tensoring with $\bbI=\halfid$) recovers the bimodule $\bbE$. (Note that the bimodule shown is $\bbI\boxtimes\bbA$ rather than $\bbA\boxtimes\bbI$; they are equivalent.)  The vertices are in on-to-one correspondence with the elements of $\Alg$ (distinguished using bold face), and our conventions label vertices according to idempotents on the outgoing type $D$ side. The type A multiplications act on the right, so that e.g. $m_2(\boldsymbol\iota_0,\rho_{12})=\boldsymbol\iota_0\cdot\rho_{12} = \boldsymbol\rho_{12}$, while the type D structure is a left action with e.g. $\delta^1(\boldsymbol\iota_0)=\rho_2\otimes\boldsymbol\rho_2$. }
\label{fig:master-bimodule}
\end{figure}

\begin{figure}[h!]
\labellist 
\tiny
\pinlabel $+$ at 109 93
\pinlabel $+$ at 141 93
\pinlabel $+$ at 172 93
\pinlabel {$\rho_{1}$} at 84 80 \pinlabel {$\rho_{1}$} at 105 108
\pinlabel {$\rho_{3}$} at 114 80 \pinlabel {$\rho_{3}$} at 137 108
\pinlabel {$\rho_{123}$} at 146 80 \pinlabel {$\rho_{123}$} at 172 108
\pinlabel {$\rho_{123}$} at 178 80 \pinlabel {$\rho_{3}$} at 198 108 \pinlabel {$\rho_{2}$} at 210 108 \pinlabel {$\rho_{1}$} at 222 108
\pinlabel {$\rho_{2}$} at 130 1 \pinlabel {$\rho_{2}$} at 151 28
\pinlabel {$\rho_{12}$} at -2 42 \pinlabel {$\rho_{12}$} at 21 71
\pinlabel {$\rho_{23}$} at 264 42 \pinlabel {$\rho_{23}$} at 287 71
\endlabellist
 \includegraphics[scale=.8]{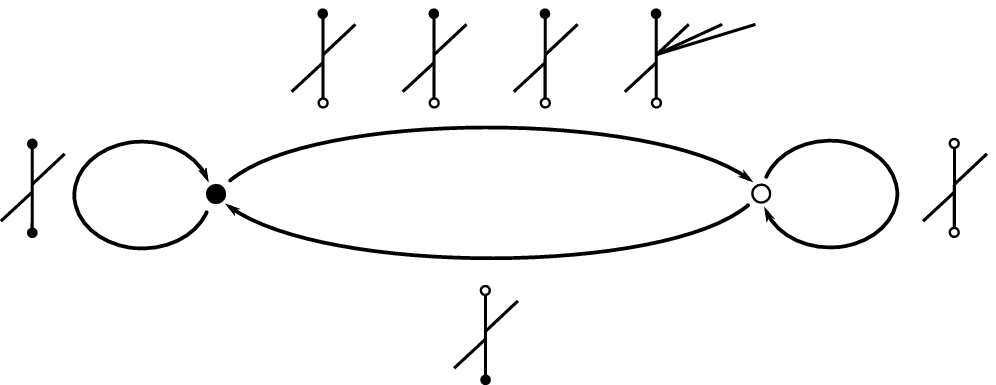}
\caption{The half-identity bimodule $\bbI=\halfid$. One can check directly from this description that $\bbI\boxtimes\bbI$ is the identity bimodule; compare \cite[Theorem 4]{LOT-bimodules}. Indeed, the difference is the presence of the operation $(\bu,\rho_3,\rho_2,\rho_1) \mapsto \rho_{123}\otimes\circ$.}
\label{fig:halfid}
\end{figure}

First, as observed in \cite{LOT2011}, the bimodule structure on $\Alg$ can be computed directly from an explicit Heegaard diagram $\mathsf{AZ}$; see Figure \ref{fig:AZ}. From this it is possible to calculate the type DA structure $\bbA$; see Figure \ref{fig:master-bimodule}. In particular, we have that $\bbA=\CFDA(\mathsf{AZ})$. On the other hand, it follows from \cite{LOT2013} that $\bbE=\halfid\boxtimes\CFDA(\mathsf{AZ})$. See \cite[Section 3]{LOT2013} for the explicit construction of a Heegaard diagram associated with a given mapping class $\phi$, from which $\CFDA(\phi)$ may be computed directly. In the present setting, $\phi$ is the elliptic involution realized as the composition of six Dehn twists; we leave it as an exercise to compute $\bbE=\CFDA(\phi)=\bbI\boxtimes\bbA$ following  \cite{LOT2013}. We have recorded the biomdule $\halfid$ in Figure \ref{fig:halfid}. 

When considering the effect of these three bimodules on the invariants $\HFhat(M)$, we regard $\bbA\boxtimes -$,  $\bbI\boxtimes -$, and $\bbE\boxtimes -$ as endofunctors $\FA$, $\FI$, and $\FE$ on the appropriate Fukaya category. Ultimately, we will show   that $\FA$ and $\FE$ have the same effect on any curve-set $\HFhat(M)$. It is interesting to see how this works on some simple examples: By construction, $\FE$ acts on immersed curves via the elliptic involution. That is, $\FE(\HFhat(M))$ is the result of rotating the square representing the torus by 180 degrees. This can be seen explicitly in examples; see Figure \ref{fig:sym-trefoil}, for instance. Notice that this operation interchanges $1 \leftrightarrow 3$ and $2\leftrightarrow 123$. In fact, since $\bbE$ is obtained by composing Dehn twists, one can apply loop calculus to systematically study the effect of $\bbE\boxtimes - $ on any decorated graph associated with an extendable type D structure. 

\labellist 
\small
\pinlabel {$\FA$} at 128 59
\pinlabel {$\longrightarrow$} at 128 48
\endlabellist
\begin{figure}[h]
\includegraphics[scale=.8]{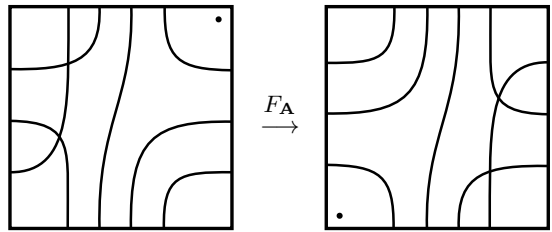}
\caption{The effects of $\FE$ and $\FA$ on the invariant of the right-hand trefoil exterior are identical.}
\label{fig:sym-trefoil}
\end{figure}

One immediately notices that $\FE(\HFhat(M)) = \FA(\HFhat(M))$ for a great majority of curves, owing to the fact that $\bbA$ and $\bbE$ differ by only one operation (the dashed edge of Figure \ref{fig:master-bimodule} labelled by $\rho_{123}\otimes$). In fact, there is a very restricted setting in which the behaviour of these two operations is different. To illustrate the behaviour we need to treat more carefully, consider the example given in Figure \ref{fig:sym-eight}. 

\labellist 
\small
\pinlabel {$\FA$} at 128 59
\pinlabel {$\longrightarrow$} at 128 48
\endlabellist
\begin{figure}[h]
\includegraphics[scale=.8]{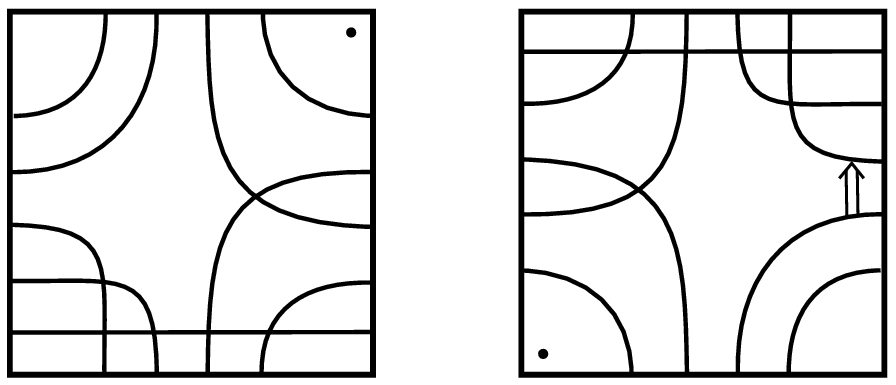}
\caption{When $M$ is the exterior of the figure eight knot, $\FE(\HFhat(M))$ and $\FA(\HFhat(M))$ are isomorphic as $\Alg$-modules after applying a single crossover arrow removal; compare Figure \ref{fig:review-example} for a summary and \cite[Section 3]{HRW} for details.}
\label{fig:sym-eight}
\end{figure}

\parpic[r]{
\labellist 
\tiny
\pinlabel {$3$} at -4 33 \pinlabel {$1$} at 63 33
\pinlabel {$2$} at 30 -2 \pinlabel {$123$} at 30 67
%\pinlabel {$0$} at 196 30.5 
%\pinlabel {$1$} at 196 39
%\pinlabel {$2$} at 203 39
%\pinlabel {$3$} at 203 30.5
\pinlabel {$z$} at 135 56
\endlabellist
 \begin{minipage}{45mm}
 \centering
 \includegraphics[scale=.8]{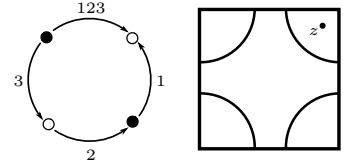}% \includegraphics[scale=.8]{figures/circ-1}
 \captionof{figure}{Two views of the curve $\sC$: as a decorated graph and as a curve in the torus.}
 \label{fig:circ-1}
  \end{minipage}%
  }
Based on the observations above, one might guess that $\bbA$ and $\bbE$ are in fact equivalent as type DA bimodules. This is not the case. Consider the curve $\sC$ consisting of a single embedded circle enclosing the basepoint, that is, $\sC$ is homotopic to the boundary circle in the torus minus (a neighbourhood of) the marked point $z$. As a decorated graph, $\sC$ this has 4 generators and one of each type of labelled edge. We will confuse the curve $\sC$ with its associated  type D structure (and the extension of this type D structure). 
%In the peg-board diagram (i.e. the cover $\R^2\setminus\Z^2$), $\sC$ encloses a single peg. 
We can now deduce that $\bbA$ and $\bbE$ are distinct type DA bimodules, since $\bbA\boxtimes\sC$ and $\bbE\boxtimes \sC$ do not agree. Indeed, the former gives rise to a non-compact curve; the calculation is summarized  in Figure \ref{fig:circ-3}. From this we conclude that $\FA$ and $\FE$ are different functors, and that $\sC$ will never arise as a summand in $\HFhat(M)$ for some three-manifold $M$ with torus boundary. We do, however, need to treat the more general case where the curve $\sC$ carries a non-trivial local system. Before doing so, we will prove that this is the only curve and local system that requires closer attention.
\labellist 
\small
%\pinlabel {$\FA$} at 128 59
%\pinlabel {$\longrightarrow$} at 128 48
\tiny
%1
\pinlabel {$3$} at -4 52 \pinlabel {$1$} at 63 52
\pinlabel {$2$} at 30 17 \pinlabel {$123$} at 30 86
%2
\pinlabel {$123$} at 151 30 
\pinlabel {$1$} at 137 40
\pinlabel {$2$} at 127 30
\pinlabel {$3$} at 118 21
\pinlabel {$1$} at 110 12
\pinlabel {$123$} at 171 92
\pinlabel {$1$} at 197 89
\pinlabel {$3$} at 189 80
\pinlabel {$2$} at 180 71
\pinlabel {$1$} at 170 61
\pinlabel {$1$} at 175.5 34 \pinlabel {$2$} at 186.5 24
\pinlabel {$1$} at 132.5 75 \pinlabel {$2$} at 123 84
%3
\pinlabel {$123$} at 320 92
\pinlabel {$1$} at 345 89
\pinlabel {$3$} at 337 80
\pinlabel {$2$} at 328 71
\pinlabel {$1$} at 318 61
\pinlabel {$3$} at 265 21
\pinlabel {$2$} at 332.5 24
%4
\pinlabel {$3$} at 395 52 \pinlabel {$1$} at 461 52
\pinlabel {$2$} at 428 17 \pinlabel {$123$} at 428 86 \pinlabel {$123$} at 428 56
\endlabellist
\begin{figure}[h]
\includegraphics[scale=.8]{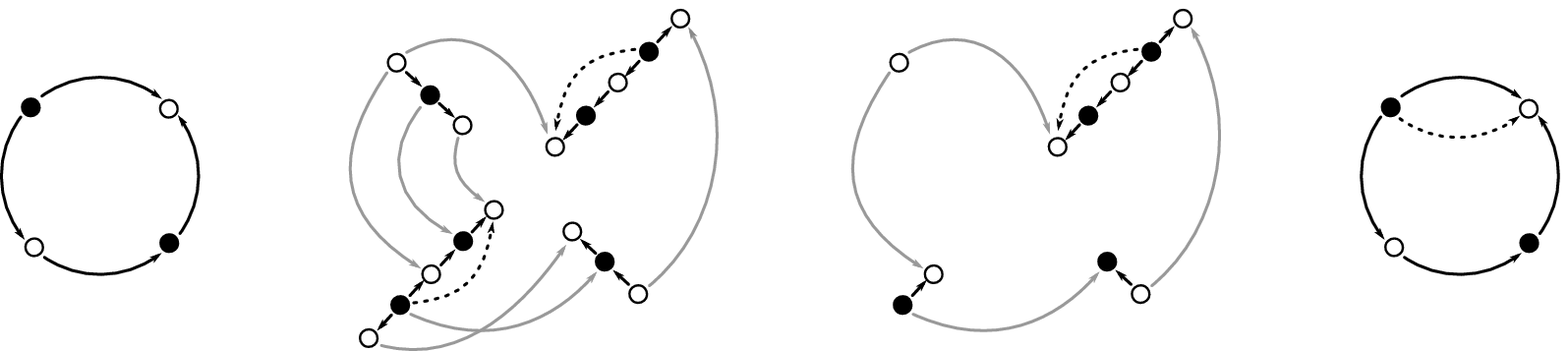}
\caption{Box tensor products of $\bbA$ and $\bbE$ with the type D structure $\sC$: For $\bbA\boxtimes\sC$ the dashed edge is included and for $\bbE\boxtimes\sC$ the dashed edge is omitted. On the left is the decorated graph for $\sC$, and beside it, the result of  $\bbA\boxtimes\sC$ and $\bbE\boxtimes\sC$ are (simultaneously) illustrated; note that unlabbled edges (in gray) should be read as differentials. Edge cancellation results in the decorated graph on the right (an intermediate step is shown), where we recall that $\delta(a)=\rho_1\otimes b + (\rho_{123}+\rho_{123})\otimes c = \rho_1\otimes b$ in $\bbA\boxtimes\sC$ since we are working modulo 2.}
\label{fig:circ-3}
\end{figure}

\begin{proof}[Proof of Theorem \ref{A-E-bim}]
In light of Proposition \ref{prop:all-the-bimodules}, our strategy of proof amounts to exhibiting the behaviour of $\FI$ on a curve-set $\HFhat(M)$. 

\parpic[r]{
\labellist 
\small
\pinlabel {$\FI$} at 93 41
\pinlabel {$\longrightarrow$} at 93 31
\tiny
%\pinlabel {$z$} at 135 56
\endlabellist
 \begin{minipage}{60mm}
 \centering
 \includegraphics[scale=.8]{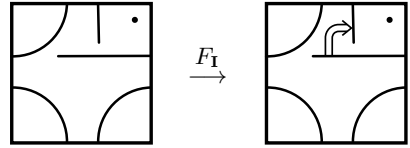}
 \captionof{figure}{The effect on a coil: the added arrow has weight $(1,\wtm)$ for some integer $\wtm$, hence depth 1, and can be removed.}
 \label{fig:coils}
  \end{minipage}%
  }
We fist observe that $F_I$ acts trivially on any curve component not containing a coil, that is, a sequence of the form  \raisebox{-4pt}{$\begin{tikzpicture}[thick, >=stealth',shorten <=0.1cm,shorten >=0.1cm] 
\draw[->] (0,0) -- (1,0);\node [above] at (0.5,-0.07) {$\scriptstyle{3}$};
\draw[->] (1,0) -- (2,0);\node [above] at (1.5,-0.07) {$\scriptstyle{2}$};
\draw[->] (2,0) -- (3,0);\node [above] at (2.5,-0.07) {$\scriptstyle{1}$};
\node at (0,0) {$\bu$};\node at (1,0) {$\circ$};\node at (2,0) {$\bu$};\node at (3,0) {$\circ$};\end{tikzpicture}$} in the associated decorated graph: Dropping the operation 
$(\bu,\rho_3,\rho_2,\rho_1) \mapsto \rho_{123}\otimes\circ$ (see Figure \ref{fig:halfid}) yields the identity bimodule $\CFDA(\mathbb{I})$. For every instance of a coil in a given curve component, $F_I$ adds an additional $\rho_{123}$ edge; this is shown in Figure \ref{fig:coils} using the crossover arrow formalism. Ignoring local systems for the moment, if there is a single isolated coil, that is not contained in a curve $\sC$, then this added crossover arrow clearly can be removed  by a change of basis (compare Figure \ref{fig:review-example} and/or consult \cite[Section 3.7]{HRW}).

With this observation made, we now consider the generic case: fix a component $\gamma$ of $\HFhat(M)$ that is not homotopic to $\sC$, decorated with an arbitrary local system. Expand the local system away from the coil, that is, along any edge that is not contained in a coil. We need to perform the arrow cancelling algorithm on the train-track $\FI(\gamma)$; we assume familiarity with the language in \cite[Section 3.7]{HRW}. The key observation is that each coil adds a new finite-depth crossover arrow, with weight $(\wtp,\wtm)$ such that $\wtp$ is a positive integer; when there is a single isolated coil (as above), and the local system is trivial, it is immediate that this weight is $(1,\wtm)$ and the arrow is removable by a change of basis. More generally, it could be that multiple coils occur in sequence, so that $\gamma$ wraps around the basepoint multiple times. In this case, one obtains a series of crossover arrows, weighted by $(1,\wtm), (5,\wtm), (9,\wtm),\ldots$, and the arrow cancelling algorithm removes these in order of increasing depth (recall that the depth of an arrow is $\min\{|\wtp|,|\wtm|\}$). Note that, during this process, new arrows may be created via composition with arrows from an (expanded) local system. Because the local system is expanded away from the coils, these arrows have the same $\wtp=1$ as the arrow they compose with, so they are removable as well. For reference, we are repeatedly applying \cite[Step 2 in the proof of Proposition 30]{HRW}.

\labellist 
\small
\pinlabel {$(V,\Phi)$} at 76 98
\small
\pinlabel $z$ at 162 165 \pinlabel $z$ at 408 165 
\endlabellist
\begin{figure}[h]
\includegraphics[scale=.8]{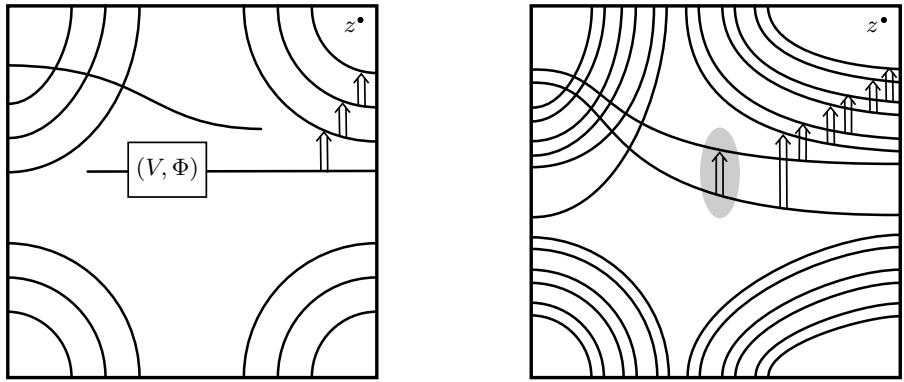}
\caption{Coils in the presence of a local system: On the left is the effect of $\FI$ on a portion of a curve containing a series of coils, together with a local system, schematically. On the right this schematic is explained in an example for a particular 2-dimensional local system -- the shaded arrow comes from the expanded local system. Arrows contributing to the local system have infinite depth, while the new arrows added by $\FI$ have finite depth, increasing left to right from depth 1. These are removed in order, since the associated weights $(\wtp,\wtm)$ have $\wtp>0$.}
\label{fig:coiling}
\end{figure}

In summary, it follows from an application of the algorithm described in \cite[Section 3.7]{HRW} that $\gamma$ and $\FI(\gamma)$ describe type D structures that are isomorphic as $\Alg$-modules, whenever $\gamma$ is not homotopic to $\sC$.  It is illustrative to take note of how this process fails for the case of the curve $\sC$: in this case any new arrows added have infinite depth, and since infinite depth arrows contribute to a local system, it is possible that applying $\FI$ changes the local system in this case.  

To complete the proof of Theorem \ref{A-E-bim}, it remains to show that components of the form $\sC$ (decorated with some non-trivial local system) do not arise as summands of invariants of three-manifolds with torus boundary. To achieve this, we appeal to the extension of (the hat version of) bordered Floer homology to the minus version, due to Lipshitz, Ozsv\'ath, and Thurston \cite{LOT-minus}. We only need a small part of this extended theory, corresponding to the $\mathbf{U}^2=0$ truncation for the Floer homology of a closed three manifold.

To describe this, we require an algebraic digression. Recall that a type D structure over $\Alg$ is equipped with a map $\delta^1\co V\to\Alg\otimes V$, which defines a collection of maps inductively by $\delta^k=\operatorname{id}\otimes\delta^1\circ\delta^{k-1}$ for $k\ge0$, where $\delta^0\co V\to V$ denotes the identity. Then writing $\mu_2$ for the multiplication in $\Alg$, the compatibility condition for the type D structure becomes $\mu_2\otimes\operatorname{id}\circ\delta^2=0$, which we will write as $\mu_2\circ\delta^2=0$, by abuse. Similarly, the compatibility condition for an extended type D structure becomes $\mu_2\circ\delta^2=\mu_0$, where we write $\mu_0=\sum_{|I|=4}\rho_I$ for the central element $U$ in $\widetilde{\Alg}$. As a result, we could instead consider the $A_\infty$ algebra $\Alg^c$ which includes the operation $\mu_0$, so that an extended type D structure is equivalent to a type D structure over $\Alg^c$ with compatibility condition $\sum_{i=0}^2\mu_i\circ\delta^i =0$ (the $\mu_1$ term is identically zero). This is sometimes referred to as a curved type D structure, where $\mu_0$ is the curvature.

The bordered invariants we wish to appeal to are defined over a slightly larger algebra $\Alg^{\mathbf{U}}$, which contains $\mu_0$  and $\mu_2$ as well as a new collection of operations $\mu_4\co\Alg^{\otimes 4} \to \Alg$. These have the property that $\mu_4(\rho_{i+3}, \rho_{i+2},\rho_{i+1},\rho_{i})=\mathbf{U}$, where the subscripts are in $\Z/4\Z$. A type D structure over $
\Alg^{\mathbf{U}}$ consists of a map $\delta^1\co V\to\Alg^{\mathbf{U}}\otimes V$ (as above) with compatibility condition $\sum_{i=0}^4\mu_i\circ\delta^i =0$ (the $\mu_1$ and $\mu_3$ terms are identically zero). Such type D structures are truncations of the full minus bordered invariants which are defined over a larger weighted $A_\infty$ algebra. The result we appeal to from Lipshitz, Ozsv\'ath, and Thurston \cite{LOT-minus} is as follows:

\begin{theorem} 
\label{thm:LOT-}
Every type D structure over $\Alg$ associated with a three-manifold with torus boundary is the restriction of a type D structure over $\Alg^{\mathbf{U}}$. \end{theorem}

By considering the coefficient of \(\mathbf{U}\) in the expression \(\mu_4 \circ \delta^4 (x)\), we deduce that

\begin{corollary}
\label{cor:LOT-}
We have 
\(\sum _{i=0}^3 D_i '\circ D_{i+1}' \circ D_{i+2}' \circ D_{i+3}' = 0\), where the \(D_i'\) are the 
coefficient maps of the reduced extended type D structure \(\CFD(M,\alpha,\beta)\), {\it i.e.}  \(\delta(x) = \sum \rho_I D_I'(x)\). (All subscripts should be interpreted \(mod \ 4\).)
\end{corollary}
For ease of notation, we  write \(\Psi' = \sum _{i=0}^3 D_i' \circ D_{i+1}' \circ D_{i+2}' \circ D_{i+3}' \).

Now suppose  that \(\sC\), equipped with some local system \((V,\varphi)\),  appears as a summand of \(\HFhat(M)\). If \(x\) is any generator  of \(\CFD(M,\alpha,\beta)\) corresponding to an intersection of \(\sC\) with \(\alpha\) or \(\beta\), an easy computation shows that \(\Psi(x) = \varphi (x)\). The fact that \(\Psi'(x)=0\) contradicts the invertibility of \(\varphi\), so no such summand can exist. This concludes the proof of Theorem~\ref{A-E-bim}.
\iffalse
Note that this (more general) notion of an extension (e.g. from $\Alg$ to $\Alg^{\mathbf{U}}$) implies the extendability condition appealed to and proved in \cite{HRW}. These extensions allow us to rule out the class of compact curves in question. Indeed, we claim that the curve $\sC$ (with any local system) does not arise as the summand of the invariant $\HFhat(M)$ where $M$ is a three manifold with torus boundary. To see this, it suffices to consider the case where the local system is non-trivial (trivial local systems are ruled out above since $\FA(\sC)$ is non-compact). Thus there are a pair of generators $x$ and $y$ in the type D structure over $\Alg^{\mathbf{U}}$ such that $\delta^4(x)=\rho_0\otimes\rho_3\otimes\rho_2\otimes\rho_1\otimes y$ and $\mu_4\circ\delta^4(x)=\mathbf{U}y$. This violates the compatibility condition, as there are no available $\delta^4$ operations to cancel against.   
\fi 
\end{proof}

Lipshitz, Ozsv{\'a}th and Thurston's paper \cite{LOT-minus} is still in preparation, and the input we need from it is much weaker than the full strength of their results,  so we outline a direct proof of Corollary~\ref{cor:LOT-}. 
In \cite[Appendix A]{HRW}, we discussed the generalizations to \cite{LOT} needed to show that \(\CFD(M,\alpha,\beta)\)  defines an extendable  type D structure. The main change   was that we needed to consider a larger class of decorated sources for the holomorphic maps used to define \(\CFD\). In particular, we needed to use decorated sources with boundary punctures labeled by Reeb chords including \(\rho_0\). 
To show that  \(\CFD(M,\alpha,\beta)\) extends to a type D structure over \(\Alg^{\mathbf{U}}\), we must consider decorated sources which contain interior punctures in addition to boundary punctures. 

To be more precise, let \((\Sigma, \mathbf{A},\mathbf{B})\) be a bordered Heegaard diagram representing  \((M,\alpha, \beta)\). We equip \(\Sigma \times [0,1]\times \R\) with a suitably generic almost complex structure \(J\). By counting holomorphic maps
\(u\co S^\triangleright \to \Sigma \times [0,1]\times \R\), where 
\(S^\triangleright\) is a decorated source as in \cite [Definition 5.2]{LOT}, we obtain an extended type D structure \(\CFD(\Sigma, \mathbf{A},\mathbf{B})\).  \(\CFD(\Sigma, \mathbf{A},\mathbf{B})\) is unreduced, and is homotopy equivalent to the reduced type D structure \(\CFD(M,\alpha,\beta)\). 

To prove Corollary~\ref{cor:LOT-}, we must use an additional class of decorated sources. 
For our purposes, it is enough to consider decorated sources  with a single interior puncture and no boundary punctures which map to east infinity. We say such a source is of {\it type 1-P}. If \(S^\triangleright\) is a source of type 1-P, we let \(\widetilde{\sM}^B(x,y;S^\triangleright)\) be the space of pseudoholomorphic maps  \(u\co S \to \Sigma \times [0,1] \times \R\)   as in  \cite[Definition 5.3]{LOT} which represent the homology class \(B \in \pi_2(x,y)\), limit to the generators \(x\) (resp. \(y\)) at \(-\infty\) (resp. \(+\infty\)), and have the additional property that near the puncture, the curve limits to a single copy of the closed Reeb orbit \(Z\) corresponding to \(\partial \Sigma\). 
We remark that if \(u\) is such a map, the composition \(\pi_\Sigma \circ u\) necessarily has multiplicity 1 near \(\partial \Sigma\).
We let 
   \(\sM^B(x,y;S^\triangleright)\) be the corresponding reduced moduli space. 
    
     If \(x\) is a generator for \((\Sigma, \mathbf{A}, \mathbf{B})\), we define 
    \[D_{\mathbf{U}}(x) = \sum_y \sum \# \sM^B(x,y;S^\triangleright) y\]
   where the inner sum runs over pairs \((S^\triangleright, B)\) such that \(S\) is a decorated source of type 1-P and \(B \in \pi_2(x,y)\) is such that the \(\mathrm{ind} \ (S^\triangleright, B) = 1\). 
  In addition, let \(D_\emptyset, D_0, D_1, D_2, D_3\) be the usual coefficient maps for the extended type \(D\) structure \(\CFD(\Sigma, \mathbf{A}, \mathbf{B})\), as defined  in  \cite[Sections 11.1 and 11.6]{LOT}. 
    
    By studying the ends of index \(2\) moduli spaces, we will show that the following relation holds:
\begin{proposition}
\label{prop:D_U sum}
$ \displaystyle D_\emptyset \circ D_\mathbf{U}+ D_\mathbf{U} \circ D_\emptyset + \sum_{i=0}^3 D_i \circ D_{i+1} \circ D_{i+2} \circ D_{i+3} = 0. $ 
\end{proposition}

If we view \(\CFD(\Sigma, \mathbf{A}, \mathbf{B})\) as a chain complex with differential \(D_\emptyset\),  the proposition says that the map \(\Psi = \sum _{i=0}^3 D_i \circ D_{i+1} \circ D_{i+2} \circ D_{i+3} = 0\) is null-homotopic. To pass from the unreduced complex \(\CFD(\Sigma,  \mathbf{A}, \mathbf{B}) \) to \(\CFD(M,\alpha, \beta)\), we cancel components of \(D_{\emptyset}\) until we arrive at an extended type D structure with \(D_\emptyset = 0 \). As we cancel, the \(D_I\)'s are progressively modified as well. Let \(D_I'\) denote the coefficient maps in the resulting type D structure, which is \(\CFD(M,\alpha,\beta)\). 

In the case of \(D_i\), (\(i=0,\ldots, 3\)), this procedure is particularly simple. The type D structure relation for \(\CFD(\Sigma,  \mathbf{A}, \mathbf{B})\)
implies that \(D_\emptyset \circ D_i + D_i \circ D_\emptyset = 0\), so the individual \(D_i\)'s are chain maps. As a group, 
\(\CFD(M,\alpha,\beta) = H(\CFD(\Sigma, \mathbf{A}, \mathbf{B} ),D_\emptyset)\), and it is not hard to see that \(D_i'=(D_i)_*\) is the map induced on homology. (This is analogous to the fact that if \(d=d_0+d_1+\ldots\) is a differential on a filtered chain complex, then \(d_1\) is a chain map with respect to \(d_0\), and the first differential on the resulting spectral sequence is \(d_{1*}\).) We conclude that
 $$ \sum_{i=0}^3 D_i' \circ D_{i+1}' \circ D_{i+2}' \circ D_{i+3}' = \Psi_* = 0$$
 since \(\Psi\) is null-homotopic.  Thus Corollary~\ref{cor:LOT-} will follow from Proposition~\ref{prop:D_U sum}.

\begin{proof}[Proof of Proposition~\ref{prop:D_U sum}] 
Suppose \(B \in \pi_2(x,y)\) has \(\partial ^\partial B = Z\). 
Define \(\sM^B(x,y;Z)\) be the union of the moduli spaces \(\sM^B(x,y;S^\triangleright)\), where the union runs over sources \(S^\triangleright\) of type 1-P with  \(\mathrm{ind} (S^\triangleright, B)=2\).
 We consider the ends of the 1-dimensional space \( \sM^B(x,y;Z)\). These ends  correspond to degenerations of the source \(S^\triangleright\) at \(\pm \infty\) or east infinity.   In the case of a degeneration at \(\pm \infty\), \(S^ \triangleright\)  must decompose  into a provincial source (no punctures labeled by east infinity) and a source of type 1-P. Standard arguments show that the number of ends of \( \sM^B(x,y;S^\triangleright)\) corresponding to degenerations of this type is the coefficient of \(y\) in   \(D_\mathbf{U} \circ D_\emptyset (x) + D_\emptyset \circ D_\mathbf{U}(x)  \).

Now we consider degenerations where the source breaks at east infinity. 
To understand them, we must study holomorphic maps \(v:T^\diamond \to \R \times Z \times [0,1] \times \R\), as in Section 5.3 of \cite{LOT}. Here  we give the target the split complex structure coming from the usual complex structures on \(\R \times Z\) and \([0,1] \times \R\).  The target  contains four Lagrangians \(A_i : = \R \times a_i  \times 1 \times \R\) which \(\partial T^\diamond\) must map to. Since  \(B\) has multiplicity \(1\) near the puncture point in \(\Sigma\), it suffices to consider only those maps such that the composition \(\pi_\Sigma \circ v:T^\diamond \to \R \times Z\) has multiplicity 1 at every point of \(\R \times Z\). Moreover, since \(B\) has an interior puncture point mapping to east  infinity, \(T^\diamond\) must have an interior puncture point which maps to east infinity as well. 

It follows that \(T^\diamond\) is either a sphere with two interior punctures or a disk with one interior puncture and one or more boundary punctures. A domain of the first type is not stable and does not contribute to the boundary of the Gromov compactification. An easy index computation shows that a domain of the second type has index \(1\) if and only if it has a single boundary puncture. We analyze this case further. 

First, \(\partial T^\diamond\) has a single component, which must map to one of the \(A_i\). There are four distinct (but isomorphic) moduli spaces, depending on which \(A_i\) \(\partial T ^\diamond \) maps to. Consider the moduli space \(\widetilde{\sM}_{0}(T)\)  consisting of maps which take \(\partial T^\diamond\)  to \(A_0\). If \(v \in \widetilde{\mathcal{M}}_0(T)\), the boundary puncture limits to the Reeb chord \(-\rho_{0123}\), and  the composition \(\pi_{\mathbb{D}} \circ v\) has as its image a single point in \([0,1] \times \R\), since \(\partial T^\diamond\) maps to \(1 \times \R\). Since the source is a disk, \(\widetilde{\mathcal{M}}_0(T)\) is transversally cut out, as in \cite[Proposition 5.16]{LOT}. By applying the Riemann mapping theorem, we see that the reduced moduli space \({\mathcal{M}_0(T)}\) consists of a single point.  

Next, consider the space \(\mathcal{M}_i^B:=\mathcal{M}^B(x,y; \{-\rho_{i\ldots i+3}\})\), as in Definition 5.68 of \cite{LOT} . This moduli space counts maps from sources with no interior punctures and a single boundary puncture which limits to the Reeb chord \(-\rho_{i\ldots i+3}\) at east infinity. Standard gluing results, as in \cite[Proposition 5.31]{LOT}, show that the ends of \(\sM^B(x,y;Z)\) corresponding to breaks at east infinity are in bijection with pairs \((a,b) \in {\sM}_i^B \times  {\sM}_i(T)\), where \(i=0,\ldots 3\). Since each \(\sM_i(T)\) consists of one point, the set of such ends is in bijection with 
the union of the \(\sM_i^B\) for \(i=0,\ldots, 3\). 

We write \[\Delta_{I}(x) = \sum_y \sum_{B\in \pi_2(x,y)} \#  \sM^B(x,y;\{-\rho_I\}) y.\] The analysis above shows that the \(\mathrm{mod} \ 2\) number of ends of  \(\sM^B(x,y;Z)\)  is the \(y\) component of the expression
\[D_\emptyset \circ D_U(x) + D_U \circ D_\emptyset (x) + \sum_{i=0}^3 \Delta_{i,i+1,i+2,i+3} (x)= 0.\]

Observe that if \(|I|=1\), then \(\Delta_I = D_I\). The statement of the proposition now follows directly from the relation above, and the following:

\begin{lemma}
\label{Lem:comp}
 \(\Delta_{IJ} = \Delta_I \circ \Delta_J\).
\end{lemma}

\begin{proof}
Suppose \(\sM^B(x,y; \{-\rho_{IJ}\})\) is 0 dimensional, and consider the 1-dimensional moduli space   \(\sM^B(x,y;\{-\rho_I\},\{-\rho_J\})\). There is a map 
\[f:= \mathrm{ev}_1 -\mathrm{ev}_2: {\sM}^B(x,y;\{-\rho_I\},\{-\rho_J\}) \to \R\] which measures the relative heights of the two punctures. Thus the mod \(2\) number of points in \(f^{-1}(t)\) is constant whenever \(t\) is a regular value for \(f\). 
As \(t \to \infty\), the curves in \(f^{-1}(t)\) degenerate to two-story buildings, and we have \(\# f^{-1}(t) = \sum_z n_{zy,I} n_{xz,J}\), where \(n_{xy,I}\) is the coefficient of \(y\) in \(\Delta_I(x)\). On the other hand, as \(t\to 0\), the curves in \(f^{-1}(t)\) degenerate to holomorphic combs consisting of a curve in  \(\sM^B(x,y;\{-\rho_{IJ}\})\) and a split curve at east infinity. Thus we have 
 \(\# f^{-1}(t) = n_{xy,IJ}\). Equating the two expressions gives the statement of the lemma. 
\end{proof}
This concludes the  proof of Proposition~\ref{prop:D_U sum}\end{proof}

\parpic[r]{
\labellist 
\tiny
\pinlabel {$3$} at 173 165 \pinlabel {$1$} at 232 98
\pinlabel {$2$} at 232 165 \pinlabel {$0$} at 173 98
\pinlabel {$y$} at 57 140
\pinlabel {$x_4$} at 100 140
\pinlabel {$x_3$} at 218 247
\pinlabel {$x_2$} at 315 140
\pinlabel {$x_1$} at 218 18
\pinlabel {$x_0$} at 3 140
\endlabellist
 \begin{minipage}{45mm}
 \centering
 \includegraphics[scale=.3]{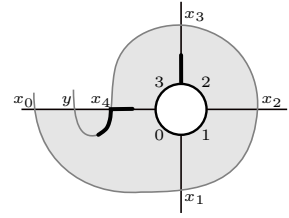}% \includegraphics[scale=.8]{figures/circ-1}
 \captionof{figure}{Part of a type D Heegaard diagram. The domain \(B\) is shaded, as is the $\beta$-curve.}
 \label{fig:m4}
  \end{minipage}%
  }

We  illustrate the ideas in the proof above with  a  local example. 
Consider a genus one Heegaard diagram which near  \(\partial \Sigma\)  is as shown in Figure~\ref{fig:m4}. It is easy to see that \(D_i(x_{i+1}) = x_{i}\), so \(D_0 \circ D_1 \circ D_2 \circ D_3 (x_4) = x_0.\) On the other hand, we have \(D_\emptyset (x_4) = y\) and \(D_\mathbf{U}(y) = x_0\). To see why  the identity holds in this case, we consider the 1-dimensional moduli space \(\sM^B(x_4,x_0;Z)\), where \(B\) is the domain shaded in the figure. Moving around in this moduli space corresponds to cutting the domain along one of the two arcs emerging from \(x_4\) and passing through the interior of \(B\). Cutting all the way along the \(\beta\) arc gives the degeneration corresponding to \(D_\mathbf{U} \circ D_\emptyset (x_4)=x_0\), while cutting all the way along the \(\alpha\) arc gives the degeneration corresponding to \(\Delta_{0123}(x_4) = x_0\). 

Similarly, to understand what happens in Lemma~\ref{Lem:comp}, let \(B'\) be the domain obtained by cutting all the way along the horizonal \(\alpha\) arc.  (This is the domain defining \(\Delta_{0123}\).) Now cut \(B'\) along the heavy vertical arc.  When the length of the cut is \(0\), we get \(B'\). When the cut reaches the \(\beta\) curve, we get two domains corresponding to the composition \(\Delta_{012}\circ \Delta_3(x_4) = x_0\). 

\begin{remark} Similar arguments can be used to give a full proof of Theorem~\ref{thm:LOT-}. We briefly sketch the new ideas involved, but will not carry out all the details here. Consider, for example, the terms carrying the coefficient \(\rho_1 \mathbf{U}\) in the expression  \(\sum_{i=0}^4 \mu_i \circ \delta^i\). These terms will be of the form 
\[D_{01} \circ D_1 \circ D_2 \circ D_3  +  D_3 \circ D_0 \circ D_1 \circ D_{12}  +  D_1 \circ  D_\mathbf{U}  +  D_\mathbf{U} \circ D_1  +  D_\emptyset \circ 
D_{1\mathbf{U}}  +  D_{1\mathbf{U}} \circ D_\emptyset \]
where \(D_{1\mathbf{U}}\) is  defined by counting index \{1\} holomorphic maps whose source has one interior puncture which maps to \(Z\) and one boundary puncture mapping to \(\rho_1\). (We say such a source is of type \(\rho_1 \mathbf{U}\).) We must show that the expression above vanishes. To do so, we consider the ends of index two moduli spaces of maps from sources of type \(\rho_1 \mathbf{U}\). These ends correspond to breaks either at \(\pm \infty\) (the last four terms in the expression above) or at east infinity (the first two.) In analyzing the breaks at east infinity, the main new analytical problem is to study curves in 
\(\R \times Z \times [0,1]\times \R\) whose limits have multiplicity 2 along \(\rho_1\) and multiplicity \(1\) along the other \(\rho_i\). Index considerations show that these all have the form of a trivial strip along \(\rho_1\) together with a curve of the type studied in the proof of Proposition~\ref{prop:D_U sum}. It follows that the number of curves with breaks at east infinity is equal to the number of index 1 curves with sources with no interior punctures and two boundary punctures labeled by \(-\rho_1\) and \(-\rho_{i,i+1,i+2,i+3}\) respectively. Finally, arguing as in Lemma~\ref{Lem:comp} shows that this number is given by the first two terms in the equation above. 
\end{remark}

%% file: sections/knot_floer.tex
% !TEX root = ../companion.tex
%knot_floer.tex
\label{sec:knot-floer}

Knot complements provide a large class of examples for which the immersed curves are relatively easy to compute; this is due to the well understood relationship between between knot Floer homology and $\CFD$ of the complement (see \cite[Chapter 11]{LOT}). The goal of this section is to explore the relationship between the knot Floer homology of a knot $K$ in $S^3$ and the immersed curves $\HFhat(M_K)$ associated with the complement $M_K = S^3 \setminus \nu(K)$. Toward that end, we first briefly review the essentials of knot Floer homology following the notation in \cite[Chapter 11]{LOT}; a more complete introduction can be found in \cite{Manolescu:HFK}. 

\subsection{Background and notation} Given a nullhomologous knot $K$ in a three-manifold $Y$, the knot Floer chain complex $\CFKminus(Y,K)$ is a free, $\Z$-graded chain complex over $\F[U]]$. As a complex it is quasi-isomorphic to $\CFminus(Y)$, but it is also endowed with a filtration, the Alexander filtration,
$$\cdots\subset\mathcal{F}_i\subset\mathcal{F}_{i+1}\subset\cdots\subset \CFKminus(Y,K)$$
with filtration level $A(x) = \min\{i|x\in\mathcal{F}_i\}$. Up to filtered chain homotopy equivalence, $\CFKminus(Y,K)$ in an invariant of the knot $K$ in $Y$. When the ambient manifold $Y$ is $S^3$, we will omit it from the notation.

Recall that $\CFKminus(Y,K)$ is defined using a doubly pointed Heegaard diagram for the pair $(Y,K)$. The differential is defined by counting certain pseudoholomorphic disks, which may be interpreted as domains in the Heegaard surface together with information about how these domains cover the two basepoints $w$ and $z$. If a homotopy class of disks $B$ from $x$ to $y$ covers the basepoints with multiplicities $n_w(B)$ and $n_z(B)$, a differential corresponding to $B$ connects $x$ to $U^{n_w(B)} y$ and lowers the filtration level by $n_z(B)$. By restricting the differential to disks which do not cover one or both basepoints, we can define several important quotient complexes. The associated graded object $\gCFKminus$ is obtained by considering only terms in the differential that do not change the Alexander grading (that is, by restricting to disks which do not cover $z$). For both $\CFKminus$ and $\gCFKminus$, further restricting to disks which do not cover $w$ and generating the complex over $\F$ instead of $\F[U]]$ produces the complexes $\CFKhat$ and $\gCFKhat$; note that this restriction is equivalent to setting $U=0$. Knot Floer homology is the homology of $\gCFKminus(Y,K)$, denoted $\HFKminus(Y,K)$, or of $\gCFKhat(Y,K)$, denoted $\HFKhat(Y,K)$.

We will need to make use of particularly nice bases for $\CFKminus(Y,K)$. Fix a representative for the filtered homotopy type of $\CFKminus(Y,K)$, which as shorthand we will denote by $C$. Note that $C$ has two filtrations, the Alexander filtration and the filtration by negative powers of $U$. We may choose $C$ so that the differential $\partial$ strictly drops one of these filtration levels; we say that such a filtered complex is reduced. The associated graded object is $gC = \oplus_i \mathcal{F}_i / \mathcal{F}_{i-1}$. Given $x$ in $C$, let $[x]\in gC$ denote the projection of $x$ to $\mathcal{F}_{A(x)}/\mathcal{F}_{A(x)-1}$. A filtered basis for $C$ is a basis $\{v_1, \ldots, v_n\}$ such that $\{[v_1],\ldots,[v_n]\}$ is a basis for the associated graded $gC$. We say that a filtered basis is vertically simplified if for each basis vector $v_i$, either $\partial v_i \in U\cdot C$ or $\partial v_i \equiv v_j + x$ for some basis element $v_j$ and $x \in U\cdot C$. In the latter case, we say there is a vertical arrow from $v_i$ to $v_j$. Similarly, we say that a filtered basis is horizontally simplified if for each basis vector $v_i$ with filtration level $A(v_i) = k$, either $A(\partial v_i) < k$ or $A(\partial v_i) = U^m \cdot v_j + x$ for some basis element $v_j$ and integer $m$ with $A(U^m \cdot v_j) = k$ and some $x$ with $A(x) < k$. The filtered complex $\CFKminus(Y,K)$ always admits a vertically simplified basis and a (possibly different) horizontally simplified basis \cite{LOT}.

\subsection{Curves from knot Floer homology}\label{sec:curves-from-CFK}

Given a knot in $S^3$, the curve invariant $\HFhat(M_K)$ of $M_K = S^3 \setminus \nu(K)$ can be readily computed from the knot Floer complex $\CFKminus(K)$. To do this, we pass through the algorithm described by Lipshitz, Ozsv{\'a}th, and Thurston in \cite{LOT} computing $\CFD(M, \mu, \lambda)$ from $\CFKminus(K)$. This algorithm makes use of a horizontally simplified basis and a vertically simplified basis for $\CFKminus(K)$. We will first consider the special case that these two bases coincide; that is, that $\CFKminus(K)$ is equipped with a basis that is both horizontally and vertically simplified. This assumption is analogous to the loop type condition:

\begin{proposition}
If $\CFKminus(K)$ admits a basis which is both horizontally and vertically simplified, then $M_K = S^3 \setminus \nu(K)$ is a loop type manifold.
\end{proposition}
\begin{proof}
\sloppy
According to the algorithm mentioned above, the generators of $\iota_0 \CFD(M, \mu, \lambda)$ are in one-to-one correspondence with the generators of $\CFKminus(K)$. In the directed graph representing $\CFD(M,\mu,\lambda)$, these generators are connected by chains of $\iota_1$-vertices referred to as horizontal chains, vertical chains, and unstable chains (these chains correspond to the segments denoted $a_k, b_k, c_k, d_k$, and $e$ in \cite{HW}). Moreover, there are exactly two chains (or two ends of the same chain), at each $\iota_0$-vertex; that is, the graph representing $\CFD(M, \mu, \lambda)$ has valence two. It follows that $M$ is a loop type manifold.
\end{proof}

\begin{remark}
It is not known whether $\CFKminus(K)$ always admits a horizontally and vertically simplified basis for any knot $K$ in $S^3$.
\end{remark}

In the presence of a horizontally and vertically simplified basis the construction of $\CFD$ from $\CFKminus$ is particularly straightforward, and since the resulting $\CFD$ is loop type we can easily extract the collection of immersed curves (with trivial local systems) $\HFhat(M_K)$ from this as described in Section \ref{sub:loop-type}. For convenience, we now describe an algorithm for recovering $\curves{M_K}$ directly from $\CFKminus(K)$ by combining these two steps. More precisely, we describe the lift $\curves{M_K,\spin}$ of $\curves{M_K}$ in the covering space $H_1(\partial M, \R)/\langle \lambda \rangle$ (here $\spin$ is the unique spin$^c$ structure on $M_K$). This space will be realized as the infinite strip $[-1/2, 1/2]\times\R$ with $(-1/2, t)$ and $(1/2,t)$ identified, and $\pi^{-1}(z)$ is the set of points $(0, n+1/2)$ for $n\in\Z$. The horizontal direction corresponds to Seifert longitude $\lambda$ in $\partial M$, and the vertical direction corresponds to the meridian $\mu$.

\begin{proposition}
Given a horizontally and vertically simplified basis for $\CFKminus(K)$, the collection of curves $\curves{M_K, \spin }$ in the infinite strip described above can be obtained from $\CFKminus(K)$ by the following procedure:
\begin{enumerate}
\item For each basis element $x$ of $\CFKminus$, place a short horizontal segment $[-1/4,1/4]\times \{t\}$ at height $t = A(x)$, where $A(x)$ denotes the Alexander grading of $x$.

\item If $\CFKminus(K)$ contains a \emph{vertical arrow} from $x$ to $y$ (that is, if $\partial x = y + Uz$ for some $z \in \CFKminus(K)$), then connect the \emph{left} endpoints of the horizontal segments corresponding to $x$ and $y$ by an arc.

\item If $\CFKminus(K)$ contains a \emph{horizontal arrow} from $x$ to $y$ (that is, if $\partial x = U^{A(y)-A(x)} y + z$ for some $z \in CFK^-(K)$ with $A(z) < A(x)$), then connect the \emph{right} endpoints of the horizontal segments corresponding to $x$ and $y$ by an arc.

\item There is now a unique horizontal segment with an unattached left endpoint, and a unique horizontal segment with an unattached right endpoint; connect these unattached endpoints to $(-1/2,0)$ and $(1/2,0)$, respectively.
\end{enumerate}
\end{proposition}
\begin{proof}
By the algorithm in \cite{LOT}, the generators of $\CFKminus$ are in one-to-one correspondence with $\iota_0$ generators of $\CFD(M, \mu, \lambda)$, which correspond to horizontal segments in the construction of $\curves{M_K, \spin}$. A vertical arrow between $x$ and $y$ of length $\ell = A(x) - A(y)$ corresponds to a chain
\[\begin{tikzpicture}[thick, >=stealth',shorten <=0.1cm,shorten >=0.1cm] 
\draw[->] (0,0) -- (1,0);\node [above] at (0.45,0) {$\scriptstyle{1}$};
\draw[<-] (1,0) -- (1.8,0);\node [above] at (1.45,0) {$\scriptstyle{23}$};
\draw[<-] (2.2,0) -- (3,0);\node [above] at (2.65,0) {$\scriptstyle{23}$};
\draw[<-] (3,0) -- (4,0);\node [above] at (3.55,0) {$\scriptstyle{123}$};
\node at (0,0) {$\bu$};\node at (1,0) {$\circ$};\node at (3,0) {$\circ$};\node at (4,0) {$\bu$};\node at (2,0) {$\cdots$};\end{tikzpicture}\]
in $\CFD(M,\mu,\lambda)$ with $\ell$ generators of idempotent $\iota_1$. In $\curves{M_K,\spin}$ this corresponds to a (downward moving) vertical segment of length $\ell$ connecting the left edge of the segment corresponding to $x$ to the left edge of the segment corresponding to $y$. Similarly, a horizontal arrow between $x$ and $y$ of length $\ell = A(y) - A(x)$ corresponds to a chain
\[\begin{tikzpicture}[thick, >=stealth',shorten <=0.1cm,shorten >=0.1cm] 
\draw[->] (0,0) -- (1,0);\node [above] at (0.45,0) {$\scriptstyle{3}$};
\draw[->] (1,0) -- (1.8,0);\node [above] at (1.35,0) {$\scriptstyle{23}$};
\draw[->] (2.2,0) -- (3,0);\node [above] at (2.55,0) {$\scriptstyle{23}$};
\draw[->] (3,0) -- (4,0);\node [above] at (3.45,0) {$\scriptstyle{2}$};
\node at (0,0) {$\bu$};\node at (1,0) {$\circ$};\node at (3,0) {$\circ$};\node at (4,0) {$\bu$};\node at (2,0) {$\cdots$};\end{tikzpicture}\]
in $\CFD(M,\mu,\lambda)$ with $\ell$ generators of idempotent $\iota_1$. In $\curves{M_K,\spin}$ this corresponds to an upward moving vertical segment of length $\ell$ connecting the right edges of the segments corresponding to $x$ and $y$. Finally, the unstable chain in the algorithm in \cite{LOT} corresponds to a path from the unmatched right edge to the unmatched left edge (moving to the right and wrapping around the cylinder $H_1(\partial M,\R)/\langle \lambda \rangle$).
\end{proof}

\begin{example}
\label{Ex:tref-cable}
The knot Floer homology of the   $(2,-1)$-cable of the left-hand trefoil knot  is shown on the right-hand side of Figure \ref{fig:knot-floer-cable-example}. Ignoring the diagonal arrows, which have both \(n_z>0\) and \(n_w>0\), we see that the complex is both horizontally and vertically simplified.  The corresponding curve is shown on the left. Compare with \cite[Figure 11.5]{LOT}, which shows the corresponding (loop-type) \(\CFD\). 
\end{example}

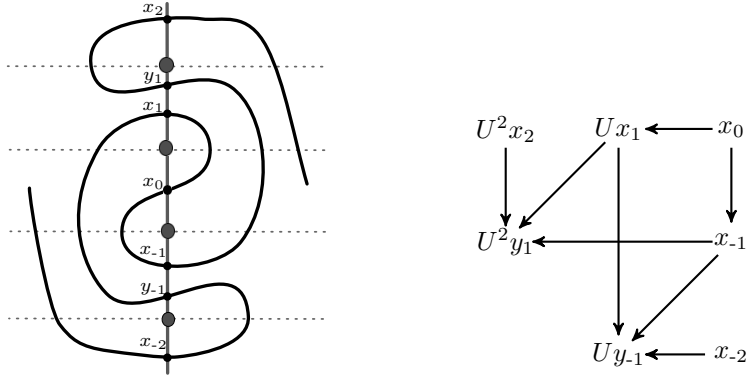
\begin{figure}[ht]
\labellist
\iffalse
\pinlabel {$-2$} at 40 20 \pinlabel {$-1$} at 120 20 \pinlabel {$0$} at 200 20 \pinlabel {$1$} at 285 20 \pinlabel {$2$} at 365 20
\pinlabel {$w$} at 80 155 \pinlabel {$z$} at 80 105
\pinlabel {$w$} at 164 155 \pinlabel {$z$} at 164 105
\pinlabel {$w$} at 244 155 \pinlabel {$z$} at 244 105
\pinlabel {$w$} at 324 155 \pinlabel {$z$} at 324 105
\fi
\scriptsize
\pinlabel {$x_{\text{-}2}$} at 142 30
\pinlabel {$y_{\text{-}1}$} at 142 83
\pinlabel {$x_{\text{-}1}$} at 142 119
\pinlabel {$x_0$} at 142 185
\pinlabel {$x_1$} at 142 259
\pinlabel {$y_1$} at 142 287
\pinlabel {$x_2$} at 142 353
\endlabellist
\includegraphics[scale=0.4]{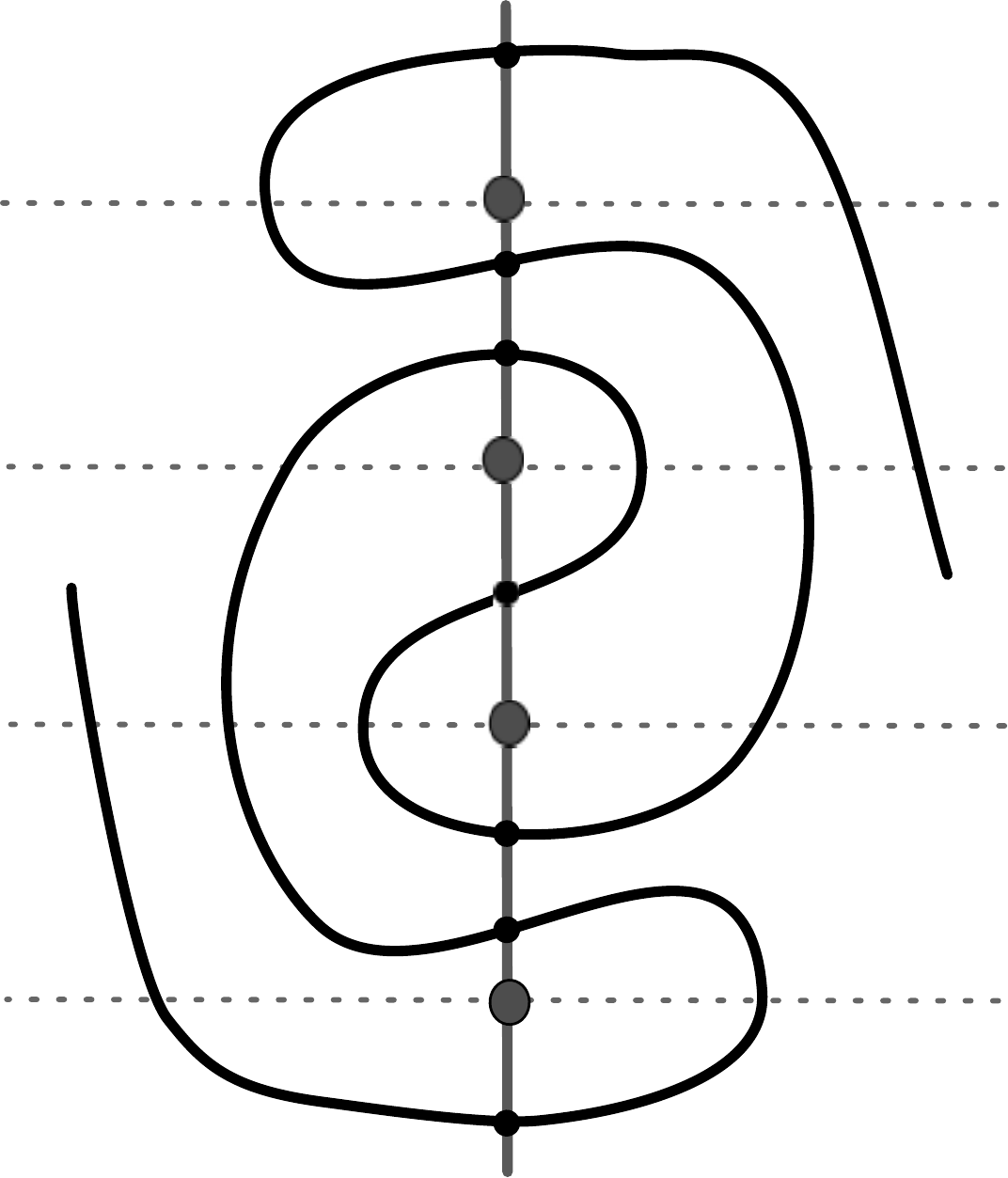}
\qquad\qquad\quad
\begin{tikzpicture}[scale=1,>=stealth', thick] 
     \node at (0,0) {$x_0$};
     \node at (-1.5,0) {${Ux_1}$}; \node at (-3,-1.5) {${U^2y_1}$};
     \node at (-3,0) {${U^2x_2}$};
      \node at (0,-1.5) {${x_{\text{-}{1}}}$}; \node at (-1.5,-3) {${Uy_{\text{-}{1}}}$};
     \node at (0,-3) {${x_{\text{-}{2}}}$};
 \draw[->,shorten <= 0.25cm, shorten >= 0.35cm] (0,0)--(-1.5,0);
  \draw[->,shorten <= 0.25cm, shorten >= 0.35cm] (0,-1.5)--(-3,-1.5) ;
    \draw[->,shorten <= 0.35cm, shorten >= 0.35cm] (0,-3)--(-1.5,-3) ;
  \draw[->,shorten <= 0.25cm, shorten >= 0.25cm] (0,0)--(0,-1.5);
   \draw[->,shorten <= 0.25cm, shorten >= 0.25cm] (-1.5,0)--(-1.5,-3);
      \draw[->,shorten <= 0.25cm, shorten >= 0.20cm] (-3,0)--(-3,-1.5);
       \draw[->,shorten <= 0.25cm, shorten >= 0.25cm] (-1.5,0)--(-3,-1.5);
              \draw[->,shorten <= 0.25cm, shorten >= 0.25cm] (0,-1.5)--(-1.5,-3) ;
\end{tikzpicture}
\caption{The $(2,-1)$-cable of the left-hand trefoil $K$. The right-hand figure shows part of \(CFK^\infty(K)\), drawn with the usual convention that a differential with \(n_z = a\) and \(n_w=b\) is represented by an arrow shifting \(a\) units down and \(b\) units to the left.}\label{fig:knot-floer-cable-example} 
\end{figure}

We remark that several common numerical invariants of a knot $K$ can be easily read off from the curve invariant $\curves{M_K,\spin}$ when it is pulled tight in a peg-board diagram. For example, the genus $g(K)$ is determined by the maximum height of the curve (here we mean a discrete height, rounded to the nearest integer). Equivalently, the genus is half of the number pegs between the minimum height and the maximum height attained by $\curves{M_K,\spin}$. Note that when pulled tight, $\curves{M_K,\spin}$ is supported in a neighborhood of a meridian passing through the peg except for one segment, which wraps around the cylinder once. This non-vertical segment encodes two important concordance invariants extracted from knot Floer homology, the Ozsv{\'a}th-Szab{\'o} invariant $\tau$ and the $\epsilon$ invariant defined by Hom \cite{Hom2014}. Starting somewhere on the non-vertical segment and following the curve rightward, the height at which the curve first hits the vertical line through the pegs (rounded to the nearest integer) is $\tau(K)$; indeed, in the construction of $\curves{M_K,\spin}$ from $\CFKminus(K)$, the horizontal segment attached to the right end of the non-vertical segment corresponds to the distinguished generator of vertical homology, and $\tau(K)$ measures the Alexander grading of this generator. The invariant $\epsilon(K)$ is determined by what $\curves{M_K,\spin}$ does after it first intersects the vertical line; the curve can turn upwards, turn downwards, or continue straight (the third option only being possible if $\tau(K) = 0$), and these correspond to $\epsilon(K)$ being $1, -1$, and $0$, respectively. It follows that the unique non-vertical segment of $\curves{M_K,\spin}$ has slope $2\tau(K) - \epsilon(K)$.

\parpic[r]{
 \begin{minipage}{45mm}
 \centering
 \includegraphics[scale=.45]{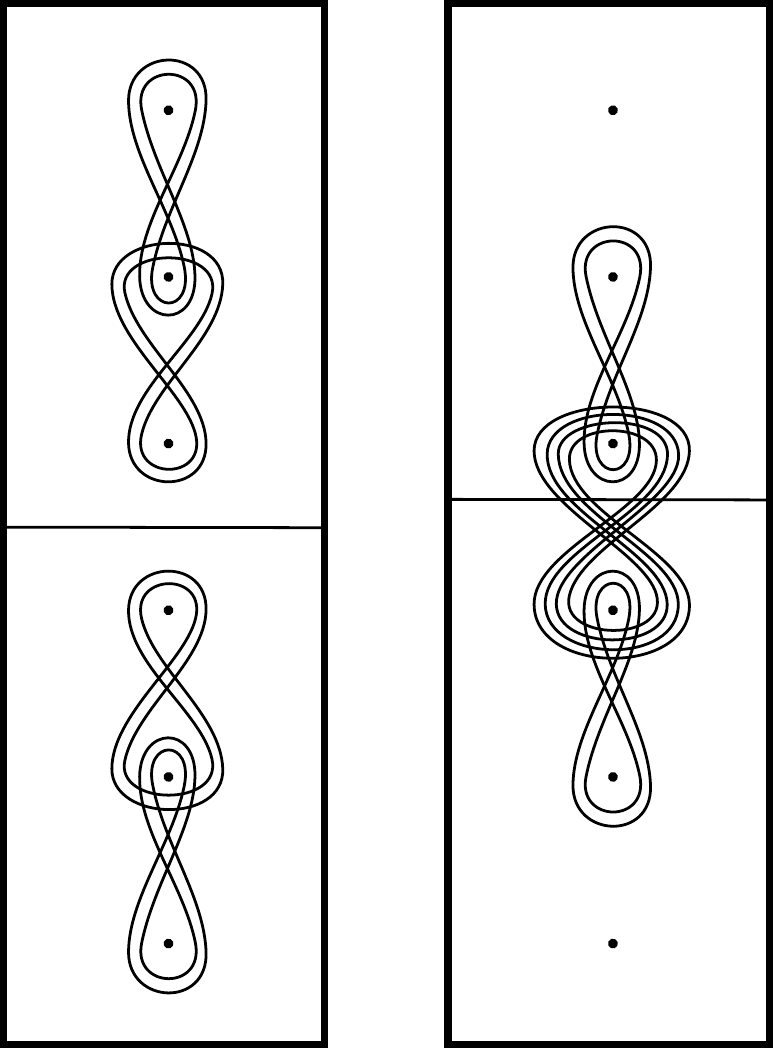}
 \captionof{figure}{}
 \label{fig:ConwayKT}
  \end{minipage}%
  }
\begin{example}
\label{Ex:conway}
The curve sets $\curves{M_K,\spin}$ for the complement of the Conway knot (left) and Kinoshita-Terasaka knot (right) are pictured in Figure \ref{fig:ConwayKT}. (The figure-eights in each diagram come in pairs, with the two curves in a pair differing in Maslov index by \(1\)). This pair of mutant knots are distinguished by their genus, and thus also by knot Floer homology. This is clearly reflected in the curves, as the maximum heights of the curve sets associated with these knots are 3 and 2, respectively. On the other hand, it is striking that the curves are otherwise similar; they differ only by a vertical shift of some of the components. In particular, the projections of these curves to the torus are identical. It is an interesting question whether this is always the case for mutant pairs; that is, does the curve set $\curves{M_K}$ in the torus, absent additional grading information, detect mutation? This is related to the conjecture that knot Floer homology ceases to detect mutation after its bigrading is collapsed to a single delta-grading \cite{BaldwinLevine}. While we proved in Corollary \ref{crl:mutation} that $\HFhat(Y)$ does not detect genus one mutation of closed 3-manifolds, this example demonstrates that mutation of knots is more subtle.

%As for all knots, each curve set has a distinguished curve component which wraps around the cylinder; in both of these knots this component is simply a horizontal line. All of the remaining components are ``figure-eight" components enclosing adjacent pegs.
\end{example}

We now turn to the case of an arbitrary basis for $\CFKminus(K)$. The algorithm in \cite{LOT} produces a representative of $\CFD(M_K,\mu,\lambda)$ which in general has vertices of valence greater than two. This corresponds to an immersed train track, which must be reduced to immersed curves by removing crossover arrows as in \cite{HRW}.  By \cite[Proposition 2.5]{HL2014}, there exists a vertically simplified basis $\{\xi_0, \ldots, \xi_{2n}\}$ and a horizontally simplified basis $\{\eta_0, \ldots, \eta_{2n}\}$ with filtered changes of basis
$$ \xi_i = \sum_{j=0}^{2n} a_{i,j} \eta_j \qquad \text{and} \qquad \eta_i = \sum_{j=0}^{2n} b_{i,j} \xi_j,$$
where $a_{i,j}, b_{i,j} \in \F[U]$ such that $a_{i,j} = 0$ if $A(a_{i,j} \eta_j) \neq A(\xi_i)$ and $b_{i,j} = 0$ if $A(b_{i,j} \xi_j) \neq \eta_i$. Following the construction in \cite{LOT}, the type D structure $\CFD(M,\mu,\lambda)$ admits corresponding bases $\{\hat\xi_0, \ldots, \hat\xi_{2n}\}$ and $\{\hat\eta_0, \ldots, \hat\eta_{2n}\}$ related by the changes of basis 
$$ \hat\xi_i = \sum_{j=0}^{2n} \hat a_{i,j} \hat\eta_j \qquad \text{and} \qquad \hat\eta_i = \sum_{j=0}^{2n} \hat b_{i,j} \hat\xi_j,$$
where $\hat a_{i,j} = a_{i,j}|_{U=0}$ and $\hat b_{i,j} = b_{i,j}|_{U=0}$. $\CFD(M, \mu, \lambda)$ has a vertical chain from $\hat\xi_{2i-1}$ to $\hat\xi_{2i}$  and a horizontal chain from $\hat\eta_{2i-1}$ to $\hat\eta_{2i}$ for $1\le i\le n$, as well as an unstable chain connecting $\hat\xi_0$ and $\hat\eta_0$.

For a given basis, the train track $\tracks$ in the cylinder representing $\CFD(M,\mu,\lambda)$ has a horizontal segment for each generator, with height determined by the filtration level of that generator. The left endpoints of these segments are connected by some train track $\tracks_{L}$, while the right endpoints are connected by some train track $\tracks_{R}$. Roughly speaking, vertical chains in $\CFD$ correspond to arcs in $\tracks_{L}$ connecting the left endpoints of the corresponding generators, and horizontal chains correspond to arcs in $\tracks_{R}$ connecting the right endpoints of the corresponding generators. If we use the basis $\{\hat\xi_i\}_{i=0}^{2n}$, then $\tracks_{L}$ is particularly simple; it is just a collection of arcs connecting the segment representing $\hat\xi_{2i-1}$ to the segment representing $\hat\xi_{2i}$ for $1\le i\le n$, and an arc from the segment representing $\hat\xi_0$ to the point $(-\frac{1}{2}, 0)$ (see, for example, Figure \ref{fig:connect-sum}(a)).  The basis $\{\hat\eta_i\}_{i=0}^{2n}$ can be obtained from $\{\hat\xi_i\}_{i=0}^{2n}$ by a sequence of elementary basis changes which replace $\hat\xi_i$ with $\hat\xi_i + \hat\xi_j$ for some $i$ and $j$ with $A(\xi_i) = A(\xi_j)$. The effect of an elementary basis change on the train track can be realized by inserting two crossover arrows from the segment corresponding to \(\hat \xi_i\) to the segment corresponding to \(\hat \xi_j\), one at each end of the segment (see Figure \ref{fig:connect-sum}(b)). The right half of the train track now consists of $\tracks_{R}$ along with some crossover arrows (one from each pair) coming from the basis change between \(\{\hat\xi_i\}\) and \(\{\hat\eta_j\}\). Since $\{\hat\eta\}_{i=0}^{2n}$ is horizontally simplified this can be replaced (up to equivalence of train tracks) by an immersed collection of arcs connecting $\hat\eta_{2i-1}$ to $\hat\eta_{2n}$ and $\hat\eta_0$ to the point $(\frac{1}{2}, 0)$ (Figure \ref{fig:connect-sum}(c)). Since the train track now has the form of immersed curves with crossover arrows, we can remove the arrows as usual to obtain a collection of immersed curves, possibly with local systems.

\begin{remark}
There is a distinguished component of the curve set: the curve which passes through the point $(\pm \frac{1}{2}, 0)$ and wraps around the cylinder. Note that this curve will never carry a nontrivial local system, since only one segment wraps around the cylinder.
\end{remark}

To illustrate this procedure, consider the knot $K = T(2,3)\#T(2,3)$. The complex $\CFKminus(K) \cong \CFKminus(T(2,3)) \otimes \CFKminus(T(2,3))$ is shown in Figure \ref{fig:connect-sum-complex}. Note that, while this complex does admit a horizontally and vertically simplified basis, the basis that arises naturally from the tensor product is neither horizontally nor vertically simplified. It is straightforward to check that 
$$\{\hat\xi_i\} = \{ax, bx, cx, ay, az, by, cy+bz, bz, cz\}$$
is a vertically simplified basis, and that 
$$\{\hat\eta_i\} = \{cz, bz, az, cy, cx, by, ay+bx, bx, ax\}$$
is a horizontally simplified basis. The corresponding train track is constructed in Figure \ref{fig:connect-sum}. We can think of this train track in three thirds, where the left third is a collection of arcs determined by $\{\xi_i\}$, the right third is a collection of arcs determined by $\{\eta_i\}$, and the middle third contains a sequence of arrows determined by the change of basis between $\{\xi_i\}$ and $\{\eta_i\}$.

The method described above requires computing a horizontally and a vertically simplified basis for $\CFKminus(K)$. In fact, in practice it is possible to construct a train track corresponding to $\CFKminus(K)$ in terms of any given basis, without finding either $\{\xi_i\}$ or $\{\eta_i\}$. The steps are as follows:
\begin{itemize}
\item[(1)] For each generator of $\CFKminus(K)$ there is a horizontal segments whose height is given by the filtration level;
\item[(2)] For each vertical arrow in $\CFKminus(K)$ connecting generators $x$ and $y$, we add a downward oriented arc connecting the left ends of the segments corresponding to $x$ and $y$;
\item[(3)] For each horizontal arrow in $\CFKminus(K)$, we add an upward oriented arc connecting the right ends of the corresponding segments; 
\item[(4)] In the resulting train track, we slide segments on each side (taking care to preserve the equivalence class of train track) until the train track has the form of arcs plus crossover arrows---we can now forget the orientation on the arcs; 
\item[(5)] There will be one free left endpoint of a horizontal segment and one free right endpoint---connect these to each other by a path wrapping around the cylinder;
\item[(6)] This train track can be reduced to immersed curves with local systems in the usual way.
\end{itemize}

We will not prove that in step (4) all extra edges can be paired off to form crossover arrows; the fact that this is possible essentially follows from the construction using horizontally and simplified bases above. In small examples, it is easy to do this step geometrically. Returning to the example of $K = T(2,3)\#T(2,3)$, the train track resulting from steps (1)-(3) appears in Figure \ref{fig:connect-sum2}(a), and an isotopy of the train track gives rise to the immersed curve with crossover arrows in Figure \ref{fig:connect-sum2}(b). Connecting the two loose ends by an arc wrapping around the cylinder produces the same train track as the previous method (Figure \ref{fig:connect-sum2}(c)).

Finally, we observe that a further shortcut is possible for a connected sum $K = K_1 \# K_2$, if for $i=1,2$ we have immersed curve sets $\Gamma_i$ corresponding to $\CFKminus(K_i)$. We start by drawing one copy of $\Gamma_1$ for each generator of $\CFKminus(K_2)$, shifted vertically according to the Alexander grading of the generator of $\CFKminus(K_2)$. This accounts for all horizontal and vertical arrows in $\CFKminus(K)$ of the form $a\otimes x \to b\otimes x$. We next add arcs from $a\otimes x$ to $a\otimes y$ for any arc from $x$ to $y$ in $\Gamma_2$ and for any generator $a$ of $\CFKminus(K_1)$. Note that whenever two generators $a$ and $b$ of $\CFKminus(K_1)$ are connected by an arc in $\Gamma_1$ on the same side as the arc from $x$ to $y$ in $\Gamma_2$, the two arcs from $a\otimes x$ to $a\otimes y$ and from $b\otimes x$ to $b\otimes y$ form a crossover arrow from the arc from $a$ to $b$ in the $x$ copy of $\Gamma_1$ to the arc from $a$ to $b$ in the $y$ copy of $\Gamma_1$. In the example above, this brings us straight to the train track in Figure \ref{fig:connect-sum2}(b). Finally, we connect the two remaining loose ends as before. To summarize, we have the following procedure:

\begin{itemize}
\item[(1)] Draw a copy of $\Gamma_1$ for each horizontal segment in $\Gamma_2$, with the appropriate vertical shift;
\item[(2)] For each arc on the right (resp. left) side of $\Gamma_2$ from $x$ to $y$, we connect the loose right (resp. left) ends of the $x$ and $y$ copies of $\Gamma_1$ and for each arc on the right (resp. left) side of $\Gamma_1$ we add a crossover arrow from the $x$ copy of that arc to the $y$ copy of that arc;
\item[(3)] We connect the remaining loose left end to the remaining loose right end by an arc wrapping around the cylinder.
\end{itemize}

\bigbreak
\begin{figure}[ht]
 \begin{tikzpicture}[scale=0.75,>=stealth', thick] 
 
 \node at (-4.4,2) {$\otimes$};
 \node at (-1,2) {$=$};
 \tiny
 \node (a) at (-7,3) {$a$};
 \node (b) at (-5,3) {$b$};
 \node (c) at (-5,1) {$c$};
 \draw[->] (b) to (a);
 \draw[->] (b) to (c);
 
 \node (x) at (-4,3) {$x$};
 \node (y) at (-2,3) {$y$};
 \node (z) at (-2,1) {$z$};
 \draw[->] (y) to (x);
 \draw[->] (y) to (z);
 
 \tiny
 \node(ax) at (.3,3.7) {$ax$};
 \node(ay) at (1.7,3.7) {$ay$};
 \node(az) at (1.7,2.3) {$az$};
 \node(bx) at (2.3,3.7) {$bx$};
 \node(by) at (3.7,3.7) {$by$};
 \node(bz) at (3.7,2.3) {$bz$};
 \node(cx) at (2.3,1.7) {$cx$};
 \node(cy) at (3.7,1.7) {$cy$};
 \node(cz) at (3.7,.3) {$cz$};
 
 \draw[->] (ay) to (ax);
 \draw[->] (ay) to (az);
 \draw[->] (by) to (bx);
 \draw[->] (by) to (bz);
 \draw[->] (cy) to (cx);
 \draw[->] (cy) to (cz);
 
 \draw[->, bend left = 20] (bx) to (ax);
 \draw[->, bend left = 20] (bx) to (cx);
 \draw[->, bend right = 20] (by) to (ay);
 \draw[->, bend left = 20] (by) to (cy);
 \draw[->, bend right = 20] (bz) to (az);
 \draw[->, bend right = 20] (bz) to (cz);

\end{tikzpicture}
\caption{The complex $\CFKminus$ of the connected sum of two trefoils.}\label{fig:connect-sum-complex}
\end{figure}
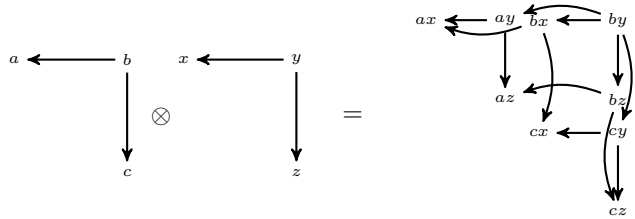

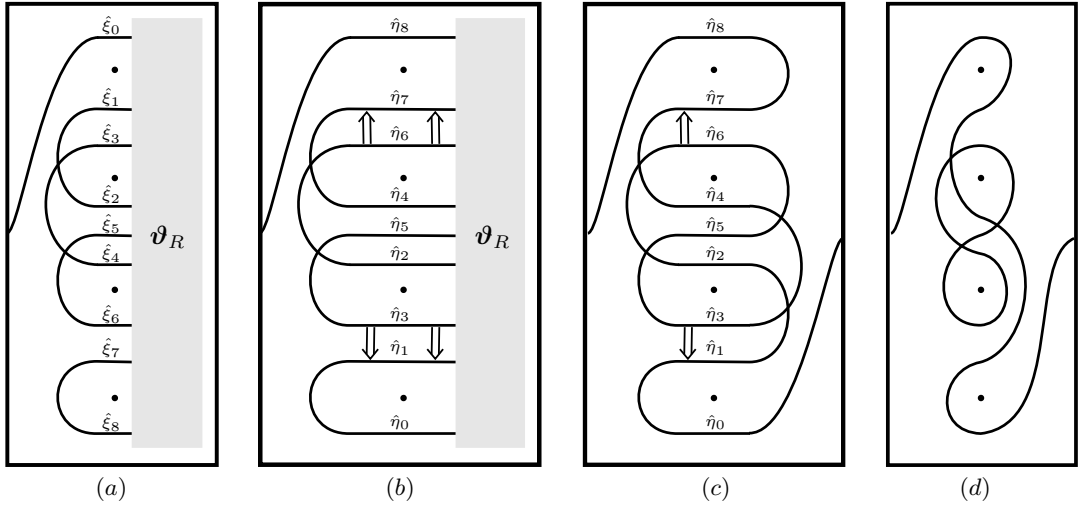
\begin{figure}[ht]
\labellist

\pinlabel {$\tracks_{R}$} at 72 102 
\pinlabel {$\tracks_{R}$} at 216 102 

\tiny
\pinlabel {$\hat\xi_0$} at 47 196 
\pinlabel {$\hat\xi_1$} at 47 164 
\pinlabel {$\hat\xi_3$} at 47 148 
\pinlabel {$\hat\xi_2$} at 47 121 
\pinlabel {$\hat\xi_5$} at 47 108
\pinlabel {$\hat\xi_4$} at 47 95 
\pinlabel {$\hat\xi_6$} at 47 68 
\pinlabel {$\hat\xi_7$} at 47 53 
\pinlabel {$\hat\xi_8$} at 47 20 

\pinlabel {$\hat\eta_8$} at 175 196 
\pinlabel {$\hat\eta_7$} at 175 164 
\pinlabel {$\hat\eta_6$} at 175 148 
\pinlabel {$\hat\eta_4$} at 175 121 
\pinlabel {$\hat\eta_5$} at 175 108 
\pinlabel {$\hat\eta_2$} at 175 95
\pinlabel {$\hat\eta_3$} at 175 68 
\pinlabel {$\hat\eta_1$} at 175 53 
\pinlabel {$\hat\eta_0$} at 175 20 

\pinlabel {$\hat\eta_8$} at 315 196 
\pinlabel {$\hat\eta_7$} at 315 164 
\pinlabel {$\hat\eta_6$} at 315 148 
\pinlabel {$\hat\eta_4$} at 315 121 
\pinlabel {$\hat\eta_5$} at 315 108 
\pinlabel {$\hat\eta_2$} at 315 95
\pinlabel {$\hat\eta_3$} at 315 68 
\pinlabel {$\hat\eta_1$} at 315 53 
\pinlabel {$\hat\eta_0$} at 315 20 

\small
\pinlabel {$(a)$} at 47 -10
\pinlabel {$(b)$} at 175 -10
\pinlabel {$(c)$} at 315 -10
\pinlabel {$(d)$} at 430 -10
\endlabellist
\includegraphics[scale=0.85]{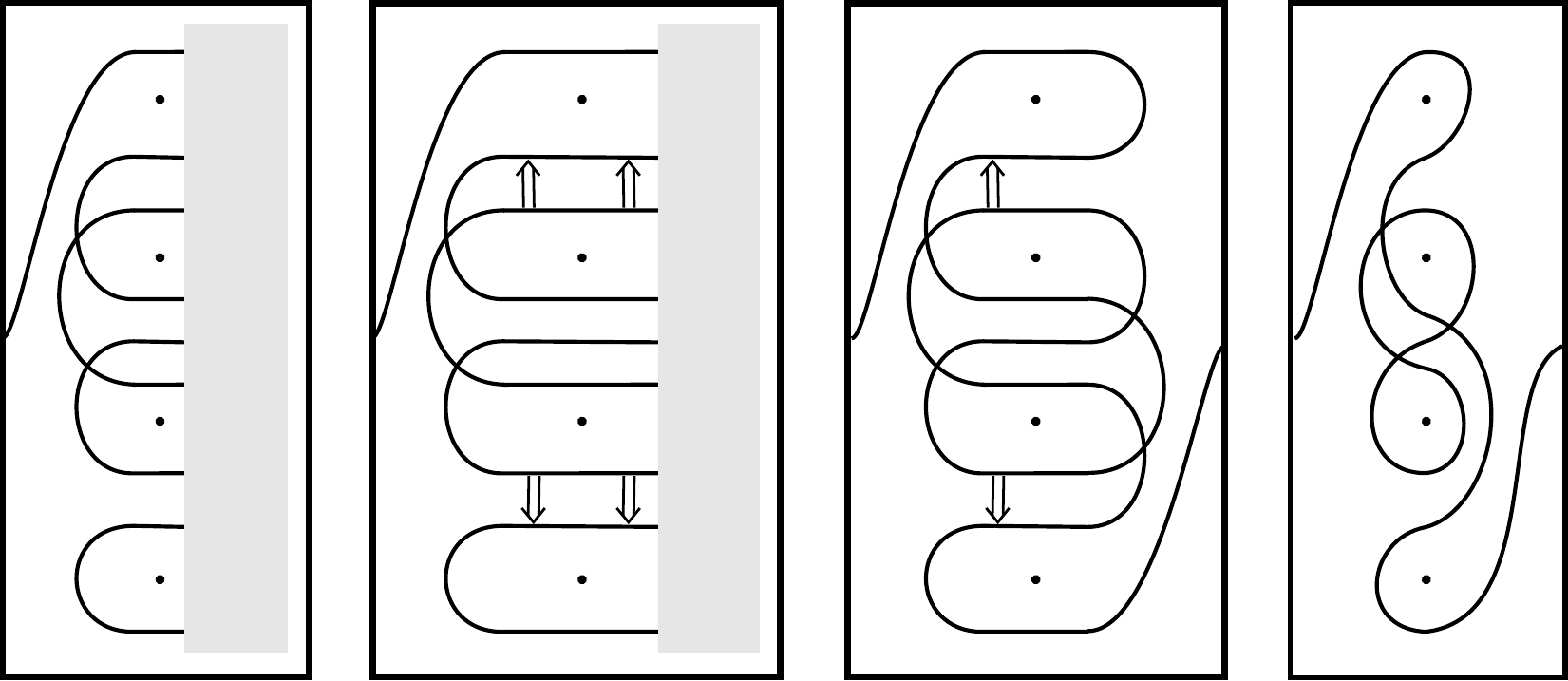}
\vspace{5 mm}
\caption{The train track constructed from $\CFKminus$ for the connected sum of two trefoils, using simplified bases; (a) if we use the vertically simplified basis, the left half of the train track is an immersed collection of arcs, while the right half is some potentially more complicated train track $\tracks_R$; (b) changing to the horizontally simplified basis amounts to inserting crossover arrows between horizontal strands at the same filtration level; (c) the right half can now be replaces with an immersed collection of arcs; (d) the arrows can be removed, resolving one crossing in the process, to obtain a set of immersed curves.}\label{fig:connect-sum}
\end{figure}

\begin{figure}[ht]
\labellist

\tiny
\pinlabel {$ax$} at 47 196 
\pinlabel {$bx$} at 47 164 
\pinlabel {$ay$} at 47 148 
\pinlabel {$cx$} at 47 121 
\pinlabel {$by$} at 47 108
\pinlabel {$az$} at 47 95 
\pinlabel {$cy$} at 47 68 
\pinlabel {$bz$} at 47 53 
\pinlabel {$cz$} at 47 20 

\pinlabel {$ax$} at 160 196 
\pinlabel {$bx$} at 160 164 
\pinlabel {$ay$} at 160 148 
\pinlabel {$cx$} at 160 121 
\pinlabel {$by$} at 160 108
\pinlabel {$az$} at 160 95 
\pinlabel {$cy$} at 160 68 
\pinlabel {$bz$} at 160 53 
\pinlabel {$cz$} at 160 20 

%\pinlabel {$\hat\eta_8$} at 175 196 
%\pinlabel {$\hat\eta_7$} at 175 164 
%\pinlabel {$\hat\eta_6$} at 175 148 
%\pinlabel {$\hat\eta_4$} at 175 121 
%\pinlabel {$\hat\eta_5$} at 175 108 
%\pinlabel {$\hat\eta_2$} at 175 95
%\pinlabel {$\hat\eta_3$} at 175 68 
%\pinlabel {$\hat\eta_1$} at 175 53 
%\pinlabel {$\hat\eta_0$} at 175 20 
%
%\pinlabel {$\hat\eta_8$} at 315 196 
%\pinlabel {$\hat\eta_7$} at 315 164 
%\pinlabel {$\hat\eta_6$} at 315 148 
%\pinlabel {$\hat\eta_4$} at 315 121 
%\pinlabel {$\hat\eta_5$} at 315 108 
%\pinlabel {$\hat\eta_2$} at 315 95
%\pinlabel {$\hat\eta_3$} at 315 68 
%\pinlabel {$\hat\eta_1$} at 315 53 
%\pinlabel {$\hat\eta_0$} at 315 20 

\small
\pinlabel {$(a)$} at 47 -10
\pinlabel {$(b)$} at 160 -10
\pinlabel {$(c)$} at 273 -10
\endlabellist
\includegraphics[scale=0.85]{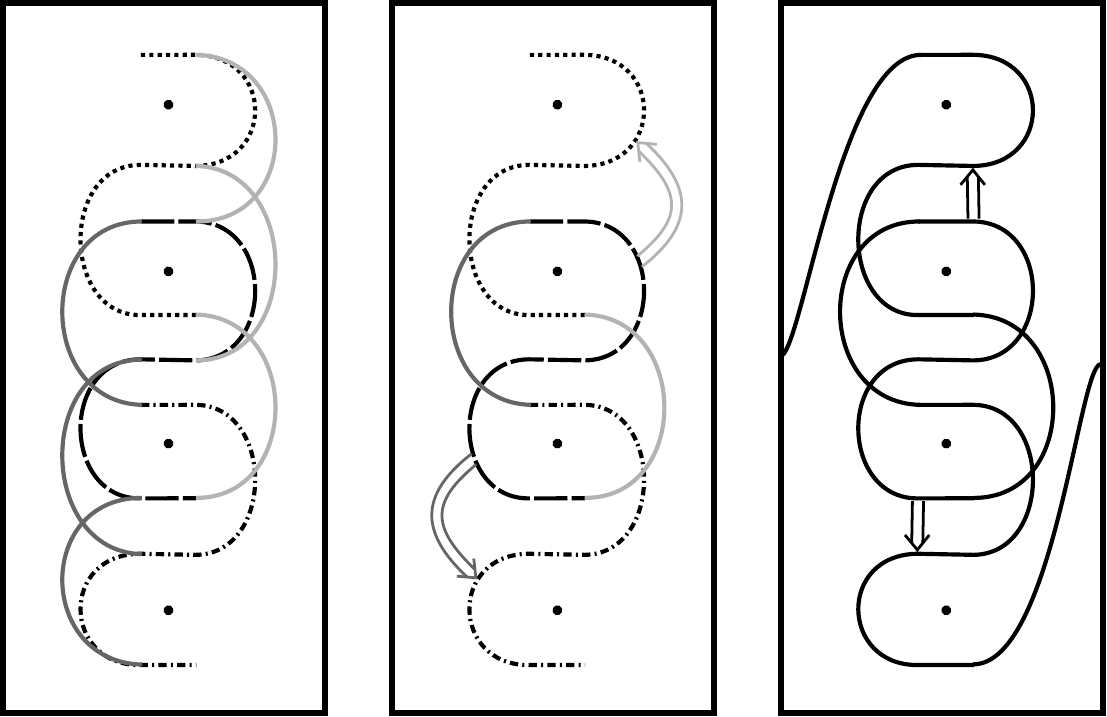}
\vspace{5 mm}
\caption{Simplified construction of the train track for a connected sum $K_1 \# K_2$ from the curves $\Gamma_1$ and $\Gamma_2$ associated with $K_1$ and $K_2$; (a) there is a copy of $\Gamma_1$ for each horizontal segment of $\Gamma_2$ (indicated by different styles of dashed line). For each arc on the right/left of $\Gamma_2$ we add segments connecting the corresponding copies of $\Gamma_1$ (indicated by dark/light gray); (b) the arcs can equivalently be viewed as arcs connecting the loose ends of copies of $\Gamma_1$ and crossover arrows connecting arcs in copies of $\Gamma_1$; (c) connecting the remaining loose ends gives a train track associated to $\CFKminus(K)$.}\label{fig:connect-sum2}
\end{figure}
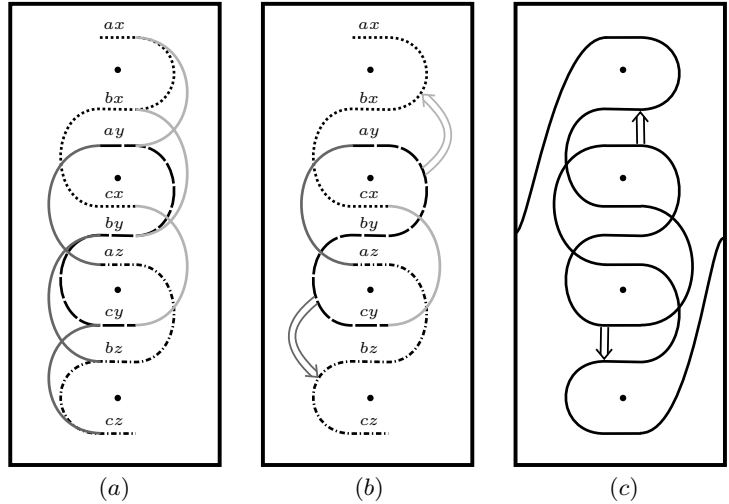

In the other direction, it is also shown in \cite{LOT} that $\HFKminus$ can be recovered from $\CFD$; we can ask if $\HFKminus$ can be recovered easily from $\curves{M_K}$ without passing through $\CFD$.

\subsection{Knot Floer homology from curves}
\label{subsec:CFK-}
Suppose that \(\mu\) is a Dehn filling slope on \(\partial M\)  and \(Y=M(\mu)\). Let \(K = K_\mu \subset M(\mu)=Y\) be the core of the Dehn filling. 
We describe how to  recover the knot Floer homology of $K$ from \(\curves{M}\). Replace the basepoint $z \in T_M$ with a marked disk $D$ (which $\curves M$ also avoids) containing two basepoints $w$ and $z$; by a slight abuse of notation, let \(\mu\) denote a representative curve of the slope $\mu$ which bisects $D$ and separates $w$ and $z$. This setup is illustrated for the right-handed trefoil, in Figure \ref{fig:knot-floer-trefoil}.

\begin{figure}[ht]
\labellist
\pinlabel {$z$} at 42 51 \pinlabel {$w$} at 72 51
\endlabellist
\includegraphics[scale=0.5]{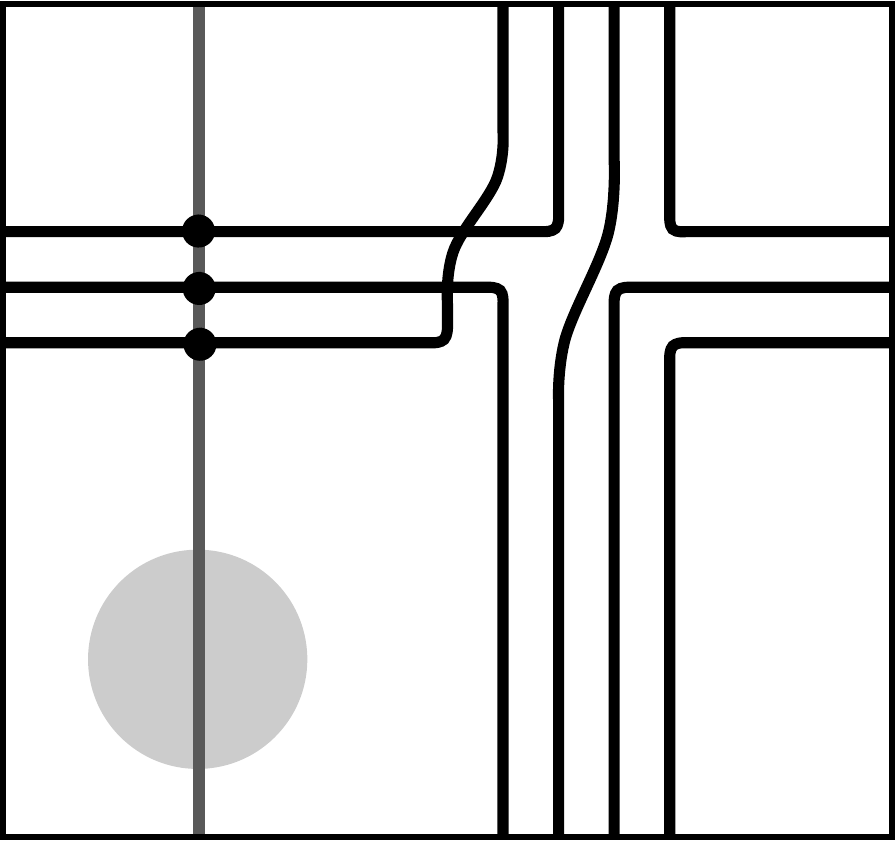}
\caption{The curve associated with the right-hand trefoil, together with the meridian $\mu$ passing through the marked disk and the basepoints $w$ and $z$.}\label{fig:knot-floer-trefoil} 
\end{figure}

Using this data, we can try to define a filtered chain complex $C^-(\curves M,\mu)$, which is a  refinement of the intersection Floer homology $HF(\curves M, \mu)$. $C^-(\curves M,\mu)$ is generated over $\F[U]$ by the intersection points between $\curves M$ and $\mu$. The differential counts immersed bigons, where a bigon covering the basepoint $w$ with multiplicity $i$ contributes with a factor of $U^i$. More precisely, the differential is defined by 
\[d(x)=\sum_{i=0}^\infty\sum_{y} U^i N_i^w(x,y)\cdot y\]
where $N_i^w(x,y)$ is the number of bigons from $x$ to $y$, counted modulo 2, covering the $w$ basepoint $i$ times. %; compare Section \ref{sub:counting}. 
Given a bigon $B$ connecting $x$ to $y$ (that is $B\in\tilde{\pi}_2(x,y)$, as in \cite[Chapter 11]{LOT}), let $n_w(B)$ and $n_z(B)$ be number of times $B$ covers $w$ and $z$, respectively. Then $z$ induces a filtration $A$ on $C^-(\curves M,\mu)$ where $A(x)-A(y) = n_z(B)-n_w(B)$ and $A(U\cdot x)=A(x)-1$. The associated graded of this object, denoted $gC^-(\curves M,\mu)$, is obtained by disallowing bigons which cross the $z$ basepoint in the differential. There are simpler versions, $\widehat{C}(\curves M,\mu)$ and $g\widehat{C}(\curves M,\mu)$, which are obtained from $C^-(\curves M,\mu)$ and $gC^-(\curves M,\mu)$ by setting $U = 0$. Equivalently, these complexes are generated over $\F$ by the intersection points between $\curves M$ and $\mu$; bigons covering the basepoint $w$ do not appear in the differential.

\noindent{\bf Warning:}  In general, we do not expect that  \(d^2=0\)  on $C^-(\curves M,\mu).$  (For examples and further discussion of this issue, see Section~\ref{sec:minus}.) However, the associated graded versions $gC^-(\curves M,\mu)$ and $g\widehat{C}(\curves M,\mu)$ are well defined and agree with knot Floer homology.

\begin{theorem}\label{thm:knot-floer}
If \(Y\) and \(K\)  are as above, the complex $gC^-(\curves M,\mu)$ (resp. $g\widehat C(\curves M,\mu))$ is filtered chain homotopy equivalent to $\gCFKminus(Y,K)$ (resp. $\gCFKhat(Y,K))$.
\end{theorem}
\begin{proof}  
This is a direct consequence of \cite[Theorem 11.19]{LOT}, which expresses \(\HFK(Y,K)\) as a box tensor product $X \boxtimes\CFD(M,\mu,\ell)$, where \(\ell\) is any curve with \(\ell \cdot \mu = 1\). 
By  \cite[Lemma 11.20]{LOT}, the type A module \(X\) has a single generator \(\mathbf{x}_0\) in idempotent \(\iota_0\), and multiplications \(m_{3+i}(\mathbf{x}_0,\rho_3,\rho_{23},\ldots, \rho_{23}, \rho_2) = U^{i+1}\mathbf{x}_0\). 
 On the other hand, intersections between $\mu$ and $\curves M$ correspond precisely to $\iota_0$ generators of $\CFD(M,\mu,\ell)$. Any bigon between intersection points not covering $z$ must cover $w$ with positive multiplicity $i$ and correspond to a sequence
\[\begin{tikzpicture}[thick, >=stealth',shorten <=0.1cm,shorten >=0.1cm] 
\draw[->] (0,0) -- (1,0);\node [above] at (0.45,0) {$\scriptstyle{3}$};
\draw[->] (1,0) -- (1.8,0);\node [above] at (1.35,0) {$\scriptstyle{23}$};
\draw[->] (2.2,0) -- (3,0);\node [above] at (2.55,0) {$\scriptstyle{23}$};
\draw[->] (3,0) -- (4,0);\node [above] at (3.45,0) {$\scriptstyle{2}$};
\node at (0,0) {$\bu$};\node at (1,0) {$\circ$};\node at (3,0) {$\circ$};\node at (4,0) {$\bu$};\node at (2,0) {$\cdots$};\end{tikzpicture}\]
in $\CFD(M,\mu,\ell)$ with $i$ generators of idempotent $\iota_1$. Thus differentials in the box tensor product precisely correspond to such bigons. 
\end{proof}

\begin{remark}
\label{Rem:L_mu}
The groups \(\HFKhat(Y,K)\) are given by \(HF(\HFhat(M),\mu)\) in \(T_M - z - w\). This group can also be computed as \(HF(\HFhat(M), L_\mu)\), where \(L_\mu\) is the {\em noncompact}  Lagrangian in the once punctured torus which has slope \(\mu\) and begins and ends at the puncture. 
\end{remark}

\begin{example}
Let \(M\) be the complement of  the  $(2,-1)$-cable of the left-hand trefoil knot in \(S^3\), as in Example~\ref{Ex:tref-cable}, and let \(\mu\) be the standard meridian. \(\HFhat(M)\)  is illustrated in  in Figure \ref{fig:knot-floer-cable-example}. The reader can easily check that 
\(C^-(\curves M, \mu) \) is isomorphic to \(HFK^-(K)\). In particular,  
\(C^-(\curves M, \mu) \) recovers the ``diagonal'' differentials in \(HFK^-(K)\), even though these differentials played no role in our calculation of \(\HFhat(M)\). Note that we do not expect this agreement of the diagonal arrows to happen in general, but it is common in simple examples.
\end{example}

The reader may object that this example  is  circular, since we constructed $\curves{M}$ using  $HFK^-(K)$. More productively, we can use the knot Floer homology of one Dehn filling to find \(\HFhat(M)\), and then use  \(\HFhat(M)\) to compute the knot Floer homology of the core of a different Dehn filling. 

\begin{example} 
\label{Ex:T(2,5)-1}
Let \(Y= -\Sigma(2,5,11)\) be the manifold obtained by \(-1\) surgery on \(T(2,5)\), and let \(K\) be the core of the surgery. Referring to Figure~\ref{Fig:T(2,5)-1}, we see that 
\( \HFKhat(Y,K) \)  is generated by the intersection points \(x_1\ldots x_9\).  The right-hand side of the figure shows (part of) \(C^-(\HFhat(M),\mu)\), which is easily seen to be a complex in this case. \end{example}

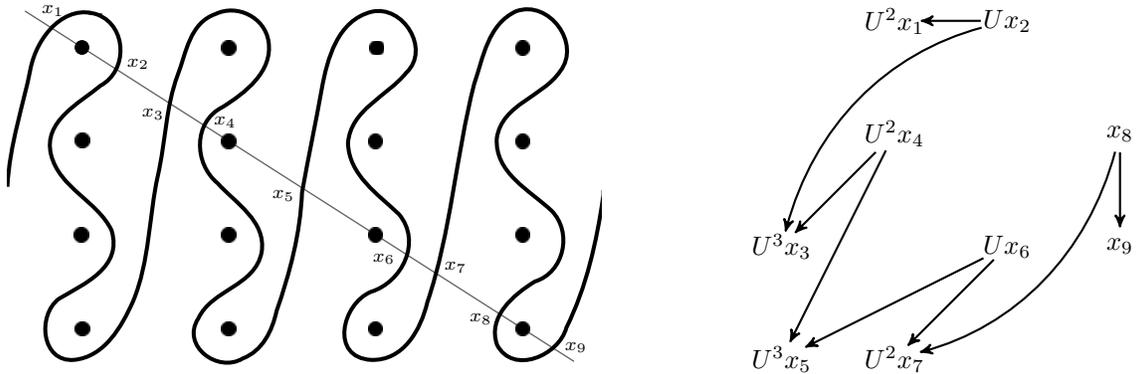
\begin{figure}[ht]
\labellist
\scriptsize
\pinlabel {$x_1$} at 27 222
\pinlabel {$x_2$} at 78 190
\pinlabel {$x_3$} at 87 160
\pinlabel {$x_4$} at 130 155
\pinlabel {$x_5$} at 165 111
\pinlabel {$x_6$} at 225 73
\pinlabel {$x_7$} at 268 68
\pinlabel {$x_8$} at 283 37
\pinlabel {$x_9$} at 340 20

\endlabellist
\includegraphics[scale=0.63]{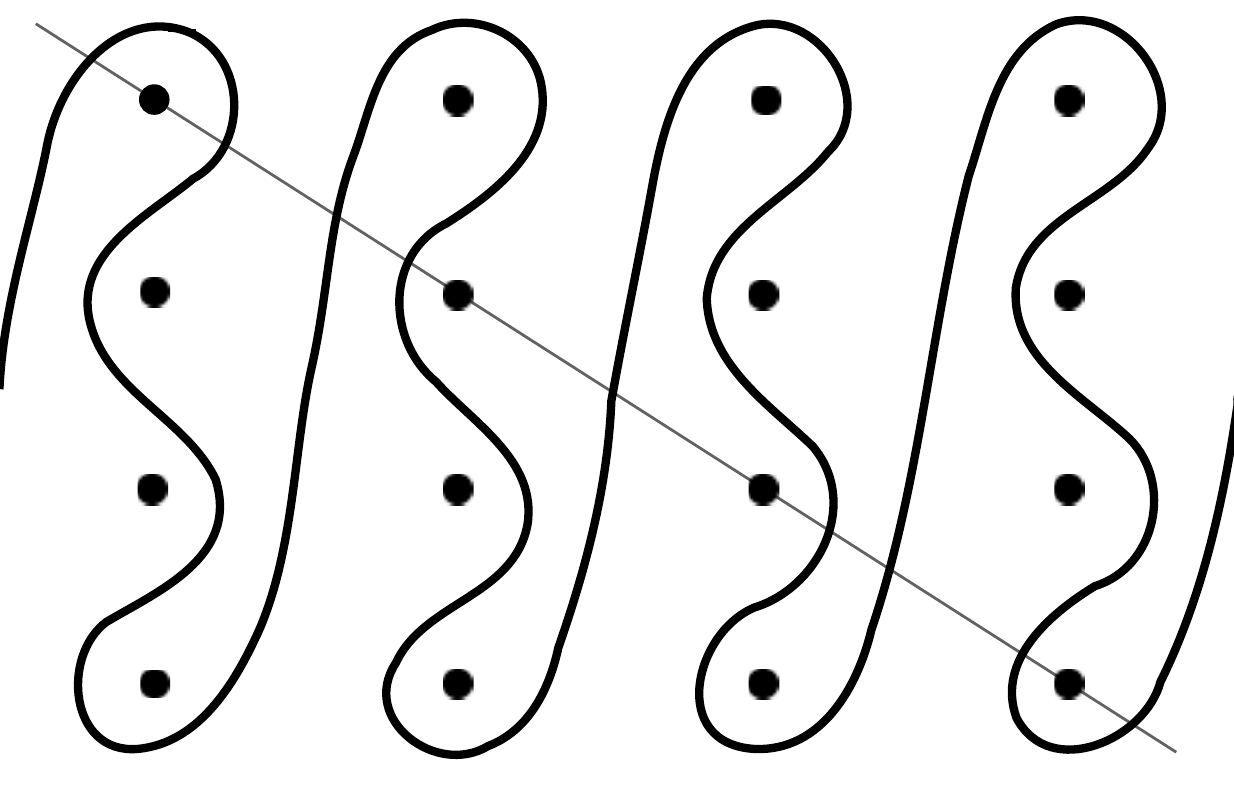}
\qquad\qquad\quad
\begin{tikzpicture}[scale=1,>=stealth', thick] 
\node at (-1.5,0) {$Ux_2$};
\node at (-3,0) {$U^2x_1$};
\node at (-4.5,-3) {$U^3x_3$};
\node at (-3,-1.5) {$U^2x_4$};
\node at (-4.5,-4.5) {$U^3x_5$};
\node at (-1.5,-3) {$Ux_6$};
\node at (-3,-4.5) {$U^2x_7$};
\node at (0,-1.5) {$x_8$};
\node at (0,-3) {$x_9$};
 \draw[->,shorten <= 0.35cm, shorten >= 0.35cm] (-1.5,0)--(-3,0);
 \draw[->,shorten <= 0.35cm, shorten >= 0.20cm] (-1.5,0) to [bend right] (-4.5,-3);
 \draw[->,shorten <= 0.35cm, shorten >= 0.25cm] (-3,-1.5)--(-4.5,-3);
  \draw[->,shorten <= 0.25cm, shorten >= 0.25cm] (-3,-1.5)--(-4.5,-4.5);
   \draw[->,shorten <= 0.35cm, shorten >= 0.35cm] (-1.5,-3)--(-4.5,-4.5);
   \draw[->,shorten <= 0.25cm, shorten >= 0.25cm] (-1.5,-3)--(-3,-4.5);
   \draw[->,shorten <= 0.25cm, shorten >= 0.35cm] (0,-1.5) to [bend left] (-3,-4.5);
    \draw[->,shorten <= 0.25cm, shorten >= 0.25cm] (0,-1.5) to  (0,-3);
 \end{tikzpicture}
\caption{\(-1\) surgery on the complement of \(T(2,5)\). On the left, intersections between between \(\HFhat(M)\) and \(L_\mu\). On the right, \(C^-(\HFhat(M),\mu)\), drawn with the usual convention.}\label{Fig:T(2,5)-1} 
\end{figure}

\subsection{\(\spinc\) structures and the Alexander grading}
The Alexander grading on \(\HFKhat(Y, K)\) is easily computed from 
 \(\HFhat(M,\spin)\). To explain this precisely, we briefly recall some facts about
 spin\(^c\) structures. For more details, we refer the reader to section 6.2 of \cite{HRW}
 and section 3.3 of \cite{RR}. 
 
The set \(\spinc(M)\) is the set of nonvanishing vector fields on \(M\) modulo a certain equivalence relation. It is a torsor over \(H^2(M) \simeq H_1(M, \partial M)\). If \(\mu\) is a simple closed curve on \(\partial M\), we can consider the sutured manifold \((M, \gamma_\mu)\), where the suture \(\gamma_\mu\) consists of two parallel copies of \(\mu\). The set \(\spinc(M, \gamma_\mu)\) consists of nonvanishing vector fields on \(M\) satisfying a certain condition on \(\partial M\), modulo the usual equivalence relation. It is  a torsor over  \(H^2(M,\partial M) \simeq H_1(M)\). There is an obvious restriction map \(\pi:\spinc(M, \gamma_\mu) \to \spinc(M)\). If \(\spin \in \spinc(M)\), we write \(\spinc(M, \gamma_\mu, \spin) := \pi^{-1}(\spin)\). \(\spinc(M, \gamma_\mu, \spin)\) is a \(H_M\)-torsor, where \(H_M\) is the image of \(H_1(\partial M)\) in \(H_1(M)\). 

Recall from  \cite{HRW} that  \(\barT_{M,\spin}\) is the cover of \(T_M\) corresponding to the kernel of the composite homomorphism \(\pi_1(\partial M) \to H_1(\partial M) \to H_1(M)\). The group of deck transformations is \(H_M\). 

Let \(L_\mu\) be noncompact Lagrangian defined in Remark~\ref{Rem:L_mu}.   The set of lifts of \(L_\mu\) to  \(\barT_{M,\spin}\) is an \(H_M\)--torsor. It can be identified with \(\spinc(M, \gamma_\mu)\) as in the proof of Proposition 47 in \cite{HRW}. If 
\(\spin \in \spinc(M)\) and \(\spinbar \in \spinc(M,\partial M, \spin)\), let \(L_{\mu,\spinbar}\) be the corresponding lift of \(L_\mu\). If we parametrize \(\partial M\) by \((\mu, \ell)\) where \(\ell \cdot \mu  = 1\), it follows immediately from the construction of 
\(\HFhat(M,\spin)\) in section 6.2 of \cite{HRW} that 
  \(\HFKhat(Y, K,\spinbar) = HF(\HFhat(M,\spin),L_{\mu,\spinbar}).\)

\iffalse
The relative \(\spinc\) decomposition of \(\HFKhat(Y,K) \simeq SHF(M,\gamma_\mu)\) can thus be determined in a way completely analogous to the calculation of \(\spinc\) structures on a closed manifold ({\it c.f.} section ??? of \cite{HRW}). Let \(\widetilde{T}_M\) be the cover of \(T_M\) corresponding to the kernel of the map \(\pi_1(T_M) \to \pi_1(\partial M)\).  The lifted invariant corresponding to the module \(X\) appearing in the proof of  Theorem~\ref{thm:knot-floer} is a single copy of \(L_\mu\) in \(\widetilde{T}_M\). If \(\spin \in \spinc(M)\) and \(\spinbar \in \spinc(M,\partial M, \spin)\), 
\(\HFKhat(M(\mu), K_\mu,\spinbar) = HF(\HFhat(M,\spin),L_{\mu,\spinbar})\). 
\fi  
 
 To define the {\em Alexander grading} on \(\HFKhat(Y,K)\), we must first fix a class \([\Sigma] \in H_2(M,\partial M)\) such that \(\partial [\Sigma] = \lambda \) (the Seifert longitude).  The fact that \(\partial M\) is a torus implies that \(\spinbar \in \spinc(M,\gamma_\mu)\) has a well-defined first Chern class \(c_1(\spinbar) \in H^2(M, \partial M)\). We define a function
 \(A:\spinc(M, \gamma_\mu) \to \frac 12 \Z\) by \(A(\spinbar) = \frac 12 \langle c_1(\spinbar), [\Sigma] \rangle \).
 If \(x \in \HFKhat(Y,K,\spinbar)\), its Alexander grading is defined to be \(A(x) := A(\spinbar) \). (This definition is most useful when \(b_1(M)=1\), in which case  \([\Sigma]\) is unique up to sign). 
 
 We can view \(\barT_{M,\spin} \) as the quotient of \(H_1(\partial M)\) by \(\langle \lambda  \rangle\). As such, there is a natural height function \(h: \barT_{\mu,\spin} \to \R\) given by \(  h(v) =  v \cdot \lambda\). 
 If \(a \in H_M\), \(c_1(\spinbar+a) = c_1(\spinbar) + 2PD(a)\), so \[A(\spinbar + a) - A(\spinbar) = \langle PD(a), [\Sigma] \rangle = a \cdot [\Sigma]= a \cdot \lambda. \] On the other hand  \(L_{\mu,\spinbar+a}\) is the result of translating \(L_{\mu,\spinbar}\) by \(a\), so the heights of \(L_{\mu,\spinbar + a}\) and \(L_{\mu,\spinbar}\) also differ by \(a \cdot \lambda\). After normalizing the height function on \(\barT_{M,\spin}\) by an overall shift, we see we have proved the following 
\begin{proposition}
\(A(\spinbar) = h(p_{\mu,\spinbar})\), where \(p_{\mu,\spinbar}\) is the midpoint of \(L_{\mu,\spinbar}\).
\end{proposition} 
 
 If \(\spin\) is a \(Spin\) structure, the normalized height function can be easily determined from the fact that the conjugation symmetry sends \(A\) to \(-A\). 
 
 \begin{example}
Let \(K \subset S^3\) be  the  $(2,-1)$-cable of the left-hand trefoil knot, as illustrated in  in Figure \ref{fig:knot-floer-cable-example}. The generators in the figure are labeled so that \(A(x_i)=A(y_i)=i\).  
\end{example}

 \begin{example} Let \(M\) be the twisted \(I\)-bundle over the Klein bottle. The homological longitude \(\lambda\) is {\em twice} a primitive curve; choose any \(\mu\) with \( \lambda \cdot \mu= 2\).  The invariant \(\HFhat(M)\) is shown in Figure~\ref{fig:twisted-I-with-NBHD}. \(\spinc(M) = \{ \spin_0, \spin_1\}\); both elements are fixed by the conjugation action. \(\HFKhat(Y,K, \spin_0)\) has two generators, both with \(A\) grading \(0\), while 
 \(\HFKhat(Y,K, \spin_1)\) has one generator with \(A\) grading \(1\) and one with \(A\) grading \(-1\). Notice that the height functions on \(\barT_{M,\spin_0}\) and \(\barT_{M,\spin_1}\) are incompatible: the dots in \(\barT_{M,\spin_0}\) are at odd heights, while the dots in \(\barT_{M,\spin_1}\) are at even heights. There is no reasonable way that we can combine \(\HFhat(M,\spin_0)\) and \(\HFhat(M,\spin_1)\) to get a single collection of curves in  \(\barT_M\). 
 \end{example}

\begin{example} 
Let \(Y\) and \(K\) be as in Example~\ref{Ex:T(2,5)-1}. Referring to Figure~\ref{Fig:T(2,5)-1}, we see that 
\( \HFKhat(Y,K) \)  is generated by the intersection points \(x_1\ldots x_9\).  The Alexander and relative Maslov gradings of the generators  are easily computed and are shown in the table below:
\begin{center}
\begin{tabular} {|r||c|c|c|c|c|c|c|c|c|}
\hline
$ i $& 1 & 2 & 3 & 4 & 5 & 6 & 7 & 8 & 9 \\
\hline $A(x_i)$ & 2 & 1 & 1 & 1 & 0 & -1 & -1 & -1  & -2 \\
\hline $M(x_i)$ & 4 & 3 & 0 & 1 & -2 & -1 & -2 & 1 & 0 \\
\hline
\end{tabular}
\end{center}
\end{example}

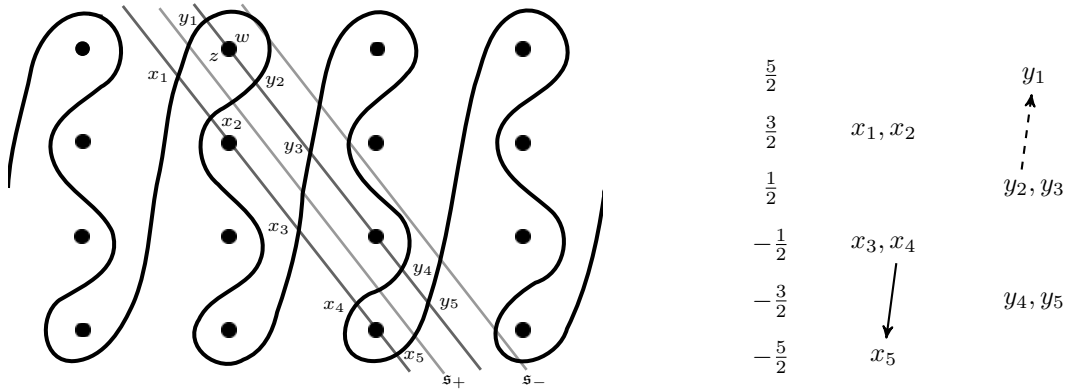
\begin{figure}[ht]
\labellist
\scriptsize
\pinlabel {$x_1$} at 90 183
\pinlabel {$x_2$} at 134 155
\pinlabel {$x_3$} at 162 92
\pinlabel {$x_4$} at 195  45
\pinlabel {$x_5$} at 243 15
\pinlabel {$y_1$} at 108 217
\pinlabel {$y_2$} at 160 180
\pinlabel {$y_3$} at 171 140
\pinlabel {$y_4$} at 248 67
\pinlabel {$y_5$} at 264  47
\pinlabel {$z$} at 123 195
\pinlabel {$w$} at 140 206
\pinlabel {$\spin_+$} at 267 0
\pinlabel {$\spin_-$} at 315 0

\endlabellist
\includegraphics[scale=0.63]{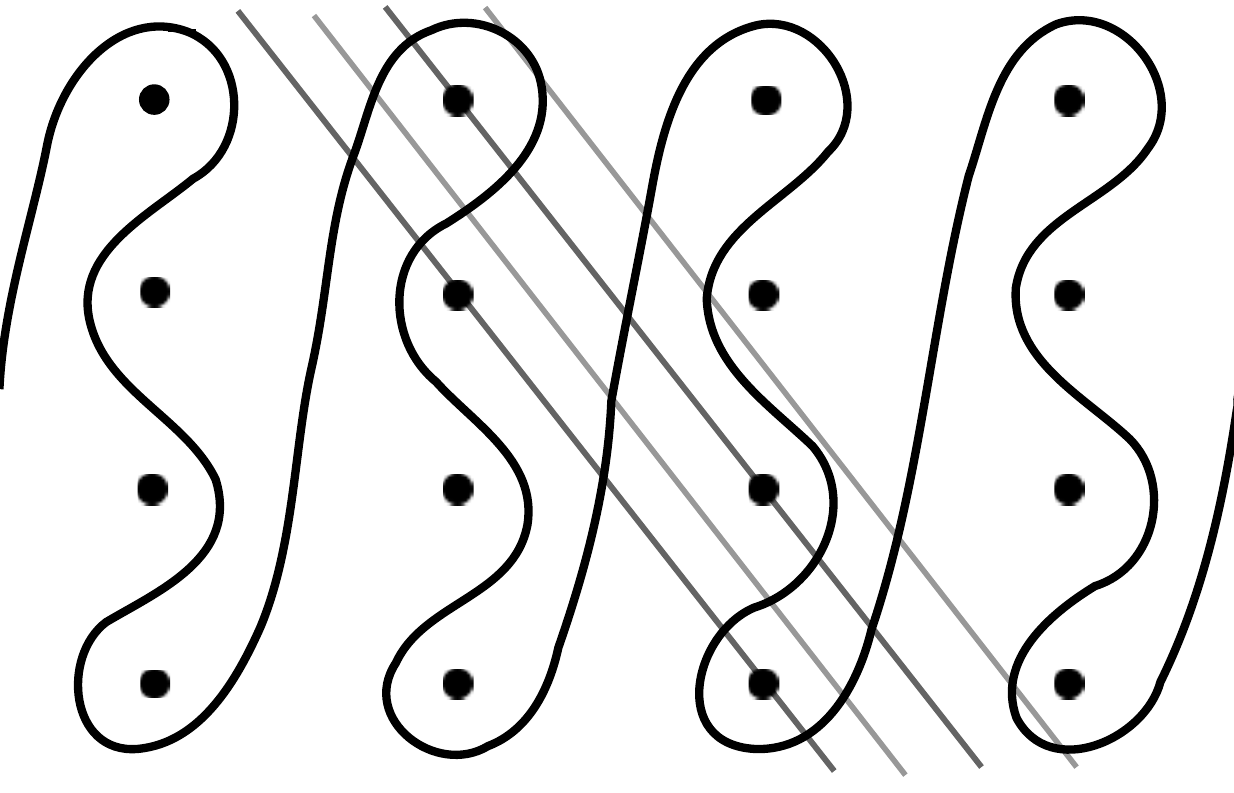}
\qquad\qquad\quad
\begin{tikzpicture}[scale=1,>=stealth', thick] 
\node at (-2,-0.5) {$x_1,x_2$};
\node at (-3.5,-0.5) {$ \frac{3}{2}$};
\node at (-2,-2.0) {$x_3,x_4$};
\node at (-3.5,-2) {$- \frac{1}{2}$};
\node at (-2,-3.5) {$x_5$};
\node at (-3.5,-3.5) {$- \frac{5}{2}$};
\node at (0,0.25) {$y_1$};
\node at (-3.5,0.25) {$ \frac{5}{2}$};
\node at (0,-1.25) {$y_2,y_3$};
\node at (-3.5,-1.25) {$ \frac{1}{2}$};
\node at (0,-2.75) {$y_4,y_5$};
\node at (-3.5,-2.75) {$- \frac{3}{2}$};
 \draw[->,shorten <= 0.25cm, shorten >= 0.25cm] (-1.8,-2)--(-2,-3.5);
  \draw[dashed,->,shorten <= 0.25cm, shorten >= 0.25cm] (-0.2,-1.25)--(0,0.25);
  \end{tikzpicture}
\caption{\(-2\) surgery on the complement of \(T(2,5)\). On the left, intersections between between \(\HFhat(M)\) and \(L_\mu\) (dark grey line segments). The light grey lines are the curves we would pair with to compute \(\HFhat(Y)\).  On the right, \(\HFKhat(Y,K))\), with Alexander grading indicated by height. The solid arrow indicates a differential with \(n_w=0,n_z=1\), while the dashed arrow indicates a differential with \(n_z=0,n_w=1\). }\label{Fig:T(2,5)-2} 
\end{figure}

\begin{example} 
Let \(Y\) be the manifold obtained by \(-2\) surgery on \(T(2,5)\), and let \(K\) be the core of the surgery. The relevant diagram is shown on the left-hand side of Figure~\ref{Fig:T(2,5)-2}; the right-hand side shows the Alexander grading on \(\HFKhat(Y,K)\) together with the differentials with \(n_z=0\) or \(n_w=0\). 
  \(\spinc(M,\partial M)\) divides into two equivalence classes. Generators in the first equivalence class are labeled \(x_i\); those in the second equivalence class are labeled \(y_i\). \(\spinc (Y) \) has two elements, \(\spin_+\) and \(\spin_-\). As the figure shows, if we forget  \(w\), the generators in the first  class give generators of \(\CFhat(Y,\spin_+)\) and the generators in the second equivalence class become generators of \(\CFhat(Y,\spin_-)\). If we forget \(z\), the roles of the two equivalence classes are reversed. 
\end{example}

%% file: sections/properties.tex
% !TEX root = ../companion.tex
 %properties.tex

In this section, we give some applications of the relation between \(\HFhat(M)\) and knot Floer homology. These include direct characterizations of the Turaev torsion and Thurston norm in terms of 
\(\HFhat(M)\), and a new proof of Theorem 1.6 of \cite{RR}, which characterizes the L-space interval of a Floer simple manifold in terms of its Turaev torsion.

\subsection{The Turaev torsion} \label{subsection:torsion}
\(\HFhat(M,\spin)\) is a compact, oriented multicurve, so it defines a class in \(H_1(\barT_{M, \spin})\). These classes (for all \(\spin \in \spinc(M)\)) determine and are determined by the Turaev torsion of \(M\).

To make this relation precise, we recall some facts about the Turaev torsion. First, we discuss \(\spinc\) structures.
The set  of relative \(\spinc\) structures \(\spinc(M, \partial M)\) is the set of nonvanishing vector fields which  point out of \(\partial M\), modulo the usual equivalence relation. It is a torsor over \(H^2(M,\partial M) \simeq H_1(M)\). Although this definition is  similar to that of \(\spinc(M,\gamma_m)\), the two boundary conditions differ, and there is no canonical way to identify the two sets. 

To pin down the sign of the Turaev torsion, we must orient \(H_*(M;\Q)\).  If \(\partial M = T^2\), then \(H_1(M;\Q) \simeq \Q \oplus H_1(M,\partial M;\Q)\) where \(\Q\) is the image of the map \(H_1(\partial M;\Q) \to H_1(M)\). We choose a generator \(m\) for the image. If \(h_0, h_1\)  are bases of of \(H_0(M;\Q)\), \(H_1(M,\partial M;\Q)\), then \((h_0, m, h_1, h_2, h_3)\) defines an orientation of \(H_*(M;\Q)\), where \(h_2\) and \(h_3\) are dual bases to \(h_1\) and \(h_0\) with respect to the intersection pairing. It is easy to see that the resulting orientation does not depend on the choice of \(h_0\) and \( h_1\). Hence choosing a homology orientation amounts to choosing a generator \(m\) for the image of \(H_1(\partial M;\Q)\). 

Other equivalent ways of fixing a homology orientation are to choose either a {\em homological longitude} \(\lambda \) ({\it i.e.} a generator of \(\ker(H_1(\partial M) \to H_1(M))) \) or the homology class of a Seifert surface ({\it i.e.} a class \([\Sigma] \in H_2(M,\partial M)\) with 
\(\partial [\Sigma] = \lambda.\) These two are related to the first one by the requirement
that \([\Sigma] \cdot m > 0\).

Once we have fixed a homology orientation, the Turaev torsion can be thought of as a function \(\tau_M\co\spinc(M, \partial M) \to \Z\). If \(b_1(M) >1\), \(\tau(\spinhat) = 0 \) for all but finitely many values of \(\spinhat \in \spinc(M,\partial M)\). If \(b_1(M)=1\), \(\tau_M(\spinhat)=1\) if \(\langle c_1(\spinhat), [\Sigma]\rangle\gg 0  \), and \(\tau_M(\spinbar)=0\) if \(\langle c_1(\spinhat), [\Sigma]\rangle \ll 0\), where \([\Sigma]\) is determined by the homology orientation as above. 

From now on, we fix a primitive curve \(\mu \in H_1(\partial M)\) with \( [\Sigma] \cdot \mu > 0\). 
If we fill in the punctures in  \(\barT_{M,\spin}\), the result is homeomorphic to \(S^1 \times \R\), so it has two ends.  We use the homology orientation  to identify these ends as {\em positive} and {\em negative}, according to the convention that for \(p \in \barT_{M,\spin}\)  \(p+n\mu\) converges to the positive end as \(n \to + \infty\). 
Similarly, the homology orientation induces an orientation on the Lagrangians \(L_{\mu, \spin}\), with the positive end of the Lagrangian pointing to the positive end of 
\(\barT_{M,\spin}\). 

A Turaev torsion for balanced sutured manifolds was defined in \cite{FJR}. It may be viewed as a function \(\tau_{(M,\gamma)}:\spinc(M,\gamma) \to \Z\). Note that every balanced sutured manifold carries a canonical homology orientation --- no additional choices need to be made. If \(M\) is a manifold with torus boundary, then with respect to the homology orientation on \(M\) for which \(\mu\) is a positive generator, we have 
 $$\sum_{\spinbar \in \spinc(M,\gamma_\mu)} \tau_{(M,\gamma_\mu)} [\spinbar] \sim (1-[\mu]) \sum_{\spinhat \in \spinc(M,\partial M)} \tau_M(\spinhat)[\spinhat]
 $$
 \cite[Prop 2.1]{RR}.
 Here \(\sim\) indicates equality up to overall multiplication by some element of \(H_1(M)\). The following lemma is an easy consequence of this fact.

\begin{lemma}
\label{Lem:taudiff}
Suppose \(\mu\) is a primitive curve in \(H_1(\partial M)\). A \(\spinc\) structure \(\spinbar \in \spinc(M,\gamma_\mu)\)  determines \(\spinc\) structures 
\(\spinbar_\pm \in \spinc(M,\partial M)\) satisfying the following properties:
\begin{enumerate}
\item \(\spinbar_+ - \spinbar_- = [\mu] \in H_1(\partial M)\)   .
\item If \(x \in H_1(M)\), \((\spinbar + x)_\pm = \spinbar_\pm + x\). 
\item  \(\tau_{(M, \gamma_\mu)}(\spinbar) = \tau_M(\spinbar_+) - \tau_M(\spinbar_-)\).  
\end{enumerate}
\end{lemma}

The set \(p^{-1}(z) \subset \barT_{M,\spin}\) is an \(H_M\) torsor.   Properties (1) and (2) mean that we can  identify  \(p^{-1}(z) = \{ z_{\spinhat} \, | \,  \spinhat \in \spinc(M, \partial M, \spin)\}\) in such a way that \(z_{\spinbar_\pm}\) is the positive/negative end of \(L_{\mu,\spin}\).

For each \(\spinhat \in \spinc(M, \partial M)\), let \(\gamma_{\spinhat}:[0,\infty) \to \barT_{M,\spin}\)  be a path from \(z_{\spinhat}\) to the negative end of \(\barT_{M,\spin}\). We will define \(n_{\spinhat}\) to be the signed intersection number of \(\gamma_{\spinhat}\) with \(\HFhat(M,\spin)\). In order to make sense of this definition, we must orient \( \HFhat(M,\spin)\) and make sense of the homology class of 
 \(\gamma_{\spinhat}\). 
 
  Recall that the \(\Z/2\) grading on \(\HFhat(M,\spin)\) gives it a well-defined orientation up to global orientation reversal. To pin down the global orientation, recall that \(\HFKhat(M(\mu),k_\mu,\spinbar) = HF(\HFhat(M,\spin), L_{\mu,\spinbar})\), so \(\tau_{(M,{\gamma_\mu})}(\spinbar) = \pm \HFhat(M,\spin) \cdot L_\mu\). \(L_\mu\) is oriented by our choice of homology orientation, and we orient \(\HFhat(M,\spin)\) so that \(\tau_{(M,\gamma_\mu)}(\spinbar) =  \HFhat(M,\spin) \cdot L_\mu\). 

Homologically, we can express \([\gamma_{\spinhat}]\) as follows. Let \(X_k\) be the manifold with boundary obtained by removing a regular neighborhood of \(p^{-1}(z)\) from \(\barT_{M,\spin}\) and  then removing everything below height \(k\). For \(k<k'\), there is a map \(H_*(X_{k'}) \to H_*(X_k)\) induced by inclusion. Then 
\( [\gamma_{\spinhat}] = - \sum_{j\geq 1} [L_{\mu,\spinbar - j\mu}]\)
  defines an element of the direct limit 
\(\varinjlim H_*(X_k,\partial X_k)\),  where \( \spinbar \in  \spinc(M, \gamma_\mu)\) is defined by the relation \(\spinbar_+ = \spinhat + \mu\). \(\HFhat(M,\spin)\) is a compactly supported oriented multicurve, so it defines a class in \(H_1(\barT_{M,\spin})\), and there is a well-defined intersection number \(n_{\spinhat} =[\gamma_{\spinhat}] \cdot [\HFhat(M,\spin)]\). 

\begin{proposition}\label{prp:torsion}
	\(\tau_M(\spinhat) = n_{\spinhat}\) for all \(\spinhat \in \spinc(M, \partial M)\). 
\end{proposition}

\begin{proof}
Given \(\spinhat \in \spinc(M, \partial M)\), define \(\spinbar \in  \spinc(M, \gamma_\mu)\) by the relation \(\spinbar_+ = \spinhat + \mu\) as above. Then 
		  \([\gamma_{\spinhat + \mu}] - [\gamma_{\spinhat}] =- [L_{\mu, \spinbar}]\). It follows that 		 
		 \begin{align*}
	n_{\spinhat + \mu } - n_{\spinhat}  & = [\HFhat(M, \spin)]\cdot [L_{\mu, \spinbar}] 
	 = \chi(\SFH(M, \gamma_{\mu},\spinbar)) \\
	& = \tau_{(M, \gamma_\mu)} (\spinbar) 
	 = \tau_M(\spinhat + \mu) - \tau_M(\spinhat)
	\end{align*}
	where in the first line we have used Proposition~\ref{prop:HFK} and in the second we used  Lemma~\ref{Lem:taudiff}. Hence if the statement of the proposition holds for \(\spinhat\), it holds for \(\spinhat + \mu\) as well. 
	
	Since \(\HFhat(M,\spin)\) is compactly supported, \(n_{\spinhat}=0\) for \(\langle c_1(\spinhat), [\Sigma]\rangle \ll 0\). In addition,  \(\tau(M,\spinhat) = 0\) for \(\langle c_1(\spinhat), [\Sigma]\rangle \ll 0\). 
	Taking \(N\) sufficiently large, we see that the proposition holds for \(\spinhat - N \mu\), and hence for \(\spinhat\). 
\end{proof}

\begin{remark}
The proposition shows that for \(\spinhat \in \spinc(M, \partial M, \spin)\),  \(\tau_{M}(\spinhat)\) is determined by the class of  \([\HFhat(M,\spin)] \in H_1(\barT_{M,\spin} - p^{-1}(z))\). Conversely, it is easy to see that any \(x \in H_1(\barT_{M,\spin} - p^{-1}(z))\) is determined by its intersection numbers with the \(\gamma_{\spinhat}\)'s. Hence the information carried by \(\tau_M\) is precisely the homology classes of \([\HFhat(M,\spin)] \in H_1(\barT_{M,\spin} - p^{-1}(z))\) as \(\spin\) runs over \(\spinc(M)\). 
\end{remark}

\begin{corollary}
Let \(p:\barT_{M,\spin} \to T_M\) be the projection. If \(b_1(M)=1\), \(p_*([\HFhat(M,\spin)]) =  \lambda\), where \(\lambda\) is the homological longitude; otherwise, \(p_*([\HFhat(M,\spin)]) = 0\). 
\end{corollary}

\begin{proof} \(H_1(\barT_{M,\spin}) \simeq \Z\) is generated by \(\lambda\). Let \(\gamma\) be a path from the positive end of \(\barT_{M,\spin}\) to the negative end; then \([\gamma] \cdot [\lambda] = 1\). If 
\(b_1(M)=1\), we know \(n_{\spinbar}= \tau_M(\spinbar) = 1\) for \(\langle c_1(\spinbar), [\Sigma] \rangle \gg0\), and  \(n_{\spinbar}= \tau_M(\spinbar) = 1\) for \(\langle c_1(\spinbar), [\Sigma] \rangle \ll0\). It follows that \(\gamma \cdot [\HFhat(M, \spin)] = 1\), which implies  
\([\HFhat(M,\spin) ]=  \lambda\). A similar argument applies when \(b_1(M)>1\). 
\end{proof}

Throughout, we have set up our orientation conventions so that  \(\lambda \cdot \mu > 0\). We chose this convention, rather than the more usual \(\mu \cdot \lambda >0\), so that \(n\) surgery on a knot in \(S^3\) corresponds to pairing with a line of slope \(n\). (The usual convention for \(K\subset S^3\) is that \(\mu \cdot \lambda >0\) on the boundary of a tubular neighborhood of \(K\), which forces \(\lambda \cdot \mu > 0 \) on the boundary of the complement of \(K\).)

\subsection{The Thurston norm}
\label{subsec:Thurston-norm}
Suppose that \(x \in H_2(M,\partial M)\) satisfies \(\partial x = \lambda\). 
We can use the  relationship between \(\barcurves{M}\) and knot Floer homology to express the Thurston norm of \(x\) in terms of the \(\HFhat(M,\spin)\). 
  If \(\bargamma \subset \overline{T}_{M,\spin}\) is a curve or collection of curves, we define 
 $$k^+(\bargamma) =  \max \{\langle c_1(\spinbar), x\rangle \, | \, z_{\spinbar}  \ \text{is not connected to} \ + \infty \ \text{in the complement of} \ \HFhat(M,\spin)  \}.$$

\begin{proposition}
Suppose that \(x\) is as above and that  \(\Sigma\) is a norm-minimizing surface representing  \(x\). Then
$$ - \chi(\Sigma) =  \max_{\spin \in \spinc(M)} k^+(\HFhat(M,\spin)).$$
\end{proposition}

\begin{proof}
 If \(\overline{\gamma}\) is a curve in \(\overline{T}_{M,\spin}\), let \(\mathcal{P}(\bargamma)\) be the set of corners of the pegboard representative of \(\bargamma\). 
 %. If \(\mathcal{P}(\bargamma) \neq \emptyset\), we define
 %$$ k^+(\bargamma) =  \max \{\langle c_1(\spinbar), [\Sigma] \rangle \, | \, z_{\spinbar} \in \mathcal{P}(\bargamma)\}. $$
 We claim that if \(\mathcal{P}(\bargamma) \neq \emptyset\), then  
 $$ \max \{\langle c_1(\spinbar), x  \rangle \, | \, z_{\spinbar} \in \mathcal{P}(\bargamma)\}= k^+(\bargamma).$$
  Indeed, if \(z_{\spinbar}\) is a highest peg in \(\mathcal{P}(\bargamma)\), then it must lie below \(\bargamma\), so \(\spinbar\) is not connected to \(+ \infty\) in the complement of \(\HFhat(M,\spin)\). Conversely, it is clear that every peg above \(z_{\spinbar}\) is connected to \(+\infty\). 
  
  Next, let \(l\) be the rational homological  longitude of \(M\) ({\it i.e.} the primitive class of which \(\lambda\) the Seifert longitude \(\lambda\) is a positive multiple) and let \(\lambda_k \subset \overline{T}_\spin\) be the curve parallel to \(\lambda\) and passing though \(z_{\spinbar}\) with \(\langle c_1(\spinbar), [\Sigma] \rangle=k\). Note that if 
 \(  \bargamma = \barcurves{M, \spin}\), Theorem~\ref{thm:knot-floer} implies that 
 $$\HFKhat(M(l), K_l, \spin_k) = HF(\bargamma, \lambda_k).$$
Here   \(\spin_k \in \spinc(M, \partial M)\) is defined to be the  relative spin\(^c\) structure which restricts to \(\spin\) on \(M\) and satisfies \(\langle c_1(\spin_k), [\Sigma] \rangle= k\). 
 
  If \(\mathcal{P}(\bargamma) \neq \emptyset\), we claim that 
  $$  \max \{k \, | \, \HFa(\bargamma,\lambda_k) \neq 0)\} = k^+(\bargamma).$$
  To see this,  pull \(\bargamma\) tight. If \(\langle c_1(\spinbar), [\Sigma] \rangle>k^+\), then the complex computing \(\HFa(\bargamma,\lambda_k)\) has no generators. Conversely, since \(\bargamma\) hangs on a peg of height \(k^+\),  some arc of \(\bargamma\)  must lie above \(\lambda_{k^+}\), and some arc of it must lie below.  Since \(\bargamma\) is pulled tight, \(\HFa(\bargamma,\lambda_{k^+})\} \neq 0\). 
  
Next  consider the case where \(\mathcal{P}(\bargamma)= \emptyset\). Then \(\bargamma\) is solid torus like. It is represented by a curve parallel to \(\lambda\), and is trapped between two rows of pegs at height \(n\) and \(n+1\). We  have \(\HFa(\bargamma,\lambda_{k})\} = 0\) for all \(k\), but \(k^+(\bargamma) = n\).  

Now we consider \(\barcurves{M}\). 
The case in which every component of \(\barcurves{M}\) is solid torus like has been studied by Gillespie \cite{Gillespie}, who showed that such an \(M\) must be boundary compressible. In this case, it is easy to see that the proposition holds. Thus we may assume that  not every component of \(\barcurves{M}\) is solid torus-like. 
Taking the \(\max\) of the relations above over all components of  \(\barcurves{M} \), we see that 
$$ \max \{\langle c_1(\spinbar), [\Sigma] \rangle \, | \, z_{\spinbar} \in \mathcal{P}(\barcurves{M})\} = 
\max \{k \,| \,  \HFKhat(M_l, K_l, \spin_k) \neq 0\} = - \chi(\Sigma) $$
where the last equality follows from the fact that  the knot  Floer homology determines the Thurston norm.  

It remains to show that \(k^+(\barcurves{M}) = \max \{\langle c_1(\spinbar), [\Sigma] \rangle \, | \, z_{\spinbar} \in \mathcal{P}(\barcurves{M})\}\) . The only way  this can fail to happen is if 
 \(\barcurves{M}\) has a solid torus-like component at height \(n>k^++1\). Suppose that we have such a component. Then it pairs nontrivially with any curve \(\mu\) which satisfies \(\mu \cdot l= 1\), so we have \(\HFKhat(M(\mu),K_\mu, n) \neq 0\). The fact that knot Floer homology determines the Thurston norm  implies that \(-\chi(\Sigma) \geq n-1 >k^+\), which is a contradiction. 
\end{proof}

\subsection{L-space slopes and torsion} 
Recall that \(\sL(M) = \{\mu \, | \, M(\mu) \ \text{is an L-space}\}\) is the set of L-space filling slopes of \(M\). Let \(S^{sing}(M)\) be the set of essential tangent slopes to \(\HFhat(M)\); that is 
 \(S^{sing}(M) = \{\alpha \, | \, \alpha \ \text{is tangent to any representative of} \ \HFhat(M)\}\). 
We showed in \cite{HRW} that \(\sL^\circ(M)\) is the complement of \(S^{sing}(M)\). Manifolds for which \(\sL^\circ(M)\) is nonempty are said to be Floer simple. 

For the rest of this section, we assume that \(M\) is Floer simple. 
As a application of Proposition~\ref{prp:torsion}, we give a short proof of Theorem 1 of \cite{RR}, which characterizes the set \(\sL(M)\) in terms of the Turaev torsion of \(M\). 

Suppose \(\alpha \in \sL^\circ_M\), and identify \(H^1(\partial M; \R)\) with \(\R^2\) by the map 
\(\beta \mapsto (\beta \cdot \alpha, \beta \cdot l)\), where \(l\) is the rational homological longitude.  Suppose further that \(\HFhat(M,\spin)\) has been pulled tight, and let \(\tildecurves{M,\spin}\) be it's preimage  under the covering map \(H_1(\partial M;\R) \setminus H_1(\partial M;\Z)
\to \barT_{M,\spin}\). 

\begin{lemma} 
 With respect to the coordinates above, the pegboard diagram for \(\tildecurves{M, \spin}\) is a graph of the form \(y = f(x)\).  
\end{lemma}

\begin{proof}
Consider the vertical line \(L_c\) given by the equation \(x = c\), where \(c\) is chosen so that \(L_c\) does not pass through any pegs. Since both \(L_c\) and  \(\tildecurves{M,\spin}\)  are pulled tight, they are in minimal position. Since \(M(\alpha)\) is an L-space, \(L_c \cap \tildecurves{M,\spin}\) contains a single point.
It follows that \(\tildecurves{M,\spin}\) is a graph, except possibly for some vertical segments joining lattice points. If such a segment exists, then \(\alpha \in S^{sing}(M)\), which contradicts \(\alpha \in \sL^\circ_M\). 
\end{proof}

%Clearly $$ S^{sing}(M) = \bigcup_{\spin \in Spin^c(M)} S^{sing}(M, \spin).$$ 

Hence  \(\tildecurves{M,\spin}\) is an embedded curve which divides  the plane into two connected components. One of these components contains all points \(h \in H_1(M, \R)\) with \( h \cdot l \ll 0 \) and the other contains all points with \( h \cdot l \gg 0 \). We call points in the first component {\em black}, and those in the second component {\em white}. Equivalently, if we identify pegs with relative spin\(^c\) structures,  black pegs have \(n_{\spinhat} = 0\), while white pegs have \(n_{\spinhat} = 1\). 

If \(\mathbf{p}\) and \(\mathbf{q}\) are two distinct pegs, let \([\mathbf{p}-\mathbf{q}] \in \hat\R \) be the slope of the line joining them. We define
$$ X_\spin = \{ [\mathbf{p}- \mathbf{q}] \, | \, \mathbf{p} \ \text{is black,} \  \mathbf{q} \ \text{is white and} \ l \cdot (\mathbf{p} - \mathbf{q})  \geq 0\} $$
to be the set of slopes of lines joining a  black peg to a white peg which is no higher than it is.

\begin{proposition}
	\label{Prop:S=X}
	Suppose \(M\) is Floer simple and not solid-torus-like, and let \(\alpha \in \sL(M)\). Then \(S^{sing}(M,\spin)\) is the smallest interval in \(\hat \R \setminus \{ \alpha\}\) which contains the set \(X_\spin\). 
\end{proposition}

\begin{proof} The set \(S^{sing}(M, \spin)\) is an interval which does not contain \(\alpha\). We first show that \(X_s \subset S^{sing}(M, \spin)\). Suppose that \(\mathbf{p}\) is a black peg, \(\mathbf{q}\) is a white peg, and that \( l \cdot (\mathbf{p} - \mathbf{q})  > 0\). Let \(\gamma\) be a curve representing \(\tildecurves{M, \spin}\), and consider the ray from \(\mathbf{p}\) to \(\mathbf{q}\). Since \(\mathbf{p}\) is black and \(\mathbf{q}\) is white, there must be some point \(x\) on the segment from \(\mathbf{p}\) to \(\mathbf{q}\) which lies on \(\gamma\). The ray  from \(\mathbf{p}\) to \(\mathbf{q}\) points down, so it must eventually reenter the black region. Thus there is some other point \(y\) on the ray past \(\mathbf{q}\) which lies on \(\gamma\). Applying the (extended) mean value theorem to \(x\) and \(y\), we see that \( [\mathbf{p}- \mathbf{q}] \in S(\gamma)\). It follows that \([\mathbf{p}- \mathbf{q}] \in S^{sing}(M, \spin)\). 
	
	To show that \(  S^{sing}(M, \spin)\) is the smallest interval containing \(X_\spin\), it suffices to show that the endpoints of \(S^{sing}(M, \spin)\) lie in \(X_\spin\). If \(M\) is not solid-torus-like, then  \(S^{sing}(M, \spin)\) is a union of intervals whose endpoints are slopes of the pegboard diagram for \(\tildecurves{M, \spin}\). Thus its endpoints are slopes of the pegboard diagram. 
	
	Under our identification of \(H_1(\partial M)\) with \(\R^2\), the slope \(\alpha\) corresponds to a vertical line, which has infinite slope.  
	Thus the endpoints of \(S^{sing}(M,\spin)\) will be the maximum and minimum values of \(f'(x)\).  
	
	At each corner of the graph, either the curve is concave up (\(f''(x)\geq0)\), and the curve lies just below a white peg, or the curve is concave down \((f''(x)\leq0)\) and the curve lies just above a black peg. Clearly the maximal value of the slope \(f'(x)\) is attained on an interval where we transition from having \(f''(x)\geq 0\) to having \(f''(x) \leq 0\). The left endpoint of the corresponding segment lies below a white peg, while the right endpoint is above a black peg. Thus the  slope is an element of \(X_{\spin}\). Similarly, the minimal value of the slope must occur on a segment where the left endpoint lies above a black peg, and the right endpoint lies above a white one. This slope is also in \(X_\spin\).  
\end{proof}

In  \cite{RR}, \(\sL(M)\) was characterized in terms of the set 
$$ D^\tau(M) = \{ \spinbar_0 - \spinbar_1 \, | \, \spinbar_0, \spinbar_1 \in Spin^c(M, \partial M),  \tau(M, \spinbar_i)=i, l \cdot  (\spinbar_0 - \spinbar_1)  \geq 0 \} \subset H_1(M) $$ 
 Let \(j_*:H_1(\partial M) \to H_1(M)\) be the inclusion, and denote by  \([j_*^{-1}(D^\tau(M))] \subset Sl(\partial M)\) the projectivization of the set \(j_*^{-1}(D^\tau(M))\). 

\begin{lemma}
	\label{Lem:X=Dtau}
	$[ \displaystyle j_*^{-1}(D^\tau(M)) ]= \bigcup_{\spin \in Spin^c(M)} X_\spin. $
\end{lemma}

\begin{proof} If \(\spinbar_0, \spinbar_1 \in Spin^c(M, \partial M)\), then \(\spinbar_0 - \spinbar_1 \in \im j_*\) if and only if \(\spinbar_0\) and \( \spinbar_1\) induce the same \(Spin^c\) structure \(\spin \in Spin^c(M)\). If this is the case, then \(j_*^{-1} (\spinbar_0 - \spinbar_1)\) is the set of differences of the form \(\mathbf{p}_0- \mathbf{p}_1\), where \(\mathbf{p}_i\) is a lattice point in \(H_1(\partial M, \R)\) whose image in \(\overline{T} =H_1(M, \R)/\ker j_*\) is \( z_{\spinbar_i}\). The condition that \(\tau(M, \spinbar_i) = i\) is equivalent to saying that \(n_{\spinbar_i} = i\) in other words, that \(\mathbf{p}_0\) is black and \(\mathbf{p}_1\) is white. Finally, the condition that \(  l \cdot (\spinbar_0 - \spinbar_1)  \geq 0\) is equivalent to saying that \( l \cdot (\mathbf{p}_0 - \mathbf{p}_1) \geq 0\). 
\end{proof}

Combining Proposition~\ref{Prop:S=X} with Lemma~\ref{Lem:X=Dtau}, we arrive at 
\begin{theorem} (Theorem 1 of \cite{RR})
	Suppose \(M\) is Floer simple and not solid torus like, and that \(\alpha \in \sL^\circ(M)\). Then \(\sL^\circ(M)\) is the largest  interval of \(Sl(\partial M)\) which contains \(\alpha\) and does not contain any element of \([j^{-1}_*(D^\tau(M))]\). 
\end{theorem}

%% file: sections/seiberg-witten.tex
% !TEX root = ../companion.tex
%seiberg-witten.tex

It is interesting to compare the invariant \(\barcurves{M}\) with the moduli space of finite energy solutions to the Seiberg-Witten equations on \(M\). In this section, we briefly sketch the way this analogy should work,  relying  mainly on the work of Morgan, Mrowka, and Szab{\'o} \cite{MMS1997} and Mrowka, Ozsv{\'a}th and Yu \cite{MOY1997}. 

\subsection{The Seiberg-Witten equations} If \(M\) is a manifold with torus boundary, we let \(M'= M \cup_{\partial M} \partial M \times [0,\infty)\). We fix a Riemannian metric \(g\) on \(M'\) which has the form \(g_{E} + dt^2\) on \(\partial M \times [0, \infty)\), where \(g_E\) is a flat metric on \(\partial M \cong T^2\). 

Next, we  choose \(\spin \in \spinc(M')\), and let \(E_\spin\) be  the principal \(\spinc(3)\) bundle over \(M'\) associated to \(\spin\). A connection \(A\)  on \(E_\spin\) induces a connection \(\widehat{A}\) on the determinant line bundle 
\(\det(\spin)\), as well as on the associated \(SO(3)\) bundle, which is the frame bundle of \(M'\). We consider the  space \(\mathcal{A}\) of connections on \(E_\spin\) which induce a fixed connection \(A_{\mathfrak{so}_3}\) on the frame bundle. (Usually \(A_{\mathfrak{so}_3}\) will be the Levi-Civita connection induced by \(g\).)   \(\mathcal{A}\) is an affine space modeled on \(\Omega^1(M'; i\R)\). Finally, we let \(W\) be the spinor bundle associated to \(E_\spin\). 

The Seiberg-Witten equations on \(M'\) are equations for a pair \((A, \psi) \in \mathcal{C} = \mathcal{A} \times \Gamma(W)\). They have the form
\begin{align*}
& \slashed{\partial} _A \psi  = 0 \\
& F_{\widehat{A}} = q(\psi)
\end{align*}
where \(q(\psi)\) is a certain quadratic function of the spinor. The gauge group \(\mathcal{G} = \operatorname{Map}(M', S^1)\) acts on  \(\mathcal{C}\) by \(\gamma(A, \psi) = (A-\gamma^{-1}d\gamma, \gamma \cdot \psi)\); the equations are invariant under this action.

\subsubsection*{The limit map}
The {energy} of a Seiberg-Witten solution \((A, \psi)\) on \(M'\)  is given by 
$$E(A, \psi) = \frac{1}{4} \int_{M'} \left(\Vert F_{\widehat{A}}\Vert^2 +4 \Vert \nabla_A \phi \Vert^2 + \Vert \phi \Vert^4 + s \Vert \phi \Vert^2  \right)$$
where \(s\) is the scalar curvature of \(M'\). 
We let \(\mathcal{M}(M, \spin)\) denote the quotient of the set of finite energy Seiberg-Witten solutions on \(M'\)  by the action of \(\mathcal{G}\). 

Let \(\mathcal{M}(\partial M, \spin|_{\partial M})\) be the set of translation invariant solutions to the Seiberg-Witten equations on \(\partial M \times \R\) modulo the action of the group \(\operatorname{Map}(\partial M, S^1)\) of translation invariant gauge transformations on \(M \times \R\). A  Seiberg-Witten solution on \(M'\) can be put in temporal gauge on the cylindrical end. Once this is done, the finite energy condition ensures that \((A, \psi)|_{\partial M \times [T, \infty)}\) limits to an element of  \(\mathcal{M}(\partial M, \spin|_{\partial M})\) as \(T \to \infty\). We thus obtain a map 
$$ j\co \mathcal{M}(M, \spin) \to \mathcal{M}(\partial M, \spin|_{\partial M})$$ 
which may be refined as follows. Let \(\overline{\mathcal{G}}_{\partial M}\subset \operatorname{Map}(\partial M, S^1)\)  be the subgroup of maps which extend to \(M\), and let \(\overline{\mathcal{M}}(\partial M, \spin|_{\partial M})\) be the quotient of the set of translation invariant solutions by \(\overline{\mathcal{G}}_{\partial M}\). Then there is a covering map \(\overline{\mathcal{M}}(\partial M, \spin|_{\partial M})
\to \mathcal{M}(\partial M, \spin|_{\partial M}) \) and a well-defined map
$$ \overline{j}\co  \mathcal{M}(M, \spin) \to \overline{\mathcal{M}}(\partial M, \spin|_{\partial M})$$
which is a lift of \(j\) to  \(\overline{\mathcal{M}}(\partial M, \spin|_{\partial M}).\)

\subsection{Structure of  \({\mathcal{M}(\partial M, \spin|_{\partial M})}\)}
\label{sub:structure} So far, everything we have said applies to an arbitrary manifold with a cylindrical end. We now use the fact that \(\partial M \cong T^2\). 
Since the Riemannian metric \(g_E\) on \(\partial M\) has non-negative scalar curvature (in fact, it is flat),  all Seiberg-Witten solutions on \(\partial M \times \R\) are {\em reducible}; that is they have \(\psi \equiv 0 \). It follows that 
  \(\mathcal{M}(\partial M, \spin|_{\partial M}) = \emptyset\)  unless \(c_1(\spin) = 0 \). Let \(\spin_0\) be the unique spin\(^c\) structure on \(M\) with  \(c_1(\spin_0)=0\).

  Choose a connection \(A_0\) on \(E_\spin(\partial M \times \R)\) such that \(F_{\widehat{A_0}} = 0\). Then \(F_{\widehat{A_0+a}} = 2da\), so \((A_0+a, 0)\) is a reducible solution to the Seiberg-Witten equations if and only \(a\) is closed. Denote the identity component of  \(\mathcal{G}_{\partial M}\) by  \(\widetilde{\mathcal{G}}_{\partial M}  = \{ e^{if} \, | \, f\co \partial M \times \R \to \R\}\). We have \(e^{if} \cdot (A_0 +a,0) = (A_o+a - i df,0)\), so  the quotient of the  space of Seiberg-Witten solutions on \(\partial M \times \R\) by \( \widetilde{\mathcal{G}}_{\partial M}\) is naturally identified with \(H^1(\partial M, \R)\). The quotient \(\mathcal{G}_{\partial M}/\widetilde{\mathcal{G}}_{\partial M} = [\partial M, S^1]= H^1(\partial M, \Z)\) acts on this space in the obvious way, so    \(\mathcal{M}(\partial M, \spin_0) =
 H^1(\partial M, \R)/H^1(\partial M, \Z)\). By Poincar\'e duality, this space can be identified with the torus \(H_1(\partial M, \R)/H_1(\partial M, \Z)\).
  
 The quotient  \(\overline{\mathcal{G}}_{\partial M}/\widetilde{\mathcal{G}}_{\partial M}\) consists of those elements of \(H^1(\partial M,\Z)\) which pull back from \(H^1(M, \Z)\).  Thus  \(\overline{\mathcal{M}}(\partial M, \spin_0) =
 H^1(\partial M, \R)/j^*(H^1(M, \Z))\). By Poincar\'e duality, this can be identified with \( H_1(\partial M, \R)/\ker j_*\). 

An important feature of \(\mathcal{M}(\partial M, \spin_0)\) is that it contains a unique point \(z=(A_0,0)\) for which \(\ker \slashed {\partial} _{A_0}\) is nontrivial. To understand this fact, we recall the structure of the Dirac operator on a Riemann surface \(\Sigma\) equipped with a \(\spinc\) structure \(\spin\). The spinor bundle \(W\) on \(\Sigma\) splits as \(W^+ \otimes W^-\) where \(W^\pm\) are complex line bundles. A connection \(A\) on \(E_\spin\) induces connections \(A^\pm\) on \(W^\pm\). Since \(\Sigma\) is a Riemann surface, the curvature 
\(F_{A^\pm}\) is automatically of type \((1,1)\), so the connections \(A^{\pm}\) induce holomorphic structures on \(W^{\pm}\). 
As holomorphic line bundles,  \(W^- = W^+ \otimes K_\Sigma^{-1}\), where \(K_\Sigma\) is the canonical bundle, and the Dirac operator \( \slashed{\partial} _A\co \Gamma(W^+) \to \Gamma(W^-)\) is given by \(\sqrt{2} \overline{\partial}_A\).  Finally, we have \(\det(\spin) = W^+ \otimes W^- =  (W^+)^2 \otimes K_\Sigma^{-1}\). 

When \(\Sigma = \partial M\) is a torus and \(\spin = \spin_0\), \(K_\Sigma\) is the trivial bundle and \(c_1(W^+) = 0\). The moduli space \(\mathcal{M}(\partial M, \spin_0)\) can be identified with \(\text{Pic}^0(\Sigma)\) via the map which assigns to \((A,0)\) the line bundle \(W^+\) with the holomorphic structure induced by \(A\). 
Then the  Dirac operator \( \slashed{\partial} _A\) has nontrivial kernel precisely when \(W^+\) has a holomorphic section. There is a unique element of \(\text{Pic}^0(\Sigma)\) with a holomorphic section; namely, the trivial bundle. Let \(A_0\) be the corresponding flat connection.

\subsubsection*{Conjugation symmetry} A \(\spinc\) structure \(\spin\) on \(M\) has a {conjugate} \(\spinc\) structure \(c(\spin)\) whose transition functions are conjugate to the transition functions for \(\spin\). We have \(W_{c(\spin)}^\pm =(W_\spin^\mp)^*\).  A connection \(A\) on \(\spin\) induces a connection \(\overline{A}\) on \(c(\spin)\), and a spinor \(\psi\) for \(\spin\) induces a spinor \(\overline{\psi}\) for \(c(\spin)\). The map
\(c\) defined by \( (A,\psi) \mapsto (\overline{A}, \overline{\psi})\) identifies \(\mathcal{M}(M, \spin)\) with 
 \(\mathcal{M}(M, c(\spin))\). 
 
On \(\partial M\),  \( c(\spin)_0 = \spin_0\), so \(c\) induces an involution of \(\mathcal{M}(\partial M, \spin_0)\). Under the identification  \(\mathcal{M}(\partial M, \spin_0) =  \text{Pic}^0(\partial M)\), we have \(c(L) = L^*\). The four fixed points of \(c\) correspond to the four spin structures on \(\partial M\). The special point \(z\) is one of these points; to specify which one, we recall the following description of spin structures on \(S^1\). Let \(V\) be a nonvanishing section of \(TS^1\). The preimage of \(V\) in the spin bundle is a double cover of \(V\); if it is a trivial double cover, we say that the spin structure is the trivial spin structure on \(S^1\), and if it is nontrivial, we say that the spin structure is nontrivial. If we write \(\partial M = S^1 \times S^1\), then  the spin structure corresponding to \(z\) is the product of the trivial spin structure with itself.

\subsubsection*{Reducible solutions}
 If \(\partial M = T^2\), then \(j^*\co H^2(M) \to H^2(\partial M)\) is the trivial map, so any \(Spin^c\) structure \(\spin\)  on \(M\) restricts to \(\spin_0\). 
Elements of \(\mathcal{M}(M, \mu, \spin) \) may be divided into {\it
  reducibles} (solutions with \(\psi \equiv 0\)) and {\it
  irreducibles} (all the rest). 
We let  \(\mathcal{M}^{red}(M, \spin)\) be the space of reducible
solutions, and similarly for \(\mathcal{M}^{irred}\).  
Arguing as we did for \(T^2\), it is easy to see that 
\(\mathcal{M}^{red}(M, \spin) =  H^1(M;\R)/H^1(M;\Z)\) if  \(c_1(\spin)\)
is torsion, and is empty if it is not. 

To describe the image of \(\mathcal{M}^{red}(M,\spin)\) under \(j\), we fix a basis \((m,l)\) for \(H_1(\partial M)\), where \(l\) is a rational homological longitude and \(m \cdot l = 1\). We identify \(\mathcal{M}(\partial M, \spin_0)\) with  \(S^1 \times S^1\) by the map which sends \((A,0)\) to \(({\hol}_m A^+, {\hol}_l A^+)\). Suppose that the order of \(l\) in \(H_1(M)\) is \(n\). Then if \(S\) is a surface in \(M\) which bounds \(n l\), 
$$ 2n \, {\hol}_l A^+ = n \,{\hol}_l \widehat{A} = \int_S F_{\widehat{A}} = 0 \in \R/(2\pi \Z)$$
so \(j(\mathcal{M}^{red}(M,\spin))\) lies on a line of the form \(\text{hol}_l A^+ = k\pi/n\) for \(k \in \Z/(2p)\). 

To pin down the value of \(k\), we fix a spin structure \(\mathfrak{t}\) on \(M\), and let \(\mathfrak{s}\) be the associated \(\spinc\) structure, so that \(c_1(\mathfrak{s})=0\). The restriction of \(\mathfrak{t}\) to 
\(\partial M\) determines a 2-torsion point \(p_{\mathfrak{t}} \in \text{Pic}^0(\partial M)\), and 
\(j(\mathcal{M}^{red}(M, \mathfrak{s}))\) is the horizontal line in \(\text{Pic}^0(\partial M)\) passing through \(p_{\mathfrak{t}}\). More generally, for \(x \in H_1(M)\), let \(l\cdot x = S \cdot x/p\), which is a well-defined element of \(\R/\Z\). Then it is not hard to show that  
\(j(\mathcal{M}^{red}(M, \mathfrak{s}+x))\) is the horizontal line given by the equation 
\(\hol_l A^+ = \hol_l p_{\mathfrak{t}} + 2 \pi l \cdot x\).

\subsection{Floer solid tori, revisited}
The solid torus \(M=S^1 \times D^2\) admits a metric of positive scalar curvature, so the Seiberg-Witten equations have only reducible solutions.  Thus \(\mathcal{M}(M) = \mathcal{M}^{red}(M) \simeq S^1 \). 
 Its image under \(j\) passes through the two points on  \( \mathcal{M}(\partial M, \spin_0)\) corresponding to spin structures on \(\partial M\) which extend over \(S^1 \times D^2\). The spin structure on \(S^1\) which extends to \(D^2\) is the one corresponding to the non-trivial double cover of \(S^1\), so \(j(\mathcal{M}(S^1\times D^2))\) is disjoint from \(z\).  \(\overline{j}(\mathcal{M}(S^1\times D^2)) = S^1 \times 0 \subset S^1 \times \R\).  It lies midway between two preimages of \(z\), and coincides with \(\barcurves{S^1\times D^2}\) (up to homotopy). 
 
   If \(M\) is the twisted \(I\)--bundle over the Klein bottle, \(M\) admits a metric of nonnegative scalar curvature, so we again have \(\mathcal{M}(M) = \mathcal{M}^{red}(M).\)  \(H^2(M) \simeq \Z \oplus \Z/2\), so there are two torsion \(\spinc\) structures \(\spin, \spin'\) on \(M\), both of which are induced by spin structures. Their images under \(j\) are two parallel horizontal lines, each passing through 2 fixed points of \(c\). 
  
  The kernel of the map \(H_1(\partial M) \to H_1(M)\) is a subgroup of the form \(2\Z \oplus 0 \subset \Z \oplus \Z\), so the cover \(\barT \simeq S^1 \times \R\), where each circle of the form \(S^1 \times n\) contains {\em two} preimages of \(z\).  \(\overline{j}(\mathcal{M}(M, \spin))\)  has the same form as \(\barcurves{M,\spin}\), as shown in Figure \ref{Fig:SW_first}.  \(\overline{j}(\mathcal{M}(M, \spin'))\) is  more interesting; it passes directly through two lifts of \(z\). To understand what is going on, note that 
  since the metric on \(M\) is flat rather than positively curved,
    reducible solutions need not be transversely cut out. Indeed, 
  as we shall see below, the two reducible solutions passing through lifts of \(z\) are not transversely cut out. When we perturb to achieve transversality, we expect that the resulting curve will resemble \(\barcurves{M,\spin'}\) as shown in Figure \ref{Fig:SW_first}.

\begin{figure}[t]
\includegraphics[scale=0.8]{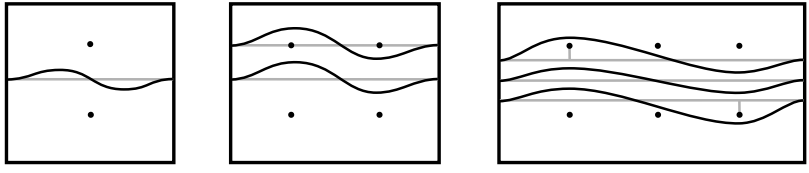}
\caption{Moduli spaces of solutions for the solid torus, the twisted $I$-bundle over the Klein bottle, and the Seifert fibered Floer solid torus with base orbifold $D^2(3,3)$. In each case, the cover $\barT$ is illustrated so that the left and right sides of each regtangle are identified (as in the examples of Figure \ref{fig:first-4-knots}). Here, and below, the moduli spaces of solutions is indicated in grey while a representative of the invariant $\liftcurves{M}$ is illustrated in black.}\label{Fig:SW_first}
\end{figure}

\subsection{Structure of  \({\mathcal{M}(M,  \spin)}\)}
The structure of \(\mathcal{M}(M, \spin) \) was described by Morgan, Mrowka, and Szab{\'o} \cite{MMS1997}. (In fact, \cite{MMS1997} studies solutions to the four-dimensional solutions to the Seiberg-Witten equations on a manifold with an end of the form \(T^3 \times [0, \infty)\), but their results can be made  to apply in the 3-dimensional case by considering solutions on \(M' \times S^1\).) It  follows from their work that 
\(\mathcal{M}(M, \spin) \) is compact. Moreover, they showed that the irreducible part of the moduli space has the following local structure.

The fact that \(\ker \partial_{A_0}\) is nontrivial implies that the moduli space 
 \( \mathcal{M}(\partial M, \spin_0)\) is not transversely cut out at the point \(z= (A_0,0)\). This has important consequences for the structure of \(\mathcal{M}(M,\spin)\). If \(j(A,\psi) = z' \neq z\), then it can be shown that the solution \((A, \psi)\) decays exponentially to \(z\) as we go down the cylindrical end. In turn, this can be used to prove that for generic \(\mu\), the moduli space \(\mathcal{M}(M, \mu, \spin) \)
 is a 1-dimensional manifold near \((A, \psi)\).
 
 In contrast, if \(j(A,\psi) = z\), the solution decays more slowly as we go along the tube. A more delicate analysis using the center manifold technique shows  that  for generic \(\mu\),  \(\mathcal{M}(M, \mu, \spin) \) is locally homeomorphic to \([0,1)\), where the point corresponding to \(0\) maps to \(z\) under \(j\). 
 
In summary,  \(\mathcal{M}(M, \mu, \spin) \) can be written as the union of \(\mathcal{M}^{red}\) and \(\mathcal{M}^{irred}\). The unperturbed moduli space \(M^{red}\) is homeomorphic to the torus \(H^1(M;\R)/H^1(M;\Z)\), while 
\(M^{irred}\) has the structure of a (possibly noncompact) \(1\)-manifold with boundary. Boundary points of \(\mathcal{M}^{irred}\) map to \(z\) under \(j\), while noncompact ends of \(\mathcal{M}^{irred}\) limit to \(\mathcal{M}^{red}\). 

\subsection{Seifert fibred spaces}
\label{subsec:SFSW}  The  Seiberg-Witten equations on closed Seifert fibred spaces were studied by Mrowka, Ozsv{\'a}th and Yu \cite{MOY1997}. Their results can be extended to Seifert fibred spaces with boundary with little change. We  sketch this process here. First, we equip \(TM'\) with a connection \(A_{\mathfrak{so}_3}\)  compatible with the \(S^1\) action on \(M'\), but which is not induced by a metric. Then it can be shown (as in \cite[Theorem 4]{MOY1997}) that any irreducible finite energy solution to the Seiberg-Witten equations on \(M'\) is invariant under the \(S^1\) action on \(M'\) induced by the Seifert fibration. 

Next, \(S^1\)--invariant solutions to the Seiberg-Witten equations are shown to correspond to finite energy solutions of the vortex equations on the base orbifold \(\Sigma\) of \(M\). These equations have the following form \cite[equations (25)-(27)]{MOY1997}
\begin{align*}
2 F_A - F_K & = i (|\alpha|^2 - |\beta|^2) \omega_\Sigma \\
\overline{\partial}_A \alpha = 0 \quad & \text{and} \quad \overline{\partial}_A^* \beta = 0 \\
\alpha = 0 \quad & \text {or} \quad \beta = 0 
\end{align*}
Here \((\alpha, \beta)\) is a section of the spinor bundle \(W\) for \(\Sigma\). As in Section \ref{sub:structure},  \(W\) can be decomposed as 
\(W = W^+ \oplus W^-\), where \(W^- = W^+ \otimes K_\Sigma^{-1}\), where \(K_\Sigma\) is the canonical bundle of \(\Sigma\), endowed with the metric connection. \(A\) is a connection on \(W^+\); it induces a Hermitian metric on \(W^+\), which we use to define both \(|\alpha|^2\) and \(\overline{\partial}_A\). Finally, \(\omega_\Sigma\) is the area form on \(\Sigma\). 

Let \(\mathcal{M}^{irred}_\alpha(M)\) be the moduli space of solutions to these equations for which \(\alpha\neq  0\), and similarly for \(\beta\). Conjugation symmetry exchanges  \(\mathcal{M}^{irred}_\alpha(M)\) and 
 \(\mathcal{M}^{irred}_\beta(M)\), so it is enough to understand  \(\mathcal{M}^{irred}_\alpha(M)\).

As we go down the tubular end, solutions to the vortex equations limit to flat \(S^1\) connections on the boundary \(S^1\). The space of such solutions modulo gauge is naturally identified with \(S^1\). If we only quotient by those gauge transformations which extend over \(\Sigma\), the resulting moduli space can be identified with \(\R\). As in the 3-dimensional case, we have a map \(j:\mathcal{M}^{irred}_\alpha(M) \to \R\) given by \(j(A, \alpha) = h_A := \frac{i}{2\pi} \int_\Sigma F_A\). 

Let \(\mathcal{D}(\Sigma)\) denote the set of effective orbifold divisors on \(\Sigma\).
In analogy with the results of \cite{MOY1997}, the moduli space \(j^{-1}(h) \cap \mathcal{M}^{irred}_\alpha(M)\) can be identified with the set
$$ \mathcal{D}_h= \{ D \in \mathcal{D}(\Sigma) \, | \, |D|\leq h   \}$$
when \(h <-\chi(\Sigma)/2\), and is empty for \(h \geq - \chi(\Sigma)/2\). The correspondence between the two is established as follows. 
Suppose \(j(A, \alpha) = h\). Then \(\alpha\) is a holomorphic section of \(W^+\) (with holomorphic structure induced by \(A\)), so  it determines a effective divisor \(D \in \mathcal{D}\). We have \(h_A \geq |D|.\) By integrating the first vortex equation, we see that $ 2h_{A}+\chi(\Sigma) < 0. $ It follows that \(|D|\leq h_A \leq - \chi(\Sigma)/2\), so the condition above is certainly necessary. The converse follows from the fact that it is possible to solve the Kazdan-Warner equation on open surfaces, as established in \cite{HT1992}. 

When \(\chi(\Sigma)< -2\), the divisor \(D\) can vary freely in \(\Sigma\), and the spaces \(\mathcal{D}_h\) will be noncompact manifolds of positive dimension. In contrast, if \(\chi(\Sigma)> -2\), \(D\) must be supported at the orbiford points of \(\Sigma\), and the moduli spaces \(\mathcal{D}_h\) will be discrete. 
If \(M\) is Seifert-fibred over \(D^2\) with two or three exceptional fibres, the latter condition holds, so  \(\mathcal{M}^{irred}_\alpha(M)\) will consist of one arc \(X_D\) for each effective orbifold divisor \(D\) with \(|D| < -\chi(\Sigma)/2\).   Each arc starts at a point of \(\mathcal{M}^{red}\) (where \(h_A = \chi(\Sigma)/2\)). Its other endpoint (where \(h_A = |D|\)) maps to \(z\) under \(j\).  The moduli space \(\mathcal{M}^{irred}_\beta(M)\) is isomorphic to \(\mathcal{M}^{irred}_\alpha(M)\); the two are exchanged by the conjugation symmetry.

To determine the image of \(\mathcal{M}^{irred}_{\alpha}(M)\) under \(\overline{j}\) recall that the vertical coordinate of \(\overline{j}(A, \psi)\) is given by \(h_{\widehat{A}} = \frac{i}{4\pi}\int_{S} F_{\widehat{A}}\), where \(\widehat{A}\) is the induced connection on \(\det \spin = W^+ \otimes W^-\), and \(S\) is a surface generating \(H_2(M, \partial M)\). Since solutions to the Seiberg-Witten equations are invariant under the \(S^1\) action, 
\(F_{\widehat{A}}\) pulls back from \(\Sigma\). If
  the projection \(\pi\co S \to \Sigma\) has degree \(d\), then
$$ h_{\widehat{A}} = \frac{i}{4\pi} \int_\Sigma F_{\widehat A} = d(h_A+ \chi(\Sigma)/2)).$$
The value of 
\(h_{\widehat{A}}\)  on \(X_D\) will vary between \(0\) and \(d(|D| + \chi(\Sigma)/2) = d|D| + \chi(S)/2\). Finally, let \(f \in H_1(\partial M)\) be the fibre slope. Since \(\widehat{A}\) pulls back from \(\Sigma\), it will have trivial holonomy along \(f\). It follows that \(\overline{j}(X_D)\) lies on a line parallel to \(f\).

%\subsubsection*{Further examples: Torus knots and other Seifert fibered spaces} 
Seifert fibered spaces provide a family of examples on which to compare the curves arising from this point of view with those defined in terms of bordered Floer homology. 
\begin{figure}[t]
\includegraphics[scale=0.8]{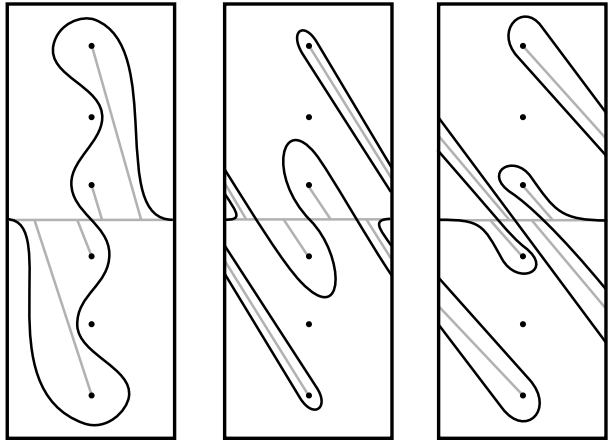}
\caption{Further examples: moduli spaces of solutions for the three Seifert fibered spaces with base orbifold $D^2(2,7)$, illustrated together with $\barcurves{M}$ in each case. In terms of the basis shown the fiber slope in each example, from left to right, is $-14$, $-\frac{14}{5}$ and $-\frac{14}{3}$. }\label{Fig:SW_2_7}
\end{figure}

 \parpic[r]{
 \begin{minipage}{55mm}
 \centering
 \includegraphics[scale=0.8]{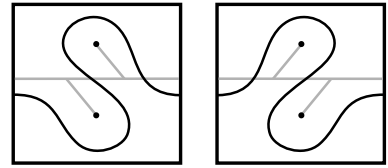}
\captionof{figure}{Curves for the left- and right-hand trefoil.}\label{Fig:trefoils-SW}
  \end{minipage}%
  }
The complement of the right-handed trefoil fibres over \(D^2\) with exceptional fibres of multiplicities \(2\) and \(3\), so \(\chi(\Sigma) = -1 + (1/2) + (1/3) = -1/6\). The only effective divisor with  \(|D|<1/12\) is the trivial divisor, so \(\mathcal{M}^{irred}_\alpha\) consists of a single arc, on which \(h_{\widehat{A}} \in [-1/2,0]\). To determine its image under \(\overline{j}\), note that \(f=l - 6m\), where \(m\) and \(\) are the standard meridian and longitude of the trefoil in \(S^3\). Thus the arc maps to a line with slope \(6\). The full moduli space is shown in Figure~\ref{Fig:trefoils-SW}.

To pass from \(\mathcal{M}(M)\) to \(\HFhat(M)\) in this example, we employ the following heuristic: first, we consider the moduli space \(\widehat{\mathcal{M}}(M)\) obtained by dividing out by the group of maps \(f\co S^1 \to M\) which satisfy \(f(p) = 1\) for some fixed point \(p \in M\). \(\widehat{\mathcal{M}}(M)\) will contain one point for each reducible point of \(\widehat{\mathcal{M}}(M)\), and an entire circle of points for each irreducible point. After an appropriate perturbation, this should reduce to a 1-dimensional space which contains roughly two points for each irreducible point of \(\widehat{\mathcal{M}}(M)\) . We expect that this moduli space should take the form of the curve shown in the figure, which is isotopic to \(\liftcurves{M}\). 

Now suppose that instead of \(T(2,3)\), we consider \(T(2,7)\). The complement fibres over \(D^2\) with exceptional fibres of multiplicities \(2\) and \(7\), so \(\chi(\Sigma) = -5/14\). Now there are two effective divisors with  \(|D|<5/28\); namely the trivial divisor and the divisor containing a single copy of the orbifold point of multiplicity \(7\).  \(\mathcal{M}^{irred}_\alpha\) consists of  two arcs,  on which the maximum values of \(\text{hol}_{\ell} A\) are \(14\cdot(5/28)=5/2\) and \(14 \cdot (5/28-1/7) = 1/2\). Both arcs map to lines of slope \(14\).

If instead of the complement of \(T(2,7)\), we considered another Seifert fibred space over \(D^2\) with exceptional fibres of multiplicities \(1/2\) and \(1/7\), the general form of the Seiberg-Witten moduli space would be similar, but the slope of the relevant arcs with respect to a standard basis \((m, l)\) for \(H_1(\partial M)\) would differ, as illustrated in Figure~\ref{Fig:SW_2_7}. (Note that the spaces in the figure are oriented so the fibre slopes are negative; {\it e.g.} the figure shows the moduli space for the complement of the left-hand \((2,7)\) torus knot.)
 
 \parpic[r]{
 \begin{minipage}{26mm}
 \centering
 \includegraphics[scale=0.8]{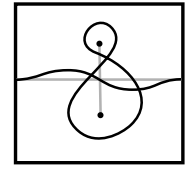}
\captionof{figure}{Curves for the figure eight.}\label{Fig:fig8-SW}
  \end{minipage}%
  }
As a final, non-Seifert fibered, example, let $M$ be the complement of the figure-8 knot in \(S^3\). Here, we cannot determine the Seiberg-Witten moduli space explicitly, but we know from the Alexander polynomial that the signed number of ends of irreducible arcs at \(z_{\pm 1/2}\) (the preimages of \(z\) closest to the reducible line) should be \(\pm 1\), respectively. We expect that with respect to an appropriate metric/deformation,  \(\mathcal{M}(M)\) should consist of a single arc of irreducibles joining \(z_{1/2}\) to \(z_{-1/2}\), together with the usual circle of reducibles. After passing to the unreduced moduli space and perturbing, the arc should become the figure-8 component of \(\HFhat(M)\). 

%% file: sections/Kh.tex
% !TEX root = ../companion.tex
%Kh.tex

If \(L\) is a link in \(S^3\), let \(\boldsymbol{\Sigma}_L\) be its two-fold branched cover. 
There is an established and, by now, well-explored relationship between the Khovanov homology of \(L\) and \(\HFhat(\boldsymbol{\Sigma}_L)\)  \cite{OSz2005}. This comes in the form of a spectral sequence, and has been recast in more algebraic terms (and calculated) using the machinery of bordered Floer homology in the work of Lipshitz, Ozsv\'ath, and Thurston \cite{LOT2014,LOT2016}. In the case of a 4-ended tangle $T$, whose two-fold branched cover $\boldsymbol{\Sigma}_T$ is a manifold with torus boundary, it is natural to ask about how the Khovanov homology of $T$ (in the sense of Bar-Natan \cite{Bar-Natan2005}, say) is captured by the invariant $\HFhat(\boldsymbol{\Sigma}_T)$. We collect the pieces and set this up.  

\subsection{Khovanov homology of tangles}

Let \(T \subset B^3 \) be a four-ended tangle. A parametrization \((\mathbf{a}, \mathbf{b})\) of \(T\) is a choice of arcs \(a_1,b_1,a_2,b_2 \subset S^2\) such that
\begin{itemize}
\item the endpoints of each arc lie on \(\partial T\);
\item each endpoint of \(T\) lies on one \(a_i\) and one \(b_i\); and
\item the union of all four arcs is an embedded circle. 
\end{itemize}

\labellist \small
\pinlabel {$a_1$} at 40 -7
\pinlabel {$a_2$} at 40 80
\pinlabel {$b_1$} at 83 36
\pinlabel {$b_2$} at -7 36
\endlabellist
\parpic[r]{
 \begin{minipage}{37mm}
 \centering
 \includegraphics[scale=0.7]{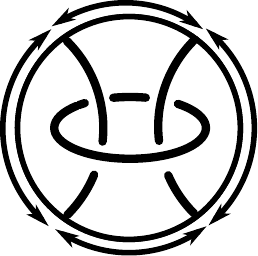}
 \captionof{figure}{A parametrized tangle}\label{fig:param}
  \end{minipage}%
  }
A planar diagram \(D\) comes with a preferred parametrization, as shown in Figure \ref{fig:param}. 
 Following Bar-Natan \cite{Bar-Natan2005}, we can view the Khovanov homology of the parametrized tangle \((T, \mathbf{a}, \mathbf{b})\) as a chain complex over a category generated by two objects  \(B_0\) and \(B_1\),
corresponding to crossingless planar diagrams with four ends. 
Morphisms between these objects are given by cobordisms, modulo certain relations.  Since we are using \(\F\) coefficients, we can work with the smaller  (undotted) version of the category  used by Bar-Natan: $\operatorname{End}(B_0\oplus B_1)$ is 6 dimensional, with 
\begin{align*}
\operatorname{Hom}(B_0,B_0) & = \langle\mathbf{1_0},\mathbf{t_0}\rangle & \qquad & 
\operatorname{Hom}(B_0,B_1) = \langle\mathbf{s_0}\rangle \\
\operatorname{Hom}(B_1,B_0) &= \langle\mathbf{s_1}\rangle & \qquad & 
\operatorname{Hom}(B_1,B_1)  = \langle\mathbf{1_1},\mathbf{t_1}\rangle
\end{align*}
where $\mathbf{s_i}$ denotes a {\bf s}addle cobordism and $\mathbf{t_i}$ denotes a trivial cobordism with a {\bf t}ube joining the two sheets. Note that the latter can be replaced by a sum of two dotted trivial cobordisms applying the neck-cutting relation. Any complex in Bar-Natan's category can be reduced to a {\em minimal} complex expressed in terms of $B_0$ and $B_1$ such that no component of the differential is an identity map; we will denote this minimal complex by $\KH(T,\mathbf{a}, \mathbf{b})$.

\labellist \small
\pinlabel {$0$} at 45 34
\pinlabel {$1$} at 115 34
\endlabellist
\parpic[r]{
 \begin{minipage}{50mm}
 \centering
 \includegraphics[scale=0.8]{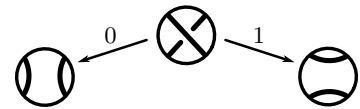}
\captionof{figure}{Resolution conventions.}\label{fig:res-KH}
  \end{minipage}%
  }
Our convention for crossing resolutions are shown in Figure \ref{fig:res-KH}. For consistency with Heegaard Floer homology \cite{OSz2005}, we have chosen the opposite of the standard convention.  As a result, the complex we work with is the dual of the complex considered in \cite{Bar-Natan2005}, or equivalently, is the complex associated with the mirror tangle. For simplicity we will ignore quantum gradings and work with unoriented tangles. (In particular, our complexes have only a relative homological grading; to fix an absolute homological grading, we would need to pick an orientation.) For example, the (parametrized) tangle $\raisebox{-3pt}{\includegraphics[scale=0.4]{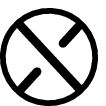}}$ is assigned the complex 
\labellist \tiny\pinlabel {$\mathbf{s}$} at 40 20 \endlabellist
$\raisebox{-3pt}{\includegraphics[scale=0.4]{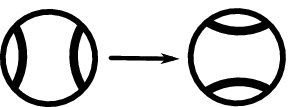}}$. 

\subsection{Filtered type D structures.} In order to calculate the Heegaard Floer homology of two-fold branched covers, Lipshitz, Ozsv\'ath, and Thurston take as input a branch set in bridge position, thought of as a knot diagram with a height function. Having made a choice of diagram for which each crossing is at a distinct height, they implement a divide-and-conquer strategy by assigning bordered objects to each crossing (as well as to the bridge caps on the top/maxima and bottom/minima). These pieces are box-tensored together to ultimately produce a filtered chain complex. Roughly speaking, the box-tensor complex inherits a filtration by enhancing the bordered objects being tensored with filtrations. 

If \((T,\mathbf{a},\mathbf{b})\) is a parametrized  4-ended tangle, its double branched cover \(\boldsymbol{\Sigma}_T\) is a manifold with torus boundary. The parametrization \((\mathbf{a},\mathbf{b})\) determines a parametrization \((\alpha, \beta)\) of \(\partial \boldsymbol{\Sigma}_T\); where \(\alpha\) is the double branched cover of either one of the \(a_i\)'s, and similarly for \(\beta\). 
 We would like to endow  $\CFD(\boldsymbol{\Sigma}_T,\alpha, \beta)$ with a filtration. Referring to \cite[Definition 2.2]{LOT2014} for the details, we are interested in extendable type D structures over the torus algebra  $\Alg$ with the properties that (1) the underlying vector space $V$ is equipped with an integer grading and (2) the differential $\partial\co\Alg\otimes V\to \Alg\otimes V$ (equivalently, the map $\delta$) does not decrease this grading. Very mild changes to  Lipshitz, Ozsv\'ath, and Thurston's construction  show that such an object  exists.  (We just compute the bordered invariants for appropriate partial closures rather than the full closures used in \cite{LOT2014}.) 
 
\labellist \small
\pinlabel {$12$} at  87 60
\pinlabel {$23$} at 87 17
\endlabellist
\parpic[r]{
 \begin{minipage}{40mm}
 \centering
 \includegraphics[scale=0.8]{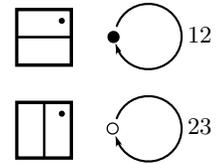}
\captionof{figure}{Basic curves and associated decorated graphs.}\label{fig:self-loops}
  \end{minipage}%
  }
The result of this construction is a filtered type D structure. Its associated graded will be a direct sum of copies of two basic objects, namely 

$\qquad\qquad\qquad\qquad\hor:= \CFD(\boldsymbol{\Sigma}_{B_1},\alpha,\beta)$

$\qquad\qquad\qquad\qquad\ver := \CFD(\boldsymbol{\Sigma}_{B_0},\alpha,\beta)$

 Geometrically, \(\hor\) and \(\ver\) are  closed circles parallel to the curves \(\alpha, \beta\) used to parametrize the torus. In terms of type D structures, \($\hor$\) is a black idempotent joined to itself by a single \(\rho_{12}\) arrow, while \(\ver\) is a white idempotent joined to itself by a single \(\rho_{23}\) arrow. In terms of decorated graphs, this dictionary is illustrated in Figure \ref{fig:self-loops}.

\labellist \small
\pinlabel {$=$} at  75 68 \pinlabel {$=$} at 122 68
\pinlabel {$=$} at 38 18 \pinlabel {$=$} at 155 18
\endlabellist
\parpic[r]{
 \begin{minipage}{60mm}
 \centering
 \includegraphics[scale=0.8]{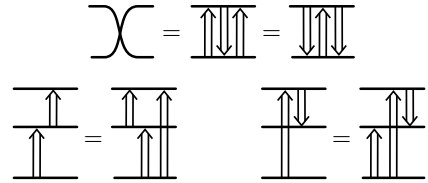}
\captionof{figure}{A quick review of some arrow-calculus; compare \cite[Figure 27]{HRW}.}\label{fig:cheat-sheet}
  \end{minipage}%
  }A theorem of  Lekili and Perutz \cite{LP2011} shows that the pair $\left\{\hor,\ver\right\}$ generates the compact part of the  Fukaya category of the punctured torus. Using  arrow calculus, it is straightforward to express any given curve-set  $\HFhat(M)$  as a twisted complex built up out of elements of this basis. For instance, here is an algorithm that one can implement in three steps. First, replace any nontrivial local systems by bundles of parallel curves joined by an appropriate collection of crossover arrows representing the monodromy. Next, pick outer-most curves covering opposite corners, and by adding a pair of clockwise arrows (covering the other two corners), add a new crossing between two curves. Since there are only finitely many pairs of such curves, one can repeat this process and in finitely many steps the new configuration will be (1) a collection of horizontal curves (possibly with a permutation between strands), (2) a collection of vertical strands (possibly with a permutation between strands), and (3) some crossover arrows between vertical and horizontal strands. Note that the latter may be interpreted as morphisms between the former. Next, the permutations (if present) can be simplified by replacing each crossing with a triples of arrows, and the result simplified if desired; see Figure \ref{fig:cheat-sheet} for a quick review of some non-trivial moves from the arrow-calculus introduced in \cite[Section 3.6]{HRW}. A simple example illustrating this process is given in Figure \ref{fig:express-basis-example}.  Notice that even in the case of an embedded curve, the choices involved can give rise to seemingly different outcomes, in the sense that the corresponding labeled graphs are not isomorphic, even though the corresponding complexes over \(\sA\) are isomorphic.
 
\begin{figure}[t]
\includegraphics[scale=0.8]{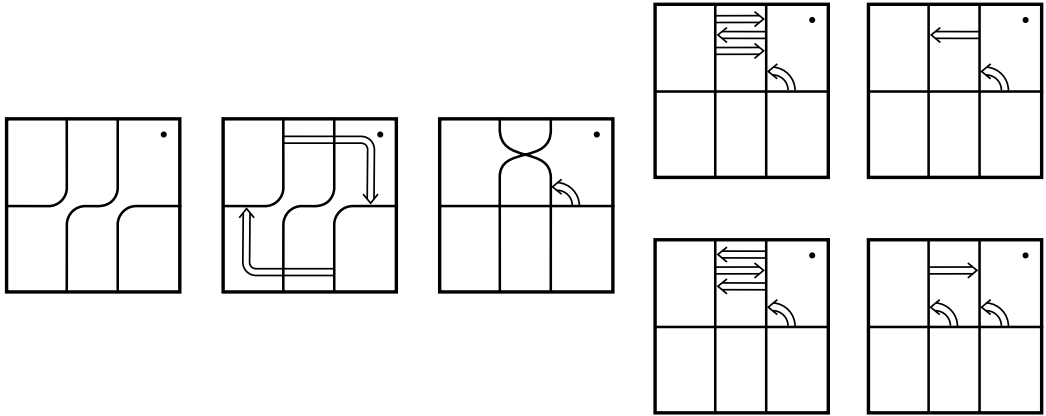}
\caption{An example illustrating the two main steps for expressing an immersed curve (far left) in terms of the basis (far right). Notice that choices are involved: in this simple example the second step removing the crossing can be achieved in two ways and we have shown both.}\label{fig:express-basis-example} 
\end{figure}

The category generated by \(\hor\) and \(\ver\) closely resembles  Bar-Natan's category. Indeed, we have

\begin{align*}
\operatorname{Hom}\Big(\ver,\ver\Big) & = \langle \mathbf{1},\rho_{23}\rangle & \qquad & 
\operatorname{Hom}\Big(\ver,\hor\Big) = \langle\rho_2 \rangle  \\
 \operatorname{Hom}\Big(\hor,\ver\Big) &= \langle \rho_1+\rho_3 \rangle & \qquad & 
 \operatorname{Hom}\Big(\hor,\hor\Big)  = \langle\mathbf{1},\rho_{12}\rangle
\end{align*}
The dictionary relating the two categories is shown in Figure~\ref{fig:BN-KH}. 
 Interestingly, the crossover arrow formalism turns out to be perfectly adapted to represent these maps. The mapping cone of each of the four non-identity maps above can be represented by a single crossover arrow, as shown. 
\labellist 
\pinlabel {$+$} at 333 112
\pinlabel {$+$} at 333 49
\endlabellist
\begin{figure}[ht]
\includegraphics[scale=0.8]{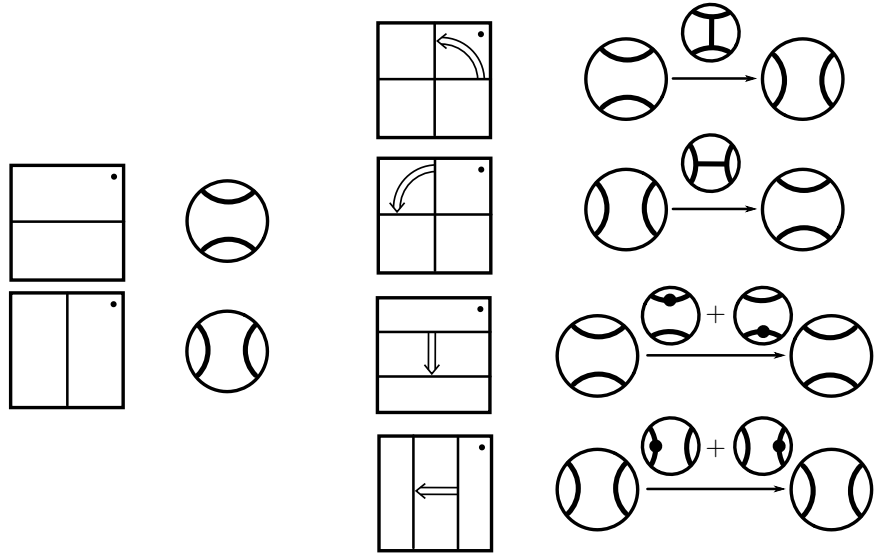}
\caption{Identifying objects (left) and morphisms (right) with Bar-Natan's category \cite{Bar-Natan2005}.} \label{fig:BN-KH}
\end{figure}  
  
\subsection{The associated graded} 
To sum up, the filtered type D structure associated to a four-ended tangle can be described graphically as follows. First, we have a collection of horizontal and vertical circles, together with an integer grading (the filtration grading) on each circle. This set of circles should be in bijection with the generators of the minimal chain complex \(\KH(T,\mathbf{a}, \mathbf{b})\), and the filtration grading should match the homological grading there. 
Second, there is  a collection of morphisms between the generators, represented by crossover arrows. These arrows should increase the filtration grading, and the set of crossover arrows which increase the filtration grading by one should be in bijection with the set of nonzero components of the differential in \(\KH(T,\mathbf{a}, \mathbf{b})\). Finally, the object of the Fukaya category represented by this type D structure should be \(\CFD(\boldsymbol{\Sigma}_T,\alpha, \beta)\). 

To make this process more precise, consider a filtered type D structure $N$ that is reduced and expressed in terms of the objects $\hor$ and $\ver$. We write $N_{\ge k}$ for the sub-complex of filtration grading at least $k$. We assume for simplicity (and, for consistency with all of the examples considered here) that there are no interesting local systems internal to the associated graded terms $N_k=N_{\ge k}/N_{> k}$. These restrictions give rise to a decomposition $\delta =\sum_{i=0}^n\alpha_i$ of the map $\delta\co N\to\Alg\otimes N$ associated with the type D structure $N$ where the $\alpha_i$ raise filtration grading by $i$. In particular, the $\alpha_i$ vanish for all sufficiently large $i$ (in practice, this will be determined by the number of crossings in a given tangle), and the $\alpha_0$ and $\alpha_1$ have {\em a priori} restrictions placed on them: For every generator $x\in N_k$, $\alpha_0(x)$ is $(\rho_{12}+\rho_{23})\otimes x$ (only one term in this sum will be non-zero). And, since $N$ is reduced, the terms arising in $\alpha_1$ can only be of the four types of morphisms shown in Figure \ref{fig:BN-KH}. 

The associated graded type D structure we would like to consider has map given by $\alpha_0+\alpha_1$, however this need not square to zero (in the appropriate sense) in general. Recall that compatibility requires that $\mu\otimes\id\circ\id\otimes\delta\circ\delta$ vanishes, so asking that $\big(\bigoplus_{k\in\Z} N_k, \alpha_0+\alpha_1\big)$ is a type D structure amounts to showing that $f = (\mu\otimes\id)\circ(\id\otimes(\alpha_0+\alpha_1))\circ(\alpha_0+\alpha_1)$ vanishes. By abuse of notation, write 
$f=\alpha_0\alpha_0+\alpha_0\alpha_1+\alpha_1\alpha_0+\alpha_1\alpha_1$%$f=\sum_{i,j<2} \alpha_i\alpha_j$ 
(so that the composite $(\mu\otimes\id)\circ(\id\otimes\delta)\circ\delta=\sum \alpha_i\alpha_j $). Now observe that, when $f(x)$ is non-zero there must be a non-zero $(\alpha_1\alpha_1)(x)$; because $\sum \alpha_i\alpha_j$ vanishes, this non-zero $(\alpha_1\alpha_1)(x)$ (formerly) cancelled with some $(\alpha_0\alpha_2)(x)$  or $(\alpha_2\alpha_0)(x)$. (Here we have made a crucial appeal to the fact that $N$ is reduced.) These are terms of the form $\rho_{12}\rho_{3}\otimes y$ or $\rho_{1}\rho_{23}\otimes y$, respectively, so that $(\alpha_1\alpha_1)(x)$ contains only summands of the form $\rho_{123}\otimes y$.  As a result, the associated graded may be viewed as a type D structure over the quotient $\Alg/(\rho_{123}=0)$, which experts will recognize as being closely related to Khovanov's algebra associated with a 2-tangle \cite{Khovanov2002}. 

\labellist \small
\pinlabel {$\mathbf{s}$} at 113 100
\pinlabel {$\mathbf{s}$} at 302.5 99.5
\endlabellist
\begin{figure}[ht]
\includegraphics[scale=0.8]{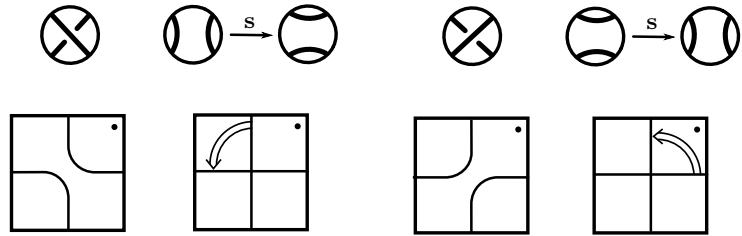}
\caption{The two simple complexes corresponding to each single-crossing tangle (above) and the corresponding curve, written in terms of basis elements in each case (below). In general, the labels $\mathbf{s}$ and $\mathbf{t}$ are determined by their source and target, so we will drop them in subsequent examples.} \label{fig:two-complexes-KH}
\end{figure}
As a simple example, consider the complexes associated to the one-crossing diagrams of Figure~\ref{fig:two-complexes-KH}. Here the entire complex is determined by the requirement that its associated graded is the Khovanov homology. An easy application of the graphical calculus shows these objects can be represented by simple closed curves of slope \(\pm 1\), as expected.

\labellist \small
%\pinlabel {$\rho_{12}$} at 145 75
%\pinlabel {$\rho_{12}$} at 232 75
%\pinlabel {$\rho_{23}$} at 318 73
%\pinlabel {$\rho_{12}$} at 189 35
%\pinlabel {$\rho_{1}+\rho_3$} at 272 35
%\pinlabel {$\rho_{3}$} at 232 8
\pinlabel {${12}$} at 145 74
\pinlabel {${12}$} at 231 74
\pinlabel {${23}$} at 316 73
\pinlabel {${12}$} at 189 33
\pinlabel {${1}$} at 288 35 \pinlabel {${3}$} at 252 11
\pinlabel {${1}$} at 170 3
\endlabellist
\begin{figure}[ht]
\includegraphics[scale=0.6]{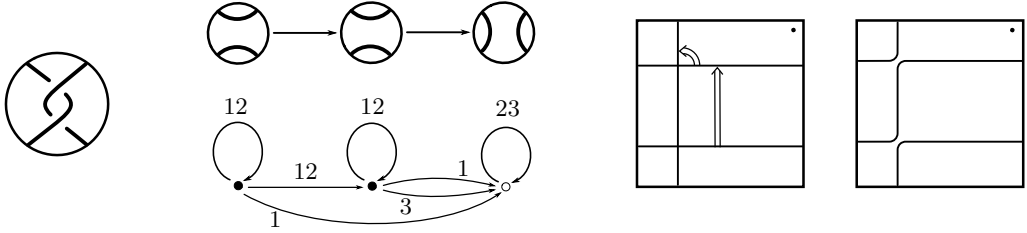}
\caption{Several views of the complexes associated to the tangle on the left. The upper figure shows the Bar-Natan complex, while the latter shows the corresponding type D structure. To the right are the graphical representation of this type D structure and its simplification. } \label{fig:not-formal}
\end{figure}

\parpic[r]{
 \begin{minipage}{38mm}
 \centering
 \includegraphics[scale=0.6]{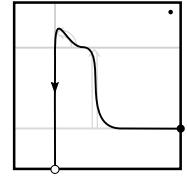}
\captionof{figure}{The train-track component representing the null-homotopy; compare Figure~\ref{fig:not-formal}.}\label{fig:nullhom}
 \end{minipage}%
 }An important distinction between Bar-Natan's category and the Fukaya category is that the Fukaya category is not formal \cite{LP2011}. This can be seen even in simple examples. 
 For example, let \(T\) be the two-crossing tangle shown in Figure~\ref{fig:not-formal}. The complex \(\KH(T,\mathbf{a}, \mathbf{b})\) is shown in the figure. If we naively try to translate this complex to a complex in the Fukaya category, the result will not have \(\partial^2=0\). This is because the composition of \(\rho_{12}\) and \(\rho_{1}+\rho_3\) is null-homotopic as an element of \(\Hom \left( \hor,\ver \right)\), but not identically \(0\). To form the true complex, we must add in the null-homotopy in the form of a component of the differential which shifts the filtration grading by 2 (that is, we add $\alpha_2=\rho_1$ in the notation introduced above). Again, the crossover arrow formalism is well-adapted to representing this phenomenon: the relevant complex is still represented by a diagram with just two crossover arrows, as shown in Figure~\ref{fig:not-formal}. (Interpreted as train tracks, this composition of arrows is shown in Figure \ref{fig:nullhom}.)

\subsection{Examples} 
We end by describing graphical complexes corresponding to some more interesting tangles. In each case, given $(T,\mathbf{a}, \mathbf{b})$, we find a train-track representative for $\HFhat(\boldsymbol{\Sigma}_T)$ with a filtration so that associated graded agrees with $\KH(T,\mathbf{a}, \mathbf{b})$. %In each case, given a parametrize tangle $(T,\mathbf{a}, \mathbf{b})$, we construct a train-track representative for the curve(s) $\HFhat(\boldsymbol{\Sigma}_T)$ (in the basis described above and using the crossover arrow formalism) that is filtered and admits an associated graded type D structure that agrees with $\KH(T,\mathbf{a}, \mathbf{b})$.

\subsubsection*{Solid tori and rational tangles} Rational tangles are precisely those with two-fold branched cover homeomorphic to the solid torus; see \cite{Lickorish1981, Montesinos1976}, for example. Moreover, there is a natural one-to-one correspondence between isotopy classes of essential simple closed curves and rational tangles (up to boundary-fixing isotopy).  The Khovanov homology of a rational tangle is described by Thompson \cite{Thompson2017}, showing that the continued fraction description of the rational tangle governs the combinatorics of the Bar Natan complex in a controlled way. Similarly, these combinatorics can be used to produce  filtered type D structures that are isomorphic (as unfiltered type D structures) to the relevant simple closed curve. An example is shown in Figure~\ref{fig:rational-KH}. 
\begin{figure}[ht]
\includegraphics[scale=0.8]{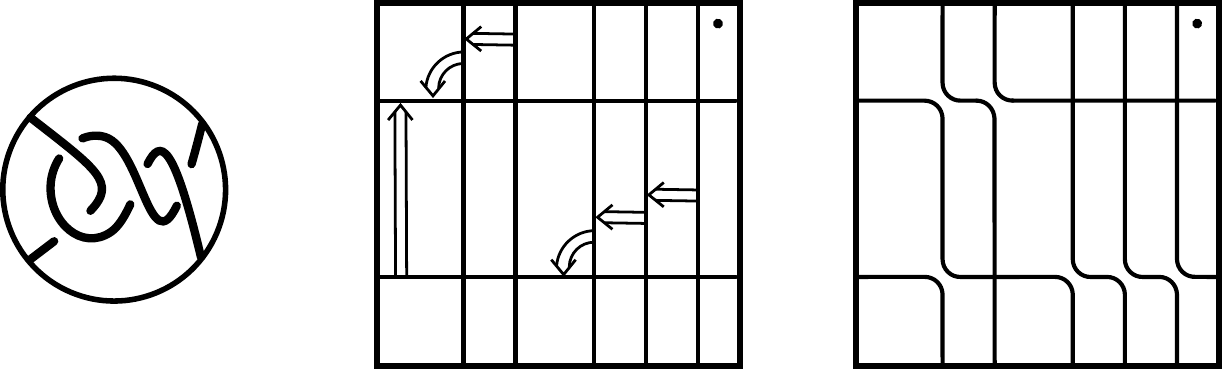}
\caption{The filtered and non-filtered type D structures associated with the two-fold branched cover of the 5/2 rational tangle. For the corresponding Bar-Natan complex, see \cite[Figure 4.2.4]{Thompson2017}, bearing in mind that our complex is the dual of the one shown there.}\label{fig:rational-KH} 
\end{figure}

Note that the (relative) filtration levels and the Bar-Natan complex are completely determined by the graphical complex: for the later, we simply ignore any tracks that correspond to traversing two (or more) crossover arrows. Conversely, Thompson shows that Bar-Natan complex of a rational tangle is a zig-zag ({\it i.e.} the underlying graph of the complex is linear). In this case, we can unambiguously reconstruct the graphical complex from the Bar-Natan complex. It may appear, {\it e.g.} that we have to choose the relative height of the two rightmost arrows in the figure, but the two positions are equivalent, since we can slide the rightmost arrow all the way around the torus.

\parpic[r]{
 \begin{minipage}{70mm}
 \centering
 \includegraphics[scale=0.8]{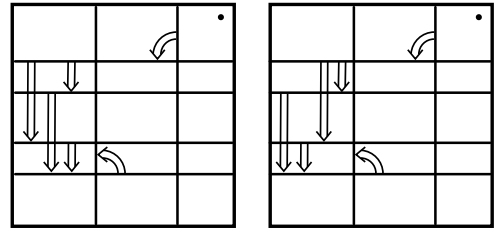}
\captionof{figure}{Distinct filtered type D structures with identical associated graded objects.}\label{fig:cautionary-remark}
  \end{minipage}%
  } {\em The twisted $I$-bundle over the Klein bottle.} 
%\subsubsection*{The twisted $I$-bundle over the Klein bottle} 
The situation is more subtle for non-rational tangles. A simple example is given by the $(-2,2)$-pretzel tangle; 
see Figure \ref{fig:twisted-I-KH}. Note that, owing to the presence of a solid torus-like component in the invariant, the algorithm described above will not give rise to an appropriate filtered type D structure. However, the Bar Natan complex for this tangle suggests a candidate. In fact, unlike in the case of rational tangles, there are two distinct candidates, as shown in Figure \ref{fig:cautionary-remark}. By appealing to the arrow calculus moves reviewed in Figure \ref{fig:cheat-sheet}, the reader can verify that these are distinct. 
Moreover, only one of these two candidates has an unfiltered invariant equal to \(\HFhat(\Sigma_T)\). \labellist 
\pinlabel {$\lambda$} at 35 93
\pinlabel {$\varphi$} at 76 135
\small
\pinlabel {$0$} at 90 64
\pinlabel {$0$} at 90 40
%\pinlabel {$1$} at 150 25
%\pinlabel {$2$} at 200 -8
\endlabellist
\begin{figure}[ht]
\includegraphics[scale=0.8]{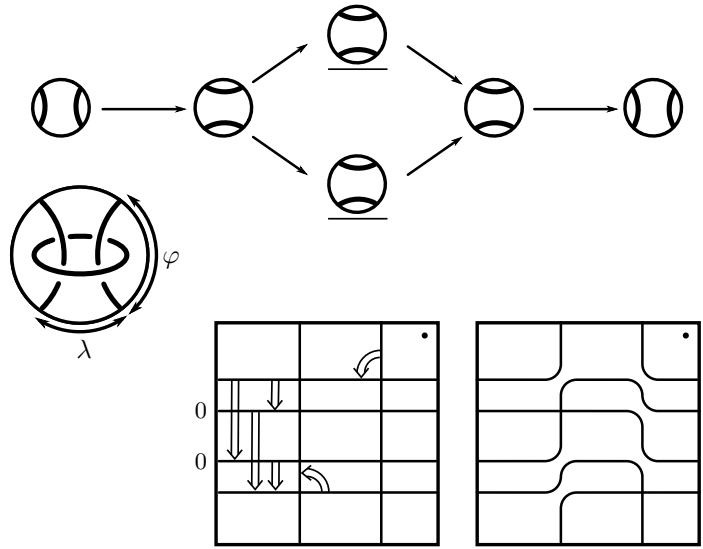}
\caption{The twisted $I$-bundle over the Klein bottle is the two-fold branched cover of the $(-2,2)$-pretzel tangle shown \cite{Montesinos1976}. The arcs in the boundary of the tangle are labelled according to the slopes they are covered by, namely, $\lambda$ is the rational longitude and $\varphi$ is the fiber slope relative to the $D^2(2,2)$ Seifert fibration in the cover. The Bar-Natan complex $\KH(T,\lambda,\varphi)$ is shown above and the filtered type D structure is shown below. After using the graphical calculus to simplify this type D structure, we arrive the collection of curves on the right. This is the invariant of the twisted I-bundle over the Klein bottle (compare with Figure~\ref{fig:twisted-I-with-NBHD}).}\label{fig:twisted-I-KH}
\end{figure}

\iffalse
\parpic[r]{
 \begin{minipage}{70mm}
 \centering
 \includegraphics[scale=0.8]{figures/cautionary-remark}
\captionof{figure}{Distinct filtered type D structures with identical associated graded objects.}\label{fig:cautionary-remark}
  \end{minipage}%
  }
This brings an interesting subtlety to light: unsurprisingly, it is possible to construct distinct filtered type D structures with identical associated graded objects. This is the case even for the simple tangle in question, as illustrated in Figure \ref{fig:cautionary-remark}. In both cases the associated graded is that of the $(-2,2)$-pretzel tangle, however removing all of the crossover arrows in each example yields on the one hand, the invariant of the twisted I-bundle over the Klein bottle (as in Figure \ref{fig:twisted-I-KH}) and, on the other, a type D structure that cannot be the invariant associated with a three-manifold. We leave the latter claim as an exercise for the interested reader (hint: compare the immersed curves obtained by arrow-removal with Figure \ref{Fig:SW_first}).
\fi 

\subsubsection*{The trefoil} As a final example, we revisit the right-hand trefoil exterior, which arises as the two-fold branched cover of the pretzel tangle shown in Figure \ref{fig:trefoil-KH}. Note that the Seifert structure on the knot exterior is reflected in the tangle \cite{Montesinos1976} (see also \cite{Watson2012}). In this case we see that different filtered type D structures, corresponding to chain homotopic Bar-Natan complexes, arise naturally. Again, this highlights the utility of the normal form provided by the structure theorem and, at the same time, the flexibility of the train track formalism for expressing filtered objects. 

\labellist 
\pinlabel {$\varphi$} at 95 155
\pinlabel {$\mu$} at 145 200
\small
\pinlabel {$0$} at 140 11
\endlabellist
\begin{figure}[ht]
\includegraphics[scale=0.8]{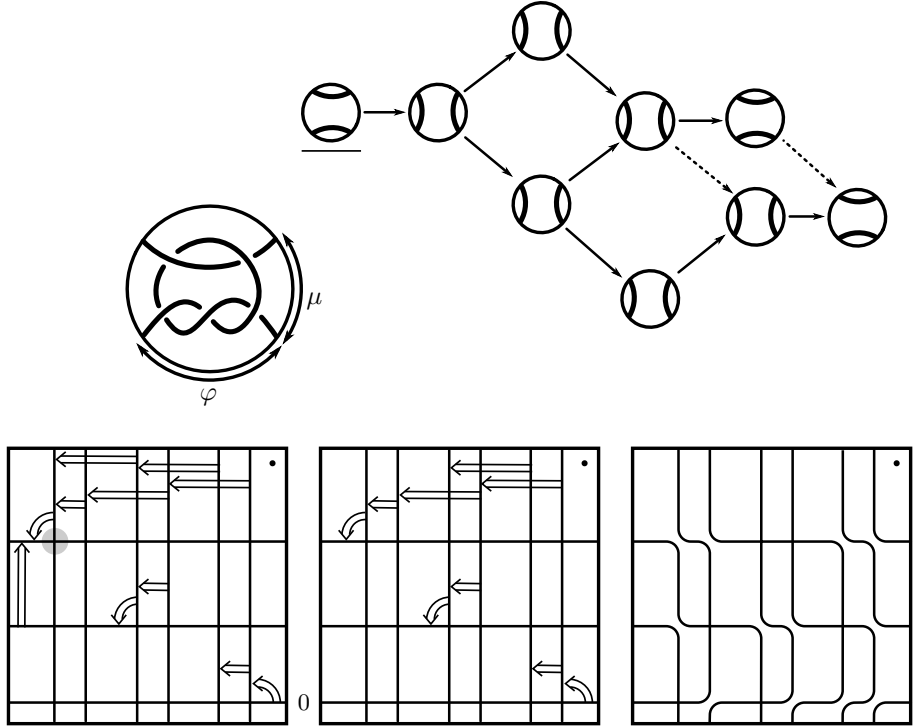}
\caption{The right-hand trefoil exterior $M$: in this example we have given two filtered type D structures -- to see they are isomorphic the shaded intersection needs to be resolved, the extra arrow removed, and finally the crossing replaced. The homotopy equivalent associated graded Bar-Natan complexes are shown above (either include or exclude the dashed differentials). Note that the trefoil exterior admits a Seifert fibration (as can be seen from the quotient tangle) and the fiber slope $\varphi$ agrees with $\lambda+6\mu$, where $(\mu,\lambda)$ are the preferred meridian-longitude pair. This agrees with the basis expressing the immersed curve shown at lower-right.}\label{fig:trefoil-KH}
\end{figure}

%% file: sections/minus.tex
% !TEX root = ../companion.tex
%minus.tex

We conclude with some speculations on how  \(HF^\pm(M_1\cup_h M_2)\)  is related to \(\HFhat(M_1)\) and \(\HFhat(M_2)\). 
In Section~\ref{subsec:CFK-}, the definition of the map \(d\) on  $C^-(\curves M,\mu)$ involved counting bigons in the whole torus, in particular those covering the basepoint. In a very similar way, we can use \(\HFhat(M_1)\) and \(\HFhat(M_2)\) to define complexes which are formally analogous to \(CF^{\pm}(M_1 \cup_h M_2)\). 

We will use coefficients in the ring of power series $\sT_- \cong\F[U]]$ or in $\sT_+\cong\F[U,U^{-1}]]/U\cdot\F[U]]$. Let $\gamma_0$ and $\gamma_1$ be two collections of  curves in $T = T^2 \setminus z$, and let $\CFminus(\gamma_0,\gamma_1)$ and $\CFplus(\gamma_0,\gamma_1)$ be generated over $\sT_-$ and $\sT_+$, respectively, by intersection points of $\gamma_0$ and $\gamma_1$. For points $x,y\in \gamma_0\cap\gamma_1$, let $N_i(x,y)$ be the mod 2 number of Whitney disks in $T^2$  connecting $x$ to $y$ and covering the basepoint $z$ with (positive) multiplicity $i$. (As usual, the oriented boundary of the disk should go from \(x\) to \(y\) along \(\gamma_0\).) We define
\[d(x)=\sum_{i=0}^\infty\sum_{y\in \gamma_0\cap\gamma_1} U^i N_i(x,y)\cdot y.\]
\begin{remark}
The characterization of the Maslov gradings in Section~\ref{sec:grad} implies that there is a well defined homological grading on \(CF^\pm(\gamma_0,\gamma_1)\) which is compatible with the homological grading on \(HF(\gamma_0,\gamma_1)\). 
\end{remark}

\noindent{\bf Warning:} For arbitrary curves $\gamma_0$ and $\gamma_1$, $d^2$ may not be zero. Problems arise when either curve has a cusp (i.e. a segment which bounds a disk in $T^2$); however curves may have two cancelling cusps, as in the curve for the figure eight knot complement, and $d^2 = 0$ for curves of this form. More formally, we can assign to each double point \(p\) of \(\gamma\) a quantity $$\mathfrak{o}_p = \sum_{\phi \in \pi_2(p,\gamma)} \# \mathcal{M} (\phi) U^{n_z(\phi)}$$ where the sum runs over homotopy classes of maps \(\phi:D^2 \to T^2\) such that \(\phi(1) = p\) and \(\phi(\partial D^2 )\subset \gamma\). 
We say \(\gamma\) is {\em good} if \(\mathfrak{o}_p = 0 \) for all double points of \(\gamma\). 

  If \(\gamma_0\) and \(\gamma_1\) are good, standard theory of Floer homology for immersed Lagrangians shows that 
 $d^2=0$ on $\CFminus(\gamma_0,\gamma_1)$ and $\CFplus(\gamma_0, \gamma_1)$. In this case, we define $\HFminus(\gamma_0, \gamma_1)$ and $\HFplus(\gamma_0, \gamma_1)$ to be the corresponding homologies.
Note that $\HFa(\gamma_0,\gamma_1)$ may be recovered from $HF^\pm(\gamma_0,\gamma_1)$ by constructing the mapping cone of the map induced by the $U$-action. 

Below, we compute $\HFplus(\gamma_0,\gamma_1)$ for a few examples. We will focus mainly on Dehn fillings of manifolds already studied earlier. That is, we consider $Y = M_0 \cup_h M_1$ with $M_0$ a loop type manifold and $M_1$ a solid torus. We set $\gamma_0 = \curves{M_0}$ and $\gamma_1 = h(\curves{M_1})$.

\subsection{Surgeries on knots} As a first calculation, let $K$ be the figure eight knot and consider the family of integer homology spheres obtained by $S^3_{1/n}(K)$ where $K$ is a non-negative integer; for negative integers, recall that $S^3_{-r}(K)\cong -S^3_r(K)$ since $K$ is amphicheiral. We compute $\HFplus(\gamma_0, \gamma_1)$, where $\gamma_0$ is the pair of immersed curves associated with the figure eight knot exterior and $\gamma_1$ is a simple closed curve of slope $\frac{1}{n}$.

\parpic[r]{
 \begin{minipage}{40mm}
 \centering
 \includegraphics[scale=0.4]{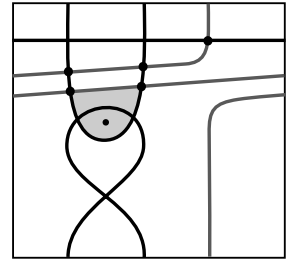}
\captionof{figure}{The intersection of curves associated with $\frac{1}{2}$-surgery on $K$.}\label{fig:fig-8-rev}
  \end{minipage}%
  }
The case $n=2$ is illustrated in Figure \ref{fig:fig-8-rev}. Notice that there are 5 intersection points, each generating a copy of $\sT_+$. However, in this case we have a non-trivial differential owing to the existence of two bigons covering the basepoint. As a result, $\HFplus(\gamma_0,\gamma_1)\cong \sT_+ \oplus \F^2$. More generally, one computes that the homology  for general $n$ is given by  $\sT_+ \oplus \F^n$, in agreement with $\HFplus(S^3_{1/n}(K))$. Notice that if the figure eight curve is a component of $\curves{M}$, then there is always a $\F^n$ summand in $\HFplus(\curves{M},\gamma_1)$ where $\gamma_1$ is a line of slope $\frac{1}{n}$ corresponding to $\frac{1}{n}$-surgery.

%\parpic[r]{
 %\begin{minipage}{40mm}
 %\centering
% \includegraphics[scale=0.4]{figures/surgery-on-trefoil}
%\captionof{figure}{The intersection of curves associated with $\frac{1}{2}$-surgery on $T_{2,3}$.}\label{fig:tref-rev}
 % \end{minipage}%
 % }
We can treat surgery on the right-hand trefoil knot $T_{2,3}$ in a similar manner, though here it is simpler to calculate by considering the lift of \(\HFhat(M)\) to the plane. In this case we calculate the $+$-version of the curves invariant to get $\sT_+$, as expected, for $+1$-surgery. In general, by inspecting the diagram in Figure \ref{fig:surgery-on-T23-T34} (which illustrates the case $n=2$) it is easy to see that  $\HFplus(\gamma_0,\gamma_1)\cong \sT_+ \oplus \F^{n-1}$.

\begin{figure}[ht]
\labellist
\endlabellist
\includegraphics[scale=0.5]{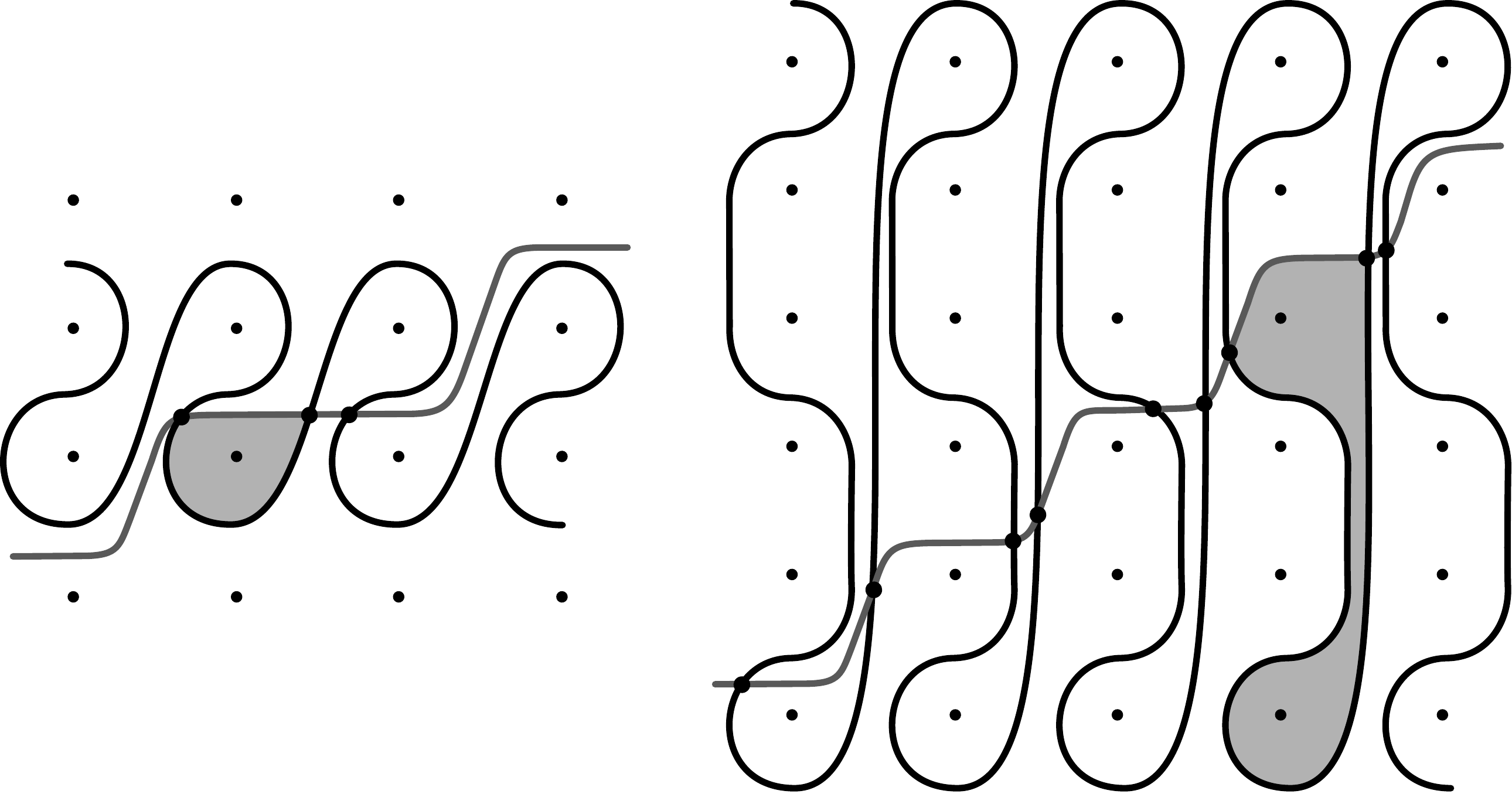}
\caption{Computing $+\frac{1}{2}$-surgery on the trefoil (left) and $+1$-surgery on the $(3,4)$-torus knot (right). }\label{fig:surgery-on-T23-T34}
\end{figure}

As a slightly more complicated example, we take \(M\) to be the complement of \(T(3,4)\); the complex \(\CFplus(\gamma_0, \gamma_1)\) is shown in Figure \ref{fig:T_34-complex}. 
The reader can easily check that \(\HFplus(\gamma_0,\gamma_1) = \sT_+ \oplus \F^4\); with a little more effort one can also check that the relative Maslov gradings are correct. (Compare {\it e.g.} with the results of Borodzik and N{\'e}methi \cite{BN2013}.)

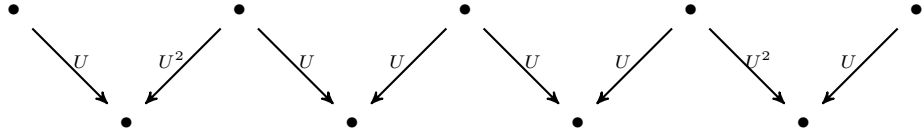
\begin{figure}[ht]
\begin{tikzpicture}[scale=1.5,>=stealth', thick] 
\node at (0,0) {$\bu$};
\node at (1,-1) {$\bu$};
\node at (2,0) {$\bu$};
\node at (3,-1) {$\bu$};
\node at (4,0) {$\bu$};
\node at (5,-1) {$\bu$};
\node at (6,0) {$\bu$};
\node at (7,-1) {$\bu$};
\node at (8,0) {$\bu$};
\draw[->,shorten <= 0.35cm, shorten >= 0.35cm] (0,0)--(1,-1)node [pos=0.6, above] {$\scriptstyle{U}$};
\draw[->,shorten <= 0.35cm, shorten >= 0.35cm] (2,0)--(1,-1)node [pos=0.6, above] {$\scriptstyle{U^2}$};
\draw[->,shorten <= 0.35cm, shorten >= 0.35cm] (2,0)--(3,-1)node [pos=0.6, above] {$\scriptstyle{U}$};
\draw[->,shorten <= 0.35cm, shorten >= 0.35cm] (4,0)--(3,-1)node [pos=0.6, above] {$\scriptstyle{U}$};
\draw[->,shorten <= 0.35cm, shorten >= 0.35cm] (4,0)--(5,-1)node [pos=0.6, above] {$\scriptstyle{U}$};
\draw[->,shorten <= 0.35cm, shorten >= 0.35cm] (6,0)--(5,-1)node [pos=0.6, above] {$\scriptstyle{U}$};
\draw[->,shorten <= 0.35cm, shorten >= 0.35cm] (6,0)--(7,-1)node [pos=0.6, above] {$\scriptstyle{U^2}$};
\draw[->,shorten <= 0.35cm, shorten >= 0.35cm] (8,0)--(7,-1)node [pos=0.6, above] {$\scriptstyle{U}$};
\end{tikzpicture}
\caption{\(\CFplus(\gamma_0,\gamma_1)\) for \(+1\) surgery on \(T(3,4)\)}\label{fig:T_34-complex}
\end{figure}

Next, we consider large integer surgeries on an arbitrary nullhomologous knot $K$ in an integer homology sphere $Y$. Let $M = Y\subset\nu(K)$; as before $\gamma_0 = \curves{M}$ and $\gamma_1$ is a simple closed curve of slope $n > 0$. Note that $\HFa(\gamma_0, \gamma_1)$ has $n$ spin$^c$ structures, which we index by integers $s$ with $|s|\le \frac{n}{2}$.

In analogy to the large integer surgery formula for Heegaard Floer homology, we relate $\mathit{HF}^\pm(\gamma_0, \gamma_1)$ to the complex $C^\infty(\curves M,\mu)$ (defined just as $C^-(\curves M,\mu)$ but with $\F[U,U^{-1}]]$ coefficients). For $s$ in $\Z$, let $A_s^-(\curves M,\mu)$ denote the subcomplex of $C^-(\curves M,\mu)$ obtained by restricting to Alexander grading less than or equal to $s$ and let $A_s^+(\curves M,\mu)$ be the corresponding quotient complex in $C^\infty(\curves M,\mu)$. We prove:

\begin{proposition}
Assume $n \ge 2g(K)+1$. For $|s|<n/2$, $\mathit{HF}^\pm(\gamma_0,\gamma_1; s) \cong H_*(A_s^\pm(\curves M,\mu))$.\end{proposition}
Note that in the case that $C^-(\curves M,\mu)$ agrees with $\CFKminus(Y, K)$, this implies that $\HFminus(\gamma_0, \gamma_1) \cong \HFminus(Y_n(K))$.

\parpic[r]{
\labellist
\scriptsize
\pinlabel {$2$} at 49 -5
\pinlabel {$1$} at 61 -5
\pinlabel {$0$} at 72 -5
\pinlabel {${\text{-}1}$} at 81 -5
\pinlabel {${\text{-}2}$} at 93 -5
\endlabellist
 \begin{minipage}{60mm}
 \centering
 \includegraphics[scale=0.75]{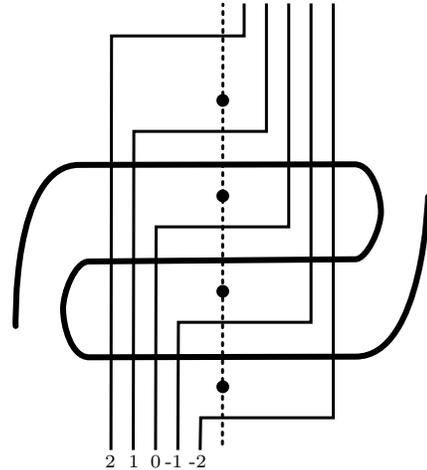}
\captionof{figure}{The intersection of curves, labelled by filtration level, associated with $5$-surgery on the right handed trefoil.}\label{fig:large-n-surgery}
  \end{minipage}%
  }
%\begin{proof}
{\em Proof.}
We prove the statement for the minus invariants, the proof for plus is similar. It is convenient to work with the lifts of $\gamma_0$ and $\gamma_1$ in $\R^2$, where the marked point lifts to points of the form $(a, b+\frac{1}{2})$ for $a,b\in\Z$. We work in a strip centered on the $y$-axis, which is a lift of the meridian $\mu$. Choose a homotopy representative for $\widetilde\gamma_0$ which lies between $y=-g(K)$ and $y=g(K)$ and which meets a neighborhood of the $y$-axis in horizontal segments. $\HFminus(\gamma_0, \gamma_1)$ has $n$ spin$^c$ structures, so we choose $n$ lifts of the line of slope $n$, crossing the $y$-axis at height $s$ with $|s|\le \frac{n}{2}$ and homotoped to lie in a neighborhood of the $y$-axis between $y=-g$ and $y=g$. See Figure \ref{fig:large-n-surgery} for the case of $+5$-surgery on the right-hand trefoil.

Consider the intersection homology with $\F[U]]$ coefficients of $\gamma_0$ with the $y$-axis, with basepoints $z$ and $w$ to the left and right, respectively, of each marked point; this is simply $C^-(\curves M,\mu)$. Note that each line of slope $n$ above is a slight perturbation of the vertical line in the relevant region; in particular, it has exactly the same intersection with $\gamma_0$. Clearly if $s \ge g$, the chain complex $\CFminus(\gamma_0, \gamma_1; s)$ is precisely the complex $C^-(M,\mu,w,z)$ where the marked point corresponds to $w$ and we ignore $z$ (that is, we forget the filtration). 

For the line corresponding to $s = g-1$, note that the intersections with $\gamma_0$ are unchanged, but one marked point has moved from the right of the left of the line. This has the effect that, for a generator $x$ with Alexander grading $g$, any bigon starting at $x$ covers an extra marked point, and any bigon connecting to $x$ covers one less marked point. To see the effect on $\HFminus$, we can replace the generator $x$ with $x' = Ux$ and keep the differential the same; indeed, if $d(x) = y$ then $d(x') = Uy$, and if $d(y)=Ux$ then $d(y)=x'$. More generally, considering the line corresponding to some integer $s$, we see that $\CFminus(\gamma_0,\gamma_1;s)$ is the same as $\CFminus(\gamma_0,\gamma_1;g)$ except that bigons from $x$ to $y$ cover one more marked point for each integer $A(y) < m \le A(x)$ with $m > s$ and one fewer marked point for each integer $A(x) < m \le A(y)$ with $m > s$. It follows that if $\{x_1, \ldots, x_k\}$ is a basis for $\CFminus(\gamma_0,\gamma_1;g) \cong C^-(\curves M,\mu)$ over $\F[U]]$, then $\{U^{\ell_1} x_1, \ldots, U^{\ell_k} x_k\}$ is a basis for $\CFminus(\gamma_0,\gamma_1;s)$ over $\F[U]]$, where $\ell_i = \text{max}(0, A(x_i)-s)$. But it is easy to see that this is a basis for $A^-(\curves M,\mu)$ as well with the same differential.\qed
%\end{proof}

%calculations.tex

\subsection{Surface bundles, revisited} Let $M_g$ be the product of a genus $g$ surface with a single connected boundary component with $S^1$. Let $\boldsymbol\gamma_g=\curves{M_g}$ be the associated invariant calculated in Section \ref{sub:bundles}. We will compute $\HFplus(\boldsymbol\gamma_g,L_0)$, where $L_0=h(\curves{D^2\times S^1})$ and $h\co \partial(D^2\times S^1)\to \partial M_g$ realises the filling giving rise to the product $Y_g=\Sigma_g\times S^1$. As in Section \ref{sub:bundles}, this is the result of intersection with a horizontal line; compare Figure \ref{fig:corollary-example}. 

\parpic[r]{
 \begin{minipage}{40mm}
 \centering
 \includegraphics[scale=0.4]{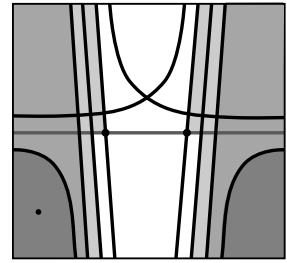}
\captionof{figure}{A bigon covering the basepoint in $(d_4d_{-4})$.}\label{fig:bundle-rev}
  \end{minipage}%
  }
The group $\HFplus(\boldsymbol\gamma_g,L_0)$ will have a contribution from each connected component of $\boldsymbol\gamma_g$. It is easy to calculate that the contribution of a $(d_0)$ or a $(d_0d_0)$ component is a summand of the form $\sT_+\oplus\sT_+$. The same is true of the contribution of a $(d_2d_{-2})$ component; in all three cases one observes that the map $\sT_+\to\sT_+$ is $2U=0$. More generally, the contribution of any component of the form $(d_{2k}d_{-2k})$, for $k>0$,  can be seen from Figure \ref{fig:bundle-rev}. Each component of this form gives rise to a summand in the chain complex isomorphic (as a group) to $\sT^{2(2k-1)}_+$ with differential described by $\bigoplus_{i=1}^{2k-1}(\sT_+\stackrel{D_i}{\to}\sT_+)$ for $D_i=U^i+U^{2k-i}$. On homology, for each $i< k$, this gives rise to summands isomorphic as modules to $\sT_i\cong H_*(\sT_+\stackrel{U^i}{\to}\sT_+)$. When $i=k$ we again get $\sT_+\oplus\sT_+$. For ease of comparison with the calculation of $\HFplus(Y_g)$ \cite{JM2008,OSz2004-knot} we decompose according to spin$^c$ structures. For the torsion spin$^c$ structure $\spin_0$ we have that \[\HFplus(\boldsymbol\gamma_g,L_0;\spin_0) \cong \HFa(\boldsymbol\gamma_g,L_0;\spin_0) \otimes \sT_+ \otimes \F\cong \HFhat(Y_g) \otimes \sT_+ \otimes \F\] in agreement with Jabuka and Mark \cite[Theorem 4.10]{JM2008} since $\dim\HFa(\boldsymbol\gamma_g,L_0;\spin_0)$ is given by twice the number of curve components of $\boldsymbol\gamma_g$, that is, $2^{g} + \sum_{i=0}^{2g}\binom{2g}{i} = 2^{g} + 2^{2g}$.

For non-torsion spin$^c$ structures $\spin_i$ we get non-trivial contributions for $0<i<k$ from each $(d_{2k}d_{-2k})$ when $k>0$. Namely, each $(d_{2k}d_{-2k})$ curve component gives rise to a $\sT_{k-i}$ summand in spin$^c$ structure $\spin_i$. Thus
\[\HFplus(\boldsymbol\gamma_g,L_0;\spin_i) \cong \bigoplus_{j=i+1}^g \sT_{j-i}^{\binom{2g}{g+j}}\] for each $0<i<g+1$. This gives  $\HFplus(\boldsymbol\gamma_g,L_0;\spin_i)\cong \HFplus(Y_g;\spin_i)$ comparing with Ozsv\'ath and Szab\'o \cite{OSz2004-knot}. 

%$\sT_{[k]}=\bigoplus_{i=1}^{2k-1}\sT_i$ where $\sT_i\cong H_*(\sT_+\stackrel{U^i}{\to}\sT_+)$. Collecting terms (with the number of terms given by Theorem \ref{thm:surface-bundle}) we have that $\HFa(\curves{M_g},h^{-1}(\gamma_1))^+$ is isomorphic to %\[\bigoplus_{k=1}^{g} \Big(\bigoplus_{i=1}^{2k-1}\sT_i\Big)^{\binom{2g}{g+k}}\oplus\sT_+^{2^{g-1}+\frac{1}{2}\binom{2g}{g}}\] \[\sT_{[1]}^{\binom{2g}{g+1}}\oplus\sT_{[2]}^{\binom{2g}{g+2}}\oplus\cdots\oplus\sT_{[g]}^{\binom{2g}{2g}}\oplus\sT_+^{2^{g-1}+\frac{1}{2}\binom{2g}{g}}\]

\subsection{Splicing trefoils, revisited} As a final example, let \(Y\) be  the result of splicing two right-hand trefoil complements; here by splice we mean the gluing which identifies the meridian of one knot complement with the Seifert longitude of the other. Figure \ref{fig:splice-of-trefoils} shows the intersection of immersed curves in the plane associated with this splice and the resulting chain group $\CFplus(\boldsymbol\gamma_0,\boldsymbol\gamma_1)$. There are 7 generators, corresponding to the 7 intersection points between the two immersed curves and it is straightforward to check that there are 8 bigons contributing to the differential. The resulting homology is $\HFplus(\boldsymbol\gamma_0,\boldsymbol\gamma_1)\cong\sT_+\oplus\F^3$.
Using the technique of \cite{RR}, we can express \(Y\) as Dehn surgery on a knot in the connected sum of two copies of the Poincar{\'e} sphere. By applying the mapping cone formula, one can check that the expression above agrees with  $\HFplus(Y)$. 

\begin{figure}[ht]
\includegraphics[scale=0.8]{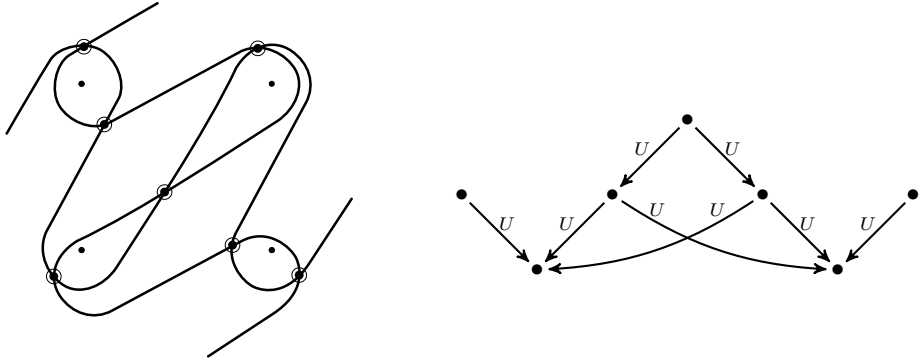} \hspace{1 cm}
\raisebox{1cm}{\begin{tikzpicture}[scale=1,>=stealth', thick] 
     \node at (0,3) {$\bu$};
     \node at (-3,2) {$\bu$}; \node at (-1,2) {$\bu$}; \node at (1,2) {$\bu$}; \node at (3,2) {$\bu$};
    \node at (-2,1) {$\bu$}; \node at (2,1) {$\bu$};
\draw[->,shorten <= 0.15cm, shorten >= 0.15cm] (0,3)--(-1,2) node [pos=0.6, above] {$\scriptstyle{U}$};
\draw[->,shorten <= 0.15cm, shorten >= 0.15cm] (0,3)--(1,2) node [pos=0.6, above] {$\scriptstyle{U}$};
\draw[->,shorten <= 0.15cm, shorten >= 0.15cm] (-3,2)--(-2,1) node [pos=0.6, above] {$\scriptstyle{U}$};
\draw[->,shorten <= 0.15cm, shorten >= 0.15cm] (-1,2)--(-2,1) node [pos=0.6, above] {$\scriptstyle{U}$};
\draw[->,shorten <= 0.15cm, shorten >= 0.15cm] (-1,2) to[bend right=15] (2,1); \node at (-0.4,1.8) {$\scriptstyle{U}$};
\draw[->,shorten <= 0.15cm, shorten >= 0.15cm] (1,2) to[bend left=15] (-2,1); \node at (0.4,1.8) {$\scriptstyle{U}$};
\draw[->,shorten <= 0.15cm, shorten >= 0.15cm] (1,2)--(2,1) node [pos=0.6, above] {$\scriptstyle{U}$};
\draw[->,shorten <= 0.15cm, shorten >= 0.15cm] (3,2)--(2,1) node [pos=0.6, above] {$\scriptstyle{U}$};
  %\draw[->,shorten <= 0.25cm, shorten >= 0.35cm] (0,-1.5)--(-3,-1.5) node [pos=0.2, above] {$\scriptstyle{U^2}$};
    %\draw[->,shorten <= 0.25cm, shorten >= 0.25cm] (0,-3)--(-1.5,-3) node [pos=0.4, above] {$\scriptstyle{U}$};
  %\draw[->,shorten <= 0.25cm, shorten >= 0.25cm] (0,0)--(0,-1.5);
   %\draw[->,shorten <= 0.25cm, shorten >= 0.25cm] (-1.5,0)--(-1.5,-3);
    %  \draw[->,shorten <= 0.25cm, shorten >= 0.25cm] (-3,0)--(-3,-1.5);
      % \draw[->,shorten <= 0.25cm, shorten >= 0.25cm] (-1.5,0)--(-3,-1.5);
        %      \draw[->,shorten <= 0.25cm, shorten >= 0.25cm] (0,-1.5)--(-1.5,-3);
\end{tikzpicture}}
\caption{Left: the curves $\curves{M_0}$ and $h(\curves{M_1})$, where $M_i$ is the right handed trefoil complement and $h$ is the splice identifying meridian to longitude. Right: the resulting chain group $\CFplus(\boldsymbol\gamma_0,\boldsymbol\gamma_1)$. }\label{fig:splice-of-trefoils}
\end{figure}

\subsection{A cautionary example} 
Next, we consider an example which we learned from Robert Lipshitz, where \(\HFplus(\HFhat(M_0),h(\HFhat(M_1))\) need not agree with \(\HFplus(M_0 \cup_h M_1)\). Let \(M_0=S^1\times D^2\), and let \(M_1 = S^1 \times D^2 \# Z\), where \(Z\) is any rational homology sphere. Then \(\CFD(M_1) = \CFD(S^1\times D^2) \otimes \HFhat(Z)\), so \(\HFhat(M_1)\) is  a disjoint union of parallel copies of \(\HFhat(S^1\times D^2)\), one for each generator of \(\HFhat(Z)\). Choose \(h\) so that \(M_0 \cup_h M_1= S^3 \# Z = Z\). Since \(\HFplus(Z)\) is not determined by \(\HFhat(Z)\),  \(\HFplus(\HFhat(M_0),h(\HFhat(M_1))\) need not be equal to  \(\HFplus(M_0 \cup_h M_1)\). Note that in this instance the invariant of \(\HFhat(M_1)\) consists of several parallel copies of the same curve. We expect that in a version of the theory which enabled us to calculate  \(\HFplus(M_0 \cup_h M_1)\), these curves would have to form a local system with nontrivial monodromy.

We end by posing the following question: 

\begin{question}
What conditions on  $M_0$ and $M_1$ guarantee that $$\HFplus(M_0\cup_h M_1) \simeq \HFplus(\curves{M_0}, h(\curves{M_1})?$$ \end{question}

At a minimum, we would conjecture that this is the case when \(M_0\) and \(M_1\) are Floer simple. But it is conceivable that it would be enough to require that \(\HFhat(M_0)\) and \(\HFhat(M_1)\) are both good (in the sense introduced at the beginning of this section) and contain no local system of multiplicity greater than one.